\documentclass[10pt, reqno]{amsart}
\usepackage{amsfonts}
\usepackage{mathrsfs}
\usepackage{amscd}
\usepackage{amsmath}
\usepackage{amssymb}
\usepackage{latexsym}
\usepackage{lscape}
\usepackage{xypic}
\usepackage{comment}
\usepackage{amscd}
\usepackage{wasysym}  
\usepackage{tikz} 
\usetikzlibrary{matrix,arrows}
\usepackage{tikz-cd}
\usepackage{appendix}
\usepackage{geometry}
\usepackage[nice]{nicefrac}
\usepackage{dsfont}
\usepackage{color}
\usepackage{graphicx}
\usepackage{todonotes}
\usepackage{theoremref}
\usepackage[english]{babel}
\usepackage{bbm}
\usepackage{theoremref}
\usepackage{leftidx}

\usepackage[hyperindex,linktocpage=true]{hyperref}
\hypersetup{
	colorlinks,
	linkcolor={red!50!black},
	citecolor={blue!50!black},
	urlcolor={blue!80!black}
}

\newtheorem{thm}{Theorem}[section]

\newtheorem{prop}[thm]{Proposition}
\newtheorem{lem}[thm]{Lemma}

\newtheorem{lem-def}[thm]{Lemma-Definition}
\newtheorem{cor}[thm]{Corollary}

\theoremstyle{definition}

\newtheorem{ex}[thm]{Example}

\newtheorem{rmk}[thm]{Remark}
\newtheorem{dfn}[thm]{Definition}
\newtheorem{constr}[thm]{Construction}
\newtheorem{nota}[thm]{Notation}

\newtheorem{ass}[thm]{Assumption}

\makeatletter
\newcommand*{\da@rightarrow}{\mathchar"0\hexnumber@\symAMSa 4B }
\newcommand*{\da@leftarrow}{\mathchar"0\hexnumber@\symAMSa 4C }
\newcommand*{\xdashrightarrow}[2][]{%
	\mathrel{%
		\mathpalette{\da@xarrow{#1}{#2}{}\da@rightarrow{\,}{}}{}%
	}%
}
\newcommand{\xdashleftarrow}[2][]{%
	\mathrel{%
		\mathpalette{\da@xarrow{#1}{#2}\da@leftarrow{}{}{\,}}{}%
	}%
}
\newcommand*{\da@xarrow}[7]{%
	\sbox0{$\ifx#7\scriptstyle\scriptscriptstyle\else\scriptstyle\fi#5#1#6\m@th$}%
	\sbox2{$\ifx#7\scriptstyle\scriptscriptstyle\else\scriptstyle\fi#5#2#6\m@th$}%
	\sbox4{$#7\dabar@\m@th$}%
	\dimen@=\wd0 %
	\ifdim\wd2 >\dimen@
	\dimen@=\wd2 %   
	\fi
	\count@=2 %
	\def\da@bars{\dabar@\dabar@}%
	\@whiledim\count@\wd4<\dimen@\do{%
		\advance\count@\@ne
		\expandafter\def\expandafter\da@bars\expandafter{%
			\da@bars
			\dabar@ 
		}%
	}%  
	\mathrel{#3}%
	\mathrel{%   
		\mathop{\da@bars}\limits
		\ifx\\#1\\%
		\else
		_{\copy0}%
		\fi
		\ifx\\#2\\%
		\else
		^{\copy2}%
		\fi
	}%   
	\mathrel{#4}%
}
\makeatother

\newcommand{\Bb}{\mathcal{B}}
\newcommand{\Cc}{\mathcal{C}}
\newcommand{\Dd}{\mathcal{D}}
\newcommand{\Ee}{\mathcal{E}}
\newcommand{\Ff}{\mathcal{F}}
\newcommand{\Gg}{\mathcal{G}}
\newcommand{\Hh}{\mathcal{H}}
\newcommand{\Ii}{\mathcal{I}}

\newcommand{\Ll}{\mathscr{L}}

\newcommand{\Oo}{\mathcal{O}}
\newcommand{\Pp}{\mathcal{P}}

\newcommand{\Ss}{\mathcal{S}}
\newcommand{\Tt}{\mathcal{T}}

\newcommand{\IA}{\mathbb{A}}

\newcommand{\ID}{\mathbb{D}}
\newcommand{\IE}{\mathbb{E}}
\newcommand{\IF}{\mathbb{F}}
\newcommand{\IG}{\mathbb{G}}

\newcommand{\II}{\mathbb{I}}

\newcommand{\IL}{\mathbb{L}}
\newcommand{\IM}{\mathbb{M}}
\newcommand{\IN}{\mathbb{N}}

\newcommand{\IQ}{\mathbb{Q}}
\newcommand{\IR}{\mathbb{R}}
\newcommand{\IS}{\mathbb{S}}

\newcommand{\IV}{\mathbb{V}}

\newcommand{\IX}{\mathbb{X}}

\newcommand{\IZ}{\mathbb{Z}}

\newcommand{\BB}{\mathfrak{B}}

\newcommand{\RR}{\mathfrak{R}}

\newcommand{\ff}{\mathbf{f}}
\let\gge\gg
\renewcommand{\gg}{\mathfrak{g}}

\newcommand{\pp}{\mathfrak{p}}

\newcommand{\bfa}{\mathbf{a}}
\newcommand{\bfb}{\mathbf{b}}
\newcommand{\bfc}{\mathbf{c}}

\newcommand{\bfe}{\mathbf{e}}

\newcommand{\bfJ}{\mathbf{J}}

\numberwithin{equation}{section}

\DeclareMathOperator{\End}{End} % Endomorphism ring
\DeclareMathOperator{\Gr}{Gr} % Affine Grassmannian
\DeclareMathOperator{\Fl}{Fl} % Affine flag variety
\DeclareMathOperator{\Spec}{Spec} % Spectrum
\DeclareMathOperator{\Spf}{Spf} % Formal spectrum
\DeclareMathOperator{\Spd}{Spd} % Diamond spectrum
\newcommand{\can}{{\operatorname{can}}}% canonical deperfections
\newcommand{\perf}{{\operatorname{perf}}} % perfection
\DeclareMathOperator{\identity}{id} % Identity map
\DeclareMathOperator{\DM}{DM} % Motives
\DeclareMathOperator{\DTM}{DTM} % Tate motives
\DeclareMathOperator{\DATM}{DATM} % Artin-Tate motives
\DeclareMathOperator{\MTM}{MTM} % Mixed Tate motives
\DeclareMathOperator{\MATM}{MATM} % Mixed Artin-Tate motives
\DeclareMathOperator{\Sch}{Sch} % Category of schemes
\DeclareMathOperator{\op}{op} % Opposite category
\DeclareMathOperator{\St}{St} % Stable
\renewcommand{\Pr}{\operatorname{Pr}} % Presentable infinity category
\DeclareMathOperator{\et}{\acute{e}t} % étale motives
\DeclareMathOperator{\PreStk}{PreStk} % Prestacks
\DeclareMathOperator{\Ani}{An} % Anima
\DeclareMathOperator{\Fun}{Fun} % Functors
\DeclareMathOperator{\Gal}{Gal} % Galois group
\DeclareMathOperator{\pr}{pr} % Projection map
\DeclareMathOperator{\CT}{CT} % Constant term
\DeclareMathOperator{\Cent}{Cent} % Centralizer
\DeclareMathOperator{\Norm}{Norm} % Normalizer
\DeclareMathOperator{\der}{der} % Derived subgroup
\DeclareMathOperator{\im}{im} % Image
\DeclareMathOperator{\adj}{ad} % Adjoint
\DeclareMathOperator{\sico}{sc} % Simply connected
\DeclareMathOperator{\Aut}{Aut} % Automorphism group
\DeclareMathOperator{\Rep}{Rep} % Representation category
\DeclareMathOperator{\IC}{IC} % Intersection complex
\DeclareMathOperator{\Corr}{Corr} % Correspondences
\DeclareMathOperator{\id}{id} % Identity morphism
\DeclareMathOperator{\Hck}{Hck} % Hecke stack
\DeclareMathOperator{\loc}{loc} % Local
\DeclareMathOperator{\Sht}{Sht} % Shtuka space
\DeclareMathOperator{\Sh}{Sh} % Shimura variety
\DeclareMathOperator{\unr}{unr} % Unramified
\DeclareMathOperator{\vsp}{vsp} % Very special
\DeclareMathOperator{\Res}{Res} % Restriction
\DeclareMathOperator{\fd}{fd} % Finite dimensional
\DeclareMathOperator{\Perv}{Perv} % Perverse sheaves
\DeclareMathOperator{\Hom}{Hom} % Hom-groups
\DeclareMathOperator{\Irr}{Irr} % Irreducible components
\DeclareMathOperator{\tr}{tr} % Trace
\DeclareMathOperator{\cl}{cl} % Classical
\DeclareMathOperator{\Ad}{Ad} % Adjoint action
\DeclareMathOperator{\tors}{tors} % Torsion
\DeclareMathOperator{\defect}{def} % Defect
\DeclareMathOperator{\Coh}{Coh} % Coherent sheaves
\DeclareMathOperator{\Comp}{Comp} % Composition
\DeclareMathOperator{\naive}{naive} % Naive integral/local models
\DeclareMathOperator{\spl}{spl} % Splitting models
\DeclareMathOperator{\Lie}{Lie} % Lie algebra
\DeclareMathOperator{\Pair}{Pair} % Pairs
\DeclareMathOperator{\Disp}{Disp} % Displays
\DeclareMathOperator{\bas}{bas} % Basic Newton stratum
\DeclareMathOperator{\Top}{top} % Top-dimensional irreducible components
\DeclareMathOperator{\Sat}{Sat} % Satake
\DeclareMathOperator{\BM}{BM} % Borel--Moore homology
\DeclareMathOperator{\JL}{JL} % Jacquet--Langlands
\DeclareMathOperator{\Perf}{Perf} % Perfect schemes
\DeclareMathOperator{\AffSch}{AffSch} % Affine schemes
\DeclareMathOperator{\BD}{BD} % Beilinson-Drinfeld
\DeclareMathOperator{\Grp}{Grp} % Category of groups
\DeclareMathOperator{\Set}{Set} % Category of sets
\DeclareMathOperator{\Nm}{Nm} % Norm
\DeclareMathOperator{\Adm}{Adm} % Admissible locus
\DeclareMathOperator{\pfp}{pfp} % Perfectly of finite presentation
\DeclareMathOperator{\AlgStk}{AlgStk} % Algebraic stack
\DeclareMathOperator{\Repr}{Repr} % Representable
\DeclareMathOperator{\Cat}{Cat} % Category
\DeclareMathOperator{\CH}{CH} % Chow groups

\newcommand{\into}{\hookrightarrow}
\newcommand{\onto}{\twoheadrightarrow}
\newcommand{\pH}{{}^{\mathrm{p}}\mathrm{H}}

\newcommand{\IHom}{\underline{\operatorname{Hom}}} % Internal Hom
\newcommand{\unit}{\mathbbm{1}} % Unit object (as not yet sure about which coefficients to use)
\newcommand{\QVect}{\IQ\text{-}\operatorname{Vect}} % Q-vector spaces
\newcommand{\grQVect}{\operatorname{gr}\text{-}\IQ\text{-}\operatorname{Vect}} % Graded vector spaces
\newcommand{\Mod}{\text{-}\operatorname{Mod}} % Modules
\newcommand{\GL}{\mathrm{GL}} % General linear group
\newcommand{\SL}{\mathrm{SL}} % Special linear group
\newcommand{\PGL}{\mathrm{PGL}} % Projective linear group
\newcommand{\PU}{\mathrm{PU}} % Projective unitary group
\newcommand{\GSp}{\mathrm{GSp}} % Group of symplectic similitudes
\renewcommand{\phi}{\varphi}
\newcommand{\Tatep}{\mathrm{Tate}_p}

\newcommand{\comp}{\mathrm{c}} % Compact objects
\newcommand{\cons}{\mathrm{c}} % Constructible objects

\newcommand{\Gmot}{\widehat{G}_{\operatorname{mot}}}
\newcommand{\HBM}{\mathrm{H}^{\BM}} % Borel--Moore homology
\newcommand{\GIT}{/\!\!/} % GIT quotient

\begin{document}

\title[Cycles on splitting models]{Cycles on splitting models of Shimura varieties}
\author[Thibaud van den Hove]{Thibaud van den Hove}

\address{Max Planck Insitut für Mathematik, Vivatsgasse 7, 53111 Bonn, Germany}
\email{vandenhove@mpim-bonn.mpg.de}

%\subjclass[2020]{Primary 14G35; Secondary 11S37, 14C25, 14F42}

\begin{abstract}
	We construct exotic Hecke correspondences between the special fibers of different PEL type Shimura varieties, when the local groups are restrictions of scalars of unramified groups.
	In particular, the local groups themselves are not necessarily unramified, and the Shimura varieties can have bad reduction.
	By adapting the methods of Xiao--Zhu in the case of good reduction, we use this to construct new instances of the geometric Jacquet--Langlands correspondence, including a motivic refinement, and verify generic instances of the Tate conjecture for the special fibers of these Shimura varieties at very special level.
	Our main tool is to resolve the integral models by the splitting models of Pappas--Rapoport.
	We also define splitting versions of the moduli stacks of local shtukas and affine Deligne--Lusztig varieties, and study their geometry.
\end{abstract}

\maketitle

\setcounter{tocdepth}{1}
\tableofcontents
\setcounter{section}{0}

%\pagebreak

\thispagestyle{empty}

\section{Introduction}

Shimura varieties form a very interesting class of algebraic varieties.
Their cohomology has been used with great succes to realize instances of the Langlands correspondence, and thanks to their arithmetic properties, many conjectures concerning general algebraic varieties are more tractable for Shimura varieties.
Examples include Langlands' conjecture that all motivic L-functions are automorphic \cite{Langlands:Lfunctions}, and the Tate conjecture (cf.~\cite{HarderLanglandsRapoport:Algebraische} for an example on the generic fiber).
In \cite{XiaoZhu:Cycles}, Xiao--Zhu have studied certain cycles on the special fibers of Shimura varieties, and used them to show further instances of the Langlands correspondence and the Tate conjecture, generalizing certain constructions from \cite{TianXiao:Tate,HelmTianXiao:Tate}.
The Shimura varieties considered in \cite{XiaoZhu:Cycles} have good reduction, i.e., have hyperspecial level at a fixed prime \(p\).
The main goal of the present paper is to generalize the methods of \cite{XiaoZhu:Cycles}, by considering Shimura varieties with certain bad reduction, and to prove instances of the Jacquet--Langlands correspondence and the Tate conjecture for such Shimura varieties.
Another goal is to upgrade this construction to the motivic level, to provide further evidence for the conjecture that the Langlands correspondence is of motivic origin.
In order to simplify the exposition, we explain these two parts seperately in the introduction.
In particular, throughout this introduction, we will state various results involving \(\ell\)-adic étale sheaves for some \(\ell\neq p\), which we will refine to the level of étale motives in the main text.
(The proofs in the \(\ell\)-adic case can be shown verbatim, and are significantly easier than the motivic versions.)

\subsection{The main theorems}

Fix a prime \(p\), which we will assume to be odd throughout the introduction.
For hyperspecial level at \(p\), Shimura varieties are expected to have (and this is known in almost all cases) canonical smooth integral models, over the ring of integers of some completion of the reflex field.
Such a hyperspecial level exists exactly when the local group at \(p\) appearing in the Shimura datum is unramified, i.e., is quasi-split and splits over an unramified extension.
We will be interested in the case where the local group is still quasi-split, but does not necessarily split over an unramified extension.
In that case, there will still be a very special level (in the sense of \cite[Definition 6.1]{Zhu:Ramified}), generalizing the notion of hyperspecial level.

The paper \cite{XiaoZhu:Cycles} contains both global and local arguments.
On the local side, the main tools are the geometric Satake equivalence, relating the category of equivariant perverse sheaves on the affine Grassmannian with representations of the Langlands dual group, the construction of cohomological correspondences on the moduli stack of local shtukas, and the study of irreducible components of certain affine Deligne--Lusztig varieties.
The geometric objects appearing all depend on a (parahoric) level.
Moreover, at very special level, there is a version of the geometric Satake equivalence \cite{Zhu:Ramified,Richarz:Affine}, and the other ingredients also generalize, to a certain degree, to very special level.

However, the global side is more problematic.
There, the crucial input in \cite{XiaoZhu:Cycles} is the construction of so-called exotic Hecke correspondences between the special fibers of the canonical models of different Shimura varieties (which usually do not lift to the generic fiber), and the fact that these special fibers are smooth.
Of course, the special fibers of Shimura varieties at very special level are rarely smooth.
Moreover, in the PEL case, the construction of exotic Hecke correspondences uses the moduli interpretation of the canonical integral models, which is generally not available when the local group is ramified \cite{Pappas:ArithmeticModuli}.
(We note that more general constructions of exotic Hecke correspondences have been announced by Xiao--Zhu and van Hoften--Sempliner, but still under the assumption that the local group splits over an unramified extension of \(\IQ_p\).)

We solve both these problems simultaneously, by considering splitting integral models of Shimura varieties, as introduced in \cite{PappasRapoport:LocalII.Splitting}.
Indeed, we will assume that the the ramification of the local group arises mostly via restriction of scalars (i.e., that it is essentially unramified, in the sense of \thref{definition essentially unramified}).
In that case, the splitting models of PEL type Shimura varieties do admit natural moduli interpretations, and they are smooth at very special level.
In fact, they provide natural resolutions of singularities of the canonical models.
Throughout the introduction, we will denote by \(\Sh_{\mu,K}^{\spl}\) the perfection of the geometric special fiber of the splitting model of a Shimura variety, at level \(K=K_pK^p\), with \(K_p=\Gg(\IZ_p)\) the \(\IZ_p\)-valued points of a parahoric model \(\Gg\).
We note that various theorems stated in the introduction can be extended beyond the essentially unramified case.
For the sake of exposition, we will keep this assumption throughout the introduction, and refer to the main text for generalizations.

By using splitting models, we obtain the following two main theorems, confirming instances of the Tate conjecture, as well as \cite[Conjecture 4.60]{Zhu:Coherent}.
We refer to \thref{main global theorem} and \thref{coro JL} for more details and the unexplained notation.

\begin{thm}\thlabel{intro--Tate}
	Let \((\IG,\IX)\) be a PEL type Shimura datum with Hodge cocharacter \(\mu\), assume that \(G:=\IG_{\IQ_p}\) is essentially unramified with connected center, and that \(\Gg\) is very special parahoric.
	Assume that the basic element in \(B(G,\mu^*)\) is very special in the sense of \thref{very special Kottwitz elements}.
	Then:
	\begin{enumerate}
		\item The basic Newton stratum \(\Sh_{\mu,K,\bas}^{\spl}\) is pure of dimension \(\langle \rho,\mu\rangle\), and its Borel--Moore homology is isomorphic to
		\[\mathrm{H}^{\BM}_{\langle2 \rho,\mu\rangle}(\Sh_{\mu,K,\bas}^{\spl}) \cong C(\IG'(\IQ)\backslash \IG'(\IA_f)/K,\overline{\IQ}_\ell) \otimes_{\overline{\IQ}_\ell} V_{\mu^*}^{\Tatep}.\]
		Here, \(V_{\mu^*}^{\Tatep} \subseteq V_{\mu^*}\) is a nontrivial subspace of the irreducible \(\widehat{G}\)-representation of highest weight \(\mu^*\), and \(\IG'\) is an inner form of \(\IG\), which is trivial at the finite places, and compact modulo center at the archimedean place.
		
		\item Let \(\pi_f=\pi_{f,p}\pi_f^p\) be an irreducible \(\Hh_{K,\overline{\IQ}_\ell}\)-module.
		If the Satake parameter of the \(p\)-component \(\pi_{f,p}\) is \(V_{\mu^*}\)-general in the sense of \thref{defi general}, the restriction of the cycle class map to the \(\pi_f\)-isotypical component
		\[\mathrm{H}^{\BM}_{\langle2 \rho,\mu\rangle}(\Sh_{\mu,K,\bas}^{\spl})[\pi_f] \subseteq \mathrm{H}^{\BM}_{\langle2 \rho,\mu\rangle}(\Sh_{\mu,K,\bas}^{\spl}) \xrightarrow{\cl} \mathrm{H}^{\langle 2\rho,\mu\rangle}_{\comp}(\Sh_{\mu,K}^{\spl},\overline{\IQ}_\ell(\langle \rho,\mu\rangle))\]
		is injective.
		\item\label{surjectivity} Assume that \((\IG,\IX)\) gives rise to a compact unitary group Shimura variety, and that \(p\) is a split place, as in \cite{ScholzeShin:Cohomology}.
		Let \(\pi_f^p\) be an irreducible \(\Hh_{K^p,\overline{\IQ}_\ell}\)-module.
		Then the \(\pi_f^p\)-isotypical copmonent of the cycle class map surjects onto
		\[\sum_{\pi_p} T^{\langle \rho,\mu\rangle}(\pi_p \pi_f^p,\overline{\IQ}_\ell) \otimes \pi_p\pi_f^p\]
		if the Satake parameters of the \(\pi_p\) are all strongly \(V_{\mu^*}\)-general (in the sense of \thref{defi general}), where the \(\pi_p\) range over those irreducible representations of the local Hecke algebra \(\Hh_{\Gg}\) such that \(\pi_p\pi_f^p\) appears in \(C(\IG'(\IQ)\backslash \IG'(\IA_f)/K,\overline{\IQ}_\ell)\).
	\end{enumerate}
\end{thm}

Here, \(T^{\langle \rho,\mu\rangle}(\pi_p\pi_f^p,\overline{\IQ}_\ell)\) denotes the space of Tate classes in the \(\pi_p\pi_f^p\)-isotypic part of the degree \(\langle \rho,\mu\rangle\) cohomology in \(\Sh_{\mu,K}^{\spl}\); cf.~§\ref{subsec:Tate}.
In particular, the Tate conjecture holds for these \(\pi_f^p\).

\begin{rmk}
	Let us briefly comment on the assumptions appearing in this theorem.
	\begin{enumerate}
		\item The restriction in \eqref{surjectivity} to the Shimura varieties considered in \cite{ScholzeShin:Cohomology} appears because the required input on the cohomology of Shimura varieties does not seem to have appeared in other (ramified) cases.
		In particular, we only consider non-endoscopic situations.
		It is likely that using the methods of this paper, endoscopic situations will also be manageable as in \cite[§2.3]{XiaoZhu:Cycles}, once the stable trace formula is known for Shimura varieties at very special level.
		This is expected to appear in a forthcoming work of Haines--Zhou--Zhu.
		\item The assumption that \(\IG_{\IQ_p}\) is essentially unramified is crucial in order to have a resolution of singularities.
		Indeed, when \(\IG_{\IQ_p}\) is a ramified unitary group, it is not expected that the canonical integral model of the Shimura variety admits a smooth resolution.
		(The only exception is when \(\IG_{\IQ_p}\) is an odd ramified unitary group, and the very special parahoric \(\Gg\) is not absolutely special, where the canonical model is already smooth \cite{HainesRicharz:Smoothness}.)
		\item In contrast to the other statements appearing in this introduction, we only consider the Tate conjecture on the level of \(\ell\)-adic cohomology, since the conjecture does not make sense motivically.
		There is a \(\IQ_p\)-adic version of the Tate conjecture, using crystalline cohomology, and our methods should yield similar results as \thref{intro--Tate} for this version.
		Since, over finite fields, the Tate conjecture is independent of \(\ell\) in a certain sense (cf.~\cite[Theorem 2.9]{Tate:Conjectures} for a precise statement, the argument also works for \(\ell=p\) when using crystalline cohomology), we omit details.
	\end{enumerate}
\end{rmk}

\begin{thm}\thlabel{intro--JL}
	Let \((B,*,\Oo_B,V_i,(,),\Ll_i,h_i)\) be (integral at \(p\)) PEL data for \(i\in \{1,2\}\), yielding Shimura data \((\IG_i,\IX_i)\) with Hodge cocharacters \(\mu_i\).
	Fix an isomorphism \(V_1\otimes \IA_f\cong V_2\otimes \IA_f\) compatibly with the PEL data, which induces an isomorphism \(\IG_{1,\IA_f}\cong \IG_{2,\IA_f}\).
	Assume \(\IG_{1,\IQ_p}\) is essentially unramified, fix a very special parahoric model \(\Gg/\IZ_p\) and a level \(K=K^p\Gg(\IZ_p)\subseteq \IG_{1,\IA_f}\).
	Then there is a natural map
	\[\Hom_{\Coh(\widehat{G}^I\sigma/\widehat{G}^I)}\left(\widetilde{V_{\mu_1}},\widetilde{V_{\mu_2}}\right) \to \Hom_{\Hh_{K^p}}\left(\mathrm{H}_c^{\langle 2\rho,\mu_1\rangle}(\Sh_{\mu_1,K}^{\spl},\overline{\IQ}_\ell(\langle \rho,\mu_1\rangle)),\mathrm{H}_c^{\langle 2\rho,\mu_2\rangle}(\Sh_{\mu_2,K}^{\spl},\overline{\IQ}_\ell(\langle \rho,\mu_2\rangle))\right)\]
	compatible with composition, where \(I\) denotes the inertia group of \(\IQ_p\).
\end{thm}

Even though these two theorems are stated for splitting models of Shimura varieties, they directly imply similar statements for the canonical models.
Indeed, since the splitting models resolve the canonical models, their cohomology agrees with the cohomology of the nearby cycles of the canonical models.
Thus, the Jacquet--Langlands transfer maps indeed yield (an underived version of) the map required in \cite[Conjecture 4.60]{Zhu:Coherent}.
Moreover, although the Tate conjecture is usually stated for smooth projective varieties, there also exists a version for singular varieties \cite[Conjecture 7.3]{Jannsen:Mixed}.
And if the Tate conjecture holds for a resolution of singularities of such a singular variety, then the Tate conjecture already holds for the singular variety \cite[Theorem 7.10]{Jannsen:Mixed}.
Thus, the theorems above are relevant, even if one is primarily interested in canonical models of Shimura varieties.

Although working with splitting models fixes the global problems alluded to above, another issue arises on the local side: the connection between the special fibers of Shimura varieties with affine Grassmannians, affine Deligne--Lusztig varieties, and moduli of local shtukas, is specific to the canonical models.
To resolve these issues, we will define splitting versions of affine Grassmannians, affine Deligne--Lusztig varieties, and moduli of local shtukas, relate these to the special fibers of the splitting models of Shimura varieties, and study their geometry.

\subsection{Splitting affine Grassmannians and moduli of local shtukas}

Recall that the special fibers of Shimura varieties at parahoric level admit local models, with the same étale-local singularities, given by a union of Schubert varieties in a partial affine flag variety.
The idea of the splitting model from \cite{PappasRapoport:LocalII.Splitting} is to resolve this by a convolution product of (unions of) smaller Schubert varieties.
To place certain results below in a more natural context, it make sense to define an unbounded version, called the \emph{splitting affine Grassmannian}.

Let us now move to the local setting, and fix a nonarchimedean local field \(F\), as well as an essentially unramified reductive group \(G\) over \(F\), with very special parahoric model \(\Gg\).
If \(G= \Res_{F'/F} H\) for some totally ramified extension \(F'/F\), unramified \(F'\)-group \(H\), and reductive \(\Oo_{F'}\)-model \(\Hh\), we define the splitting affine Grassmannian \(\Gr_{\Gg}^{\spl}\) as the \([F':F]\)-fold convolution product of \(\Gr_{\Hh}=LH/L^+\Hh \cong \Gr_{\Gg}\).
In general, we have to be careful with the connected components, and we refer to \thref{definition splitting grassmannians} for a precise definition.
There is a natural map to the usual affine Grassmannian \(\Gr_{\Gg}\), which we will call the canonical affine Grassmannian to distinguish the two, and write \(\Gr_{\Gg}^{\can}\).
Moreover, \(\Gr_{\Gg}^{\spl}\) admits a natural \(L^+\Gg\)-action, and the corresponding quotient can be viewed as a splitting Hecke stack.

The (canonical) affine Grassmannian is of fundamental importance in the local Langlands correspondence, thanks to the geometric Satake equivalence.
This is an equivalence between \(L^+\Gg\)-equivariant perverse sheaves on \(\Gr_{\Gg}^{\can}\), and representations of the Langlands dual group \(\widehat{G}\) (or, in the ramified case,  the inertia-invariants \(\widehat{G}^I\) of the Langlands dual group).
Using the description of the splitting Grassmannian as a convolution product, we can similarly describe equivariant perverse sheaves on it, cf.~\thref{Satake for splitting models}.

\begin{prop}\thlabel{intro-theorem satake}
	There is a natural commutative diagram
	\[\begin{tikzcd}
		\Perv_{L^+\Gg}(\Gr_{\Gg}^{\spl}) \arrow[d] \arrow[r, "\cong"] & \Rep_{\overline{\IQ}_\ell}^{\fd}(\widehat{G}) \arrow[d]\\
		\Perv_{L^+\Gg}(\Gr_{\Gg}^{\can}) \arrow[r, "\cong"'] & \Rep_{\overline{\IQ}_\ell}^{\fd}(\widehat{G}^I),
	\end{tikzcd}\]
	where the vertical arrows are given by pushforward, and restriction of representations respectively.
\end{prop}

We emphasize that although the bottom and right arrows are monoidal for the natural convolution product on \(\Perv_{L^+\Gg}(\Gr_{\Gg}^{\can})\), the category \(\Perv_{L^+\Gg}(\Gr_{\Gg}^{\spl})\) does not admit a natural monoidal structure.
Instead, the proposition should be seen as a geometric way to produce representations of \(\widehat{G}^I\), which are restrictions of irreducible representations of \(\widehat{G}\).
Indeed, although the simple representations of \(\widehat{G}^I\) correspond to the intersection complexes on certain Schubert varieties in \(\Gr_{\Gg}^{\can}\), it is more complicated to directly construct the perverse sheaf corresponding to the restrictions of \(\widehat{G}\)-representations.
(Another way to do this would be to use nearby cycles, as in \cite{Zhu:Ramified}.
But this would leave the realm of perfect geometry over \(\overline{\IF}_q\), and is hence not suited for our current purposes.)
Of particular interest to us is the \(\widehat{G}\)-representation whose highest weight is (the dual of) the Hodge cocharacter associated with a Shimura datum, as well as its restriction to \(\widehat{G}^I\).

Now that we have a splitting version of the affine Grassmannian and local Hecke stack, we can define splitting versions of any object which admits a map to the (canonical) local Hecke stack.
This includes the moduli stack of local shtukas, which we write \(\Sht_{\Gg}^{\spl}\to \Sht_{\Gg}^{\can} = LG/\Ad_{\sigma^{-1}} L^+\Gg\).
Here, \(\sigma\) denotes the arithmetic Frobenius of \(F\).
In particular, both objects admit a map to the Kottwitz stack \(LG/\Ad_{\sigma^{-1}} LG\).
Taking the self-fiber products of \(\Sht_{\Gg}^{\spl}\) and \(\Sht_{\Gg}^{\can}\) over the Kottwitz stack then yields natural Hecke correspondences over \(\Sht_{\Gg}^{\spl}\) and \(\Sht_{\Gg}^{\can}\).
Following \cite[§5.4]{XiaoZhu:Cycles}, we can define categories \(\mathrm{P}^{\Corr}(\Sht_{\Gg}^{\spl})\) and \(\mathrm{P}^{\Corr}(\Sht_{\Gg}^{\can})\), whose objects are perverse sheaves on \(\Sht_{\Gg}^{\spl}\) and \(\Sht_{\Gg}^{\can}\) respectively, and whose morphisms are cohomological correspondences supported on the Hecke correspondences above.
We refer to §\ref{subsec:motives on shtukas} for details.
By h-descent, this category is an incarnation of the category of \(\ell\)-adic sheaves on the Kottwitz stack parametrizing \(G\)-isocrystals, i.e., the constructible side of the (conjectural) categorical arithmetic local Langlands correspondence from \cite{Zhu:Coherent,Zhu:Tame}.

By applying a construction that resembles the categorical trace of Frobenius, we can construct cohomological correspondences on \(\Sht_{\Gg}^{\spl}\) and \(\Sht_{\Gg}^{\can}\), cf.~\thref{Main local theorem}.

\begin{thm}\thlabel{intro-theorem shtuka}
	Let \(G/F\) be an essentially unramified quasi-split reductive group, with very special parahoric \(\Gg/\Oo_F\).
	Then there exists a commutative diagram
	\[\begin{tikzcd}[column sep=small]
		\Coh_{\widehat{G}\text{-}\mathrm{fr}}^{\widehat{G}^I}(\widehat{G}^I\sigma) \arrow[ddd] \arrow[rrr] &&&\mathrm{P}^{\Corr}(\Sht_{\Gg}^{\spl}) \arrow[ddd]\\
		& \Rep_{\overline{\IQ}_\ell}^{\fd}(\widehat{G}) \arrow[d] \arrow[r, "\Sat"] \arrow[ul] & \Perv_{L^+\Gg}(\Gr_{\Gg}^{\spl}) \arrow[d] \arrow[ur] & \\
		& \Rep_{\overline{\IQ}_\ell}^{\fd}(\widehat{G}^I) \arrow[r, "\Sat"'] \arrow[ld] & \Perv_{L^+\Gg}(\Gr_{\Gg}^{\can}) \arrow[rd] & \\
		\Coh_{\widehat{G}^I\text{-}\mathrm{fr}}^{\widehat{G}^I}(\widehat{G}^I\sigma) \arrow[rrr] &&& \mathrm{P}^{\Corr}(\Sht_{\Gg}^{\can}),
	\end{tikzcd}\]
	where \(\Coh_{\widehat{G}\text{-}\mathrm{fr}}^{\widehat{G}^I}(\widehat{G}^I\sigma) \subseteq \Coh_{\widehat{G}^I\text{-}\mathrm{fr}}^{\widehat{G}^I}(\widehat{G}^I\sigma) \subseteq \Coh(\widehat{G}^I\sigma/\widehat{G}^I)\) are the subcategories of coherent sheaves on \(\widehat{G}^I\sigma /\widehat{G}^I\), generated by free vector bundles pulled back from representations of \(\widehat{G}\), respectively \(\widehat{G}^I\).
\end{thm}

By \cite{vdH:SphericalParameters}, the stack \(\widehat{G}^{I}\sigma/\widehat{G}\) is a closed substack of the stack of local Langlands parameter from \cite{Zhu:Coherent}.
In particular, the bottom arrow is a part of the conjectural categorical local Langlands correspondence.
In constrast to most literature on the subject, we emphasize that there is no assumption on the ramification, and in particular the theorem also applies to wildly ramified groups.
In fact, the bottom half of the diagram exists for general quasi-split groups, without the assumption that \(G\) is essentially unramified.

The interpretation of the bottom arrow as a part of the categorical local Langlands correspondence, suggests that it should be fully faithful.
We do not address this question here, but since the leftmost arrow is clearly fully faithful, it suggests that all arrows in the outer square are fully faithful.
For the rightmost arrow, this is indeed the case, and we give a direct proof in \thref{spl vs can fully faithful}.
Consequently, to construct the upper half of the diagram, it suffices to check the factorization on the level of objects, which is easy.
Nevertheless, for the application to \thref{intro--Tate}, it is important to have a direct construction of the uppermost arrow.

\subsection{Exotic Hecke correspondences}

Now, let us go back to the global setting.
In order to apply \thref{intro-theorem shtuka} to prove Theorems \ref{intro--Tate} and \ref{intro--JL}, we will pull back the cohomological correspondences between the moduli of local shtukas to the Shimura varieties.
For this, we will construct exotic Hecke correspondences, in situations where the local groups are allowed to be ramified.
As mentioned before, we will use splitting models of Shimura varieties.
Since exotic Hecke correspondences are characterized via the moduli of local shtukas, we will need the moduli of splitting local shtukas, as introduced in the previous subsection.

For \(i\in \{1,2\}\), let \((\IG_i,\IX_i)\) be two Shimura data arising from PEL data \((B,*,\Oo_B,V_i,(,),\Ll_i,h_i)\).
In particular, the semisimple \(\IQ\)-algebra \(B\) is the same in both PEL data.
Denote their reflex fields by \(\IE_i\), and their Hodge cocharacters by \(\mu_i\).
We moreover fix an isomorphism \(V_1\otimes \IA_f \cong V_2\cong \IA_f\), compatibly with the additional structures.
This yields an isomorphism \(\IG_{1,\IA_f} \cong \IG_{2,\IA_f}\), so that \(\IG_1\) and \(\IG_2\) are in particular isomorphic at \(p\).

By using moduli interpretations, \cite{PappasRapoport:LocalII.Splitting} have defined splitting integral models of Shimura varieties; let us denote the perfections of their special fibers at some level \(K\subseteq \IG_1(\IA_f)\cong \IG_2(\IA_f)\) by \(\Sh_{\mu_i,K}^{\spl}\).
(For the introduction, it will suffice to consider them over \(\overline{\IF}_p\).)
Roughly speaking, they parametrize abelian varieties with PEL structure, together with a suitable filtration on their cotangent bundle.
Similarly, in PEL type situations, we can show that a certain bounded part of the moduli of splitting local shtukas admits a moduli interpretation, in terms of \(p\)-divisible groups with extra structure.
For this, we were inspired by the methods of Hoff \cite{Hoff:Parahoric}.
By using these moduli interpretations, we can construct exotic Hecke correspondences, cf.~\thref{Exotic Hecke correspondences}.

\begin{thm}\thlabel{intro--exotic}
	Let \((\IG_i,\IX_i)\) be as above.
	Assume \(\IG_{1,\IQ_p}\cong \IG_{2,\IQ_p}\) is essentially unramified, and let \(\Gg\) be a very special parahoric model.
	Then there exists a diagram with cartesian squares
	\[\begin{tikzcd}
		\Sh_{\mu_1,K}^{\spl} \arrow[d] & \Sh_{\mu_1\mid \mu_2}^{\spl} \arrow[l] \arrow[r] \arrow[d] & \Sh_{\mu_2,K} \arrow[d] \\
		\Sht_{\Gg,\leq\mu_{1}}^{\spl} & \Sht_{\mu_{1}\mid \mu_{2}}^{\spl} \arrow[l] \arrow[r] & \Sht_{\Gg{,}\leq \mu_{2}}^{\spl},
	\end{tikzcd}\]
	where the horizontal arrows are representably by ind-(perfectly proper) ind-schemes.
	Here, \(\Sht_{\mu_{1}\mid \mu_{2}}^{\spl}\) is a suitable Hecke correspondence between the moduli of splitting local shtukas.
\end{thm}

The construction for \thref{intro--exotic} only uses the moduli interpretations of \(\Sh_{\mu_i{,}K}^{\spl}\) and \(\Sht_{\Gg{,}\leq \mu_{i}}^{\spl}\).
In particular, we can similarly construct exotic Hecke correspondences between the special fibers \(\Sh_{\mu_i,K}^{\naive}\) of the naive integral models of Shimura varieties from \cite{RapoportZink:Period}.
Moreover, the schematic image of a natural map \(\Sh_{\mu_i,K}^{\spl} \to \Sh_{\mu_i,K}^{\naive}\) agrees with the special fiber of the canonical model of the Shimura variety.
Using this, we can construct exotic correspondences between the special fibers of canonical models of Shimura varieties, as the schematic image of a natural map \(\Sh_{\mu_1\mid \mu_2}^{\spl} \to \Sh_{\mu_1\mid \mu_2}^{\naive}\).
We refer to \thref{Exotic Hecke correspondences} for details.

In fact, the assumptions on \(\IG_{1,\IQ_p}\cong \IG_{2,\IQ_p}\) and \(\Gg\) can be substantially weakened: it suffices that the splitting models from \cite{PappasRapoport:LocalII.Splitting} are flat and absolutely weakly normal.
This is known in the essentially unramified case for general parahorics, as well as some ramified unitary group cases \cite{BijakowskiHernandez:GeometryAR}.
For more general unitary groups, it is likely that modified versions of this splitting model will admit exotic Hecke correspondence (and that these will induce exotic Hecke correspondences between the canonical models), but we do not pursue this here.

Combining \thref{intro--exotic} with the construction of cohomological correspondences on the moduli of (splitting) local shtukas from \thref{intro-theorem shtuka}, we obtain the geometric Jacquet--Langlands correspondence from \thref{intro--JL}.
For the Tate conjecture, we can obtain a similar diagram as in \thref{intro--exotic}, involving the special fiber of the splitting model of a more general Hodge type Shimura variety, and a zero-dimensional (weak) Shimura variety, via Rapoport--Zink uniformization (\thref{RZuniformization}.)
The resulting Jacquet--Langlands transfer map can then be identified with a cycle class map.
To study the injectivity and surjectivity of this map, we need more precise information about the irreducible components in the basic Newton stratum of \(\Sh_{\mu,K}^{\spl}\).
By Rapoport--Zink uniformization we can reduce this to studying irreducible components of splitting versions of affine Deligne--Lusztig varieties.

\subsection{Irreducible components of splitting affine Deligne--Lusztig varieties}

As discussed in the previous subsection, our next task is to introduce splitting versions of affine Deligne--Lusztig varieties.
Let \(G/F\) be an essentially unramified reductive group, and \(\mu\in X_*(T)^+\) a dominant cocharacter.
Then we have corresponding Schubert varieties \(\Gr_{\Gg,\leq \mu}^{\spl}\subseteq \Gr_{\Gg}^{\spl}\) and \(\Gr_{\Gg,\leq \mu_I}^{\can}\subseteq \Gr_{\Gg}^{\can}\), where \(\mu_I\in X_*(T)_I^+\) is the image of \(\mu\).
By definition, the (canonical) affine Deligne--Lusztig variety associated with \(\mu_I\) and \(b\in G(\breve{F})\) is
\[X_{\leq \mu_I}^{\can}(b) := \left\{gL^+\Gg\in LG/L^+\Gg\mid g^{-1}b \sigma(g)\in L^+\Gg\backslash \Gr_{\Gg,\leq \mu_I}^{\can}\right\}\subseteq \Gr_{\Gg}^{\can}.\]
In particular, it admits a map to \(L^+\Gg\backslash\Gr_{\Gg,\leq \mu_I}^{\can}\), and we define the \emph{splitting affine Deligne--Lusztig variety} associated with \(\mu\) and \(b\) via the following cartesian diagram:
\[\begin{tikzcd}[column sep=huge]
	X_{\leq \mu}^{\spl}(b) \arrow[rr] \arrow[d] && L^+\Gg\backslash \Gr_{\Gg,\leq \mu}^{\spl} \arrow[d]\\
	X_{\leq,\mu_I}^{\can}(b) \arrow[rr, "gL^+\Gg\mapsto g^{-1}b\sigma(g)"'] && L^+\Gg\backslash \Gr_{\Gg,\leq \mu_I}^{\can}.
\end{tikzcd}\]

In particular, the splitting versions of the Shimura variety and the affine Deligne--Lusztig varieties are obtained via pullback along the same map.
By Rapoport--Zink uniformization for the canonical models, this implies that the basic Newton stratum of the special fiber of the splitting model of the Shimura variety admits Rapoport--Zink uniformization by splitting versions of affine Deligne--Lusztig varieties, \thref{RZuniformization}.

To understand the irreducible components of \(X_{\leq \mu}^{\spl}(b)\), it suffices to understand the geometry of \(X_{\leq \mu_I'}^{\can}(b)\) (for general \(\mu_I'\leq \mu_I\)), as well as the fibers of the map \(\Gr_{\Gg,\leq \mu}^{\spl}\to \Gr_{\Gg,\leq \mu_I}^{\can}\).
But this is essentially a convolution morphism, whose fibers are well understood, and admit a representation-theoretic interpretation in terms of the multiplicities appearing in the tensor products of irreducible representations of the (inertia-invariants of the) dual group; cf.~\cite{Haines:Structure}.
Moreover, in the canonical and (essentially) unramified case, the irreducible components of affine Deligne--Lusztig varieties are well-understood, by \cite{XiaoZhu:Cycles,Nie:Irreducible,ZhouZhu:Twisted}.
This yields the following theorem, cf.~\thref{Irreducible components of splitting adlv}.
We denote by \(\IM\IV^{\can}\) and \(\IM\IV^{\spl}\) the sets of various Mirkovic--Vilonen cycles in the canonical and splitting affine Grassmannians, which form natural bases of the irreducible representations of \(\widehat{G}^I\) and \(\widehat{G}\).
Moreover, \(\lambda_b\in X_*(T)_{\Gamma_F}\) denotes the \emph{best integral approximation} of \(b\in B(G)\), in the sense of \cite{HamacherViehmann:Irreducible,ZhouZhu:Twisted}, and \(\Gamma_F:=\Gal(\overline{F}/F)\) the absolute Galois group of \(F\).

\begin{thm}
	Let \(\mu\in X_*(T)^+\) and \(b\in B(G,\mu)\).
	There exists a natural commutative diagram of sets
	\[\begin{tikzcd}
			J_b(F)\backslash \Irr^{\Top}X_{\leq \mu}^{\spl}(b) \arrow[d] \arrow[r, "\cong"] & \bigsqcup_{\lambda\in X^*(\widehat{T}),\lambda_{\Gamma_F}=\lambda_b}\IM\IV_\mu^{\spl}(\lambda) \arrow[d]\\
			\bigsqcup_{\mu_I'\leq \mu_I\in X_*(T)_I^+} J_b(F)\backslash \Irr^{\Top}X_{\mu_I'}^{\can}(b) \arrow[r, "\cong"] & \bigsqcup_{\mu_I'\leq \mu_I\in X_*(T)_I^+} \bigsqcup_{\lambda_I\in X^*(\widehat{T})_I,(\lambda_I)_{\Gamma_F}=\lambda_b} \IM\IV_{\mu_I'}^{\can}(\lambda_I),
	\end{tikzcd}\]
	where the horizontal maps are bijections.
	Moreover, the right horizontal arrow is given by decomposing the irreducible \(\widehat{G}\)-representation of highest weight \(\mu\) into irreducible \(\widehat{G}^I\)-representations and then comparing weight spaces under the map \(X_*(T)\to X_*(T)_I\).
\end{thm}

This theorem illustrates another advantage of the splitting models.
When studying the cohomology of Shimura varieties (defined over a number field), it is the \(\widehat{G}\)-representation \(V_{\mu^*}\) with highest weight the (dual of the) Hodge cocharacter \(\mu^*\) that shows up, rather than the \(\widehat{G}^I\)-representation \(V_{\mu_I^*}\) of highest weight \(\mu_I^*\).
In particular, when using known results on the cohomology of Shimura varieties to deduce instances of the Tate conjecture, we need \(V_{\mu^*}\) rather than \(V_{\mu_I^*}\) to show up.
This does not happen for \(X_{\leq \mu_I^*}^{\can}(b)\), but it does happen for \(X_{\leq \mu^*}^{\spl}(b)\).

If moreover \(b\) is \emph{is very special} in the sense of \thref{very special Kottwitz elements} (which is a generalization of the notion of unramified elements from \cite[§4.2]{XiaoZhu:Cycles} to the quasi-split case), we have more refined information about the irreducible components, cf.~\thref{irreducible comp of ADLV in very special case}.
Moreover, for applications towards the Tate conjecture, we need a precise construction of the irreducible components of splitting affine Deligne--Lusztig varieties.
This is done in §\ref{sec--adlv and semi-infinite orbits}, by studying the relation between affine Deligne--Lusztig varieties and semi-infinite orbits, and bootstrapping from the case where \(G/F\) is unramified, which is handled in \cite[§4.4]{XiaoZhu:Cycles}.

The above gives us a good understanding of the cycle class map for the special fibers of splitting models of Shimura varieties, at least for cycles in the basic Newton stratum.
To deduce \thref{intro--Tate}, it remains to understand a certain Chevalley restriction map (as in \cite[§1.4]{XiaoZhu:Cycles}) for the group \(\widehat{G}^I\), as well as the Frobenius action on the cohomology of the Shimura variety (on the generic fiber, so that no choice of model is involved).
Since our local group is essentially unramified, the Chevalley restriction map can be understood by bootstrapping from the situation in \cite{XiaoZhu:Vector}.
Moreover, the cohomology of compact unitary group Shimura varieties has already been studied in \cite{ScholzeShin:Cohomology}, and the assumptions appearing in loc.~cit.~correspond exactly to the local group being essentially unramified.

\subsection{Motivic aspects of the Langlands program}

Another goal of this article is to show that the construction from \cite{XiaoZhu:Cycles} can be carried out motivically, i.e., by replacing the use of étale cohomology with étale motives.
This provides further evidence for the conjecture that the Langlands program is of motivic origin.
More precisely, the construction of (a part of) the categorical local Langlands functor in \thref{intro-theorem shtuka} holds motivically as well.
This is a first step towards a motivic refinement of the arithmetic categorical local Langlands correspondence, as asked for in \cite[Remark 4.41]{Zhu:Coherent}.
We will explain in detail the motivic nature of this correspondence in upcoming work.

Recall that the main cohomological ingredient of \cite{XiaoZhu:Cycles}, and hence of the present paper, is the geometric Satake equivalence.
In the situation of ramified groups, a motivic refinement of this equivalence have been constructed in \cite{vdH:RamifiedSatake}.
With this additional input, most theorems presented in the introduction can be refined motivically, and we refer to the main text for details.
As a concrete application, we deduce that the Jacquet--Langlands correspondence from \cite[Conjecture 4.60]{Zhu:Coherent} and \thref{intro--JL} is of motivic origin, and in particular independent of the auxiliary prime \(\ell\).
Moreover, when the special fibers of the splitting models of the Shimura varieties are proper, we can construct a Jacquet--Langlands transfer between the (rationalized) higher Chow groups of these special fibers, cf.~\thref{JL for Chow and crystalline}.
By applying the rigid realization functor, we can further deduce a Jacquet--Langlands transfer between their crystalline cohomology.

\begin{rmk}
	Let us clarify the relation of our arguments with those appearing in previous work, in particular when using motives.
	\begin{enumerate}
		\item It should be clear that our arguments are inspired by \cite{XiaoZhu:Cycles}, and our construction of motivic correspondences on the moduli of (either canonical or splitting) local shtukas is similar in spirit to the construction of loc.~cit.
		In order to emphasize the motivic difficulties that show up, in particular related to the Tate twists that we do not trivialize (cf.~\thref{twisting has no effect}, and the appearance of \(\widehat{G}^Iq^{-1}\sigma/\widehat{G}^I\) in \thref{Main local theorem}, rather than \(\widehat{G}^I\sigma/\widehat{G}^I\) as in \thref{intro-theorem shtuka}), we have explained the construction of the functor from \thref{intro-theorem shtuka} in detail.
		However, to check that certain properties of our construction hold, we will refer to \cite{XiaoZhu:Cycles} when no new ideas are required.
		We have also taken this opportunity to explain how several arguments of \cite{XiaoZhu:Cycles}, especially involving correspondences between moduli of restricted local shtukas, can be simplified.
		\item In \cite{Zhu:Tame}, Zhu has constructed a larger part of the arithmetic categorical local Langlangs correspondence.
		The approach in \cite{Zhu:Tame} is similar in spirit to \cite{XiaoZhu:Cycles}, by using the categorical trace of Frobenius, but relies on a more robust and general framework.
		We have followed the construction of \cite{XiaoZhu:Cycles}, since it is needed to work with abelian categories (as a derived version of the motivic and/or ramified Satake equivalence is not yet available), and we wanted to see explicitly how certain motivic difficulties could appear and be resolved.
		In upcoming work, we will come back to a motivic refinement of the arithmetic categorical local Langlands correspondence, and clarify the relation with the categorical trace of Frobenius.
		\item Finally, Yu \cite{Yu:Geometric} has refined certain arguments of \cite{XiaoZhu:Cycles} to the level of \(\overline{\IZ}_\ell\)-sheaves, and used them to construct a geometric Jacquet--Langlands transfer for the cohomology of Shimura varieties with nontrivial coefficients systems.
		His arguments do not immediately carry over to the motivic setting: in the setting of \thref{intro--JL}, he uses that \(\IG_1\) and \(\IG_2\) are isomorphic over \(\IQ_\ell\), which suffices when working with \(\ell\)-adic cohomology.
		On the other hand, we are using motives with \(\IQ\)-coefficients, and it is not necessarily true that \(\IG_1\cong \IG_2\) over \(\IQ\).
		Instead, since we are mainly working with PEL type Shimura varieties, we use the moduli interpretations to replace the arguments from \cite{Yu:Geometric} by a motivic argument.
		This still allows us to construct geometric Jacquet--Langlands transfers involving nontrivial local systems, without needing to work with integral coefficients, and we refer to \thref{JL transfer for local systems} for details.
	\end{enumerate} 
\end{rmk}

\subsection{Outline}

Let us now give an overview of the paper.

In §\ref{Sec:Models of Shimura varieties}, we recall the naive, canonical, and splitting local models and integral models of Shimura varieties, as well as the relevant PEL data.
We use this in §\ref{Sec:Displays} to define various stacks of (bounded) local shtukas, and provide moduli interpretations in the naive and splitting cases.
Using this, we construct exotic Hecke correspondences in §\ref{Sec:Exotic}.
The above sections are mostly moduli-theoretic.

Next, we move to a more group-theoretic approach.
In §\ref{Sec:Splitting Gr}, we define and study splitting versions of the affine Grassmannian for essentially unramified groups, and prove \thref{intro-theorem satake}.
Next, we introduce splitting versions of affine Deligne--Lusztig varieties in §\ref{Sec:ADLV}, and study their irreducible components.
In §\ref{Sec:Motivic corr}, we introduce the group-theoretic version of the moduli stack of splitting local shtukas, and construct motivic correspondences as in \thref{intro-theorem shtuka}, by using the motivic Satake equivalence from \cite{vdH:RamifiedSatake}.
Throughout these sections, we mostly use motives with \(\IZ[\frac{1}{p}]\)-coefficients.
But from §\ref{Subsec:Motivic correspondences on shtukas} onwards, we will change to \(\IQ\)-coefficients, so that the Peter--Weyl theorem applies.

In the last part of the paper, we put everything together.
By pulling back the motivic correspondences on the moduli of local shtukas to the exotic Hecke correspondences, we construct the motivic Jacquet--Langlands transfer in §\ref{Sec:JacquetLanglands}.
And in §\ref{Sec:Tate}, we deduce generic instances of the Tate conjecture for the special fibers of the splitting models of Shimura varieties.

Finally, in the Appendix \ref{Appendix:motives}, we recall the theory of étale motives that we use throughout the paper.
We also review the notion of cohomological correspondences.
The results of this appendix are probably well known, but we have included them since they have not appeared in the literature in the specific form needed in this paper.

\subsection{Notation}\label{Subsec:Notation}

Throughout this paper, we fix a prime \(p\); this will be assumed odd at various places, but not in general.
We will use \(F\) to denote a non-archimedean local field of residue characteristic \(p\), and fix a uniformizer \(\varpi\in F\).
Its ring of integers will be denoted \(\Oo_F\), its residue field by \(k=k_F\cong\IF_q\), its Galois group by \(\Gamma_F=\Gal(\overline{F}/F)\), and its inertia group by \(I=I_F\).
Similarly, we denote the residual Galois group by \(\Gamma_k=\Gal(\overline{k}/k)\).
We moreover fix a lift of the (arithmetic) \(q\)-Frobenius \(\sigma\in \Gamma_F\), or equivalently, a section \(\Gamma_k\to \Gamma_F\) of the natural quotient \(\Gamma_F\to \Gamma_k\).

We denote by \(\Perf_k\) the category of perfect \(k\)-schemes, and by \(\AffSch_k^{\perf}\subseteq \Perf_k\) the subcategory of perfect affine \(k\)-schemes.
More generally, for a ring \(R\), we denote by \(\AffSch_R\) the category of affine \(R\)-schemes.
For a perfect \(k\)-algebra \(R\), we have the ring of ramified Witt vectors \(W_{\Oo_F}(R):=W(R) \widehat{\otimes}_{W(k)} \Oo_F\), as well as the truncated version \(W_{\Oo_F,n}(R) := W_{\Oo_F}(R) \otimes_{\Oo_F} (\Oo_F/\varpi^{n+1})\) for \(n\geq 0\).
These rings of ramified Witt vectors are equipped with a natural Frobenius automorphism, which we also denote by \(\sigma\).

Next, we recall some group-theoretic notation.
For a smooth affine group scheme \(H\) over \(F\), we denote its loop group by \(LH\colon (\AffSch_k^{\perf})^{\op}\to \Grp\colon \Spec R \mapsto H(W_{\Oo_F}(R)\otimes_{\Oo_F} F)\); this is representable by an ind-(perfect scheme) \cite[Proposition 1.1]{Zhu:Affine}.
For a smooth affine group scheme \(\Hh/\Oo_F\), we have the positive loop group \(L^+\Hh\colon (\AffSch_{k}^{\perf})^{\op}\to \Grp\colon \Spec R \mapsto \Hh(W_{\Oo_F}(R))\).
This is a closed subgroup scheme of \(LH\) \cite[Lemma 1.2]{Zhu:Affine}, and is representable by a pro-smooth affine group scheme.
More precisely, it is the limit \(L^+\Hh:=\varprojlim_{n} L^n\Hh\), where \(L^n\Hh\colon (\AffSch_{k}^{\perf})^{\op} \to \Grp \colon \Spec R \mapsto \Hh(W_{\Oo_F,n}(R))\) are the truncated loop groups.
We will then call the étale sheafification \((LH/L^+\Hh)^{\et}=:\Fl_{\Hh}\) (as a functor \((\AffSch_k^{\perf})^{\op} \to \Set\)) the affine flag variety of \(\Hh\).

Now, assume that \(H=G\) is reductive and \(\Gg/\Oo_F\) a parahoric model, corresponding to a facet \(\ff\) in the Bruhat--Tits building of \(H\).
We will always fix a maximal split torus \(A\subset G\), a maximal \(\breve{F}\)-split torus \(S\) containing \(A\), and let \(T:=\Cent_G H\), which is a maximal torus.
We will denote groups of cocharacters by \(X_*(-)\), and groups of characters by \(X^*(-)\).
We denote the Iwahori--Weyl group by \(\tilde{W}=\Norm_G(S)(\breve{F})/\Tt(\breve{\Oo_F})\), where \(\Tt\) is the connected Néron model of \(T\), which is automatically contained in \(\Gg\).
Let \(W_{\ff} = \left(\Norm_S(G)(\breve{F}) \cap \Gg(\breve{\Oo_F}) \right)/\Tt(\breve{\Oo_F})\subseteq \tilde{W}\) denote the subgroup corresponding to \(\ff\), so that the \(L^+\Gg\)-orbits in \(\Fl_{\Gg}\) (after base change to \(\Spec \overline{k}\)) are indexed by \(W_{\ff} \backslash \tilde{W} /W_{\ff}\) \cite[Proposition 2.8]{Richarz:Schubert}; we denote this decomposition by
\[\Fl_{\Gg} = \bigsqcup_{w\in W_{\ff} \backslash \tilde{W} /W_{\ff}} \Fl_{\Gg,w}.\]
Such an orbit is called a Schubert cell, and we denote its closure (called a Schubert variety) by \(\Fl_{\Gg,\leq w} = \bigsqcup_{w'\leq w} \Fl_{\Gg,w'}\), where \(\leq\) denotes the Bruhat partial order.
For \(\mu\in X_*(T)\), we denote the \(\mu\)-admissible subset from \cite[(3.6)]{Rapoport:Guide} by \(\Adm_\mu^{\ff}\), and write \(\Fl_{\Gg,\preccurlyeq \mu} := \bigsqcup_{w\in \Adm_\mu^{\ff}} \Fl_{\Gg,w}\subseteq \Fl_{\Gg}\).
We denote the length function arising from the quasi-Coxeter structure of \(\tilde{W}\) by \(l\colon \tilde{W} \to \IZ_{\geq 0}\), and extend it to the double coset space via
\[l\colon W_{\ff} \backslash \tilde{W}/W_\ff \colon w \mapsto \max_{v\in W_\ff} \min_{v'\in W_\ff} l(vwv').\]
Then the orbit \(\Fl_{\Gg,w}\) is of dimension \(l(w)\) \cite[Proposition 2.8]{Richarz:Schubert}.

When \(H=G\) is reductive and \(\Hh=\Gg\) a very special parahoric model (in the sense of \cite[Definition 6.1]{Zhu:Ramified}), we will also write \(\Gr_{\Gg}^{\can} = \Fl_{\Gg}\), and call it the (canonical, twisted) affine Grassmannian of \(\Hh\) (we will use the subscript \(\can\) throughout, in order to distinguish it from the splitting affine Grassmannians introduced in \thref{definition splitting grassmannians}).
Such a reductive group \(G\) is automatically quasi-split \cite[Lemma 6.1]{Zhu:Ramified}, and we fix a Borel \(T\subseteq B \subseteq G\), which yields a notion of dominant cocharacters \(X_*(T)^+\subseteq X_*(T)\).
Moreover, \(X_*(T)\) is equipped with a \(\Gamma_F\)-action, and the inertia-coinvariants \(X_*(T)_I\) also admit a subset of \(B\)-dominant elements \(X_*(T)_I^+\).
Denote the natural pairing \(X^*(T)\times X_*(T) \to \IZ\) by \(\langle-,-\rangle\).
This pairing is \(I\)-invariant, and hence descends to a pairing \(X^*(T) \times X_*(T)_I \to \IZ\).
Using the above notation, \(W_{\ff}=W^I\) is isomorphic to the inertia-invariants of the (absolute) finite Weyl group \(W\), and we have \(W_{\ff}\backslash \tilde{W}/W_{\ff} \cong X_*(T)_I^+\).
We denote the corresponding Schubert cells and Schubert varieties by \(\Gr_{\Gg,(\leq) \mu_I}\), whose dimension is \(\langle 2\rho,\mu_I\rangle\), where \(\rho\) is the half-sum of the positive absolute roots of \(G\) (with respect to \(B\)).
For an element \(\lambda\in X_*(T)\), we will denote its image in the inertia-coinvariants by \(\lambda_I\in X_*(T)_I\).
Generally, to clarify whether we are working in \(X_*(T)\) or \(X_*(T)_I\), we will denote elements in \(X_*(T)_I\) with a subscript \((-)_I\), even if they do not come from a fixed lift in \(X_*(T)\).
Using this notation, we have \(\Fl_{\Gg,\preccurlyeq \mu} = \Gr_{\Gg,\leq \mu_I}^{\can}\).

\subsection{Acknowledgments}

I thank Remy van Dobben de Bruyn, Arnaud Eteve, Tom Haines, Manuel Hoff, Pol van Hoften, Timo Richarz, Peter Scholze, Torsten Wedhorn, Liang Xiao, Jize Yu, Mingjia Zhang, Rong Zhou, and Xinwen Zhu for helpful conversations and/or useful feedback on earlier versions of this paper.
This project started while I was a PhD-student at the TU Darmstadt, where I was supported by the European research council (ERC) under the European Union’s Horizon 2020 research and innovation programme (grant agreement No 101002592), and the LOEWE professorship in Algebra (through Timo Richarz), project number LOEWE/4b//519/05/01.002(0004)/87.
This project was finished at the Max Planck Institute for Mathematics in Bonn, and I thank them for their hospitality and excellent working conditions.

\section{Models of Shimura varieties}\label{Sec:Models of Shimura varieties}

Since one of the main objectives of this paper is to study the geometry of splitting models of Shimura varieties, we start by recalling the necessary definitions.
Whenever we will use PEL type Shimura varieties, we will assume \(p\neq 2\); in particular we assume this throughout this section.

\subsection{PEL and splitting data}

Let us start by recalling the data appearing in the moduli interpretations of Shimura varieties.

\begin{nota}\thlabel{PEL data}
	Throughout, we will consider the following notions of PEL data.
	\begin{enumerate}
		\item A local rational PEL datum is a tuple \((B,*,V,(,),\mu)\), with \(B\) a finite dimensional semisimple \(\IQ_p\)-algebra with involution \(*\), \(V\) a finite (left) \(B\)-module, and \((,)\colon V\times V\to \IQ_p\) a nondegenerate alternating bilinear form satisfying \((bv,w)=(v,b^*w)\) for \(v,w\in V\) and \(b\in B\).
		This defines an algebraic group \(G\) over \(\IQ_p\) via
		\[G(R)=\left\{g\in \GL_B(V\otimes_{\IQ_p} R)\mid (gv,gw)=c(g)(v,w), c(g)\in R^\times\right\},\]
		which we will always assume to be connected.
		Then \(\mu\) is a cocharacter \(\mu\colon \IG_m\to G\), defined over a finite extension of \(\IQ_p\), with field of definition \(E/\IQ_p\); its residue field will be denoted \(k_E\cong \IF_q\).
		Over such a finite extension, we obtain an eigenspace decomposition, and we assume only the weights \(0\) and \(1\) can appear, i.e., we have \(V=V_0\oplus V_1\); in particular \(\mu\) is minuscule.
		Finally, for a local rational PEL datum as above, we will denote by \(F\) the center of \(B\), and \(F_1\subseteq F\) the invariants under \(*\).
		\item A local integral PEL datum (refining \((B,*,V,(,),\mu)\)) consists of a *-stable maximal order \(\Oo_B\subseteq B\), and a selfdual (with respect to \((,)\)) multichain \(\Ll\) of \(\Oo_B\)-lattices in \(V\) \cite[Definition 3.4]{RapoportZink:Period}.
		Let \(\Gg/\IZ_p\) be the automorphism group of \(\Ll\); this is a smooth affine group scheme with generic fiber \(G\), whose neutral connected component is parahoric \cite[Remark 21.6.7]{ScholzeWeinstein:Berkeley}, and we will assume \(\Gg\) is actually a connected Bruhat--Tits stabilizer group scheme.
		We denote \(K_p:=\Gg(\IZ_p)\subseteq G(\IQ_p)\), which is the stabilizer of \(\Ll\).
		\item A global rational PEL datum consists of a tuple \((B,*,V,(,),h)\), with \(B\) a finite dimensional semisimple \(\IQ\)-algebra \(B\) with positive involution \(*\) and \(V\) a finite dimensional left \(B\)-module with nondegenerate bilinear form \((,)\colon V\times V\to \IQ\) satisfying \((bv,w)=(v,b^*w)\) for \(v,w\in V\) and \(b\in B\).
		Finally, \(h\colon \mathbb{C}\to \End_{B\otimes \IR}(V\otimes \IR)\) is a \(*\)-homomorphism such that the symmetric bilinear form \((V\otimes \IR)\times (V\otimes \IR)\to \IR\colon (v,w)\mapsto (v,h(i)w)\) is positive-definite.
		The data above defines an algebraic group \(\IG\) over \(\IQ\) via
		\[\IG(R) = \left\{x\in \GL_B(V\otimes_{\IQ} R)\mid (gv,gw)=c(g)(v,w), c(g)\in \IQ\right\}.\]
		As in the local case we will assume that \(\IG\) is connected (which amounts to excluding type D in Kottwitz's classification), but now we will also assume that \(\IG\) satisfies the Hasse principle (i.e., that \(\mathrm{H}^1(\IQ,\IG) \to \prod_v \mathrm{H}^1(\IQ_v,\IG)\) is injective).
		Then the global rational PEL datum above defines a Shimura datum \((\IG,\IX)\) with reflex field \(\IE\).
		We denote the canonical \(\IE\)-model of the associated Shimura variety at (sufficiently small) level \(K\subseteq \IG(\IA_f)\) by \(\Sh_K(\IG,\IX)\).
		\item Finally, a global integral (at \(p\)) PEL datum consists of a global rational datum, together with a local integral datum refining the local rational datum which is obtained as the base change of the global datum to \(\IQ_p\) (and \(\mu\) is induced by \(h\)).
		Note that we then have \(G=\IG_{\IQ_p}\).
		In this situation, we will also fix a place \(\pp\) of \(\IE\) lying over \(p\), as well as the \(\pp\)-adic completion \(E:=\IE_{\pp}\).
		Moreover, \(K^p\subseteq \IG(\IA_f^p)\) will denote a sufficiently small compact open subgroup, and \(K:=K^pK_p\).
	\end{enumerate}
	Recall that \(\Ll\) is selfdual if \(\Lambda\in \Ll\) implies \(\hat{\Lambda}\in \Ll\), where \(\hat{\Lambda}\) is the image of \(\Hom_{\IZ_p}(\Lambda,\IZ_p)\) under the isomorphism \(\Hom_{\IQ_p}(V\otimes \IQ_p,\IQ_p)\cong V\otimes \IQ_p\) induced by \((,)\).
\end{nota}

\begin{rmk}\thlabel{remarks about PEL data}
	\begin{enumerate}
		\item We assume the Hasse principle to ensure that the Shimura varieties defined via moduli interpretations below agree with Shimura varieties in the sense of Deligne (otherwise, they would be disjoint unions of such Shimura varieties).
		Many of the techniques of this paper should still work without this assumptions, and we leave the details for the interested reader.
		\item\label{Hodge embedding} Given a global integral PEL datum \((B,*,V,(,),h,\Oo_B,\Ll)\), we can obtain a second global integral PEL datum \((\IQ,*,V,(,),h,\IZ_p,\Ll)\).
		The corresponding algebraic group \(\IG'/\IQ\) agrees with \(\GSp(V)\), and there is a natural closed immersion \(i\colon \IG\subseteq \IG'\), which is compatible with the cocharacters \(\mu\).
		Moreover, we clearly have \(K_p = i^{-1}\left(\GSp(\Ll)(\IZ_p)\right)\cap G(\IQ_p)\).
	\end{enumerate}
\end{rmk}

The following notation will be used when defining splitting structures.

\begin{nota}\thlabel{notation before splitting}
	Fix a local integral PEL datum as above.
	\begin{enumerate}
		\item Let \(\Upsilon\) denote the set of homomorphisms \(F\to \overline{\IQ}_p\), and for \(\tau\in \Upsilon\) write \(F_\tau:=\tau(F)\) and \(F_\tau^+:=\tau(F^+)\) for the images in \(\overline{\IQ}_p\).
		\item We denote by \(\sim\) the equivalence relation on \(\Upsilon\), such that \(\tau\sim \tau'\) when their restrictions to \(F^+\) are in the same \(\Gal(\overline{\IQ}_p/\IQ_p)\)-orbit.
		\item For each equivalence class \([\tau]\), we fix a representative \(\tau\), and write \(F_{[\tau]}^+:=F_{\tau}^+\) and \(F_{[\tau]}:= F_{[\tau]}^+\otimes_{F^+} F\).
		Then we have \(F\cong \prod_{[\tau]\in \Upsilon/\sim} F_{[\tau]}\) and \(F^+\cong \prod_{[\tau]\in \Upsilon/\sim} F_{[\tau]}^+\).
		Similar factorizations exist for the rings of integers \(\Oo_F\) and \(\Oo_{F^+}\), as well as any module under a ring in \(\{F^+,F,\Oo_{F^+},\Oo_F\}\) (such as \(V\otimes \IQ_p\)).
		\item For an equivalence class \([\tau]\in \Upsilon/\sim\), we fix an ordering of its elements \((\tau_{[\tau],i})_{i=0,\ldots,[F_{[\tau]}\colon \IQ_p]-1}\), in such a way that any two elements with the same restriction to \(F^+\) are consecutive.
		\item We define the integers \((r_\tau)_{\tau\in \Upsilon}\) such that for any \(a\in F\) we have
		\[\det(T-a\mid V_1) = \prod_{\tau\in \Upsilon}(T-\tau(a))^{r_\tau}.\]
		Note that for each \(\tau\in \Upsilon\), the \(F_{[\tau]}\)-module \(V_{[\tau]}\) is free of rank \(r_{\tau}+r_{\tau\circ *}\).
		\item Finally, we denote by \(E^{\Gal}\) the Galois closure of \(E\) in \(\overline{\IQ}_p\); this contains all \(F_\tau\).
	\end{enumerate}
\end{nota}

For any \(b\in B^\times\) which normalizes \(\Oo_B\), and any \(\Oo_B\otimes_{\IZ_p} \Oo_S\)-module \(M\), denote by \(M^b\) its restriction of scalars along the isomorphism \(\Oo_B\to \Oo_B\colon x\mapsto b^{-1}xb\).

\begin{dfn}\thlabel{Defi:sets of pairs}
	Fix a local integral PEL datum, and let \(S\) be any \(\Oo_E\)-scheme.
	An \emph{\(\Ll\)-set of polarized \(\Oo_B\otimes_{\IZ_p} \Oo_S\)-modules} is a functor 
	\[\underline{\mathscr{H}}\colon \Ll\to \Oo_B\otimes_{\IZ_p} \Oo_S\Mod\colon \Lambda\mapsto \mathscr{H}_\Lambda,\]
	equipped with \(\Oo_B\otimes_{\IZ_p}\Oo_S\)-linear \emph{periodicity isomorphisms}
	\[\theta^b_\Lambda\colon \mathscr{H}_\Lambda^b\xrightarrow{\cong} \mathscr{H}_{b\Lambda},\]
	satisfying the following condition:
	\begin{enumerate}
		\item Each \(\mathscr{H}_\Lambda\) is a finite locally free \(\Oo_S\)-module.
		\item For each \(\Lambda\in \Ll\), there exists a perfect pairing \(\mathscr{H}_\Lambda\times \mathscr{H}_{\hat{\Lambda}} \to L\), which induces an isomorphism \(\mathscr{H}_\Lambda\xrightarrow{\cong} \mathscr{H}_{\hat{\Lambda}}^\vee\otimes L\), for some line bundle \(L\) on \(S\) (independent of \(\Lambda\in \Ll\)).
		These isomorphisms are moreover required to be compatible (in the natural way, cf.~\cite[Definition 2.1.12 (5)]{Lan:Compactifications}).
	\end{enumerate}
	Note that this invertible line bundle does not appear in \cite{RapoportZink:Period, Lan:Compactifications}, but will be necessary for us \cite[Remark 21.6.3]{ScholzeWeinstein:Berkeley}.
\end{dfn}

We can now define splitting structures.
This definition will be slightly different from \cite[Definition 2.3.3]{Lan:Compactifications}, but will be directly applicable for local models, rather than only for integral models of Shimura varieties.

\begin{dfn}
	Let \(S\) be an \(\Oo_E\)-scheme, and \(\underline{\mathscr{H}}\) an \(\Ll\)-set of polarized \(\Oo_B\otimes_{\IZ_p} \Oo_S\)-modules.
	A \emph{splitting structure} for \(\mathscr{H}\) consists of data \(\left\{(\underline{\mathscr{F}}_{[\tau]}^i, \underline{j}_{[\tau]}^i)_{[\tau]\in \Upsilon/\sim}\mid 0\leq i<[F_{[\tau]}\colon \IQ_p]\right\}\) as follows:
	\begin{enumerate}
		\item For each \([\tau]\in \Upsilon/\sim\) and \(0\leq i<[F_{[\tau]}\colon \IQ_p]\), the object \(\underline{\mathscr{F}}_{[\tau]}^i\) is a functor
		\[\Ll\to \Oo_B\otimes_{\IZ_p}\Oo_S\Mod\colon \Lambda\mapsto \mathscr{F}_{\Lambda,[\tau]}^i,\]
		and \(\underline{j}_{[\tau]}^i = \{j_{\Lambda,[\tau]}^i\}_{\Lambda\in \Ll}\) consists of injections \(\mathscr{F}_{\Lambda,[\tau]}^i \into \mathscr{H}_{\Lambda,[\tau]}\).
		\item Each \(\mathscr{F}_{\Lambda,[\tau]}^i\) and \(\mathscr{H}_{\Lambda,[\tau]}/\mathscr{F}_{\Lambda,[\tau]}^i\) is a finite locally free \(\Oo_S\)-module.
		\item Let \(\Lambda\in \Ll\) and \([\tau]\in \Upsilon/\sim\). Then for each \(0\leq i<[F_{[\tau]}\colon \IQ_p]\), we have \(\mathscr{F}_{\Lambda,[\tau]}^{i+1}\subseteq \mathscr{F}_{\Lambda,[\tau]}^i\) (where \(\mathscr{F}_{\Lambda,[\tau]}^{[F_{[\tau]}\colon \IQ_p]}=0\)), and the quotient is a locally free \(\Oo_S\)-module of rank \(r_{\tau_{[\tau],i}}\), 
		which is moreover annihilated by \(a\otimes 1 - 1\otimes \tau_{[\tau],i}(a)\) for any \(a\in \Oo_{F_{[\tau]}}\).
		\item Let \(\Lambda,[\tau],i\) be as above, and let \(b\in B^\times\) which normalizes \(\Oo_B\).
		Then the periodicity isomorphisms \(\theta_\Lambda^b\colon \mathscr{H}_{\Lambda}^b\cong \mathscr{H}_{b\Lambda}\) restrict to periodicity isomorphisms for \(\mathscr{F}_{\Lambda,[\tau]}^i\), i.e., there is a (necessarily unique) commutative diagram
		\[\begin{tikzcd}
			(\mathscr{F}_{\Lambda,[\tau]}^i)^b \arrow[r] \arrow[d, "\theta^{b,i}_{\Lambda,[\tau]}"'] & (\mathscr{H}_{\Lambda,[\tau]})^b\arrow[d, "\theta^b_\Lambda"]\\
			\mathscr{F}_{b\Lambda,[\tau]}^i \arrow[r] & \mathscr{H}_{b\Lambda,[\tau]}.
		\end{tikzcd}\]
		\item The \emph{determinant condition} holds, i.e., there is an equality of polynomials
		\(\det_{\Oo_S}(a\mid \mathscr{F}_{\Lambda,[\tau]}^0) = \det_{\Oo_S}(a\mid V_1)\) in \(a\in \Oo_B\).
		\item Let \(\Lambda,[\tau],i\) be as usual, and consider the perfect pairing \(\mathscr{H}_{\Lambda,[\tau]}\times \mathscr{H}_{\hat{\Lambda},[\tau]}\to L\).
		Then the orthogonal complement \((\mathscr{F}_{\Lambda,[\tau]}^i)^\bot\) of \(\mathscr{F}_{\Lambda,[\tau]}^i\) under this pairing satisfies
		\[\left(\prod_{k=0}^i (a\otimes 1 - 1 \otimes \tau_{[\tau],k}(a))\right)\left((\mathscr{F}_{\Lambda,[\tau]}^i)^\bot\right)\subseteq \mathscr{F}_{\hat{\Lambda},[\tau]}^i,\]
		for any \(a\in \Oo_{F_{[\tau]}}\) and \(0\leq i\leq [F_{[\tau]}\colon \IQ_p]\) divisible by \([F_{[\tau]}\colon F_{[\tau]}^+]\).
	\end{enumerate}
\end{dfn}

There is an obvious notion of isomorphisms of splitting structures, cf.~\cite[Definition 2.3.5]{Lan:Compactifications}.

\subsection{Local models}

Local models where first systematically studied (in the PEL case) in \cite{RapoportZink:Period}, in order to model the singularities of integral models of Shimura varieties and their special fibers.
Let us fix a local integral PEL datum; then we can attach to it various local models.
The following definition was introduced in \cite[Definition 3.27]{RapoportZink:Period}.

\begin{dfn}
	The \emph{naive local model} \(\mathscr{M}_{\Gg,\preccurlyeq\mu}^{\naive}\) is the projective \(\Oo_E\)-scheme representing the functor sending a scheme \(S\) to the following data, up to isomorphism:
	\begin{enumerate}
		\item A functor \(\Ll\mapsto \Oo_B\otimes_{\IZ_p} \Oo_S\Mod\colon \Lambda\mapsto t_{\Lambda}\),
		\item A natural transformation \(\phi_\Lambda\colon \Lambda\otimes_{\IZ_p} \Oo_S \to t_\Lambda\).
	\end{enumerate}
	These are moreover required to satisfy the following conditions:
	\begin{enumerate}
		\item Each \(t_\Lambda\) is finite locally free as an \(\Oo_S\)-module,
		\item Each \(\phi_\Lambda\) is surjective,
		\item The natural periodicity isomorphisms \(\Lambda^b\otimes_{\IZ_p} \Oo_S\cong b\Lambda\otimes_{\IZ_p} \Oo_S\) induce (necessarily unique) periodicity isomorphisms \(t_\Lambda^b\cong t_{b\Lambda}\) via \(\phi_\Lambda\) and \(\phi_{b\Lambda}\).
		\item There is an equality \(\det_{\Oo_S}(b\mid t_\Lambda) = \det_L(b\mid V_0)\) of polynomials in \(b\in \Oo_B\),
		\item The composite map \(t_\Lambda^\vee \xrightarrow{\phi_\Lambda^\vee} (\Lambda\otimes_{\IZ_p} \Oo_S)^\vee \cong \hat{\Lambda}\otimes_{\IZ_p} \Oo_S \xrightarrow{\phi_{\hat{\Lambda}}} t_{\hat{\Lambda}}\) vanishes.
	\end{enumerate}
\end{dfn}

In case \(G\) does not split over an unramified extension, the local model \(\mathscr{M}_{\Gg,\preccurlyeq\mu}^{\naive}\) is typically not flat over \(\Oo_E\) \cite{Pappas:ArithmeticModuli}.
Instead, we can obtain a flat model as follows.

\begin{dfn}
	The \emph{canonical local model} \(\mathscr{M}_{\Gg,\preccurlyeq\mu}^{\can}\) is the absolute weak normalization of the schematic closure in \(\mathscr{M}_{\Gg,\preccurlyeq\mu}^{\naive}\) of its generic fiber.
	In other words, it is the absolute weak normalization of the unique flat closed subscheme of \(\mathscr{M}_{\Gg,\preccurlyeq\mu}^{\naive}\) with the same generic fiber \cite[Proposition 2.8.5]{EGA4.2}.
\end{dfn}

Recall that a scheme is \emph{absolutely weakly normal} if every universal homeomorphism to it is an isomorphism.
Moreover, for any scheme, the category of universal homeomorphisms to it has an initial object, called its \emph{absolute weak normalization}.

It will be useful to compare this canonical local model to group-theoretic definitions of local models.
Recall the (schematic) local models conjectured in \cite[Conjecture 21.4.1]{ScholzeWeinstein:Berkeley}, whose existence has been proven in \cite{AGLR:Local,GleasonLourenco:Tubular}.

\begin{prop}\thlabel{comparison local models}
	The canonical local model \(\mathscr{M}_{\Gg,\preccurlyeq\mu}^{\can}\) agrees with the local model from \cite[Conjecture 21.4.1]{ScholzeWeinstein:Berkeley}.
\end{prop}

\begin{proof}
	The naive local model \(\mathscr{M}_{\Gg,\preccurlyeq\mu}^{\naive}\) is projective over \(\Oo_E\), hence the same holds for the canonical local model.
	By definition, \(\mathscr{M}_{\Gg,\preccurlyeq\mu}^{\can}\) is absolutely weakly normal, and it is flat by \cite[Proposition 2.18 (1)]{AGLR:Local}.
	
	Now, let \(\Gr_{\Gg,\Oo_E}^{\BD}\) be the Beilinson--Drinfeld Grassmannian from \cite[§20.3]{ScholzeWeinstein:Berkeley}, defined over \(\Spd \Oo_E\).
	By \cite[Corollary 21.6.10]{ScholzeWeinstein:Berkeley}, there is a closed immersion \(\mathscr{M}_{\Gg,\preccurlyeq\mu}^{\naive,\diamondsuit} \into \Gr_{\Gg,\Oo_E}^{\BD}\), which induces an isomorphism \(\mathscr{M}_{\Gg,\preccurlyeq\mu}^{\naive,\diamondsuit} \cong \Gr_{\Gg,E,\mu}^{\BD}\) after base change to \(\Spd E\); here \((-)^{\diamondsuit}\) is the functor defined in \cite[§18.1]{ScholzeWeinstein:Berkeley}.
	Since \((-)^\diamondsuit\) preserves closed immersions, and is invariant under taking absolute weak normalization \cite[Lemma 2.13]{AGLR:Local}, we also have closed immersions
	\[\mathscr{M}_{\Gg,\preccurlyeq\mu}^{\can,\diamondsuit} \into \mathscr{M}_{\Gg,\preccurlyeq\mu}^{\naive,\diamondsuit} \into \Gr_{\Gg,\Oo_E}^{\BD}.\]
	By definition, the v-sheaf local model \(\IM_{\Gg,\preccurlyeq\mu}^{\mathrm{AGLR}}\) from \cite[Definition 4.11]{AGLR:Local} is the v-closure of \(\Gr_{\Gg,E,\mu}^{\BD}\) inside \(\Gr_{\Gg,\Oo_E}^{\BD}\), i.e., the smallest closed sub-v-sheaf of \(\Gr_{\Gg,\Oo_E}^{\BD}\) containing \(\Gr_{\Gg,E,\mu}^{\BD}\).
	In particular, there is a closed immersion \(\IM^{\mathrm{AGLR}}_{\Gg,\preccurlyeq \mu} \into \mathscr{M}_{\Gg,\preccurlyeq\mu}^{\can,\diamondsuit}\), which induces an isomorphism on the generic fiber.
	By \cite[Theorem 7.21]{AGLR:Local} and \cite[Corollary 1.4]{GleasonLourenco:Tubular}, there is a normal flat projective \(\Oo_E\)-scheme \(\mathscr{M}_{\Gg,\preccurlyeq \mu}^{\mathrm{AGLR}}\) equipped with an isomorphism \(\mathscr{M}_{\Gg,\preccurlyeq\mu}^{\mathrm{AGLR},\diamondsuit} \cong \IM_{\Gg,\preccurlyeq\mu}^{\mathrm{AGLR}}\); this is exactly the schematic local model satisfying \cite[Conjecture 21.4.1]{ScholzeWeinstein:Berkeley}.
	
	By \cite[Theorem 2.16]{AGLR:Local} the closed immersion \(\IM^{\mathrm{AGLR}}_{\Gg,\preccurlyeq \mu} \into \mathscr{M}_{\Gg,\preccurlyeq\mu}^{\can,\diamondsuit}\) lifts to a unique morphism \(\alpha\colon \mathscr{M}_{\Gg,\preccurlyeq\mu}^{\mathrm{AGLR}} \to \mathscr{M}_{\Gg,\preccurlyeq\mu}^{\can}\).
	We claim \(\alpha\) is an isomorphism.
	As both schemes are absolutely weakly normal, it suffices to show \(\alpha\) is a universal homeomorphism.
	By flatness of \(\mathscr{M}_{\Gg,\preccurlyeq \mu}^{\can}\) and properness of \(\alpha\), it suffices to see that \(\alpha\) is universally injective.
	This can be checked fiberwise over \(\Spec \Oo_E\); as we already know \(\alpha\) induces an isomorphism on the generic fiber, it suffices to check the special fiber.
	In particular, by \cite[Lemma 3.4]{BhattScholze:Projectivity}, we may pass to perfections.
	But since \(\alpha^\diamondsuit\) is a closed immersion, hence a monomorphism by definition \cite[Definition 5.6]{Scholze:Etale}, \(\alpha^{\diamondsuit}\) induces an injection on field-valued points, hence the same holds for \(\alpha\) by \cite[Proposition 18.3.1]{ScholzeWeinstein:Berkeley}.
	This implies \(\alpha\) is universally injective, concluding the proof.
\end{proof}

\begin{cor}\thlabel{normality of local model}
	The canonical local model \(\mathscr{M}_{\Gg,\preccurlyeq\mu}^{\can}\) is normal.
\end{cor}
\begin{proof}
	This follows from \cite[Corollary 1.4]{GleasonLourenco:Tubular}.
\end{proof}

\begin{cor}\thlabel{local model and admissible locus}
	The perfection of the special fiber of \(\mathscr{M}_{\Gg,\preccurlyeq\mu}^{\can}\) agrees with the \(\mu\)-admissible locus of the affine flag variety \(\Fl_{\Gg,\preccurlyeq\mu} \subseteq \Fl_{\Gg}\).
\end{cor}
\begin{proof}
	By \thref{comparison local models}, this follows from \cite[Theorem 6.16]{AGLR:Local}.
\end{proof}

Finally, we have the splitting local model, introduced in \cite[§14]{PappasRapoport:LocalII.Splitting}.
Let \(\underline{\Lambda}\) be the \(\Ll\)-set of polarized \(\Oo_B\otimes_{\IZ_p} \Oo_S\)-modules given by \(\Lambda\mapsto \Lambda\otimes_{\IZ_p} \Oo_S\).

\begin{dfn}
	The \emph{splitting local model} \(\mathscr{M}_{\Gg,\preccurlyeq\mu}^{\spl}\) is the projective \(\Oo_{E^{\Gal}}\)-scheme representing the functor sending a scheme \(S\) to the set of isomorphism classes of splitting structures for \(\underline{\Lambda}\).
\end{dfn}

	Our approach throughout this paper will be to deduce results for the canonical models from similar results for the naive and splitting models, which are more accessible thanks to their moduli interpretations.
	For this, we will often assume that the splitting model is flat and absolutely weakly normal.
	Although this does not hold in general (as pointed out in \cite[§14]{PappasRapoport:LocalII.Splitting}, orthogonal groups and ramified unitary groups can cause problems), it is known in many cases.
	We refer to \cite{PappasRapoport:LocalII.Splitting,BijakowskiHernandez:GeometryPEL,BijakowskiHernandez:GeometryAR,ShenZheng:Fzips} for examples, and will highlight certain examples where the splitting models are even smooth in \thref{example splitting models} below.

\begin{ass}\thlabel{assumption splitting flat}
	The splitting model \(\mathscr{M}_{\Gg,\preccurlyeq\mu}^{\spl}\) (attached to a fixed local integral PEL datum) is flat over \(\Oo_{E^{\Gal}}\), and absolutely weakly normal.
\end{ass}

There is always a natural map \(\mathscr{M}_{\Gg,\preccurlyeq\mu}^{\spl}\to \mathscr{M}_{\Gg,\preccurlyeq\mu}^{\naive}\), sending a splitting structure to the assignment 
\[\Lambda\mapsto \left(\Lambda \otimes_{\IZ_p} \Oo_S\right)/\left(\bigoplus_{[\tau]\in \Upsilon/\sim} \mathscr{F}_{\Lambda,[\tau]}^0\right).\]
If we moreover assume \thref{assumption splitting flat}, this map factors uniquely through \(\mathscr{M}_{\Gg,\preccurlyeq \mu}^{\can}\), yielding maps
\[\mathscr{M}_{\Gg,\preccurlyeq\mu}^{\spl}\to \mathscr{M}_{\Gg,\preccurlyeq\mu}^{\can}\to \mathscr{M}_{\Gg,\preccurlyeq\mu}^{\naive}.\]
Furthermore, the parahoric group scheme \(\Gg\otimes_{\IZ_p} \Oo_E\) acts on the three local models (where we have to further base change to \(\Oo_{E^{\Gal}}\) for the splitting model), for which the above maps are equivariant.

\begin{lem}\thlabel{flatness of splitting models implies closure}
	Under \thref{assumption splitting flat}, the canonical model \(\mathscr{M}_{\Gg,\preccurlyeq\mu}^{\can}\) is the absolute weak normalization of the schematic image of the natural map \(\mathscr{M}_{\Gg,\preccurlyeq\mu}^{\spl} \to \mathscr{M}_{\Gg,\preccurlyeq\mu}^{\naive}\).
\end{lem}
\begin{proof}
	Since \(\mathscr{M}_{\Gg,\preccurlyeq\mu}^{\spl}\) is flat by assumption, it is the schematic closure (in itself) of its generic fiber.
	Moreover, \(\mathscr{M}_{\Gg,\preccurlyeq\mu}^{\spl}\to \mathscr{M}_{\Gg,\preccurlyeq\mu}^{\naive}\) induces an isomorphisms on generic fibers, up to a change of base field.
	Hence, the schematic closure of \(\mathscr{M}_{\Gg,\preccurlyeq\mu}^{\spl} \to\mathscr{M}_{\Gg,\preccurlyeq\mu}^{\naive}\) is the schematic closure in \(\mathscr{M}_{\Gg,\preccurlyeq\mu}^{\naive}\) of its generic fiber.
	Taking the absolute weak normalization then yields exactly the canonical local model.
\end{proof}

For the next corollary, we will work with perfections of special fibers of local models.
Since absolute weak normalization maps are universal homeomorphisms, which become isomorphisms after perfection, we see that the natural map \((\mathscr{M}_{\Gg,\preccurlyeq\mu}^{\can})_{\IF_q}^{\perf} \to (\mathscr{M}_{\Gg,\preccurlyeq\mu}^{\naive})_{\IF_q}^{\perf}\) is a closed immersion.

\begin{cor}\thlabel{corollary proper surjective}
	Under \thref{assumption splitting flat}, the natural map \((\mathscr{M}_{\Gg,\preccurlyeq\mu}^{\spl})_{\IF_q}^{\perf} \to (\mathscr{M}_{\Gg,\preccurlyeq\mu}^{\naive})_{\IF_q}^{\perf}\) factors as
	\[(\mathscr{M}_{\Gg,\preccurlyeq\mu}^{\spl})_{\IF_q}^{\perf} \to (\mathscr{M}_{\Gg,\preccurlyeq\mu}^{\can})_{\IF_q}^{\perf} \to (\mathscr{M}_{\Gg,\preccurlyeq\mu}^{\naive})_{\IF_q}^{\perf},\]
	where the first map is a perfectly proper surjective, and the second map a closed immersion.
\end{cor}
\begin{proof}
	The factorization follows from \thref{flatness of splitting models implies closure}, and the fact that universal homeomorphisms become isomorphisms after perfection.
	The second map is a closed immersion by the discussion above.
	Finally, the map is perfectly proper since all schemes are (perfectly) projective, and surjectivity again follows from \thref{flatness of splitting models implies closure}.
\end{proof}

\begin{ex}\thlabel{example splitting models}
	Let us give some examples where \thref{assumption splitting flat} is satisfied, and make the special fiber of the splitting model more explicit.
	We will assume that \(G_{\adj}=\Res_{F'/F} G'\) is the restriction of scalars of an unramified \(F'\)-group \(G'\) along a totally ramified extension \(F'/F\), and that \(\Gg\) is a very special parahoric of \(G\).
	(Recall that local models are invariant under central isogenies.)
	Then the corresponding splitting local model \(\mathscr{M}_{\Gg,\preccurlyeq \mu}^{\spl}\) is smooth over \(\Oo_{E^{\Gal}}\) by \cite[Proposition 2.3]{ShenZheng:Fzips}.
	More precisely, the proof shows that the special fiber of \(\mathscr{M}_{\Gg,\preccurlyeq\mu}^{\spl}\) is a convolution product of minuscule Schubert varieties in the (power series) affine Grassmannian for \(G'\).
	Note that by \cite[Lemma 3.15]{AGLR:Local}, these agree with the convolution product of minuscule Schubert varieties in the Witt vector affine Grassmannian from \cite{Zhu:Affine,BhattScholze:Projectivity}.
	
	More generally, under the same assumption on \(G_{\adj}\), but if \(\Gg\) is a general parahoric, \cite[Proposition A.6]{ShenZheng:Fzips} show that the splitting model \(\mathscr{M}_{\Gg,\preccurlyeq\mu}^{\spl}\) is flat over \(\Oo_{E^{\Gal}}\), and that its special fiber is given by the convolution product of admissible loci in convolution partial affine flag varieties.
	The same proof (i.e., by reducing to the case of local models for unramified groups) shows that \(\mathscr{M}_{\Gg,\preccurlyeq\mu}^{\spl}\) is normal, and in particular absolutely weakly normal.
	
	Finally, local models are compatible with products, and invariant under central isogenies.
	Moreover, the local model of a Weil restriction along an unramified restriction corresponds to the Weil restriction of the local model along the corresponding extension of residue fields.
	Thus, the example above suffices to describe the splitting local models more generally when \(G\) is essentially unramified in the sense of \thref{definition essentially unramified} below.
\end{ex}

\subsection{Integral models of Shimura varieties}\label{subsec:integral models}

Next, we move to the global setting, and introduce integral models of PEL type Shimura varieties.
For this, we fix a global integral PEL datum, and let \(\Sh_K(\IG,\IX)/\IE\) be the Shimura variety associated to the underlying rational datum.
For a place \(\pp\) of \(\IE\) with completion \(E\), there is a unique canonical integral model of this Shimura variety defined over \(\Oo_E\) \cite{Pappas:Integral,PappasRapoport:padic}.
However, this canonical model does not have a nice moduli interpretation in general.
To remedy this, we will also consider more general integral models of \(\Sh_K(\IG,\IX)\), which do have a moduli interpretation, and which are related to the canonical integral model.
The first model was defined in \cite[Definition 6.9]{RapoportZink:Period}, and we refer to loc.~cit.~for the unexplained terminology.

\begin{dfn}\thlabel{Defi:naive model}
	The \emph{naive integral model} \(\mathscr{S}_{K^pK_p}^{\naive}(\IG,\IX)\) is the quasi-projective \(\Oo_E\)-scheme representing the functor sending a test scheme \(S\) to the following data, up to isomorphism:
	\begin{enumerate}
		\item An \(\Ll\)-set \(A=\{A_\Lambda\}_{\Lambda\in \Ll}\) of abelian varieties over \(S\) (\cite[Definition 6.5]{RapoportZink:Period}),
		\item A \(\IQ\)-homogeneous polarization \(\overline{\lambda}\) of \(A\) (\cite[Definition 6.7]{RapoportZink:Period}),
		\item A \(K^p\)-level structure \(\overline{\eta}\colon H_1(A,\IA_f^p)\cong V\otimes \IA_f^p \mod K^p\), respecting the bilinear forms on both sides up to a scalar in \((\IA_f^p)^\times\),
	\end{enumerate}
	such that the determinant condition \(\det_{\Oo_S}(b\mid \mathrm{Lie} A_\Lambda) = \det_{\overline{\IQ}_p}(b\mid V_0)\) (as polynomials in \(b\in \Oo_B\)) holds for each \(\Lambda\in \Ll\).
\end{dfn}

As for the local model, this naive integral model is not flat in general, and hence cannot be the canonical integral model.
However, we can obtain a flat model as before.

\begin{dfn}\thlabel{Defi:canonical model}
	The canonical integral model \(\mathscr{S}_{K^pK_p}^{\can}(\IG,\IX)\) is the normalization of the schematic closure of the inclusion
	\[\Sh_K(\IG,\IX)_E\to \mathscr{S}_{K^pK_p}^{\naive}(\IG,\IX).\]
\end{dfn}

This definition is suitable for our purposes, and agrees with more modern constructions in the literature.
It will follow from \eqref{diagram of models} below, combined with \thref{normality of local model}, that it in fact suffices to take the absolute weak normalization in the definition above.
However, for the proof of the following proposition, it will be convenient to already know \(\mathscr{S}_{K^pK_p}^{\can}(\IG,\IX)\) is normal.

\begin{prop}
	The integral model from \thref{Defi:canonical model} satisfies the conditions from \cite[Conjecture 4.2.2]{PappasRapoport:padic}.
\end{prop}
\begin{proof}
	By \thref{remarks about PEL data} \eqref{Hodge embedding}, we can find a Hodge embedding of \(\Sh_K(\IG,\IX)\) into some Siegel modular variety, which satisfies the condition of \cite[Theorem 7.1.8]{KisinPappasZhou:Integral}.
	This Siegel modular variety admits a natural integral model with parahoric level at \(p\) given by the same lattice multichain \(\Ll\), defined by the usual moduli interpretation, which we will denote \(\mathscr{S}\).
	Consider the open and closed immersions
	\[\Sh_K(\IG,\IX) \into \mathscr{S}_{K^pK_p}^{\naive}(\IG,\IX) \into \mathscr{S},\]
	arising naturally from \thref{Defi:naive model}.
	Then the schematic closures of \(\Sh_K(\IG,\IX)\) in \(\mathscr{S}_{K^pK_p}^{\naive}(\IG,\IX)\) and \(\mathscr{S}\) agree, and \(\mathscr{S}_{K^pK_p}^{\can}(\IG,\IX)\) is by definition the normalization of this schematic closure.
	Thus \(\mathscr{S}_{K^pK_p}^{\can}(\IG,\IX)\) agrees with the integral model constructed in \cite[Theorem 7.1.8]{KisinPappasZhou:Integral}, so that it satisfies \cite[Conjecture 4.2.2]{PappasRapoport:padic} by \cite[Remark 7.1.9 b)]{KisinPappasZhou:Integral}.
\end{proof}

Finally, we also have a splitting integral model \cite[§15]{PappasRapoport:LocalII.Splitting}.

\begin{dfn}\thlabel{Defi:splitting model}
	The splitting model \(\mathscr{S}_{K^pK_p}^{\spl}(\IG,\IX)\) is the quasi-projective \(\Oo_{E^{\Gal}}\)-scheme representing the functor sending a scheme \(S\) the set of isomorphism classes of tuples \(\{(\{A_\Lambda\}_{\Lambda\in \Ll}, \overline{\lambda},\overline{\eta}, \underline{\mathscr{F}}_{[\tau]}^i, \underline{j}_{[\tau]}^i)\}\). 
	Here, \((\{A_\Lambda\}_{\Lambda\in \Ll}, \overline{\lambda},\overline{\eta})\) is an \(S\)-valued point of \(\mathscr{S}_{K^pK_p}^{\naive}(\IG,\IX)\).
	Moreover, \((\underline{\mathscr{F}}_{[\tau]}^i, \underline{j}_{[\tau]}^i)\) is a splitting structure for the \(\Ll\)-set of polarized \(\Oo_B\otimes_{\IZ} \Oo_S\)-modules \(\underline{\mathscr{H}}\), sending \(\Lambda\) to the Lie algebra of the universal vector extension of \(A_\Lambda\) (or equivalently, the first de Rham cohomology of \(A_\Lambda\) \cite[§4]{MazurMessing:Universal}).
	Finally, each \(\mathscr{F}_{\Lambda,[\tau]}^0\) is required to agree with the kernel of the natural surjection \(\mathscr{H}_{\Lambda,[\tau]}\to (\Lie A_\Lambda)_{[\tau]}\).
\end{dfn}

We refer to \cite[Lemma 2.3.9]{Lan:Compactifications} for an explicit proof that the generic fiber of this model agrees with \(\Sh_K(\IG,\IX)_{E^{\Gal}}\), and to \cite[Lemma 2.2.7]{Lan:Compactifications} for a proof of the fact that the Lie algebra of the universal vector extension determines an \(\Ll\)-set of polarized \(\Oo_B\otimes_{\IZ_p} \Oo_S\)-modules.
Note that there is always an obvious map \(\mathscr{S}_{K^pK_p}^{\spl}(\IG,\IX)\to \mathscr{S}^{\naive}_{K^pK_p}(\IG,\IX)\), obtained by forgetting the splitting structure.
If \thref{assumption splitting flat} holds, then as explained in \cite[§15]{PappasRapoport:LocalII.Splitting}, one has the following local model diagram with cartesian squares, where the vertical arrows are smooth of relative dimension \(\dim G\):
\begin{equation}\label{diagram of models}
	\begin{tikzcd}
		\mathscr{S}_{K^pK_p}^{\spl}(\IG,\IX) \arrow[d] \arrow[r] & \mathscr{S}_{K^pK_p}^{\can}(\IG,\IX) \arrow[d] \arrow[r] & \mathscr{S}_{K^pK_p}^{\naive}(\IG,\IX) \arrow[d] \\
		{[}\Gg_{\Oo_{E^{\Gal}}} \backslash \mathscr{M}_{\Gg,\preccurlyeq\mu}^{\spl}{]} \arrow[r] & {[}\Gg_{\Oo_{E}} \backslash \mathscr{M}_{\Gg,\preccurlyeq\mu}^{\can}{]} \arrow[r] & {[}\Gg_{\Oo_{E}} \backslash \mathscr{M}_{\Gg,\preccurlyeq\mu}^{\naive}{]}.
	\end{tikzcd}
\end{equation}
In particular, if any of the local models is flat or smooth, the same holds for the corresponding integral model of \(\Sh_K(\IG,\IX)\).
Finally, the following lemma can be proven similarly as \thref{flatness of splitting models implies closure}.

\begin{lem}
	If \thref{assumption splitting flat} holds, then the canonical model \(\mathscr{S}_{K^pK_p}^{\can}(\IG,\IX)\) is the absolute weak normalization of the schematic closure of the natural map \(\mathscr{S}_{K^pK_p}^{\spl}(\IG,\IX) \to \mathscr{S}_{K^pK_p}^{\naive}(\IG,\IX)\).
\end{lem}

\section{Displays and local shtukas}\label{Sec:Displays}

A key tool in \cite{XiaoZhu:Cycles} is the moduli stack of local shtukas, which is defined group-theoretically.
On the other hand, in order to construct exotic correspondences between the special fibers of canonical models of Shimura varieties, we will construct exotic correspondences between special fibers of splitting and naive models, which are defined via linear algebra.
To characterize these, we will define ``naive" and ``splitting" analogues of the moduli of local shtukas, also via linear algebra.
We will also need natural maps between these variants, such that the stack of (canonical) local shtukas is a closed substack of the stack of naive local shtukas, similar to the situation for local models and integral models of Shimura varieties.
Moreover, we will want natural maps from various models of Shimura varieties to moduli of local shtukas to be pro-perfectly smooth, for which it is easier to work with formal schemes, without passing to perfections.
In the canonical case, this was obtained by Hoff \cite{Hoff:Parahoric}, by constructing a stack of displays with extra structure, which can be obtained from the local model, and whose special fiber yields the moduli of local shtukas after perfection.
The goal of this section will be to generalize this to different local models, yielding naive and splitting versions of displays and local shtukas.
Throughout this section, we again assume \(p\neq 2\).

\begin{rmk}
	Note that \cite{Hoff:Parahoric} uses (among other things) the flatness of the canonical local model, so certain techniques do not directly apply to naive (or even splitting) models.
	Nevertheless, this can be resolved by appealing to their moduli interpretation.
	In particular, the assumption \cite[2.28]{Hoff:Parahoric} will not be relevant for us.
\end{rmk}

\subsection{Displays}

Let us fix an integral local PEL datum as in \thref{PEL data}, and let \(E\) be the field of definition of \(\mu\).
Throughout this section, \(R\) will denote a \(p\)-complete \(\Oo_E\)- or \(\Oo_{E^{\Gal}}\)-algebra, depending on whether we are in the naive, canonical, or splitting situation.
We also denote by \(\mathrm{M}_{\Gg,\preccurlyeq\mu}^{\spl}\to \mathrm{M}_{\Gg,\preccurlyeq\mu}^{\can}\to \mathrm{M}_{\Gg,\preccurlyeq\mu}^{\naive}\) the \(p\)-completions of \(\mathscr{M}_{\Gg,\preccurlyeq\mu}^{\spl}\to \mathscr{M}_{\Gg,\preccurlyeq\mu}^{\can}\to \mathscr{M}_{\Gg,\preccurlyeq\mu}^{\naive}\).

\begin{nota}
	\begin{enumerate}
		\item In this section and the next only, we will consider the ring of Witt vectors \(W(R)\) and its truncated version \(W_n(R)\) for any \(p\)-complete \(\IZ_p\)-algebra \(R\).
		We will denote the augmentation ideal, i.e., the kernel of \(W_n(R) \to R\), by \(I_R\subseteq W_n(R)\).
		If \(R\) is moreover a complete noetherian local \(\IZ_p\)-algebra with finite reside field (or an algebraic closure thereof), we will denote by \(\widehat{W}(R)\subseteq W(R)\) the associated Zink ring from \cite{Zink:Dieudonne}.
		Using this notation, we can extend the usual definition of the positive loop group \(L^+\Gg\) and its truncations \(L^n\Gg\) to arbitrary \(p\)-complete \(\IZ_p\)-algebras.
		\item Throughout this section, we will treat the case of truncated and non-truncated objects simultaneously, and omit them from the terminology.
		For this, we fix some \(m\geq 0\), and we allow \(m=\infty\).
		Then the \(\infty\)-truncated Witt vectors are defined to be the usual Witt vectors.
	\end{enumerate}
\end{nota}

\begin{dfn}
	A naive (resp.~canonical, resp.~splitting) \((\Gg,\mu)\)-pair over \(R\) is a pair \((\Pp,q)\), with \(\Pp\) a \(\Gg\)-torsor over \(W_m(R)\), and \(q\) a \(\Gg\)-equivariant map \(\Pp_R \to (\mathrm{M}_{\Gg,\preccurlyeq\mu}^{\naive})_R\) (resp.~\(\Pp_R\to (\mathrm{M}_{\Gg,\preccurlyeq\mu}^{\can})_R\), resp.~\(\Pp_R\to (\mathrm{M}_{\Gg,\preccurlyeq\mu}^{\spl})_R\)).
	
	If \(R\) is complete, noetherian, local, and with residue field \(\overline{\IF}_p\), a (naive, canonical, or splitting) \emph{Dieudonné \((\Gg,\mu)\)-pair} is defined as a similar pair \((\Pp,q)\), where instead \(\Pp\) is a \(\Gg\)-torsor over \(\widehat{W}(R)\).
\end{dfn}

There is an obvious notion of isomorphism of pairs.

\begin{lem}\thlabel{lemma pairs}
	Sending \(R\) to the groupoid of naive \((\Gg,\mu)\)-pairs defines a stack over \(\Spf \Oo_E\), denoted \(\Pair_{\Gg,\preccurlyeq\mu}^{\naive,(m)}\).
	There is a natural equivalence \(\Pair_{\Gg,\preccurlyeq\mu}^{\naive,(m)} \cong (L^m\Gg\backslash \mathrm{M}_{\Gg,\preccurlyeq\mu}^{\naive})^{\et}\).
	Similar statements hold for canonical and splitting pairs (where the latter define a stack over \(\Spf \Oo_{E^{\Gal}}\)).
\end{lem}
\begin{proof}
	It suffices to note that \(\Gg\)-torsors over \(W_m(R)\) are equivalent to \(L^m\Gg\)-torsors over \(R\) \cite[Lemma 2.12]{BueltelHedayatzadeh:Windows}, the rest is clear.
\end{proof}

In order to give linear algebraic moduli descriptions for these stacks in the naive and splitting case, we recall the notions of polarized multichains as in \cite[Definition 3.14]{RapoportZink:Period} (cf.~also \cite[Definition 21.6.2]{ScholzeWeinstein:Berkeley}).

\begin{dfn}\thlabel{defi-multichains}
	\begin{enumerate}
		\item\label{defi-multi1} An \emph{\(\Ll\)-multichain of \(\Oo_B\otimes_{\IZ_p} R\)-lattices} \(\{M_\Lambda\}\) is a functor
		\[\Ll\to \Oo_B \otimes_{\IZ_p} R\Mod\colon \Lambda\mapsto M_{\Lambda}\]
		where the transition maps are injective,
		together with periodicity isomorphisms \(\theta_\Lambda^b\colon M_\Lambda^b\cong M_{b\Lambda}\) for each \(b\in B^\times\) normalizing \(\Oo_B\).
		This data is moreover required to satisfy:
		\begin{enumerate}
			\item Locally on \(\Spec R\), \(M_{\Lambda}\) is isomorphic to \(\Lambda \otimes_{\IZ_p} R\) as \(\Oo_B \otimes_{\IZ_p} R\)-modules.
			\item In case \(F\) is a field, then for any two adjacent lattices \(\Lambda\subseteq \Lambda'\), the quotient \(M_{\Lambda'}/M_{\Lambda}\) is isomorphic to \(\Lambda'/\Lambda \otimes_{\IZ_p} R\) as \(\Oo_B\otimes_{\IZ_p} R\)-modules, locally on \(\Spec R\).
			In general, \(F\) decomposes as a product of fields, which induces a decomposition of \(V\) and \(\Ll\), and a similar condition is required after the projection onto each factor in this decomposition.
			\item The periodicity isomorphisms \(\theta_\Lambda^b\) commute with the transition maps \(M_{\Lambda} \to M_{\Lambda'}\).
			\item For \(b\in B^\times \cap\Oo_B\) normalizing \(\Oo_B\), the composition \(M_\Lambda^b \cong M_{b\Lambda} \to M_{\Lambda}\) is multiplication by \(b\).
		\end{enumerate}
		\item An \emph{\(\Ll\)-multichain of polarized \(\Oo_B \otimes_{\IZ_p} R\)-lattices} is a tuple \((\{M_\Lambda\},L,\lambda)\), where \(\{M_\Lambda\}\) is a multichain of \(\Oo_B\otimes_{\IZ_p} R\)-lattices, \(L\) is an invertible \(R\)-module, and \(\lambda\colon \{M_\Lambda\}\cong \{M^\vee_{\hat{\Lambda}}\otimes L\}\) is an antisymmetric isomorphism of multichains.
		\item An \emph{\(\Ll\)-multichain of polarized pairs} is a tuple \((\{M_\Lambda\},\{M_{1,\Lambda}\},L,\lambda)\).
		 Here, \((\{M_\Lambda\},L,\lambda)\) is an \(\Ll\)-multichain of polarized \(\Oo_B \otimes_{\IZ_p} W_m(R)\)-lattices, and \(\Lambda\mapsto M_{1,\Lambda}\subseteq M_\Lambda\) is a functor, where each \(M_{1,\Lambda}\subseteq M_\Lambda\) is the preimage of a direct summand \(M_{1,\Lambda}/I_RM_{\Lambda}\subseteq M_\Lambda/I_RM_\Lambda\),
		 such that \(\det_R(b\mid M_\Lambda/M_{1,\Lambda}) = \det_L(b\mid V_0)\) as polynomials in \(b\in \Oo_B\).
		 Finally, we assume the isomorphisms \(M_{\Lambda}\cong M^\vee_{\hat{\Lambda}}\otimes L\) induced by \(\lambda\) identify the submodules \(M_{1,\Lambda}\cong M^*_{1,\hat{\Lambda}}\otimes L\), where 
		 \[M^*_{1,\hat{\Lambda}} = \{\phi\in \Hom_{W_m(R)}(M_{\hat{\Lambda}},W_m(R))\mid \phi(M_{1,\hat{\Lambda}})\subseteq I_R\}.\]
		 \item If \(R\) is moreover complete, noetherian, local, and with residue field \(\overline{\IF}_p\), a \emph{Dieudonné \(\Ll\)-multichain of polarized pairs} is defined similarly as in (3), where instead \((\{M_{\Lambda}\},L,\lambda)\) is an \(\Ll\)-multichain of polarized \(\Oo_B \otimes_{\IZ_p} \widehat{W}(R)\)-lattices.
	\end{enumerate}
\end{dfn}

\begin{prop}\thlabel{moduli description of pairs}
	\begin{enumerate}
		\item The groupoid of naive \((\Gg,\mu)\)-pairs over \(R\) is naturally equivalent to the groupoid of \(\Ll\)-multichains of polarized pairs over \(R\).
		\item The groupoid of splitting \((\Gg,\mu)\)-pairs over \(R\) is naturally equivalent to the groupoid of \(\Ll\)-chains of polarized pairs over \(R\), together with a splitting structure on \(M_{1,\Lambda}/I_RM_\Lambda\).
	\end{enumerate}
	When \(R\) is complete, noetherian, local, and with residue field \(\overline{\IF}_p\), the same holds for the Dieudonné versions.
\end{prop}
\begin{proof}
	An \(\Ll\)-multichain of polarized lattices is equivalent to a \(\Gg\)-torsor over \(W_m(R)\) \cite[Corollary 21.6.6]{ScholzeWeinstein:Berkeley}.
	The additional datum \(M_{1,\Lambda}/I_RM_\Lambda\subseteq M_\Lambda/I_RM_\Lambda \cong M_\Lambda\otimes_{W_m(R)} R\) (or rather, the quotient by this submodule) is then equivalent to a map \(\Pp_R\to (\mathrm{M}_{\Gg,\mu}^{\naive})_R\).
	
	In the splitting case, a splitting structure on \(M_{1,\Lambda}/I_RM_\Lambda\) is equivalent to a lift \(\Pp_R\to (\mathrm{M}_{\Gg,\mu}^{\spl})_R\).
	The Dieudonné case can be handled similarly-
\end{proof}

Next, we would like to define displays.
We fix some \(n>0\) (where \(n=\infty\) is allowed), such that \(m\geq n+1\).
Recall that to a pair \((M,M_1)\), one can associate a finite projective \(W_n(R)\)-module \(\widetilde{M_1}\) as in \cite[Definition 2.7]{Hoff:EKOR}; this forms a \(\sigma\)-linear functor.
Let us denote by \((-)^\sigma\) the Frobenius twist.

\begin{prop}\thlabel{chains to pairs}
	The functor \((M,M_1)\mapsto \widetilde{M_1}\) above induces a natural functor
	\[(\{M_\Lambda\},\{M_{1,\Lambda}\},L,\lambda)\mapsto (\widetilde{M_{1,\Lambda}}, L^\sigma,\widetilde{\lambda}),\]
	from the category of \(\Ll\)-multichains of polarized pairs over \(R\) to the category of \(\Ll\)-multichains of polarized \(\Oo_B\otimes_{\IZ_p} W_n(R)\)-lattices.
	
	If \(R\) is complete, noetherian, local, and with residue field \(\overline{\IF}_p\), we get a similar functor from the category of Dieudonné \(\Ll\)-multichains of polarized pairs to the category of \(\Ll\)-multichains of polarized \(\Oo_B \otimes_{\IZ_p} \widehat{W}(R)\)-lattices.
\end{prop}
\begin{proof}
	By \cite[Definition 2.7]{Hoff:EKOR}, the assignment \((M,M_1) \mapsto \widetilde{M_1}\) is functorial and preserves injections.
	In particular, we get the required periodicity isomorphisms \(\widetilde{M_{1,\Lambda}}^b\cong \widetilde{M_{1,b\Lambda}}\).
	Thus, to see that \(\{\widetilde{M_{1,\Lambda}}\}\) forms an \(\Ll\)-multichain of \(\Oo_B \otimes_{\IZ_p} W_m(R)\)-lattices, we have to check the conditions of \thref{defi-multichains} \eqref{defi-multi1} are satisfied.
	Conditions (c) and (d) follow from functoriality, (a) follows from functoriality and the fact that \(\widetilde{M_{1,\Lambda}}\) has the correct rank (this is included in \cite[Definition 2,7]{Hoff:EKOR}) , and finally (b) can be checked similarly to \cite[Proposition 3.5]{Hoff:EKOR}.
	Finally, the fact that \((\widetilde{M_{1,\Lambda}}, L^\sigma,\widetilde{\lambda})\) forms a multichain of polarized lattices follows since \((M,M_1)\mapsto \widetilde{M_1}\) is compatible with duals and twists \cite[Remark 3.10]{Hoff:EKOR}.
	
	The Dieudonné case can be handled similarly (cf.~also \cite[Proposition 4.3.3]{Hoff:Thesis}).
\end{proof}

In particular, by \cite[Corollary 21.6.6]{ScholzeWeinstein:Berkeley} and \cite[Lemma 2.12]{BueltelHedayatzadeh:Windows}, we get a morphism
\[\Pair_{\Gg,\preccurlyeq\mu}^{\naive,(m)} \to \{L^n\Gg\text{-torsors over } R\}.\]

\begin{dfn}
	The above discussion determines an \(L^n\Gg\)-torsor over \((L^m\Gg\backslash \mathrm{M}_{\Gg,\preccurlyeq\mu}^{\naive})^{\et}\), i.e., an \(L^m\Gg\)-equivariant \(L^n\Gg\)-torsor over \(\mathrm{M}_{\Gg,\preccurlyeq\mu}^{\naive}\), denoted \(\mathrm{M}_{\Gg,\preccurlyeq\mu}^{\naive,(n)}\).
	These can be pulled back to \(L^m\Gg\)-equivariant \(L^n\Gg\)-torsors \(\mathrm{M}_{\Gg,\preccurlyeq\mu}^{\can,(n)}\to \mathrm{M}_{\Gg,\preccurlyeq\mu}^{\can}\) and \(\mathrm{M}_{\Gg,\preccurlyeq\mu}^{\spl,(n)}\to \mathrm{M}_{\Gg,\preccurlyeq\mu}^{\spl}\).
\end{dfn}

Note that, in the canonical case, this torsor agrees with the one constructed in \cite{Hoff:Parahoric} (when the construction in loc.~cit.~applies): this follows from \cite[Propositions 4.3.8 and 4.6.7]{Hoff:Thesis}.
We can now introduce the notions of displays, attached to the various local models.

\begin{dfn}\thlabel{defi display}
	The stack of naive \((\Gg,\mu)\)-displays is the quotient stack \(\Disp_{\Gg,\preccurlyeq\mu}^{\naive,(m,n)}:=(L^m\Gg\backslash \mathrm{M}_{\Gg,\preccurlyeq\mu}^{\naive,(n)})^{\et}\), where the quotient is taken for the diagonal action.
	The stacks of canonical and splitting \((\Gg,\mu)\)-displays are defined similarly.
\end{dfn}

In the naive and splitting cases, these have natural moduli interpretations.

\begin{prop}\thlabel{moduli description of displays}
	The groupoid \(\Disp_{\Gg,\preccurlyeq\mu}^{\naive,(m,n)}(R)\) is equivalent to the groupoid of tuples \[(\{M_\Lambda\},\{M_{1,\Lambda}\},L,\lambda,\Psi),\] with \((\{M_\Lambda\},\{M_{1,\Lambda}\},L,\lambda)\in \Pair_{\Gg,\preccurlyeq\mu}^{\naive,(m)}(R)\), and \(\Psi\colon (\{M_\Lambda\},L,\lambda)\cong (\widetilde{M_{1,\Lambda}}, L^\sigma,\widetilde{\lambda})\) an isomorphism of \(\Ll\)-multichains of polarized \(\Oo_B\otimes_{\IZ_p} W_n(R)\)-lattices.
	
	A similar description holds for \(\Disp_{\Gg,\preccurlyeq\mu}^{\spl,(m,n)}(R)\), by replacing \(\Pair_{\Gg,\preccurlyeq\mu}^{\naive,(m)}(R)\) with \(\Pair_{\Gg,\preccurlyeq\mu}^{\spl,(m)}(R)\).
\end{prop}
\begin{proof}
	This follows from the definitions, as well as \cite[Corollary 21.6.6]{ScholzeWeinstein:Berkeley} and \cite[Lemma 2.12]{BueltelHedayatzadeh:Windows}.
\end{proof}

The groupoid of canonical displays has been described in \cite{Hoff:Parahoric}, but this description will not be relevant for us.
On the other hand, the above proposition suggests the following notion of Dieudonné displays, which we will only need in the naive case.

\begin{dfn}
	Let \(R\) be a complete, noetherian, local \(\Oo_E\)-algebra, with residue field \(\overline{\IF}_p\).
	Then a \emph{naive Dieudonné \((\Gg,\mu)\)-display} over \(R\) is the datum of a Dieudonné \(\Ll\)-multichain of polarized pairs \((\{M_\Lambda\},\{M_{1,\Lambda}\},L,\lambda)\) over \(R\), together with an isomorphism \((\{M_{\Lambda}\},L,\lambda) \cong (\widetilde{M_{1,\Lambda}},L^\sigma,\widetilde{\lambda})\) of \(\Ll\)-multichains of polarized \(\Oo_B \otimes_{\IZ_p} \widehat{W}(R)\)-lattices.
	A \emph{splitting Dieudonné \((\Gg,\mu)\)-display} can be defined similarly.
\end{dfn}

In order to connect the stacks of displays to Shimura varieties, we will use chains of \(p\)-divisible groups as in \cite[Definition 24.3.1]{ScholzeWeinstein:Berkeley}.

\begin{dfn}
	\begin{enumerate}
		\item An \(\Ll\)-multichain of \(\Oo_B\)-\(p\)-divisible groups over \(R\) is a functor \(\Lambda\mapsto X_\Lambda\) to \(p\)-divisible groups over \(R\) with an \(\Oo_B\)-action, together with periodicity isomorphisms \(\theta_\Lambda^b\colon X_\Lambda^b\cong X_{b\Lambda}\), satisfying the following conditions:
		\begin{enumerate}
			\item If \(M_\Lambda\) is the Lie algebra of the universal vector extension of \(X_\Lambda\), then \(\Lambda\mapsto M_\Lambda\) defines an \(\Ll\)-multichain of \(\Oo_B\otimes_{\IZ_p} R\)-lattices.
			\item The periodicity isomorphisms \(\theta_\Lambda^b\) commute with the transition maps \(X_\Lambda\to X_{\Lambda'}\).
			\item For \(b\in B^\times \cap \Oo_B\) normalizing \(\Oo_B\), the composition \(X_\Lambda^b\cong X_{b\Lambda} \to X_\Lambda\) is multiplication by \(b\).
		\end{enumerate}
		\item An \(\Ll\)-multichain of polarized \(\Oo_B\)-\(p\)-divisible groups is a tuple \((\{X_\Lambda\},\IL,\lambda)\), with \({X_\Lambda}\) an \(\Ll\)-multichain of \(\Oo_B\)-\(p\)-divisible groups, \(\IL\) an étale rank 1 \(\IZ_p\)-local system on \(\Spec R\), and \(\lambda\colon X_\Lambda\cong X_{\hat{\Lambda}}^\vee\otimes_{\IZ_p} \IL\) an antisymmetric isomorphism of multichains.
	\end{enumerate}
\end{dfn}

We now fix a global integral (at \(p\)) PEL datum as in \thref{PEL data}, lifting the local PEL datum fixed at the beginning of this section.
We will denote the corresponding Shimura datum by \((\IG,\IX)\), and by \(\widehat{\mathscr{S}}_{K^pK_p}^{\naive}(\IG,\IX)\) etc.~the \(p\)-adic completions of the various integral models of \(\Sh_K(\IG,\IX)\). 

\begin{prop}\thlabel{shimura to display}
	There is a natural commutative diagram
	\[\begin{tikzcd}
		\widehat{\mathscr{S}}^{\spl}_{K^pK_p}(\IG,\IX) \arrow[d] \arrow[r] & \widehat{\mathscr{S}}^{\can}_{K^pK_p}(\IG,\IX) \arrow[d] \arrow[r] & \widehat{\mathscr{S}}^{\naive}_{K^pK_p}(\IG,\IX) \arrow[d]\\
		\Disp_{\Gg,\preccurlyeq\mu}^{\spl,(m,n)} \arrow[r] \arrow[d] & \Disp_{\Gg,\preccurlyeq\mu}^{\can,(m,n)} \arrow[r] \arrow[d] & \Disp_{\Gg,\preccurlyeq\mu}^{\naive,(m,n)} \arrow[d]\\
		\Pair_{\Gg,\preccurlyeq\mu}^{\spl,(m)} \arrow[r] \arrow[d] & \Pair_{\Gg,\preccurlyeq\mu}^{\can,(m)} \arrow[r] \arrow[d] & \Pair_{\Gg,\preccurlyeq\mu}^{\naive,(m)} \arrow[d]\\
		\Gg_{\Oo_{E^{\Gal}}}\backslash \mathrm{M}_{\Gg,\preccurlyeq\mu}^{\spl} \arrow[r] & \Gg_{\Oo_{E}} \backslash \mathrm{M}_{\Gg,\preccurlyeq\mu}^{\can} \arrow[r] & \Gg_{\Oo_{E}} \backslash \mathrm{M}_{\Gg,\preccurlyeq\mu}^{\naive},
	\end{tikzcd}\]
	with cartesian squares.
\end{prop}
\begin{proof}
	The models \(\mathscr{S}_{K^pK_p}^{\naive}(\IG,\IX)\) and \(\mathscr{S}_{K^pK_p}^{\spl}(\IG,\IX)\) parametrize abelian varieties with extra structure, as in \thref{Defi:naive model}.
	In particular, we can associate to any \(R\)-valued point of their completions an \(\Ll\)-multichain of polarized \(\Oo_B\)-\(p\)-divisible groups, and then also a (naive or splitting) display by \cite[Proposition 2.1]{Lau:Smoothness} and \thref{moduli description of displays} (although we use the contravariant normalization as in \cite[Remark 1.30]{Hoff:Parahoric}).
	We note that to pass from polarized \(p\)-divisible groups to polarized displays, the \(\IZ_p\)-local system \(\IL\) induces a line bundle on \(R\) via \(R\otimes_{\IZ_p} \IL\).
	This yields the desired maps \(\widehat{\mathscr{S}}_{K^pK_p}^{\naive}(\IG,\IX) \to \Disp_{\Gg,\preccurlyeq\mu}^{\naive,(m,n)}\) and \(\widehat{\mathscr{S}}_{K^pK_p}^{\spl}(\IG,\IX) \to \Disp_{\Gg,\preccurlyeq\mu}^{\spl,(m,n)}\).
	Combined with \thref{lemma pairs} and \thref{defi display}, we get the following commutative diagram:
	\[\begin{tikzcd}
		\widehat{\mathscr{S}}^{\spl}_{K^pK_p}(\IG,\IX) \arrow[d] \arrow[r] & \widehat{\mathscr{S}}^{\can}_{K^pK_p}(\IG,\IX) \arrow[r] & \widehat{\mathscr{S}}^{\naive}_{K^pK_p}(\IG,\IX) \arrow[d]\\
		\Disp_{\Gg,\preccurlyeq\mu}^{\spl,(m,n)} \arrow[r] \arrow[d] & \Disp_{\Gg,\preccurlyeq\mu}^{\can,(m,n)} \arrow[r] \arrow[d] & \Disp_{\Gg,\preccurlyeq\mu}^{\naive,(m,n)} \arrow[d]\\
		\Pair_{\Gg,\preccurlyeq\mu}^{\spl,(m)} \arrow[r] \arrow[d] & \Pair_{\Gg,\preccurlyeq\mu}^{\can,(m)} \arrow[r] \arrow[d] & \Pair_{\Gg,\preccurlyeq\mu}^{\naive,(m)} \arrow[d]\\
		\Gg_{\Oo_{E^{\Gal}}}\backslash \mathrm{M}_{\Gg,\preccurlyeq\mu}^{\spl} \arrow[r] & \Gg_{\Oo_{E}} \backslash \mathrm{M}_{\Gg,\preccurlyeq\mu}^{\can} \arrow[r] & \Gg_{\Oo_{E}} \backslash \mathrm{M}_{\Gg,\preccurlyeq\mu}^{\naive}.
	\end{tikzcd}\]
	Since we already have a map \(\mathscr{S}_{K^pK_p}^{\can}(\IG,\IX) \to \Gg_{\Oo_{E}} \backslash \mathscr{M}_{\Gg,\preccurlyeq \mu}^{\can}\) \eqref{diagram of models}, the remaining map \(\widehat{\mathscr{S}}_{K^pK_p}^{\can}(\IG,\IX) \to \Disp_{\Gg,\preccurlyeq\mu}^{\can,(m,n)}\) will follow once we know the squares in the two bottom rows are cartesian.
	The cartesianness of the upper squares will then follow since the squares in \eqref{diagram of models} are cartesian.
	
	Since \(\Pair_{\Gg,\preccurlyeq\mu}^{\naive,(m)} \cong (L^{m}\Gg\backslash \mathrm{M}_{\Gg,\preccurlyeq\mu}^{\naive})^{\et}\) and \(\Disp_{\Gg,\preccurlyeq\mu}^{\naive,(m,n)}\cong (L^m\Gg\backslash \mathrm{M}_{\Gg,\preccurlyeq\mu}^{\naive,(n)})^{\et}\), and similarly for the canonical and splitting versions, the middle vertical arrows are \(L^m\Gg\)-quotients of \(L^n\Gg\)-torsors, so that the middle squares are cartesian.
	On the other hand, the descriptions \(\Pair_{\Gg,\preccurlyeq\mu}^{\naive,(m)}\cong (L^m\Gg \backslash \mathrm{M}_{\Gg,\preccurlyeq \mu}^{\naive})^{\et}\), and so on, also shows that the lower squares are cartesian, as desired.
\end{proof}

Finally, the maps \(\widehat{\mathscr{S}}_{K^pK_p}(\IG,\IX) \to \Disp_{\Gg,\preccurlyeq\mu}^{(m,n)}\) have the following important geometric property, generalizing \cite[Theorem 2.56]{Hoff:Parahoric} to different local models.

\begin{thm}\thlabel{smoothness of shimura to display}
	Assume that \(m,n\) are finite.
	Then the maps \(\widehat{\mathscr{S}}^{\naive}_{K^pK_p}(\IG,\IX) \to \Disp_{\Gg,\preccurlyeq\mu}^{\naive,(m,n)}\), \(\widehat{\mathscr{S}}^{\can}_{K^pK_p}(\IG,\IX) \to \Disp_{\Gg,\preccurlyeq\mu}^{\can,(m,n)}\), and \(\widehat{\mathscr{S}}^{\spl}_{K^pK_p}(\IG,\IX) \to \Disp_{\Gg,\preccurlyeq\mu}^{\spl,(m,n)}\) are smooth.
\end{thm}
\begin{proof}
	By \thref{shimura to display}, it suffices to consider the naive models (although the case of splitting models can also be handled directly, via a similar argument).
	
	Let \(x\in \mathscr{S}_{K^pK_p}^{\naive}(\IG,\IX)(\overline{\IF}_p)\) be a point, and let \(R\) be the completed local ring of \(\mathscr{S}_{K^pK_p}^{\naive}(\IG,\IX)\) at \(x\).
	Then the moduli interpretation of \(\mathscr{S}_{K^pK_p}^{\naive}(\IG,\IX)\) gives an abelian variety with extra structure over \(R\), and hence an \(\Ll\)-multichain of polarized \(\Oo_B\)-\(p\)-divisible groups over \(R\).
	By Dieudonné theory \cite{Zink:Dieudonne}, this yields a naive Dieudonné \((\Gg,\mu)\)-display \(M_x\) over \(R\).
	We claim that \(M_x\) is the universal deformation of its base change to \(x\).
	In other words, the deformation functor sending a complete noetherian local \(\Oo_E\)-algebra \(R\) with residue field \(\overline{\IF}_p\) to the set of deformations (as naive Dieudonné \((\Gg,\mu)\)-displays) of the base change of \(M_x\) to \(x\) over \(R\), is pro-represented by the completed local ring of \(\mathscr{S}_{K^pK_p}^{\naive}(\IG,\IX)\) at \(x\), with universal object given by \(M_x\).
	Indeed, recall that the Serre--Tate theorem \cite[Theorem V.2.3]{Messing:Crystals} tells us that the deformation theory of an abelian variety and its underlying \(p\)-divisible group agree.
	Moreover, over a complete noetherian local \(\Oo_E\)-algebra \(R\) with residue field \(\overline{\IF}_p\), the category of naive \((\Gg,\mu)\)-displays is equivalent to the category of \(\Ll\)-multichains of polarized \(\Oo_B\)-\(p\)-divisible groups; this follows from the main theorem of \cite{Zink:Dieudonne}, combined with \cite[Corollary 4.3.2 and Lemma 4.7.4]{Hoff:Thesis} to handle the extra structure.
	This shows that \(M_x\) is indeed the universal deformation of its base change to \(x\).
	
	Now, to prove that \(\widehat{\mathscr{S}}^{\naive}_{K^pK_p}(\IG,\IX) \to \Disp_{\Gg,\preccurlyeq\mu}^{\naive,(m,n)}\) is smooth, the infinitesimal lifting criterion \cite[Lemma 4.9]{Hoff:EKOR} implies it suffices to show every commutative diagram
	\[\begin{tikzcd}
		\Spec R \arrow[d] \arrow[r] & \widehat{\mathscr{S}}_{K^pK_p}^{\naive}(\IG,\IX) \arrow[d]\\
		\Spec R' \arrow[r] & \Disp_{\Gg,\preccurlyeq\mu}^{\naive,(m,n)},
	\end{tikzcd}\]
	where \(R'\to R\) is a surjection of Artinian local \(\Oo_E\)-algebras with residue fields \(\overline{\IF}_p\),
	admits a lift \(\Spec R'\to \widehat{\mathscr{S}}_{K^pK_p}^{\naive}(\IG,\IX)\).
	But this follows from the universality of the Dieudonné display \(M_x\), and the fact that \(\Disp_{\Gg,\preccurlyeq\mu}^{\naive,(\infty,\infty)} \to \Disp_{\Gg,\preccurlyeq\mu}^{\naive,(m,n)}\) is formally smooth (which in turn directly follows from \thref{defi display}).
\end{proof}

\subsection{\(\mu\)-bounded local shtukas}\label{Sec: bounded local shtukas}

The stack of local sthukas (at level \(\Gg\)) from \cite[§5.2]{XiaoZhu:Cycles} was defined as the étale stackification of the quotient of \(LG\) by the Frobenius-twisted conjugation action of \(L^+\Gg\), where everything is defined over \(\IF_q\).
This group-theoretic definition yields connections to affine flag varieties, affine Deligne--Lusztig varieties, and the categorical local Langlands correspondance via the stack of \(G\)-isocrystals \cite{Zhu:Coherent,Zhu:Tame}.
Nevertheless, in order to connect it to Shimura varieties, it is useful to have linear algebraic descriptions.
The following definition is motivated by \cite[Proposition 2.41]{Hoff:Parahoric}, and depends only on the local integral PEL datum.

\begin{dfn}
	The \emph{stack of naive local shtukas} is defined as the perfection of the special fiber of \(\Disp_{\Gg,\preccurlyeq\mu}^{\naive,(m.n)}\); it is denoted by
	\[\Sht_{\Gg,\preccurlyeq\mu}^{\naive,(m,n)}:=(\Disp_{\Gg,\preccurlyeq\mu}^{\naive,(m,n)})^{\perf}_{\IF_q}.\]
	The stack of \emph{canonical} or \emph{splitting local shtukas} are defined and denoted similarly.
\end{dfn}

\begin{rmk}\thlabel{group theoretic interpretation of sht}
	\thref{defi display} implies that this definition of \(\Sht_{\Gg,\preccurlyeq}^{\can}\) agrees (up to renormalization) with the group-theoretic definitions of the moduli of truncated local sthukas from \cite{XiaoZhu:Cycles,ShenYuZhang:EKOR,Zhu:Tame} (cf.~also \cite[Proposition 2.41]{Hoff:Parahoric}).
	In particular, when \(m=n=\infty\), we see that \(\Sht_{\Gg,\preccurlyeq}^{\can,\infty,\infty}\) admits a natural closed immersion into \((LG/\Ad_{\sigma^{-1}}L^+\Gg)^{\et}\).
\end{rmk}
From the definitions, we get natural maps
\[\Sht^{\spl,(m,n)}_{\Gg,\preccurlyeq\mu}\to \Sht^{\can,(m,n)}_{\Gg,\preccurlyeq\mu}\to \Sht^{\naive,(m,n)}_{\Gg,\preccurlyeq\mu},\]
where the second map is a closed immersion.

We can also provide moduli descriptions for \(\Sht_{\Gg,\preccurlyeq\mu}^{\naive}\) and \(\Sht_{\Gg,\preccurlyeq\mu}^{\spl}\).

\begin{lem}\thlabel{moduli description shtukas}
	Let \(R\) be a perfect \(\IF_q\)-algebra.
	Then \(\Sht_{\Gg,\preccurlyeq\mu}^{\naive}(R)\) is equivalent to the groupoid of \(\Ll\)-multichains of polarized \(\Oo_B\)-\(p\)-divisible groups \(\{X_\Lambda\}_{\Lambda\in \Ll}\) such that \(\det_R(b\mid \Lie X_\Lambda) = \det_{\overline{\IQ}_p}(b\mid V_0)\) as polynomials in \(b\in \Oo_B\).
	
	Moreover, \(\Sht_{\Gg,\preccurlyeq\mu}^{\spl}(R)\) is equivalent to the groupoid of tuples \((X,\underline{\mathscr{F}}_{[\tau]}^i, \underline{j}_{[\tau]}^i)\), where \(X\in \Sht_{\Gg,\mu}^{\naive}(R)\) and \((\underline{\mathscr{F}}_{[\tau]}^i, \underline{j}_{[\tau]}^i)\) is a splitting structure for the \(\Ll\)-set of polarized \(\Oo_B\otimes_{\IZ} R\)-modules \(\underline{\mathscr{H}}\), sending \(\Lambda\) to the Lie algebra of the universal vector extension of \(X_\Lambda\).
\end{lem}
\begin{proof}
	By Artin--Schreier--Witt theory (cf.~the proof of \cite[Theorem 12.3.4]{ScholzeWeinstein:Berkeley}, or \cite[Lemma 2.22]{BartlingHoff:Moduli}), étale rank 1 \(\IZ_p\)-local systems on \(\Spec R\) are equivalent to line bundles on \(\Spec W(R)\) with an isomorphism to their Frobenius-twist, when \(R\) is a perfect \(\IF_p\)-algebra.
	Thus, the lemma follows from Propositions \ref{moduli description of pairs} and \ref{moduli description of displays}, since displays are equivalent to \(p\)-divisible groups over perfect rings by Dieudonné theory \cite[Theorem 6.4]{Lau:Smoothness}.
\end{proof}

\begin{rmk}
	We could have used the above moduli interpretations as definitions of \(\Sht_{\Gg,\preccurlyeq\mu}^{\naive}\) and \(\Sht_{\Gg,\preccurlyeq\mu}^{\spl}\).
	However, this would make it less clear how to connect these stacks to the group-theoretical definition of \(\Sht_{\Gg,\preccurlyeq\mu}^{\can}\), whence the connection to local models.
\end{rmk}

Finally, let us make explicit the connection between the special fibers of Shimura varieties and moduli of local shtukas, for which we fix a global integral (at \(p\)) PEL datum.
Let us denote by \(\Sh_{\mu,K}^{\naive}\) the perfection of the special fiber of \(\mathscr{S}^{\naive}(\IG,\IX)\), and similarly for the canonical and splitting models.
Applying \thref{shimura to display} and passing to the perfection of the special fiber, we obtain the following diagram with cartesian squares:
\begin{equation}\label{shimura varieties to shtukas}\begin{tikzcd}
		\Sh_{\mu,K}^{\spl} \arrow[r] \arrow[d, "\loc_p^{\spl,(m,n)}"'] & \Sh_{\mu,K}^{\can} \arrow[r] \arrow[d, "\loc_p^{\can,(m,n)}"] & \Sh_{\mu,K}^{\naive} \arrow[d, "\loc_p^{\naive,(m,n)}"]\\
		\Sht_{\Gg,\preccurlyeq\mu}^{\spl,(m,n)} \arrow[r] & \Sht_{\Gg,\preccurlyeq\mu}^{\can,(m,n)} \arrow[r] & \Sht_{\Gg,\preccurlyeq\mu}^{\naive,(m,n)}.
\end{tikzcd}\end{equation}
These maps \(\loc_p\) are usually called the \emph{crystalline period maps}.
At least for the naive and splitting models, they  are given by sending an abelian variety with extra structure to its underlying \(p\)-divisible group (with similar extra structure).
\thref{smoothness of shimura to display} immediately implies the following.

\begin{cor}\thlabel{smoothness of shimura to shtuka}
	If \(m>n\) are finite, the crystalline period map \(\loc_p^{\naive,(m,n)}\colon \Sh_{\mu,K}^{\naive} \to \Sht_{\Gg,\preccurlyeq\mu}^{\naive,(m,n)}\) is perfectly smooth.
	The same holds for \(\loc_p^{\can,(m,n)}\) and \(\loc_p^{\spl,(m,n)}\).
\end{cor}

Finally, we note that a splitting structure on an abelian variety is equivalent to a splitting structure on the associated \(p\)-divisible group \cite[1.12]{MazurMessing:Universal}.
This gives another proof that the big rectangle in \eqref{shimura varieties to shtukas} is cartesian.

\section{Exotic Hecke correspondences}\label{Sec:Exotic}

Exotic Hecke correspondences are correspondences between the special fibers of possibly different Shimura varieties, which do not necessarily lift to the integral models.
They have striking applications in the Langlands program, such as geometric realizations of Jacquet--Langlands transfers between the cohomology of different Shimura varieties.
Examples of such exotic correspondences have appeared in \cite{Helm:Towards,TianXiao:Tate,HelmTianXiao:Tate}, and more general constructions have been studied in \cite{XiaoZhu:Cycles}.
However, all these references concern Shimura varieties with good reduction, i.e., with hyperspecial level at \(p\), so that the groups appearing in the Shimura data are unramified at \(p\).
We will give examples of such exotic correspondences in ramified cases.
In this section, we will only work in positive characteristic and with perfect objects.
We also set \(m=n=\infty\), i.e., we work with non-truncated objects, and remove \((m,n)\) from the notation.
In order to use moduli interpretations, we continue to assume \(p\neq 2\).

\subsection{Hecke correspondences between moduli of local shtukas}

In order to characterize exotic Hecke correspondences, let us define certain correspondences on moduli of local shtukas.
We fix two local integral PEL data \((B,*,V,(,),\mu_i,\Oo_B,\Ll)\) for \(i=1,2\), i.e., only the cocharacters \(\mu_i\) are allowed to differ.
Consider the quotient stack \(\BB(G)':=(LG/\Ad_{\sigma^{-1}} LG)^{\et}\).

\begin{lem}
	There is a natural morphism \(\Sht_{\Gg,\preccurlyeq\mu}^{\naive}\to \BB(G)'\), extending the usual map \(\Sht_{\Gg,\preccurlyeq\mu}^{\can}\subseteq LG/\Ad_{\sigma^{-1}} L^+\Gg\to LG/\Ad_{\sigma^{-1}} LG = \BB(G)'\) induced by \thref{group theoretic interpretation of sht}.
\end{lem}
\begin{proof}
	By \cite[Corollary 21.6.10 and Proposition 18.3.1]{ScholzeWeinstein:Berkeley}, there is a natural closed immersion \(L^+\Gg\backslash \mathscr{M}_{\Gg,\preccurlyeq\mu}^{\naive}\to L^+\Gg\backslash \Fl_{\Gg}\), extending the identification of the canonical local model with the \(\mu\)-admissible locus from \thref{local model and admissible locus}.
	Base changing this closed immersion along \(LG/\Ad_{\sigma^{-1}} L^+\Gg\to L^+\Gg\backslash \Fl_{\Gg}\) yields an \(L^+\Gg\)-fibration over \(L^+\Gg\backslash\mathscr{M}_{\Gg,\preccurlyeq\mu}^{\naive}\), which can be identified with \(\Sht_{\Gg,\preccurlyeq\mu}^{\naive}\).
	We obtained the desired map as the composition
	\[\Sht_{\Gg,\preccurlyeq\mu}^{\naive}\to LG/\Ad_{\sigma^{-1}} L^+\Gg\to LG/\Ad_{\sigma^{-1}} LG = \BB(G)'.\]
\end{proof}

\begin{rmk}\thlabel{remark modified Kottwitz stack}
	\begin{enumerate}
		\item The stack \(\BB(G)':= (LG /\Ad_{\sigma^{-1}} LG)^{\et}\) is similar to the Kottwitz stack \(\BB(G) := (LG/ \Ad_\sigma LG)^{\et}\) from \cite{Zhu:Coherent,Zhu:Tame}, and explains our notation \(\BB(G)'\).
		We have chosen to work with \(\BB(G)'\), to follow the conventions of \cite{XiaoZhu:Cycles,Hoff:Parahoric}.
		Nevertheless, we could also have used \(\BB(G)\) instead, with minor modifications.
		\item By \cite[Lemma 3.23]{Zhu:Tame}, \(\BB(G)\) classifies \(G\)-isocrystals.
		Moreover, for perfect \(\IF_q\)-algebras \(R\), Dieudonné theory yields a fully faithful functor from \(p\)-divisible groups up to isogeny to the category of isocrystals; a similar result holds for \(p\)-divisible groups with extra structure.
		This yields an equivalent way to define the map \(\Sht_{\Gg,\preccurlyeq\mu}^{\naive} \to \BB(G)'\).
	\end{enumerate}
\end{rmk}

\begin{dfn}\thlabel{defi correspon shtuka}
	The Hecke correspondence \(\Sht_{\Gg,\preccurlyeq\mu_1}^{\naive}\leftarrow \Sht_{\mu_1\mid \mu_2}^{\naive} \to \Sht_{\Gg,\preccurlyeq\mu_2}^{\naive}\) is defined as the fiber product
	\[\begin{tikzcd}
		&\Sht_{\mu_1\mid \mu_2}^{\naive} \arrow[ld] \arrow[rd] & \\
		\Sht_{\Gg,\preccurlyeq\mu_1}^{\naive} \arrow[rd] && \Sht_{\Gg,\preccurlyeq\mu_2}^{\naive} \arrow[ld]\\
		& \BB(G)'.&
	\end{tikzcd}\]
	The Hecke correspondences \(\Sht_{\Gg,\preccurlyeq\mu_1}^{\spl}\leftarrow \Sht_{\mu_1\mid \mu_2}^{\spl} \to \Sht_{\Gg,\preccurlyeq\mu_2}^{\spl}\) and \(\Sht_{\Gg,\preccurlyeq\mu_1}^{\can}\leftarrow \Sht_{\mu_1\mid \mu_2}^{\can} \to \Sht_{\Gg,\preccurlyeq\mu_2}^{\can}\) are defined similarly.
\end{dfn}

For the stack of canonical local shtukas, \thref{defi correspon shtuka} agrees with the similar definition in \cite[Definition 5.2.8]{XiaoZhu:Cycles}.
Moreover, for splitting and naive local shtukas, we have the following moduli interpretation, which follow immediately from the definitions.

\begin{lem}\thlabel{moduli description hecke corr of shtukas}
	Let \(R\) be any perfect \(\IF_q\)-algebra
	\begin{enumerate}
		\item The groupoid \(\Sht_{\mu_1\mid \mu_2}^{\naive}(R)\) is naturally equivalent to the groupoid of triples \((\Ee_1,\Ee_2,\beta)\), where \(\Ee_i=\{\Ee_{i,\Lambda}\}_{\Lambda\in \Ll}\in \Sht_{\Gg,\preccurlyeq\mu_i}^{\naive}\), and \(\beta = \{\beta_\Lambda \colon \Ee_{1,\Lambda} \sim \Ee_{2,\Lambda}\}_{\Lambda\in \Ll}\) is a quasi-isogeny \(\Ll\)-multichains of polarized \(\Oo_B\)-\(p\)-divisible groups (i.e., compatible with the additional structures, in the natural sense).
		\item The same holds when replacing ``naive" by ``spl".
	\end{enumerate}
\end{lem}
Note that \(G\)-isocrystals do not depend on an integral model \(\Gg\), hence we do not need to ask for \(\beta\) to respect the splitting structures.
In particular, the points of these Hecke correspondences of local shtukas can be viewed as modifications of \(p\)-divisible groups with extra structure.

The Hecke correspondences for the various stacks of local shtukas can be related as follows.

\begin{lem}\thlabel{diagram of moduli of shtukas}
	There is a natural commutative diagram
	\[\begin{tikzcd}
		\Sht_{\Gg,\preccurlyeq\mu_1}^{\spl} \arrow[d] & \Sht_{\mu_1\mid \mu_2}^{\spl} \arrow[l] \arrow[r] \arrow[d] & \Sht_{\Gg,\preccurlyeq\mu_2}^{\spl} \arrow[d] \\
		\Sht_{\Gg,\preccurlyeq\mu_1}^{\can} \arrow[d] & \Sht_{\mu_1\mid \mu_2}^{\can} \arrow[l] \arrow[r] \arrow[d] & \Sht_{\Gg,\preccurlyeq\mu_2}^{\can} \arrow[d] \\
		\Sht_{\Gg,\preccurlyeq\mu_1}^{\naive} & \Sht_{\mu_1\mid \mu_2}^{\naive} \arrow[l] \arrow[r] & \Sht_{\Gg,\preccurlyeq\mu_2}^{\naive}.
	\end{tikzcd}\]
	Moreover, the lower vertical maps are closed immersions, and if \thref{assumption splitting flat} hold for both local integral PEL data, the upper vertical maps are proper surjective.
\end{lem}
\begin{proof}
	The existence and commutativity of the diagram follows from the definitions.
	It remains to see the middle vertical arrows are surjective and a closed immersion respectively, the rest is clear.
	Since the map \(\Sht_{\mu_1\mid \mu_2}^{\can} \to \Sht_{\mu_1\mid \mu_2}^{\naive}\) can be factored as
	\[\Sht_{\mu_1\mid \mu_2}^{\can} = \Sht_{\Gg,\preccurlyeq\mu_1}^{\can}\times_{\BB(G)'} \Sht_{\Gg,\preccurlyeq\mu_2}^{\can} \to \Sht_{\Gg,\preccurlyeq\mu_1}^{\naive} \times_{\BB(G)'} \Sht_{\Gg,\preccurlyeq\mu_2}^{\can} \to \Sht_{\Gg,\preccurlyeq\mu_1}^{\naive} \times_{\BB(G)'}\Sht_{\Gg,\preccurlyeq\mu_2}^{\naive} = \Sht_{\mu_1\mid \mu_2}^{\naive},\]
	it is a closed immersion by \thref{corollary proper surjective}.
	
	Now, assume \thref{assumption splitting flat} holds for both local integral PEL data.
	Let \(X^{\naive}\) be the fiber product of \(\Sht_{\Gg,\preccurlyeq\mu_1}^{\spl}\) and \(\Sht_{\mu_1\mid \mu_2}^{\naive}\) over \(\Sht_{\Gg,\preccurlyeq\mu_1}^{\naive}\), and define \(X^{\can}\) similarly.
	Using moduli interpretations, we get an isomorphism
	\[\Sht_{\mu_1\mid \mu_2}^{\spl} \cong X^{\naive} \times_{(\Gg_{\Oo_E} \backslash \mathscr{M}_{\Gg,\preccurlyeq\mu_2}^{\naive})_{\IF_q}^{\perf}} (\Gg_{\Oo_{E^{\Gal}}} \backslash \mathscr{M}_{\Gg,\preccurlyeq\mu_2}^{\spl})_{\IF_q}^{\perf}.\]
	This yields a further isomorphism
	\[\Sht_{\mu_1\mid \mu_2}^{\spl} \cong X^{\can} \times_{(\Gg_{\Oo_E} \backslash \mathscr{M}_{\Gg,\preccurlyeq\mu_2}^{\can})_{\IF_q}^{\perf}} (\Gg_{\Oo_{E^{\Gal}}} \backslash \mathscr{M}_{\Gg,\preccurlyeq\mu_2}^{\spl})_{\IF_q}^{\perf},\]
	so that the composition \(\Sht_{\mu_1\mid \mu_2}^{\spl}\to X^{\can}\to \Sht_{\mu_1\mid \mu_2}^{\can}\) is proper surjective by \thref{corollary proper surjective}.
\end{proof}

\begin{cor}\thlabel{from corr to sht}
	The maps \(\Sht_{\mu_1\mid \mu_2}^{\naive}\to \Sht_{\Gg{,}\preccurlyeq\mu_i}^{\naive}\) are representable by ind-(perfectly proper schemes).
	The same is true for the canonical versions, and under \thref{assumption splitting flat} for the splitting versions as well.
\end{cor}
\begin{proof}
	Since \(\Sht_{\Gg,\preccurlyeq\mu_i}^{\naive}\) and \(\Sht_{\Gg{,}\preccurlyeq\mu_i}^{\can}\) admit natural closed embeddings into \(LG/\Ad_{\sigma^{-1}} L^+\Gg\) by \cite[Corollary 21.6.10]{ScholzeWeinstein:Berkeley}, the claim for naive and canonical models follows from base change, since \(LG/\Ad_{\sigma^{-1}} L^+\Gg\to \BB(G)'\) is representable by ind-(perfectly proper schemes) \cite[Lemma 3.28]{Zhu:Tame}.
	
	For the splitting models, there is a commutative diagram
	\[\begin{tikzcd}
		\Sht_{\mu_1\mid \mu_2}^{\spl} \arrow[d] \arrow[r] & \Sht_{\Gg{,}\preccurlyeq\mu_i}^{\spl} \arrow[d]\\
		\Sht_{\mu_1\mid \mu_2}^{\naive} \arrow[r] & \Sht_{\Gg{,}\preccurlyeq\mu_i}^{\naive}.
	\end{tikzcd}\]
	Here, the right arrow is representable by perfectly proper schemes since this holds on the level of local models, and the same holds for the left arrow by \thref{diagram of moduli of shtukas}.
	Since we already know the corollary for the naive models, it follows for splitting models as well.
\end{proof}

\subsection{Exotic correspondences via splitting models}

Finally, we can construct exotic Hecke correspondences between the special fibers of various models of Shimura varieties, by using moduli interpretations.
Let us fix two global integral PEL data \((B,*, \Oo_B,V_i,(,)_i,\Ll_i,h_i)\), with \(i=1,2\).
This determines two Shimura data \((\IG_i,\IX_i)\).
Let \(\IE_i\) be their respective reflex fields, and \(\mu_i\) the dominant representatives of the Hodge cocharacters determined by \(\IX_i\).
Consider their composite \(\IE\), and fix a place \(\pp\) of \(\IE\) with completion \(E=\IE_{\pp}\); below we will consider all objects as living over the residue field \(k=k_E\).
Suppose that we are given an isomorphism \(V_1\otimes \IA_f\cong V_2\otimes \IA_f\), compatibly with the pairing \((,)_i\) and the \(\Oo_B\)-structure.
Then we also have \(\IG_{1,\IA_f}\cong \IG_{2,\IA_f}\), and hence the \(\IG_i\) have Galois-equivariantly isomorphic dual groups.
Finally, we also assume that \(\mu_{1\mid Z(\widehat{G})^{\Gamma_{\IQ_p}}} = \mu_{2\mid Z(\widehat{G})^{\Gamma_{\IQ_p}}}\).

\begin{thm}\thlabel{Exotic Hecke correspondences}
	Using the notation above, assume that \thref{assumption splitting flat} holds for the two local integral PEL data obtained from \((B,*, \Oo_B,V_i,(,)_i,\Ll_i,h_i)\).
	Then there exist (necessarily unique) ind-(perfect schemes) \(\Sh_{\mu_1\mid \mu_2}^{\spl}\), \(\Sh_{\mu_1\mid \mu_2}^{\can}\) and \(\Sh_{\mu_1\mid \mu_2}^{\naive}\), fitting into the following commutative diagram
	\begin{equation}\label{big diagram of correspondences}\begin{tikzcd}
		\Sh_{\mu_1{,}K}^{\spl} \arrow[dd] \arrow[rd]&& \Sh_{\mu_1\mid \mu_2}^{\spl} \arrow[ll] \arrow[rr] \arrow[dd] \arrow[rd] && \Sh_{\mu_2{,}K}^{\spl} \arrow[dd] \arrow[rd] &\\
		& \Sht_{\Gg{,}\preccurlyeq\mu_1}^{\spl} \arrow[dd] && \Sht_{\mu_1\mid \mu_2}^{\spl} \arrow[ll] \arrow[rr] \arrow[dd] && \Sht_{\Gg{,}\preccurlyeq\mu_2}^{\spl} \arrow[dd]\\
		\Sh_{\mu_1{,}K}^{\can} \arrow[dd] \arrow[rd]&& \Sh_{\mu_1\mid \mu_2}^{\can} \arrow[ll] \arrow[rr] \arrow[dd] \arrow[rd] && \Sh_{\mu_2{,}K}^{\can} \arrow[dd] \arrow[rd] &\\
		& \Sht_{\Gg{,}\preccurlyeq\mu_1}^{\can} \arrow[dd] && \Sht_{\mu_1\mid \mu_2}^{\can} \arrow[ll] \arrow[rr] \arrow[dd] && \Sht_{\Gg{,}\preccurlyeq\mu_2}^{\can} \arrow[dd]\\
		\Sh_{\mu_1{,}K}^{\naive}\arrow[rd]&& \Sh_{\mu_1\mid \mu_2}^{\naive} \arrow[ll] \arrow[rr] \arrow[rd] && \Sh_{\mu_2{,}K}^{\naive} \arrow[rd] &\\
		& \Sht_{\Gg{,}\preccurlyeq\mu_1}^{\naive} && \Sht_{\mu_1\mid \mu_2}^{\naive} \arrow[ll] \arrow[rr] && \Sht_{\Gg{,}\preccurlyeq\mu_2}^{\naive},
	\end{tikzcd}\end{equation}
	where all squares of the form
	\[\begin{tikzcd}
		\Sh_{\mu_i{,}K}^? \arrow[d] & \Sh_{\mu_1\mid \mu_2}^? \arrow[l] \arrow[d]\\
		\Sht_{\Gg{,}\preccurlyeq\mu_i}^? & \Sht_{\mu_1\mid \mu_2}^? \arrow[l]
	\end{tikzcd}\]
	are cartesian.
\end{thm}
\begin{proof}
	The part of \eqref{big diagram of correspondences} consisting of moduli of local shtukas was constructed in \thref{diagram of moduli of shtukas}, whereas the maps \(\Sh_{\mu_i,K}\to \Sht_{\Gg,\preccurlyeq\mu_i}\) are the crystalline period maps from \eqref{shimura varieties to shtukas}.
	It remains to construct the \(\Sh_{\mu_1\mid \mu_2}\), and show they make the required squares cartesian.
	The representability of these prestacks by ind-(perfect schemes) will then follow from \thref{from corr to sht}.
	We first start with the naive and splitting models.
	
	By \thref{moduli description shtukas} and \thref{moduli description hecke corr of shtukas}, a point of \(\Sh_{\mu_1,K}^{\naive} \times_{\Sht_{\Gg{,}\preccurlyeq\mu_1}^{\naive}} \Sht_{\mu_1\mid \mu_2}^{\naive}\) consists of a homogeneously polarized \(\Ll\)-set \(A_1=\{A_{1,\Lambda}\}_{\Lambda\in \Ll}\) of abelian varieties, a level structure \(\eta\), and a modification of the corresponding \(\Ll\)-multichain of polarized \(\Oo_B\)-\(p\)-divisible groups.
	By a classical argument, cf.~e.g.~\cite[6.13]{RapoportZink:Period}, such a modification of the \(p\)-divisible groups of \(A_1\) yields a second homogeneously polarized \(\Ll\)-set \(A_2 = \{A_{2,\Lambda}\}_{\Lambda\in \Ll}\), together with a quasi-isogeny \(A_1 \sim A_2\).
	Using this quasi-isogeny, we can equip \(A_2\) with a level structure induced by \(\eta\).
	Since the determinant condition for abelian varieties can be checked on the level of \(p\)-divisible groups, this yields a point of \(\Sh_{\mu_2,K}^{\naive}\), and hence a map
	\[\Sh_{\mu_1,K}^{\naive} \times_{\Sht_{\Gg{,}\preccurlyeq\mu_1}^{\naive}} \Sht_{\mu_1\mid \mu_2}^{\naive} \to \Sh_{\mu_2,K}^{\naive} \times_{\Sht_{\Gg{,}\preccurlyeq\mu_2}^{\naive}} \Sht_{\mu_1\mid \mu_2}^{\naive}.\]
	Similarly, we can construct a map in the inverse direction, and these are mutual inverses.
	
	The case of splitting models is similar, since a splitting structure on an abelian variety (in characteristic \(p\)) is equivalent to a splitting structure on the associated \(p\)-divisible group \cite[1.12]{MazurMessing:Universal}.
	
	Note that there is a natural morphism \(\Sh_{\mu_1\mid \mu_2}^{\spl}\to \Sh_{\mu_1\mid \mu_2}^{\naive}\), which is ind-(perfectly proper).
	Indeed, it sits in the commutative diagram
	\[\begin{tikzcd}
		\Sh_{\mu_1\mid \mu_2}^{\spl} \arrow[r] \arrow[d] & \Sh_{\mu_1,K}^{\spl} \arrow[d] \\
		\Sh_{\mu_1\mid \mu_2}^{\naive} \arrow[r] & \Sh_{\mu_1,K}^{\naive},
	\end{tikzcd}\]
	where all other arrows are ind-perfectly proper by \thref{from corr to sht} and \eqref{diagram of models}.
	
	We now define \(\Sh_{\mu_1\mid \mu_2}^{\can}\subseteq \Sh_{\mu_1\mid \mu_2}^{\naive}\) as the schematic image of \(\Sh_{\mu_1\mid \mu_2}^{\spl}\to \Sh_{\mu_1\mid \mu_2}^{\naive}\).
	By properness of this map, \(\Sh_{\mu_1\mid \mu_2}^{\spl}\) surjects onto \(\Sh_{\mu_1\mid \mu_2}^{\can}\).
	Thus, as the composition
	\[\Sh_{\mu_1\mid \mu_2}^{\spl}\to \Sh_{\mu_i,K}^{\spl}\to \Sh_{\mu_i,K}^{\naive}\]
	factors through \(\Sh_{\mu_i,K}^{\can}\subseteq \Sh_{\mu_i,K}^{\naive}\) (by \thref{corollary proper surjective} and \eqref{diagram of models}), the map 
	\[\Sh_{\mu_1\mid \mu_2}^{\can} \to \Sh_{\mu_1\mid \mu_2}^{\naive} \to \Sh_{\mu_i,K}^{\naive}\]
	factors through \(\Sh_{\mu_i,K}^{\can}\) as well.
	This gives all the arrows in \eqref{big diagram of correspondences}, and it remains to show the desired squares are cartesian in the canonical case.
	
	Let \(X_i\) denote the fiber product of \(\Sh_{\mu_i,K}^{\can}\) and \(\Sht_{\mu_1\mid \mu_2}^{\can}\) over \(\Sht_{\Gg,\preccurlyeq\mu_i}^{\can}\).
	Then there is a natural map \(\Sh_{\mu_1\mid \mu_2}^{\can}\to X_i\), which is a closed immersion as both ind-(perfect schemes) admit natural closed immersions in \(\Sh_{\mu_1\mid \mu_2}^{\naive}\).
	It thus remains to show this map is surjective.
	For this, let \(x=(y,z)\in X\) be a point, with \(y\in \Sh_{\mu_i,K}^{\can}\) and \(z\in \Sht_{\mu_1\mid \mu_2}^{\can}\).
	As \(\Sht_{\mu_1\mid \mu_2}^{\spl}\to \Sht_{\mu_1\mid \mu_2}^{\can}\) is surjective by \thref{diagram of moduli of shtukas}, we can find a preimage \(z'\in \Sht_{\mu_1\mid \mu_2}^{\spl}\) of \(z\).
	Since the square
	\[\begin{tikzcd}
		\Sh_{\mu_i,K}^{\spl} \arrow[r] \arrow[d] & \Sht_{\Gg,\mu_i}^{\spl} \arrow[d]\\
		\Sh_{\mu_i,K}^{\can} \arrow[r] & \Sht_{\Gg,\preccurlyeq\mu_i}^{\can}
	\end{tikzcd}\]
	appearing in \eqref{shimura varieties to shtukas}
	is cartesian, we can find a (unique) \(y'\in \Sh_{\mu_i,K}^{\spl}\) mapping to \(y\in \Sh_{\mu_i,K}^{\can}\) and with the same image in \(\Sht_{\Gg,\preccurlyeq\mu_i}^{\spl}\) as \(z'\).
	This gives a point \(x'=(y',z')\in \Sh_{\mu_1\mid \mu_2}^{\spl}\), whose image in \(\Sh_{\mu_1\mid \mu_2}^{\can}\) maps to \(x\in X\).
	Since we are working with ind-(perfect schemes), this shows that \(\Sh_{\mu_1\mid \mu_2}^{\can}\to X_i\) is an isomorphism, as desired.	
\end{proof}

\section{Mixed Tate motives on splitting affine Grassmannians}\label{Sec:Splitting Gr}

We now move to a more group-theoretic approach to study splitting models.
For now, we allow \(p\) to be an arbitrary prime number.
Moreover, the methods used throughout the rest of the paper will not need moduli interpretations (except when appealing to the exotic Hecke correspondences constructed before), and hence we will work in greater generality than in the previous sections.
At many places, we will assume our groups are essentially unramified.

\begin{dfn}\thlabel{definition essentially unramified}
	Let \(G/F\) be a reductive group.
	Its adjoint quotient splits as a product \(G_{\adj} = \prod_i \Res_{F_i/F} H_i\), where each \(F_i/F\) is a finite separable field extension, and \(H_i/F_i\) is an absolutely simple reductive group.
	Then \(G\) is called \emph{essentially unramified} if each \(H_i\) is an unramified \(F_i\)-group.
\end{dfn}

\begin{nota}\thlabel{notation essentially unramified}
	Whenever \(G/F\) is essentially unramified and quasi-split, we fix the following notation.
	First, we write \(G_{\adj} = \prod_i \Res_{F_i/F} H_i\) for the decomposition of the adjoint quotient into restrictions of scalars of absolutely simple unramified groups.
	For each \(i\), we fix a maximal \(F_i\)-split torus \(A_i\subseteq H_i\), contained in a maximally \(\breve{F}_i\)-split torus \(S_i\subseteq H_i\), and let \(T_i:=\Cent_{H_i} S_i\).
	As in \cite[§4.2]{HainesRicharz:TestWeil}, we can construct natural tori \(A_i \subseteq S_i\subseteq T_i\subseteq \Res_{F_i/F} H_i\), with \(A_i\) maximally \(F\)-split, \(S_i\) maximally \(\breve{F}\)-split, and \(T_i\) maximal.
	By \thref{tori and central isogenies} below, we obtain a similar series of tori \(A\subseteq S\subseteq T \subseteq G\).
	We also fix compatible Borels of \(G\) and \(H_i\), so that we have compatible systems of positive roots and dominant (co)characters for these groups.
	The reduced Bruhat--Tits building for G agrees with the product of the reduced Bruhat--Tits buildings for the \(H_i\), and we use this to fix a very special parahoric \(\Gg/\Oo_F\), inducing hyperspecial parahorics \(\Hh_i\) of \(H_i\).
	
	Finally, we let \(k_i\) be the residue field of \(F_i\), and \(F_i^{\unr}\subseteq F_i\) the maximal subextension which is unramified over \(F\). 
	We also let \(e_i=[F_i\colon F_i^{\unr}]\) be the ramification index of \(F_i/F\) and \(f_i=[F_i^{\unr}:F]\) the residual degree; in particular \([F_i:F]=e_if_i\).
\end{nota}

\begin{lem}\thlabel{tori and central isogenies}
	Let \(G/F\) be quasi-split, and \(A_{\adj} \subseteq S_{\adj} \subseteq T_{\adj} \subseteq G_{\adj}\) three tori, where \(T_{\adj}\) is maximal, \(S_{\adj}\) is maximally \(\breve{F}\)-split, and \(A_{\adj}\) is maximally \(F\)-split.
	Then the reduced neutral component \(T\subseteq G\) of the preimage of \(T_{\adj}\) is a maximal torus.
	Moreover, it contains a unique maximally \(\breve{F}\)-split torus \(S\subseteq G\), which in turn contains a unique maximal \(F\)-split torus \(A\subseteq G\).
	Finally, the images of \(A\subseteq S\subseteq T\subseteq G\) in \(G_{\adj}\) are \(A_{\adj} \subseteq S_{\adj} \subseteq T_{\adj} \subseteq G_{\adj}\).
\end{lem}
\begin{proof}
	Let \(T\) be the reduced neutral component of the preimage of \(T_{\adj}\) in \(G\).
	This is of multiplicative type by \cite[Theorems 12.9 and 15.39]{Milne:Algebraic}, and hence a torus since it is reduced and connected.
	By looking at the ranks of \(Z_G\) and \(G_{\adj}\), we see that \(T\subseteq G\) is a maximal torus.
	
	Next, we define \(S\subseteq T\) as the subtorus with character lattice \(X^*(S) = X^*(T)_{I}/X^*(T)_{I,\tors}\) the quotient of \(X^*(T)_I\) by its torsion subgroup.
	In particular, it sits in the diagram
	\[\begin{tikzcd}
		X^*(T_{\adj}) \arrow[r, hook] \arrow[d, two heads] & X^*(T) \arrow[d, two heads]\\
		X^*(T_{\adj})_I/X^*(T_{\adj})_{I{,}\tors} \arrow[r] & X^*(T)_I/X^*(T)_{I{,}\tors},
	\end{tikzcd}\]
	corresponding to 
	\[\begin{tikzcd}
		T_{\adj} & T \arrow[l, two heads]\\
		S_{\adj} \arrow[u, hook] & S \arrow[l] \arrow[u, hook].
	\end{tikzcd}\]
	To see that \(S \to S_{\adj}\) is surjective, it suffices to show that \(X^*(T_{\adj})_I/X^*(T_{\adj})_{I,\tors} \to X^*(T)_I/X^*(T)_{I,\tors}\) is injective.
	On the one hand, this is true after rationalizing, so that the kernel is torsion.
	But by construction, \(X^*(T_{\adj})_I/X^*(T_{\adj})_{I,\tors}\) is torsionfree, we get the desired injectivity.
	(In fact, since \(T_{\adj}\) is the maximal torus in an adjoint group, the Galois action permutes a basis of \(X^*(T_{\adj})\), so that \(X^*(T_{\adj})_I\) is already torsionfree.)
	
	Similarly, we define \(A\subseteq S\subseteq T\), with character lattice \(X^*(A) = X^*(T)_{\Gamma_F}/X^*(T)_{\Gamma_F,\tors}\), and the same argument as above shows that its image in \(G_{\adj}\) is \(A_{\adj}\).
	Finally, we see that \(S\subseteq G\) (resp.~\(A\subseteq G\)) is maximally \(\breve{F}\)-split (resp.~maximally split) by looking at the ranks.
\end{proof}

Although our main results will concern essentially unramified groups, this assumption is not always necessary, and we will specify at the beginning of each (sub)section whether we impose it.
Note that when the group \(G/\IQ_p\) associated with an integral local PEL Shimura datum is essentially unramified, then the special fiber of the corresponding splitting model (as defined in \thref{Defi:splitting model}) is isomorphic to a convolution product of Schubert varieties, \thref{example splitting models}.
In particular, the group-theoretic methods from the following sections will be applicable to these splitting models, defined via linear algebra.

Throughout this section, we will use the theory of étale motives \cite{CisinskiDeglise:Etale}, as recalled in Appendix \ref{Appendix:motives}.
For now, we will use motives with \(\IZ[\frac{1}{p}]\)-coefficients; starting in §\ref{Subsec:Motivic correspondences on shtukas} we will switch to \(\IQ\)-coefficients.

\subsection{Mixed Tate motives on convolution affine flag varieties}

Since the bounded parts of the splitting local models arose via the convolution products of Schubert varieties in affine Grassmannians (at least at very special level, and in the essentially unramified case), let us study more generally convolution products of (unbounded) affine Grassmannians, and in particular their motivic aspects.
This will immediately imply similar results for splitting models of affine Grassmannians (once we have defined these), but we do not need to restrict to essentially unramified groups for this subsection.
In fact, most of this subsection works for more general partial affine flag varieties.
For now, let \(G/F\) be an arbitrary quasi-split reductive group, which we will assume to be residually split throughout this subsection (i.e., the reductive quotient of the special fiber of any parahoric model of \(G\) is split).
In particular, the maximal \(\breve{F}\)-split torus \(S\subseteq G\) is already split over \(F\).
We also fix an Iwahori model \(\Ii/\Oo_F\) and a parahoric model \(\Gg\supseteq \Ii\) corresponding to a facet \(\ff\), and refer to §\ref{Subsec:Notation} and \cite[§3]{vdH:RamifiedSatake} for the basics on partial affine flag varieties.

Fix an integer \(t\), and consider the \(t\)-fold convolution affine flag variety
\[\Fl_{\Gg}^{\widetilde{\times} t} := \Fl_{\Gg} \widetilde{\times} \ldots \Fl_{\Gg} := LG \overset{L^+\Gg}{\times} \ldots \overset{L^+\Gg}{\times} \Fl_{\Gg}\]
from e.g.~\cite[Definition 5.1]{vdH:RamifiedSatake}.
This is representable by an ind-(perfect projective) ind-scheme, and admits a convolution morphism \(m\colon \Fl_{\Gg} \widetilde{\times} \ldots \widetilde{\times} \Fl_{\Gg} \to \Fl_{\Gg}\).
It is moreover stratified by convolution Schubert cells
\[\Fl_{\Gg,w_\bullet} := \Fl_{\Gg,w_1} \widetilde{\times} \ldots \widetilde{\times} \Fl_{\Gg,w_t},\]
for \(w_\bullet = (w_1,\ldots,w_t)\in (W_{\ff}\backslash \tilde{W}/W_{\ff})^t\), whose closure is given by
\[\Fl_{\Gg,\leq w_\bullet} := \Fl_{\Gg,\leq w_1} \widetilde{\times} \ldots \widetilde{\times} \Fl_{\Gg,w_t},\]
called a convolution Schubert variety.
These convolution Schubert cells are stable for the natural \(L^+\Gg\)-action, induced by the action on the first factor.

In \cite[Theorem 3.7]{vdH:RamifiedSatake}, we showed that Schubert stratifications of partial affine flag varieties are Whitney--Tate and perfectly cellular (in the sense of \cite[Definitions 3.2 and 3.3]{RicharzScholbach:Witt}).
As a consequence, partial affine flag varieties admit nice categories of stratified Tate motives, which can be equipped with the motivic t-structure (in contrast to the full category of motives on a scheme, where the existence of the motivic t-structure relies on standard conjectures on algebraic cycles).
It turns out that convolution affine flag varieties also admit Whitney--Tate stratifications into perfectly cellular schemes.

\begin{lem}\thlabel{group acting simply transitively on orbit}
	For any \(w\in \tilde{W}\), there is a subscheme \(X\subseteq L^+\Ii\), isomorphic to \(\IA_k^{l(w),\perf}\), such that the action on \(w\in \Fl_{\Ii}\) induces an isomorphism \(X\cong \Fl_{\Ii,w}\).
\end{lem}
\begin{proof}
	For any relative affine root \(\psi=\alpha+n\), with \(\alpha\in \Phi(G,S)\) a relative root and \(n\in \IZ\), we denote by  \(U_\psi\) the image of \(\IG_a^{\perf} \to LG\colon x\mapsto u_{\alpha}([x]\varpi^n)\). 
	Here, \([x]\) is the Teichmüller representative of \(x\), and \(u_\alpha\) is the root homomorphism corresponding to \(\alpha\).
	
	Fix a reduced decomposition \(\dot{w}=s_1s_2\ldots s_{l(w)} \tau\), with the \(s_i\) simple affine reflections, and \(\tau\in \tilde{W}\) of length 0; we may assume \(\tau=1\).
	Let \(\bfa\) be the alcove corresponding to \(\Ii\), and let \(\bfa_i=s_1\ldots s_i(\bfa)\) for \(i=1,\ldots,l(w)\).
	Let \(\psi_i\) be the unique positive relative affine root, such that the corresponding reflection hyperplane contains the wall separating \(\bfa_{i-1}\) and \(\bfa_i\).
	Then it follows from the proof of \cite[Proposition 3.6 (1)]{vdH:RamifiedSatake} that the action map
	\[\prod_{i=1}^{l(w)} U_{\psi_i} \to \Fl_{\Ii}\colon (y_i) \mapsto y_1 y_2 \ldots y_{l(w)} w\]
	induces an isomorphism onto \(\Fl_{\Ii, w}\).
\end{proof}

\begin{prop}\thlabel{convolution Fl WT}
	The stratification of \(\Fl_{\Gg} \widetilde{\times} \ldots \widetilde{\times} \Fl_{\Gg}\) by convolution Schubert cells is Whitney--Tate with perfectly cellular strata.
\end{prop}
\begin{proof}
	For simplicity, we only handle the convolution product of two affine flag varieties.
	Let \(v,v',w,w'\in W_{\ff} \backslash \tilde{W} /W_{\ff}\) correspond to Schubert cells \(\Fl_{\Gg,w}\xrightarrow{\iota_w}\Fl_{\Gg}\) etc.
	Then \((\iota_v\widetilde{\times} \iota_w)^*(\iota_{v'}\widetilde{\times} \iota_{w'})_*\unit\) can be identified with the twisted exterior product \((\iota_v^*\iota_{v',*}\unit) \widetilde{\boxtimes} (\iota_w^*\iota_{w',*}\unit)\) (as defined in \cite[Notation 6.10]{vdH:RamifiedSatake}, but without the truncation).
	Indeed, this follows from the compatibility of both *-pullback and *-pushforward with exterior products \cite[Proposition 2.1.20]{JinYang:Kunneth}.
	Thus, since \(\iota_v^*\iota_{v',*}\unit\in \DM(\Fl_{\Gg,v})\) and \(\iota_w^*\iota_{w',*}\unit\in \DM(\Fl_{\Gg,w})\) are Tate by \cite[Theorem 3.7]{vdH:RamifiedSatake}, it suffices to show that the twisted exterior product of the monoidal units is Tate.
	But this clearly agrees with the monoidal unit of \(\Fl_{\Gg,v}\widetilde{\times} \Fl_{\Gg,w}\).
	
	It remains to show that \(\Fl_{\Gg,v}\widetilde{\times}\Fl_{\Gg,w}\) is perfectly cellular.
	By \cite[Proposition 3.6 (3)]{vdH:RamifiedSatake}, we may assume \(\Gg=\Ii\) is an Iwahori model of \(G\).
	But then we can apply \thref{group acting simply transitively on orbit} to find a subscheme \(H\subseteq L^+\Ii\), isomorphic to some perfect affine space, such that the action map induces an isomorphism \(H\times \left( \{v\} \widetilde{\times} \Fl_{\Gg,w}\right) \cong \Fl_{\Gg,v} \widetilde{\times} \Fl_{\Gg,w}\).
	Since \(\Fl_{\Gg,w} \cong \{v\} \widetilde{\times} \Fl_{\Gg,w}\) is already perfectly cellular by \cite[Proposition 3.6 (1)]{vdH:RamifiedSatake}, we are done.
\end{proof}

In particular, the category \(\DTM(\Fl_{\Gg} \widetilde{\times} \ldots \widetilde{\times} \Fl_{\Gg})\) of stratified Tate motives on the convolution affine flag variety (as in \cite[Definition 2.11]{vdH:RamifiedSatake}) admits a natural t-structure by \cite[Proposition 2.12]{vdH:RamifiedSatake}, whose heart \(\MTM(\Fl_{\Gg} \widetilde{\times} \ldots \widetilde{\times} \Fl_{\Gg})\) is called the category of mixed Tate motives.
Similarly, we can define equivariant versions \(\DTM_{L^+\Gg}(\Fl_{\Gg} \widetilde{\times} \ldots \widetilde{\times} \Fl_{\Gg})\) and \(\MTM_{L^+\Gg}(\Fl_{\Gg} \widetilde{\times} \ldots \widetilde{\times} \Fl_{\Gg})\).

In order to define suitable categories of stratified Tate motives on the moduli stack of local shtukas, it will be useful to consider natural torsors over the (convolution) affine flag variety as well.
Namely, consider the \(L^+\Gg\)-torsor over \(\Fl_{\Gg} \widetilde{\times} \ldots \widetilde{\times} \Fl_{\Gg}\) given by \(LG \overset{L^+\Gg}{\times} \ldots \overset{L^+\Gg}{\times} LG\); in particular this is just \(LG\) when \(t=1\).
We will denote te restriction of this torsor to a convolution Schubert cell or variety by \(\Fl_{\Gg,(\leq) w_\bullet}^{(\infty)} \to \Fl_{\Gg,(\leq) w_\bullet}\).
Moreover, for any \(n\geq 0\), the quotient \(L^+\Gg \to L^n\Gg\) induces an \(L^n\Gg\)-torsor \(\Fl_{\Gg,(\leq) w_\bullet}^{(n)} \to \Fl_{\Gg,(\leq) w_\bullet}\).
We then have stratifications
\begin{equation}\label{WT torsor over Fl}LG \overset{L^+\Gg}{\times} \ldots \overset{L^+\Gg}{\times} LG = \bigsqcup_{w_\bullet \in (W_\ff \backslash \tilde{W} /W_\ff)^t} \Fl_{\Gg,w_\bullet}^{(\infty)} \quad \text{ and } \quad LG \overset{L^+\Gg}{\times} \ldots \overset{L^+\Gg}{\times} \Fl_{\Gg}^{(n)} = \bigsqcup_{w_\bullet \in (W_\ff \backslash \tilde{W} /W_\ff)^t} \Fl_{\Gg,w_\bullet}^{(n)}.\end{equation}

\begin{cor}\thlabel{WT torsors over Fl}
	For finite \(n\geq 0\), the stratification \eqref{WT torsor over Fl} is Whitney--Tate with perfectly cellular strata.
\end{cor}
\begin{proof}
	Let \(v_\bullet \leq w_\bullet\), and consider the diagram with cartesian squares
	\[\begin{tikzcd}
		\Fl_{\Gg,v_\bullet}^{(n)} \arrow[d, "f_{v_\bullet}"'] \arrow[r, "\iota_{v_\bullet}^{(n)}"] & \Fl_{\Gg,\leq w_\bullet}^{(n)} \arrow[d, "f_{\leq w_\bullet}"] & \Fl_{\Gg,w_\bullet}^{(n)} \arrow[d, "f_{w_\bullet}"] \arrow[l, "\iota_{w_\bullet}^{(n)}"'] \\
		\Fl_{\Gg,v_\bullet}\arrow[r, "\iota_{v_\bullet}"'] & \Fl_{\Gg,\leq w_\bullet} & \Fl_{\Gg,w_\bullet}, \arrow[l, "\iota_{w_\bullet}"]
	\end{tikzcd}\]
	where the vertical arrows are \(L^n\Gg\)-torsors.
	Since \(L^n\Gg\) is perfectly smooth, this gives equivalences
	\[(\iota_{v_\bullet}^{(n)})^* (\iota_{w_\bullet}^{(n)})_* f_{w_\bullet}^* \cong (\iota_{v_\bullet}^{(n)})^* f_{\leq w_\bullet}^* (\iota_{w_\bullet})_* \cong f_{v_\bullet}^* (\iota_{v_\bullet})^* (\iota_{w_\bullet})_*.\]
	Whitney--Tateness of the stratification then follows from \thref{convolution Fl WT}, since *-pullback along the vertical maps preserves stratified Tate motives.
	
	It remains to show each \(\Fl_{\Gg, w_\bullet}^{(n)}\) is perfectly cellular.
	Note that \(\Fl_{\Gg, w_\bullet}^{(n)} = \Fl_{\Gg,w_1} \widetilde{\times} \ldots \widetilde{\times} \Fl_{\Gg,w_{t-1}}\widetilde{\times} \Fl_{\Gg,w_t}^{(n)}\), so by repeatedly using \thref{group acting simply transitively on orbit} (and decomposing the first factor into its Iwahori-orbits), we may assume \(t=1\).
	Similarly, \thref{group acting simply transitively on orbit} gives a stratification of \(\Fl_{\Gg, w}\), where each stratum is the product of a perfect affine space with a fiber of \(\Fl_{\Gg, w}^{(n)} \to \Fl_{\Gg, w}\), i.e., \(L^n\Gg\).
	But \(L^n\Gg\) is an extension of its unipotent radical, which is split unipotent, and a split reductive group, which is cellular by the Bruhat decomposition, concluding the lemma.
\end{proof}

\begin{rmk}\thlabel{remarks about DATM on LG}
	\begin{enumerate}
		\item The previous corollary assumed \(n\) to be finite, in order for the strata to be perfectly of finite type.
		However, we have \(L^+\Gg=\varprojlim_{n>0} L^n\Gg\) and hence \(\Fl_{\Gg,\leq w_\bullet}^{(\infty)}=\varprojlim_{n>0} \Fl_{\Gg,\leq w_\bullet}^{(n)}\).
		Hence, since \(\DM\) is Kan extended from pfp schemes, we have 
		\[\DM(\Fl_{\Gg,\leq w_\bullet}^{(\infty)}) \cong \varinjlim_{n>0} \DM(\Fl_{\Gg,\leq w_\bullet}^{(n)})\]
		by \eqref{kancolim},
		where the transition maps are given by !-pullback.
		But these are fully faithful, since the kernels of \(L^+\Gg\to L^m\Gg\to L^n\Gg\) are split pro-unipotent.
		In particular, for \(m\geq n\), such !-pullback induces equivalences \(\DTM(\Fl_{\Gg,\leq w_\bullet}^{(n)}) \cong \DTM(\Fl_{\Gg,\leq w_\bullet}^{(m)})\).
		We will then sometimes write \(\DTM(\Fl_{\Gg,\leq w_\bullet}^{(\infty)})\) for the essential image of the fully faithful !-pullback \(\DTM(\Fl_{\Gg,\leq w_\bullet}^{(n)}) \to \DM(\Fl_{\Gg,\leq w_\bullet}^{(\infty)})\).
		We will use similar notation for the equivariant and unbounded variants.
		In particular, this yields a category \(\DTM(LG)\) of stratified Tate motives on the loop group \(LG\), where the stratification is induced by the stratification on \(\Fl_{\Gg}\); this depends on the choice of parahoric \(\Gg\).
		\item By \thref{WT torsors over Fl}, we could also apply \cite[Proposition 2.12]{vdH:RamifiedSatake} to equip each \(\DTM(\Fl_{\Gg,\leq w_\bullet}^{(n)})\) with a t-structure.
		However, for the purposes of defining mixed Tate motives on moduli of local shtukas, we will renormalize the t-structure.
		Namely, instead of the t-structure from \cite[Proposition 2.7]{vdH:RamifiedSatake}, we will use the t-structure obtained by gluing the t-structures on \(\DTM(\Fl_{\Gg,w_\bullet}^{(n)})\), where the connective part is generated by 
		\[\{f_{w_\bullet}^!\unit(m)[\sum_{i=1}^t l(w_i)] \mid m\in \IZ\},\]
		where \(f_{w_\bullet}\colon \Fl_{\Gg,w_\bullet}^{(n)}\to \Fl_{\Gg,w_\bullet}\) is the \(L^n\Gg\)-torsor.
		We will still denote the resulting heart by \(\MTM(\Fl_{\Gg,\leq w_\bullet}^{(n)})\), and call it the category of (stratified) mixed Tate motives.
		In particular, this makes the equivalence \(\DTM(\Fl_{\Gg,\leq w_\bullet}^{(n)}) \cong \DTM(\Fl_{\Gg,\leq w_\bullet}^{(m)})\) for \(m\geq n\) (given by !-pullback) t-exact, and hence yields a t-structure on \(\DTM(\Fl_{\Gg,\leq w_\bullet}^{(\infty)})\).
		We make similar considerations for the equivariant and unbounded variants.
		\item All the stratifications above are in fact anti-effective (in the sense of \cite[Definition 2.6]{CassvdHScholbach:Geometric}), but this will not be relevant for the present paper.
		\item If \(G\) is quasi-split but not residually split, we can use Artin--Tate motives \(\DATM\) instead of \(\DTM\), as in \cite[§2.4]{vdH:RamifiedSatake}.
		The above remarks then remain valid.
		These categories \(\DATM\), and their hearts \(\MATM\), depend on a choice of field extension \(k'/k\) after which \(G\) becomes residually split.
		We will not use Artin--Tate motives in the rest of this paper.
	\end{enumerate}
\end{rmk}

We now specialize to the case where \(\Gg\) is very special, and write \(\Gr_{\Gg}^{\can}=\Fl_{\Gg}\).
Then mixed Tate motives on \(\Gr_{\Gg}^{\can}\) are related to the inertia-invariants of the Langlands dual group via the motivic Satake equivalence \cite[Theorem 1.1]{vdH:RamifiedSatake}.
The following result allows us to relate motives on convolution Grassmannians to Langlands dual data as well.
We denote by \(\Hck_{\Gg}^{\can}:= (L^+\Gg\backslash \Gr_{\Gg}^{\can})^{\et}\) the local Hecke stack, and define \(\DTM(\Hck_{\Gg}):= \DTM_{L^+\Gg}(\Gr_{\Gg}^{\can})\) (which is equipped with a t-structure).

\begin{prop}\thlabel{Satake for convolution Gr}
	The !-pullback along both morphisms
	\[\Gr_{\Gg}^{\can} \widetilde{\times} \ldots \widetilde{\times} \Gr_{\Gg}^{\can} \to L^+\Gg \backslash \Gr_{\Gg}^{\can} \widetilde{\times} \ldots \widetilde{\times} \Gr_{\Gg}^{\can} \to \Hck_{\Gg}^{\can} \times \ldots \times \Hck_{\Gg}^{\can}\]
	induce equivalences on stratified mixed Tate motives (i.e., on \(\MTM\)).
\end{prop}
If \(G\) is not necessarily residually split (but still quasi-split), the analogous result for Artin Tate motives still holds, by the same proof (or by using Galois-descent).
\begin{proof}
	For simplicity, we will again only consider the case \(t=2\).
	First, note that both maps preserve the stratifications, so that !-pullback along them preserves stratified Tate motives.
	By definition of the t-structures, this pullback is also t-exact, so that we get induced functors between mixed Tate motives.
	Combining \cite[(2.2.8), and Lemmas 2.2.12 and 3.2.12]{RicharzScholbach:Intersection}, we see that both functors are fully faithful when restricted to \(\MTM\).
	
	Now, consider the following diagram, consisting of \(L^+\Gg\)-torsors:
	\begin{equation}\label{diagram for convolution grassmannians}\begin{tikzcd}
			&LG \times \Gr_{\Gg}^{\can} \arrow[dl] \arrow[dr] &\\
			\Gr_{\Gg}^{\can} \widetilde{\times} \Gr_{\Gg}^{\can} \arrow[dr] && \Gr_{\Gg}^{\can} \times \Gr_{\Gg}^{\can} \arrow[dl]\\
			&\Gr_{\Gg}^{\can} \times \Hck_{\Gg}^{\can}.&
	\end{tikzcd}\end{equation}
	As above, !-pullback along these morphisms preserves mixed Tate motives and is fully faithful when restricted to \(\MTM\) (recall that we defined \(\MTM(LG)\) in \thref{remarks about DATM on LG}).
	(For the rest of the proof, we may safely restrict to finite unions of Schubert varieties, in which case we may replace the \(L^+\Gg\)-torsors by torsors over finite type quotients with split pro-unipotent kernel.
	We omit this to simplify the notation.)
	Consider also the following actions and projection map:
	\[a_1\colon L^+\Gg\times LG \times \Gr_{\Gg}^{\can} \to LG \times \Gr_{\Gg}^{\can}\colon (g,h,x)\mapsto (hg^{-1},x),\]
	\[a_2\colon L^+\Gg\times LG \times \Gr_{\Gg}^{\can} \to LG \times \Gr_{\Gg}^{\can}\colon (g,h,x)\mapsto (hg^{-1},gx),\]
	\[a_3\colon L^+\Gg\times LG \times \Gr_{\Gg}^{\can} \to LG \times \Gr_{\Gg}^{\can}\colon (g,h,x)\mapsto (h,gx),\]
	\[p\colon L^+\Gg\times LG \times \Gr_{\Gg}^{\can} \to LG \times \Gr_{\Gg}^{\can}\colon (g,h,x)\mapsto (h,x).\]
	Recall from \cite[Proposition 3.2.20]{RicharzScholbach:Intersection} that the full subcategory \(\MTM(\Gr_{\Gg}^{\can} \times \Gr_{\Gg}^{\can})\subseteq \MTM(LG\times \Gr_{\Gg}^{\can})\) consists of those mixed Tate motives \(\Ff\) such that \(p^!\Ff\cong a_1^!\Ff\).
	Similarly, \(\MTM(\Gr_{\Gg}^{\can} \widetilde{\times} \Gr_{\Gg}^{\can})\subseteq \MTM(LG\times \Gr_{\Gg}^{\can})\) consists exactly of those objects \(\Ff\) for which \(p^!\Ff\cong a_2^!\Ff\).
	Moreover, by \cite[Proposition 6.21]{vdH:RamifiedSatake}, every \(\Ff\in \MTM(LG\times \Gr_{\Gg}^{\can})\) satisfies \(p^!\Ff\cong a_3^!\Ff\).
	
	We claim that pullback along the two lower arrows of \eqref{diagram for convolution grassmannians} induce equivalences on \(\MTM\); for the lower right arrow this follows from \cite[Proposition 6.21]{vdH:RamifiedSatake}.
	Let \(\Ff\in \MTM(\Gr_{\Gg}^{\can} \widetilde{\times} \Gr_{\Gg}^{\can}) \subseteq \MTM(LG\times \Gr_{\Gg}^{\can})\),
	and consider the morphisms
	\[L^+\Gg\times X \xrightarrow{\nabla\colon (g,x)\mapsto (g^{-1},g,x)} L^+\Gg\times L^+\Gg\times X \xrightarrow{\id\times a_2} L^+\Gg\times X\xrightarrow{a_3} X,\]
	where we have denoted \(X:=LG\times \Gr_{\Gg}^{\can}\).
	Then the composition agrees with \(a_1\).
	But then 
	\[a_1^!\Ff\cong \nabla^!\circ (\id \times a_2)^! \circ a_3^!(\Ff)\cong \nabla^!\circ (\id \times a_2)^! \circ p^!(\Ff) \cong a_2^!\Ff\cong p^!\Ff.\]
	Thus \(\Ff\in \MTM(\Gr_{\Gg}^{\can}\times \Hck_{\Gg}^{\can})\), so that the lower left arrow in \eqref{diagram for convolution grassmannians} induces an essentially surjective !-pullback on \(\MATM\), which is then also an equivalence, .
	We can then conclude the proposition by applying \cite[Proposition 6.21]{vdH:RamifiedSatake}.
\end{proof}

Combining this with the motivic Satake equivalence \cite[Theorem 1.1]{vdH:RamifiedSatake}, and since the Hecke stack is compatible with products, we deduce
\begin{equation}\label{eq: Satake for convolution Gr}\MTM_{L^+\Gg}(\Gr_{\Gg}^{\can,\widetilde{\times}t}) \cong \bigotimes_{\MTM(\Spec k)}^t \Rep_{\widehat{G}^I}(\MTM(\Spec k))\cong \Rep_{(\widehat{G}^I)^t}(\MTM(\Spec k)).\end{equation}
Although this makes \(\MTM_{L^+\Gg}(\Gr_{\Gg}^{\can,\widetilde{\times} t})\) symmetric monoidal via transport of structures, it is not clear how to describe this monoidal structure directly.

\subsection{Unbounded splitting affine Grassmannians}

In the previous sections, we considered splitting versions of local models, whose special fiber was (in nice cases) the convolution product of minuscule Schubert cells.
One can naturally extend this notion, and define \emph{unbounded} splitting affine Grassmannians.
These will be closely related to the convolution affine Grassmannians studied above, but in order to connect the geometric and representation-theoretic aspects, we have to be careful about their connected components.
For the rest of this section, we will assume \(G/F\) is essentially unramified (and hence use \thref{notation essentially unramified}).
We also omit the assumption that \(G\) is residually split (although we still assume it is quasi-split); instead we will immediately consider our geometry to live over \(\overline{k}\). 
In particular, \(\Res_{k_i/k} \Gr_{\Hh_i}^{\can}\) will implicitly denote the base change to \(\overline{k}\) of the restriction of scalars of \(\Gr_{\Hh_i}^{\can}\) along \(k_i/k\); we hope this will not cause confusion. 

Recall that the connected components of affine Grassmannians agree with the Galois-coinvariants of the Borovoi fundamental groups \cite[Proposition 1.21]{Zhu:Affine}.
Consequently, the connected components of convolution affine Grassmannians are identified with products of those.

\begin{dfn}\thlabel{definition splitting grassmannians}
	For each \(i\), let \(F_i^{\unr}\subseteq F_i\) be as in \thref{notation essentially unramified}, and fix an ordering on the set of embeddings \(F_i^{\unr} \to \overline{F}\) (of which there are \(e_i\)).
	According to this choice, consider \(\prod_i \Res_{k_i/k}\left(\Gr_{\Hh_i}^{\can,\widetilde{\times} e_i}\right)\), whose group of geometric connected components is canonically isomorphic to \(\prod_i \pi_1(H_i)^{e_if_i} \cong \pi_1(G_{\adj})\).
	
	The \emph{splitting affine Grassmannian} is the ind-(perfect projective) ind-\(\overline{k}\)-scheme given by the fiber product
	\[\Gr_{\Gg}^{\spl} := \pi_1(G) \times_{\pi_1(G_{\adj})}\prod_i \Res_{k_i/k}\left(\Gr_{\Hh_i}^{\can,\widetilde{\times} e_i}\right).\]
	It admits a natural \(L^+\Gg\)-action, and the \emph{splitting Hecke stack} is defined as the étale quotient
	\[\Hck_{\Gg}^{\spl}:=(L^+\Gg\backslash \Gr_{\Gg}^{\spl})^{\et}.\]
\end{dfn}

\begin{rmk}
	\begin{enumerate}
		\item The identification \(\prod_i \pi_1(H_i)^{e_if_i} \cong \pi_1(G_{\adj})\) is not compatible with the Galois actions (cf.~\cite[§5]{Borel:Automorphic}).
		For this reason, we only define \(\Gr_{\Gg}^{\spl}\) over \(\overline{k}\), rather than \(k\).
		When making the connections to Shimura varieties later on, we will only need a certain subscheme of \(\Gr_{\Gg}^{\spl}\), which is actually naturally defined over \(\Spec k\), cf.~\cite[§5.4]{Levin:Local}.
		Nevertheless, we have given the definition above, since it places the results of Sections \ref{Sec:Splitting Gr}-\ref{Sec:Motivic corr} in a natural more general context.
		\item By construction, there are canonical bijections 
		\[\pi_0(\Hck_{\Gg}^{\spl}) \cong \pi_0(\Gr_{\Gg}^{\spl}) \cong \pi_1(G).\]
		\item There is a natural map \(\Gr_{\Gg}^{\can} \to \Gr_{\Gg_{\adj}}^{\can} \cong \prod_i \Res_{k_i/k}\left(\Gr_{\Hh_i}^{\can}\right)\), which induces an isomorphism on each connected component of \(\Gr_{\Gg}^{\can}\) by \cite[§3]{HainesRicharz:Smoothness}.
		This yields maps
		\[\Gr_{\Gg}^{\spl} \to \Gr_{\Gg}^{\can} \qquad \text{ and } \qquad \Hck_{\Gg}^{\spl} \to \Hck_{\Gg}^{\can}.\]
		Since \(\pi_1(G)\to \pi_1(G)_I\) is surjective, these maps are surjective as well. 
	\end{enumerate} 
\end{rmk}

The convolution Grassmannians \(\Gr_{\Hh_i}^{\can,\widetilde{\times} e_i}\) admit a stratification by convolution products of Schubert cells in \(\Gr_{\Hh_i}^{\can}\), inducing a stratification of \[\prod_i \Res_{k_i/k}\left(\Gr_{\Hh_i}^{\can,\widetilde{\times} e_i} \right) \times_{\Spec k} \Spec \overline{k} \cong \prod_i \prod_{f_i} \Gr_{\Hh_i}^{\can,\widetilde{\times} e_i}.\]

\begin{prop}\thlabel{strata in splitting Gr}
	Pulling back the above stratification to \(\Gr_{\Gg}^{\spl}\) yields an \(L^+\Gg\)-stable stratification of \(\Gr_{\Gg}^{\spl}\), naturally indexed by \(X_*(T)^+\).
\end{prop}
\begin{proof}
	This follows since the stratification of \(\prod_i \prod_{f_i} \Gr_{\Hh_i}^{\can,\widetilde{\times} e_i}\) is indexed by \(\prod_i (X_*(T_{H_i})^+)^{e_if_i}\), from the description of the connected components of all objects involved, and from the fact that the following diagram has cartesian squares:
	\begin{equation}\label{equation indices}\begin{tikzcd}
		X_*(T)^+ \arrow[d] \arrow[r] & X_*(T) \arrow[d] \arrow[r] & \pi_1(G)\arrow[d] \\
		\prod_i (X_*(T_{H_i})^+)^{e_if_i} \arrow[r] & \prod_i X_*(T_{H_i})^{e_if_i} \arrow[r] & \prod_i \pi_1(H_i)^{e_if_i}.
	\end{tikzcd}\end{equation}
	Indeed, the left square is cartesian since a cocharacter of \(T\) is dominant exactly when it is dominant as a cocharacter of \(G_{\adj}\). 
	Moreover the right square is cartesian since the middle vertical arrow induces an isomorphism between the kernels of both horizontal arrows, being generated by the same coroots.
\end{proof}

\begin{dfn}\thlabel{defi splitting Schubert}
	\begin{enumerate}
		\item For \(\mu\in X_*(T)^+\), we denote the corresponding stratum in \(\Gr_{\Gg}^{\spl}\) obtained from the previous proposition by \(\Gr_{\Gg,\mu}^{\spl}\), and its closure by \(\Gr_{\Gg,\leq \mu}^{\spl}\).
		These are called \emph{splitting Schubert cells}, resp.~\emph{splitting Schubert varieties}.
		\item Each \(\Gr_{\Gg,\leq \mu}^{\spl}\) admits a natural \(L^+\Gg\)-action.
		In particular, for a tuple \(\mu_\bullet = (\mu_1,\ldots,\mu_t)\in (X_*(T)^+)^t\), we define the convolution products
		\[\Gr_{\Gg,\mu\bullet}^{\spl} := \Gr_{\Gg,\mu_1}^{\spl} \widetilde{\times} \ldots \widetilde{\times} \Gr_{\Gg,\mu_t}^{\spl} \quad \text{ and } \Gr_{\Gg,\leq \mu\bullet}^{\spl} := \Gr_{\Gg,\leq \mu_1}^{\spl} \widetilde{\times} \ldots \widetilde{\times} \Gr_{\Gg,\leq \mu_t}^{\spl}.\]
		More generally, we can also take the convolution product between the strata in \(\Gr_{\Gg}^{\can}\) and \(\Gr_{\Gg}^{\spl}\).
		All these objects admit left \(L^+\Gg\)-actions, and the quotients, denoted \(\Hck_{\Gg,\mu_\bullet}^{\spl}:=(L^+\Gg \backslash \Gr_{\Gg,\leq \mu_\bullet}^{\spl})^{\et}\) etc.~will be called \emph{convolution Hecke stacks}.
	\end{enumerate}
\end{dfn}

\begin{rmk}\thlabel{remarks on splitting grassmannians}
	\begin{enumerate}
		\item Throughout the rest of the paper, we will often compare the geometry of canonical affine Grassmannians (whose Schubert stratification is indexed by \(X_*(T)_I^+\)) and splitting affine Grassmannians (whose Schubert stratification is indexed by \(X_*(T)^+\)). 
		Recall from §\ref{Subsec:Notation} that for \(\mu\in X_*(T)\), we will denote by \(\mu_I\) its image in \(X_*(T)_I\).
		Moreover, we will always denote elements in \(X_*(T)_I\) with a subscript \(I\), even if we do not fix a lift in \(X_*(T)\).
		\item The canonical morphism \(\Gr_{\Gg}^{\spl}\to \Gr_{\Gg}^{\can}\) restricts to a surjective birational map \(\Gr_{\Gg,\leq \mu}^{\spl}\to \Gr_{\Gg,\leq \mu_I}^{\can}\).
		However, it does not in general restrict to a map \(\Gr_{\Gg,\mu}^{\spl}\to \Gr_{\Gg,\mu_I}^{\can}\).
		\item The obvious closure relations hold for the stratification \(\Gr_{\Gg}^{\spl} = \bigsqcup_{\mu\in X_*(T)^+} \Gr_{\Gg,\mu}^{\spl}\).
		\item In contrast to the canonical case, there is in general no natural map \(\Gr_{\Gg,\leq \mu_\bullet}^{\spl} \to \Gr_{\Gg,\leq \sum_i \mu_i}^{\spl}\).
		However, there is still a natural convolution map \[m_{\mu_\bullet} \colon \Gr_{\Gg,\leq \mu_\bullet}^{\spl} \to \Gr_{\Gg}^{\can}.\]
	\end{enumerate}
\end{rmk}

The following corollary is immediate from \thref{convolution Fl WT} and the definition of \(\Gr_{\Gg}^{\spl}\).

\begin{cor}
	The stratification
	\(\Gr_{\Gg}^{\spl} = \bigsqcup_{\mu\in X_*(T)^+} \Gr_{\Gg,\mu}^{\spl}\)
	is Whitney--Tate with perfectly cellular strata.
\end{cor}

Consequently, we can define \(\DTM(\Gr_{\Gg}^{\spl})\) and \(\DTM(\Hck_{\Gg}^{\spl}):=\DTM_{L^+\Gg}(\Gr_{\Gg}^{\spl})\), and equip them with t-structures with hearts \(\MTM\), similar to the case of convolution Grassmannians (\thref{remarks about DATM on LG}).
The connection to representation theory via the Satake equivalence can then be stated as follows.

\begin{prop}\thlabel{Satake for splitting models}
	There are natural equivalences fitting into a commutative diagram
	\[\begin{tikzcd}
		\Rep_{\widehat{G}}(\MTM(\Spec \overline{k})) \arrow[d] \arrow[r, "\cong"', "\Sat"] & \MTM(\Hck_{\Gg}^{\spl}) \arrow[d]\\
		\Rep_{\widehat{G}^I}(\MTM(\Spec \overline{k})) \arrow[r, "\cong", "\Sat"'] & \MTM(\Hck_{\Gg}^{\can}),
	\end{tikzcd}\]
	where the vertical arrows are given by restriction of representations and proper pushforward along \(\Hck_{\Gg}^{\spl}\to \Hck_{\Gg}^{\can}\) respectively.
\end{prop}
\begin{proof}
	First, the fact that pushforward along \(\Hck_{\Gg}^{\spl}\to \Hck_{\Gg}^{\can}\) preserves \(\MTM\) follows from the similar assertion for convolution Grassmannians \cite[Proposition 5.4]{vdH:RamifiedSatake}.
	
	\thref{Satake for convolution Gr} and \cite[Theorem 1.1]{vdH:RamifiedSatake} give a commutative diagram
	\begin{equation}\label{diagram Sat for spl 1}\begin{tikzcd}
		\Rep_{\prod_i \widehat{H}_i^{e_if_i}}(\MTM(\Spec \overline{k})) \arrow[d] \arrow[r, "\cong"] & \MTM(\prod_i \Gr_{\Hh_i}^{\can,\widetilde{\times} e_i}) \arrow[d]\\
		\Rep_{\prod_i \widehat{H}_i^{f_i}}(\MTM(\Spec \overline{k})) \arrow[r, "\cong"] & \MTM(\prod_i \Gr_{\Hh_i}^{\can}),
	\end{tikzcd}\end{equation}
	where the left vertical arrow is given by restriction along the inertia invariants \(\prod_i \widehat{H}_i^{f_i}\subseteq \prod_i \widehat{H}_i^{e_if_i}\), and the right vertical arrow by proper pushforward.
	Moreover, an object in \(\MTM(\Gr_{\Gg}^{\spl})\) is equivalent to an object in \(\MTM(\prod_i \Gr_{\Hh_i}^{\can,\widehat{\times} e_i})\) together with a refinement of its \(\prod_i \pi_1(H_i)^{e_if_i}\)-grading (arising via the connected components) to a \(\pi_1(G)\)-grading.
	Using the above equivalence, we see that \(\MTM(\Gr_{\Gg}^{\spl})\) is equivalent to representations (internally in \(\MTM(\Spec \overline{k})\)) of 
	\[\prod_i \widehat{H}_i^{e_if_i} \overset{D(\prod_i \pi_1(H_i)^{e_if_i})}{\times} D(\pi_1(G)),\] 
	where \(D(-)\) denotes the diagonalizable group scheme with given character group.
	But this is exactly \(\widehat{G}\), giving an equivalence \begin{equation}\label{diagram Sat for spl 2}\begin{tikzcd}
		\Rep_{\widehat{G}}(\MTM(\Spec \overline{k})) \arrow[d] \arrow[r, "\cong"] & \MTM(\Gr_{\Gg}^{\spl}) \arrow[d]\\
		\Rep_{\prod_i \widehat{H}_i^{e_if_i}}(\MTM(\Spec \overline{k})) \arrow[r, "\cong"] & \MTM(\prod_i \Gr_{\Hh_i}^{\can,\widetilde{\times} e_i}).
	\end{tikzcd}\end{equation}
	(We note that a similar argument was used while identifying the dual group in the motivic Satake equivalence \cite[Lemma 9.5]{vdH:RamifiedSatake}.)
	A similar argument also shows that the composition of the diagrams \eqref{diagram Sat for spl 1} and \eqref{diagram Sat for spl 2} factors through \(\Rep_{\widehat{G}^I}(\MTM(\Spec \overline{k})) \cong \MTM(\Gr_{\Gg}^{\can})\).
	
	Finally, we can replace all category of mixed Tate motives by their \(L^+\Gg\)-equivariant versions by \thref{Satake for convolution Gr}, giving the desired equivalence.
\end{proof}

Again, it is not clear how to directly construct a monoidal structure on \(\MTM(\Hck_{\Gg}^{\spl})\) that makes the above equivalence monoidal.

\begin{nota}\thlabel{notation representations}
	Recall the IC-functors as defined in e.g.~\cite[Definition 6.17]{vdH:RamifiedSatake}.
	For \(\mu_I\in X_*(T)^+_I\) and \(n\in \IZ\), we will denote by \(V_{\mu_I}^{\can}(n)\in \Rep_{\widehat{G}^I}(\MTM(\Spec \overline{k}))\) the object corresponding to \(\IC_{\mu_I}(\unit(n)) \in \MTM(\Hck_{\Gg}^{\can})\).
	
	Similarly, we can define IC-functors
	\[\IC_\mu\colon \MTM(\Spec \overline{k}) \to \MTM(\Hck_{\Gg}^{\spl}),\]
	for \(\mu\in X_*(T)^+\).
	Again, we will denote by \(V_\mu^{\spl}(n) \in \Rep_{\widehat{G}}(\MTM(\Spec \overline{k}))\) the object corresponding to \(\IC_\mu(\unit(n))\in \MTM(\Hck_{\Gg}^{\spl})\).
\end{nota}

\subsection{Satake correspondences}

We now recall and generalize the Satake correspondences from \cite[(3.1.14)]{XiaoZhu:Cycles}.

\begin{dfn}
	Let \(\mu_\bullet,\lambda_\bullet,\kappa_\bullet\) be tuples of elements in \(X_*(T)^+\).
	The \emph{Satake correspondence} \(\Gr_{\lambda_\bullet\mid\mu_\bullet}^{\spl}\) is defined as
	\[\Gr_{\lambda_\bullet\mid\mu_\bullet}^{\spl} := \Gr_{\Gg,\leq \mu_\bullet}^{\spl} \times_{\Gr_{\Gg}^{\can}} \Gr_{\Gg,\leq \lambda_\bullet}^{\spl}.\]
	These correspondences admit composition morphisms
	\begin{equation}\label{composition of Satake corr}\Gr_{\kappa_\bullet\mid \lambda_\bullet}^{\spl} \times_{\Gr_{\Gg,\lambda_\bullet}^{\spl}} \Gr_{\lambda_\bullet\mid\mu_\bullet}^{\spl} \to \Gr_{\kappa_\bullet\mid \mu_\bullet}^{\spl}.\end{equation}
	Moreover, they admit a natural (diagonal) \(L^+\Gg\)-action, and taking the étale quotient yields the Satake correspondence between Hecke stacks
	\[\Hck_{\lambda_\bullet\mid \mu_\bullet}^{\spl} := (L^+\Gg \backslash \Gr_{\lambda_\bullet\mid \mu_\bullet}^{\spl})^{\et}.\]
\end{dfn}

For a tuple \(\mu_\bullet\), let us denote \(\langle 2\rho,\mu_\bullet \rangle := \sum_i \langle 2\rho,\mu_i\rangle\).
Recall from §\ref{Subsec:Notation} that \(\HBM\) denotes the \(\ell\)-adic étale Borel--Moore homology, for which we have chosen an auxiliary prime \(\ell\neq p\).
In the proposition below, we will use the notion of motivic correspondences \(\Corr\), for which we refer to \thref{defi motivic corr}.

\begin{prop}\thlabel{basic properties of Satake correspondences}
	Let \(\mu_\bullet,\lambda_\bullet\) be tuples of elements in \(X_*(T)^+\).
	\begin{enumerate}
		\item The Satake correspondence \(\Gr_{\lambda_\bullet\mid\mu_\bullet}^{\spl}\) has dimension \(\leq \langle \rho,\lambda_\bullet+\mu_\bullet\rangle\).
		\item For any \(\Ff_1\in \MTM(\Gr_{\Gg,\lambda_\bullet}^{\spl})\) and \(\Ff_2\in \MTM(\Gr_{\Gg,\mu_\bullet}^{\spl})\), there is a natural isomorphism 
		\[\Hom_{\MTM(\Gr_{\Gg})^{\can}}(m_{\lambda_\bullet,*}\Ff_1,m_{\mu_\bullet,*}\Ff_2) \cong \Corr_{\Gr_{\lambda_\bullet\mid \mu_\bullet}^{\spl}}(\Ff_1,\Ff_2).\]
		Similar isomorphisms exist if \(\Ff_1,\Ff_2\) are assumed to be \(L^+\Gg\)-equivariant instead, i.e., they lie in \(\MTM(\Hck_{\Gg}^{\spl})\).
		Moreover, these isomorphisms are compatible with composition (the composition of correspondences is defined via \thref{composition of correspondences} along \eqref{composition of Satake corr}).
		\item \label{homs and irr} Let \(V_{\mu_\bullet} = \bigotimes_{i=1}^t V_{\mu_i}\), where \(V_{\mu_i}\) is the irreducible \(\overline{\IQ}_\ell\)-linear \(\widehat{G}\)-representation of highest weight \(\mu_i\), and similarly for \(\lambda_\bullet\).
		Then there is a canonical isomorphism
		\[\Hom_{\widehat{G}^I}(V_{\lambda_\bullet},V_{\mu_\bullet}) \cong \HBM_{\langle 2\rho,\lambda_\bullet + \mu_\bullet \rangle}(\Gr_{\lambda_\bullet\mid \mu_\bullet}^{\spl}).\]
	\end{enumerate}
	Similar results hold when replacing \(\Gr_{\lambda_\bullet \mid \mu_\bullet}^{\spl}\) by a fiber product over \(\Gr_{\Gg}^{\can}\) involving convolution products of canonical Schubert varieties, or even convolution products of canonical Schubert varieties with splitting Schubert varieties.
\end{prop}
\begin{proof}
	This is well known.
	Briefly, (1) follows from the semi-smallness of the convolution morphism \cite[Lemma 5.6]{vdH:RamifiedSatake}.
	(2) follows from proper base change along the cartesian diagram
	\[\begin{tikzcd}
		\Gr_{\lambda_\bullet\mid \mu_\bullet}^{\spl} \arrow[d] \arrow[r] & \Gr_{\Gg,\mu_\bullet}^{\spl} \arrow[d, "m_{\mu_\bullet}"]\\
		\Gr_{\Gg,\lambda_\bullet}^{\spl} \arrow[r, "m_{\lambda_\bullet}"'] & \Gr_{\Gg}^{\can}.
	\end{tikzcd}\]
	Finally, (3) can be proven similar to \cite[Proposition 3.1.10 (3)]{XiaoZhu:Cycles}.
\end{proof}

For applications later on, we will need to consider correspondences constructed using both canonical and splitting models of affine Grassmannians.
To not overload the notation, we will only state the results in the generality we need, and leave the obvious generalizations to the reader.

\begin{dfn}\thlabel{specific Satake corr}
	Let \(\mu\in X_*(T)^+\) and \(\lambda_I,\nu_I\in X_*(T)_I^+\).
	The Satake correspondence \(\Gr_{\nu_I,\mu\mid \lambda_I}^{\spl}\) is defined as the fiber product
	\[\Gr_{\nu_I,\mu\mid \lambda_I}^{\spl} := (\Gr_{\Gg,\leq \nu_I}^{\can} \widetilde{\times} \Gr_{\Gg,\leq \mu}^{\spl}) \times_{\Gr_{\Gg}^{\can}} \Gr_{\Gg,\leq \lambda_I}^{\can}.\]
	By \thref{basic properties of Satake correspondences}, this scheme has dimension \(\leq \langle \rho,\mu_I+\nu_I+\lambda_I\rangle\).
	We denote the set of \(\langle \rho,\mu_I+\nu_I+\lambda_I\rangle\)-dimensional irreducible components by
	\[\IS_{\nu_I,\mu\mid \lambda_I}^{\spl}:= \Irr^{\langle \rho,\mu_I+\nu_I+\lambda_I\rangle}(\Gr_{\nu_I,\mu\mid \lambda_I}^{\spl}).\]
	For \(\bfa\in \IS_{\nu_I,\mu\mid \lambda_I}^{\spl}\), the corresponding irreducible component will be denoted by \(\Gr_{\nu_I,\mu\mid \lambda_I}^{\spl,\bfa}\subseteq \Gr_{\nu_I,\mu\mid \lambda_I}^{\spl}\).
	
	Similarly, we define
	\[\Gr_{\nu_I,\mu_I\mid \lambda_I}^{\can}:= (\Gr_{\Gg,\leq \nu_I}^{\can} \widetilde{\times} \Gr_{\Gg,\leq \mu_I}^{\can}) \times_{\Gr_{\Gg}^{\can}} \Gr_{\Gg,\leq \lambda_I}^{\can},\]
	which is also of dimension \(\leq \langle \rho,\mu_I+\nu_I+\lambda_I \rangle\).
	We also set
	\[\IS_{\nu_I,\mu\mid \lambda_I}^{\can} := \Irr^{\langle \rho,\nu_I+\mu_I+\lambda_I\rangle}(\Gr_{\nu_I,\mu_I\mid \lambda_I}^{\can}).\]
	Both versions admit natural \(L^+\Gg\)-actions, and we denote their quotients by
	\[\Hck_{\nu_I,\mu\mid \lambda_I}^{\spl} \quad \text{ and } \quad \Hck_{\nu_I,\mu_I\mid \lambda_I}^{\can}.\]
\end{dfn}

For \(\mu\in X_*(T)\), we denote by \(V_\mu\) the irreducible \(\IQ\)-linear \(\widehat{G}\)-representation of highest weight \(\mu\), and similar for \(\lambda_I\in X_*(T)_I^+\) and \(\widehat{G}^I\) (using \cite[Lemma 4.10]{Zhu:Ramified}).
By \thref{basic properties of Satake correspondences} (3), \(\IS_{\nu_I,\mu\mid \lambda_I}^{\spl}\) forms a natural basis of \(\Hom_{\widehat{G}^I}(V_{\nu_I} \otimes V_\mu, V_{\lambda_I})\).
Similarly, \(\IS_{\nu_I,\mu_I\mid \lambda_I}^{\can}\) forms a natural basis of \(\Hom_{\widehat{G}^I}(V_{\nu_I} \otimes V_{\mu_I}, V_{\lambda_I})\).

\begin{lem}\thlabel{comparing Satake cycles}
	Let \(\mu\in X_*(T)^+\) and \(\lambda_I,\nu_I\in X_*(T)_I^+\).
	Then there is a natural map
	\[\beta^{\IS}\colon \IS_{\nu_I,\mu\mid \lambda_I}^{\spl} \to \coprod_{\mu_I'\leq \mu_I\in X_*(T)_I^+} \IS_{\nu_I,\mu_I'\mid \lambda_I}^{\can}.\]
	Moreover, for any \(\bfa\in \IS_{\nu_I,\mu\mid \lambda_I}^{\spl}\), the image of \(\Gr_{\nu_I,\mu\mid \lambda_I}^{\spl, \bfa}\) in \(\Gr_{\nu_I,\mu_I\mid \lambda_I}^{\can}\) contains \(\Gr_{\nu_I,\mu_I'\mid \lambda_I}^{\can, \beta^{\IS}(\bfa)}\).
\end{lem}
\begin{proof}
	Let \(\bfa\in \IS_{\nu_I,\mu\mid \lambda_I}^{\spl}\).
	Then the image of the generic point of \(\Gr_{\nu_I,\mu\mid \lambda_I}^{\spl,\bfa}\) under the natural map 
	\begin{equation}\label{SatakeEq1}\Gr_{\nu_I,\mu\mid \lambda_I}^{\spl} \to \Gr_{\nu_I,\mu_I\mid \lambda_I}^{\can}\end{equation}
	lies in 
	\begin{equation}\label{SatakeEq2}(\Gr_{\Gg,\leq \nu_I}^{\can} \widetilde{\times} \Gr_{\Gg,\mu_I'}^{\can})\times_{\Gr_{\Gg}^{\can}} \Gr_{\Gg,\leq \lambda_I}^{\can}\end{equation}
	for some unique \(\mu_I'\leq \mu_I\in X_*(T)_I^+\).
	Let \(X\) be the intersection of the image of \(\Gr_{\nu_I,\mu\mid \lambda_I}^{\spl,\bfa}\) in \(\Gr_{\nu_I,\mu_I\mid \lambda_I}^{\can}\) with \eqref{SatakeEq2}.
	
	Note that the fibers of \eqref{SatakeEq1} over \eqref{SatakeEq2} have dimension at most \(\langle \rho,\mu_I-\mu_I'\rangle\).
	Since \(\Gr_{\nu_I,\mu\mid \lambda_I}^{\spl,\bfa}\) has dimension \(\langle \rho,\nu_I+\mu_I+\lambda_I\rangle\), the dimension of \(X\) is at least
	\[\langle \rho,\nu_I+\mu_I+\lambda_I\rangle - \langle \rho,\mu_I-\mu_I'\rangle = \langle \rho,\nu_I+\mu_I'+\lambda_I\rangle.\]
	Thus the closure of \(X\) in \(\Gr_{\nu_I,\mu_I'\mid \lambda_I}^{\can}\) defines an element in \(\IS_{\nu_I,\mu_I'\mid \lambda_I}^{\can}\), giving us the desired map.
	
	To prove the final statement, note that \(\Gr_{\nu_I,\mu\mid \lambda_I}^{\spl}\to \Gr_{\nu_I,\mu_I\mid \lambda_I}^{\can}\) is perfectly proper, so that the image of \(\Gr_{\nu_I,\mu\mid \lambda_I}^{\spl,\bfa}\) is closed.
	Its intersection with the irreducible \(\Gr_{\nu_I,\mu_I'\mid \lambda_I}^{\can,\beta^{\IS}(\bfa)}\) is then closed, non-empty, and contains the generic point.
	Hence, this intersection must be the whole \(\Gr_{\nu_I,\mu_I'\mid \lambda_I}^{\can,\beta^{\IS}(\bfa)}\), concluding the proof.
\end{proof}

\subsection{Semi-infinite orbits}

Next, we study semi-infinite orbits in splitting affine Grassmannians, generalizing \cite[§3.2]{vdH:RamifiedSatake}, although we will not need semi-infinite orbits for general parabolics.
Recall that \(G/F\) is essentially unramified, and that all the geometry is over \(\Spec \overline{k}\).

Recall from e.g.~\cite[§3.2]{vdH:RamifiedSatake}, that the twisted affine Grassmannian admits a stratification into semi-infinite orbits
\[\Gr_{\Gg}^{\can} = \bigsqcup_{\nu_I\in X_*(T)_I} \Ss_{\nu_I}^{\can},\]
corresponding to the affine Grassmannian \(\Gr_{\Bb}^{\can}:=LB/L^+\Bb\) (where \(\Bb\) is the schematic closure of \(B\) in \(\Gg\)), or equivalently, the attractor locus for the \(\IG_m\)-action on \(\Gr_{\Gg}^{\can}\) induced by a regular dominant cocharacter \cite[Theorem 5.2]{AGLR:Local}.
As in \cite[§5.1]{vdH:RamifiedSatake}, one can take the twisted products of these semi-infinite orbits, yielding a decomposition of convolution Grassmannians.
Via \thref{definition splitting grassmannians}, this immediately yields a decomposition of splitting affine Grassmannians, which we denote
\begin{equation}
	\Gr_{\Gg}^{\spl} = \bigsqcup_{\nu\in X_*(T)} \Ss_\nu^{\spl};
\end{equation}
the fact that the index set is \(X_*(T)\) follows from \eqref{equation indices}.
This is a stratification, since this already holds for canonical Grassmannians, and hence for convolution Grassmannians.
Note that under the natural map
\[\Gr_{\Gg}^{\spl}\to \Gr_{\Gg}^{\can},\]
\(\Ss_{\nu}^{\spl}\) gets sent to \(\Ss_{\nu_I}^{\can}\); this follows from \cite[(3.2.7)]{XiaoZhu:Cycles}.

We now have two stratifications of \(\Gr_{\Gg}^{\can}\) and \(\Gr_{\Gg}^{\spl}\), by Schubert cells and by semi-infinite orbits.
Recall from \cite[Proposition 3.11]{vdH:RamifiedSatake} that if \(\Gr_{\Gg,\mu_I}^{\can} \cap \Ss_{\nu_I}^{\can}\) is non-empty, it is equidimensional of dimension \(\langle \rho,\mu_I+\nu_I\rangle\).
Using \cite[(3.2.8)]{XiaoZhu:Cycles}, which also holds for general quasi-split groups, it follows that if \(\Gr_{\Gg,\mu}^{\spl} \cap \Ss_\nu^{\spl}\) is non-empty, then it is equidimensional of dimension \(\langle \rho,\mu+\nu\rangle\).

\begin{dfn}\thlabel{Defi MV sets}
	Let \(\mu_I\in X_*(T)_I^+\) and \(\nu_I\in X_*(T)_I\).
	Then the set of \emph{(canonical) Mirkovic--Vilonen (or MV) cycles} \(\IM\IV^{\can}_{\mu_I}(\nu_I)\) is defined as the set of irreducible components of \(\Ss_{\nu_I}^{\can}\cap \Gr_{\Gg,\mu_I}^{\can}\).
	These are in bijection with the irreducible components of \(\Ss_{\nu_I}^{\can}\cap \Gr_{\Gg,\leq \mu_I}^{\can}\), and we will often identify these two sets.
	For \(\bfb\in \IM\IV_{\mu_I}^{\can}(\nu_I)\), we denote the corresponding irreducible component by \((\Ss_{\nu_I}^{\can}\cap \Gr_{\Gg,\leq \mu_I}^{\can})^{\bfb}\).
	
	Similarly, if \(\mu\in X_*(T)^+\) and \(\nu\in X_*(T)\), the set of \emph{splitting Mirkovic--Vilonen cycles} \(\IM\IV^{\spl}_\mu(\nu)\) is defined as the set of irreducible components of \(\Ss_\nu^{\spl}\cap \Gr_{\Gg,\mu}^{\spl}\).
\end{dfn}

\begin{lem}\thlabel{comparing MV cycles}
	Let \(\mu\in X_*(T)^+\) and \(\nu\in X_*(T)\).
	Then there is a natural map
	\[\beta^{\IM\IV}\colon \IM\IV^{\spl}_\mu(\nu) \to \coprod_{\mu_I'\leq \mu_I\in X_*(T)_I^+} \IM\IV^{\can}_{\mu_I'}(\nu_I).\]
\end{lem}
\begin{proof}
	Let \((\Ss_\nu^{\spl}\cap \Gr_{\Gg,\leq \mu}^{\spl})^{\bfb}\subseteq \Gr_{\Gg}^{\spl}\) correspond to some \(\bfb\in \IM\IV^{\spl}_\mu(\nu)\).
	Then its image \(X\) under the natural map \(\Gr_{\Gg}^{\spl}\to \Gr_{\Gg}^{\can}\) is still irreducible, and lies in \(\Ss_{\nu_I}^{\can}\cap \Gr_{\Gg,\leq \mu_I}^{\can}\).
	Let \(\mu_I'\in X_*(T)_I^+\) such that \(\Gr_{\Gg,\mu_I'}^{\can}\) contains the image of the unique generic point of \((\Ss_\nu^{\spl}\cap \Gr_{\Gg,\leq \mu}^{\spl})^{\bfb}\); then \(X\cap \Gr_{\Gg,\mu_I'}^{\can}\) is open dense in \(X\), and \(\mu'\leq \mu\in X_*(T)_I^+\).
	
	Now, the fibers of the convolution map \(\Gr_{\Gg}^{\spl}\to \Gr_{\Gg}^{\can}\) over \(\Gr_{\Gg,\mu_I'}^{\can}\) have dimension at most \(\langle \rho,\mu_I-\mu_I'\rangle\).
	Since \((\Ss_\nu^{\spl}\cap \Gr_{\Gg,\leq \mu}^{\spl})^{\bfb}\) has dimension \(\langle \rho,\mu_I+\nu_I\rangle\), \(X\) has dimension at least 
	\[\langle \rho,\mu_I+\nu_I-\mu_I+\mu_I'\rangle = \langle \rho,\mu_I'+\nu_I\rangle.\]
	Thus, the closure of \(X\cap \Gr_{\Gg,\mu_I'}^{\can}\) in \(\Gr_{\Gg,\leq \mu_I'}^{\can}\cap \Ss_{\nu_I}^{\can}\) determines an element in \(\IM\IV^{\can}_{\mu_I'}(\nu_I)\), yielding the desired map.
\end{proof}

\begin{rmk}\thlabel{remark MV cycles}
	Via the geometric Satake equivalence, the canonical Mirkovic--Vilonen cycles yield natural bases of irreducible representations (in characteristic \(0\)) of \(\widehat{G}^I\).
	Consequently, the splitting Mirkovic--Vilonen cycles yield natural bases of irreducible representations of \(\widehat{G}\).
	Thus, the map constructed in the lemma above gives a geometric description of the restriction functor from \(\widehat{G}\)-representations to \(\widehat{G}^I\)-representations, using \thref{Satake for splitting models}.
	
	On the other hand, the canonical Mirkovic--Vilonen cycles can be defined for arbitrary quasi-split \(G\), not necessarily essentially unramified.
	Moreover, in the context of the ramified Satake equivalence, restriction of representations along \(\widehat{G}^I\subseteq \widehat{G}\) is given by nearby cycles \cite{Zhu:Ramified} along a deformation of a split affine Grassmannian into \(\Gr_{\Gg}\).
	On this split affine Grassmannian, one also has MV cycles.
	It is an interesting question to find a direct map from MV-cycles on such a split affine Grassmannian to canonical MV cycles on \(\Gr_{\Gg}\), which would correspond to the restriction along \(\widehat{G}^I\subseteq \widehat{G}\).
\end{rmk}

\begin{prop}\thlabel{Satake to MV cycles}
	Let \(\mu\in X_*(T)^+\), and \(\nu_I,\lambda_I\in X_*(T)_I\) such that \(\nu_I\) and \(\lambda_I+\nu_I\) are dominant.
	Then there is a unique injective map
	\[\iota_{\IM\IV}^{\spl}\colon \IS_{\nu_I,\mu\mid \lambda_I+\nu_I}^{\spl} \to \coprod_{\lambda\in X_*(T)\colon \lambda\mapsto \lambda_I\in X_*(T)_I} \IM\IV^{\spl}_\mu(\lambda),\]
	characterized such that for each \(\bfa\in \IS_{\nu_I,\mu\mid \lambda_I+\nu_I}^{\spl}\), we have 
	\[\Gr_{\nu_I,\mu\mid \lambda_I+\nu_I}^{\spl,\bfa} \cap (\bigsqcup_{\lambda\in X_*(T)\colon \lambda\mapsto \lambda_I\in X_*(T)_I} \Ss_{\nu_I}^{\can}\widetilde{\times}\Ss_\lambda^{\spl}) = (\Ss_{\nu_I}^{\can}\cap \Gr_{\Gg,\leq \nu_I}^{\can}) \widetilde{\times} (\Ss_\lambda^{\spl}\cap \Gr_{\Gg,\leq \mu}^{\spl})^{\iota_{\IM\IV}^{\spl}(\bfa)}\]
	as subschemes of \(\Gr_{\Gg}^{\can}\widetilde{\times} \Gr_{\Gg}^{\spl}\).
	
	Similarly, there is a unique injective map
	\[\iota_{\IM\IV}^{\can} \colon \IS_{\nu_I,\mu_I\mid \lambda_I+\nu_I}^{\can} \to \IM\IV^{\can}_{\mu_I}(\lambda_I),\]
	such that for each \(\bfa\in \IS_{\nu_I,\mu_I\mid \lambda_I+\nu_I}^{\can}\) we have 
	\[\Gr_{\nu_I,\mu_I\mid \lambda_I+\nu_I}^{\can,\bfa}\cap (\Ss_{\nu_I}^{\can} \widetilde{\times} \Ss_{\lambda_I}^{\can}) = (\Ss_{\nu_I}^{\can} \cap \Gr_{\Gg,\leq \nu_I}^{\can}) \widetilde{\times} (\Ss_{\lambda_I}^{\can} \cap \Gr_{\Gg,\leq \mu_I}^{\can})^{\iota_{\IM\IV}^{\can}(\bfa)}\]
	as subschemes of \(\Gr_{\Gg}^{\can} \widetilde{\times} \Gr_{\Gg}^{\can}\).
	
	Finally, the diagram
	\begin{equation}\label{comparison Satake vs MV}\begin{tikzcd}
		\IS_{\nu_I,\mu\mid \lambda_I+\nu_I} \arrow[r, "\iota_{\IM\IV}^{\spl}"] \arrow[d, "\beta^{\IS}"'] & \coprod_{\lambda\in X_*(T)\colon \lambda\mapsto \lambda_I\in X_*(T)_I} \IM\IV^{\spl}_\mu(\lambda) \arrow[d, "\beta^{\IM\IV}"]\\
		\coprod_{\mu_I'\leq \mu_I\in X_*(T)_I^+} \IS_{\nu_I,\mu_I'\mid \lambda_I+\nu_I} \arrow[r, "\iota_{\IM\IV}^{\can}"'] & \coprod_{\mu_I'\leq \mu_I\in X_*(T)_I^+} \IM\IV^{\can}_{\mu_I'}(\lambda_I)
	\end{tikzcd}\end{equation}
	commutes and is cartesian.
\end{prop}
\begin{proof}
	The map \(\iota_{\IM\IV}^{\can}\) be constructed similar to \cite[Lemma 3.2.7]{XiaoZhu:Cycles}; we explain how to deduce \(\iota_{\IM\IV}^{\spl}\).
	
	Let \(\bfa\in \IS_{\nu_I,\mu\mid \lambda_I+\nu_I}^{\spl}\).
	It suffices to show the intersection
	\[Z^{\spl}:=\Gr_{\nu_I,\mu\mid \lambda_I+\nu_I}^{\spl,\bfa} \cap (\bigsqcup_{\lambda\in X_*(T)\colon \lambda\mapsto \lambda_I\in X_*(T)_I} \Ss_{\nu_I}^{\can}\widetilde{\times}\Ss_\lambda^{\spl}) \subseteq \Gr_{\Gg}^{\can} \widetilde{\times} \Gr_{\Gg}^{\spl}\]
	is irreducible of dimension \(\langle \rho,\lambda_I+\mu_I+2\nu_I\rangle\).
	Indeed, then this subscheme must be of the form \((\Ss_{\nu_I}^{\can}\cap \Gr_{\Gg,\leq \nu_I}^{\can}) \widetilde{\times} (\Ss_\lambda^{\spl}\cap \Gr_{\Gg,\leq \mu}^{\spl})^{\iota_{\IM\IV}^{\spl}(\bfa)}\) for some uniquely defined \(\lambda\in X_*(T)\) lifting \(\lambda_I\) and \(\iota_{\IM\IV}^{\spl}(\bfa)\in \IM\IV^{\spl}_\mu(\lambda)\), determining the map \(\iota_{\IM\IV}^{\spl}\).
	As the base change of the open immersion \[(\Ss_{\nu_I}^{\can}\cap \Gr_{\Gg,\leq \nu_I}^{\can}) \times (\Ss_{\lambda_I+\nu_I}^{\can}\cap \Gr_{\Gg,\leq \lambda_I+\nu_I}^{\can}) \subseteq \Gr_{\Gg,\leq \nu_I}^{\can} \times \Gr_{\Gg,\leq \lambda_I+\nu_I}^{\can},\]
	we see that \(Z\) is an open subscheme of \(\Gr_{\nu_I,\mu\mid \lambda_I+\nu_I}^{\spl,\bfa}\).
	In particular, the irreducibility and desired dimension will follow, once we know that \(Z^{\spl}\) is nonempty.
	
	For this, note that by \thref{comparing Satake cycles}, the image of \(Z^{\spl}\) under \(\Gr_{\nu_I,\mu\mid \lambda_I+\nu_I}^{\spl}\to \Gr_{\nu_I,\mu_I\mid \lambda_I+\nu_I}^{\can}\) contains
	\[Z^{\can} := \Gr_{\nu_I,\mu_I'\mid \lambda_I+\nu_I}^{\can,\beta^{\IS}(\bfa)} \cap \Ss_{\nu_I}^{\can} \widetilde{\times} \Ss_{\lambda_I}^{\can} \subseteq \Gr_{\Gg}^{\can} \widetilde{\times} \Gr_{\Gg}^{\can},\]
	for some \(\mu_I'\leq \mu_I\).
	Since \(Z^{\can}\) is non-empty by the proof of \cite[Lemma 3.2.7]{XiaoZhu:Cycles}, the same holds for \(Z^{\spl}\).
	
	The commutativity of \eqref{comparison Satake vs MV} follows by construction.
	To see that it is moreover cartesian, since \(\iota_{\IM\IV}^{\spl}\) is injective it suffices to observe that the fibers of \(\beta^{\IS}\) and \(\beta^{\IM\IV}\) over an element indexed by \(\mu_I'\leq \mu_I\) have the same cardinality.
	Indeed, this cardinality is equal to the multiplicity of \(V_{\mu_I'}\) in the restriction of \(V_\mu\) to \(\widehat{G}^I\), which follows from \thref{basic properties of Satake correspondences} (3) and \thref{remark MV cycles} respectively.
\end{proof}

\section{Splitting affine Deligne--Lusztig varieties}\label{Sec:ADLV}

To understand Newton strata of Shimura varieties, it is important to study the irreducible components of affine Deligne--Lusztig varieties, which are related to Shimura varieties via Rapoport--Zink uniformization.
Our goal in this section is to define variants of the affine Deligne--Lusztig varieties in the context of splitting models, and study their irreducible components.

\subsection{Recollections on the Kottwitz set}

For now, we let \(G\) be any quasi-split reductive group over the nonarchimedean local field \(F\), and \(\Gg/\Oo_F\) a very special parahoric model.
Recall that affine Deligne--Lusztig varieties are attached to pairs of elements: a Schubert cell in \(\Gr_{\Gg}^{\can}\), and an element in the Kottwitz set \(B(G)\), classifying \(G\)-isocrystals \cite{Kottwitz:Isocrystals1,Kottwitz:Isocrystals2}.
This Kottwitz set is defined as \(B(G):=G(\breve{F})/b\sim g^{-1}b\sigma(g)\), i.e., the set of \(\sigma\)-twisted conjugacy classes in \(G(\breve{F})\).
To such a \(b\in B(G)\) we can also associate the twisted centralizer group, which is the reductive \(F\)-group given by
\[J_b\colon R\mapsto \{g\in G(R\otimes_F \breve{F})\mid g^{-1}b\sigma(g)=b\}.\]
Up to isomorphism, this does not depend on the choice of representative of \(b\in B(G)\).
Kottwitz \cite{Kottwitz:Isocrystals1} defines two invariants of \(B(G)\): the Newton map 
\[\nu\colon B(G)\to (X_*(T)^+_{\IQ})^{\Gamma_F}\colon b\mapsto \nu_b,\]
where we recall that \(T\subseteq B\subseteq G\) are a maximal torus and Borel defined over \(F\), and the Kottwitz map
\[\kappa\colon B(G)\to \pi_1(G)_{\Gamma_F}.\]
Together, they characterize \(B(G)\) in the sense that \(\nu\times \kappa\colon B(G)\to (X_*(T)^+_{\IQ})^{\Gamma_F}\times \pi_1(G)_{\Gamma_F}\) is injective.
This induces a partial order \(\leq\) on \(B(G)\), by declaring \(b\leq b'\) when \(\kappa(b)=\kappa(b')\) and \(\nu_b\leq \nu_{b'}\) with respect to the (rationalized) dominance order.
The minimal elements for this partial order are exactly the basic elements, defined as those \(b\) for which \(\nu_b\) is central.
Moreover, \(\kappa\) induces a bijection between the basic elements in \(B(G)\) and \(\pi_1(G)_{\Gamma_F}\) \cite[Proposition 5.6]{Kottwitz:Isocrystals1}.

For a dominant cocharacter \(\mu\), we denote
\[B(G,\mu)=\left\{b\in B(G)\mid \kappa(b) = \mu^\flat \text{ and } \nu(b) \leq \overline{\mu}\right\},\]
where \(\mu^\flat\) is the image of \(\mu\) in \(\pi_1(G)_{\Gamma_F}\), and \(\overline{\mu} = \frac{1}{|\Gal(\widetilde{F}/F)|}\sum_{\gamma\in \Gal(\widetilde{F}/F)}\gamma(\mu)\), for some finite \(\widetilde{F}/F\) splitting \(G\).
Then \(B(G,\mu)\) contains a unique basic element by \cite[§5]{Kottwitz:Isocrystals1}.

In case \(G\) splits over an unramified extension, the so-called unramified elements of \(B(G)\) are introduced and studied in \cite[§4.2]{XiaoZhu:Cycles}; these are the elements in the image of \(B(T)\to B(G)\), or equivalently, those elements for which \(J_b\) is an unramified \(F\)-group.
The following notion is the natural generalization to general quasi-split groups.

\begin{dfn}\thlabel{very special Kottwitz elements}
	The set \(B(G)_{\vsp}\) of \emph{very special} elements in \(B(G)\) is the image of the map \(B(T)\to B(G)\).
	Since any two choices of \(T\) are rationally conjugate, \(B(G)_{\vsp}\) is independent of this choice.
\end{dfn}

The same definition appears in \cite[Definition 2.2]{Hamann:Geometric}, where these elements are still called unramified. 
However, if \(G\) is not unramified, then neither will be \(J_b\).
Instead, we will see below that \(b\in B(G)_{\vsp}\) if and only if \(J_b\) is quasi-split, or equivalently, \(J_b\) admits a very special parahoric \cite[Lemma 6.1]{Zhu:Ramified}; whence the different terminology.

Denote by \(W_0 = W^{\Gamma_F}\) the relative Weyl group of \(G\); it acts on \(X_*(T)_{\Gamma_F}\).

\begin{lem}\thlabel{classification of spherical elements}
	The map \(X_*(T)\to B(G)\colon \mu\mapsto [\varpi^\mu]\) induces a bijection
	\[X_*(T)_{\Gamma_F}/W_0 \cong B(G)_{\vsp}.\]
\end{lem}
\begin{proof}
	This is proven in \cite[Lemma 4.2.2]{XiaoZhu:Cycles} in case \(G\) is unramified, and \cite[Lemma 2.3]{Hamann:Geometric} in general. 
\end{proof}

Now, consider the Langlands dual group \(\widehat{G}\) of \(G\), defined over \(\IQ\).
This is a pinned reductive group, which comes equipped with a \(\Gamma_F\)-action preserving the maximal torus and Borel \(\widehat{T}\subseteq \widehat{B}\subseteq \widehat{G}\).
Then the fixed points \(\widehat{G}^{\Gamma_F}\) have a split reductive neutral component and constant \(\pi_0\),  \cite[Proposition 5.4 (1) and (7)]{ALRR:Fixed}.
Moreover, a maximal diagonalizable subgroup is given by \(\widehat{T}^{\Gamma_F}\), with character group \(X_*(T)_{\Gamma_F}\).
The same argument as \cite[§5.5]{Haines:Dualities} (cf.~also \cite[Lemma 4.10]{Zhu:Ramified}) then shows that the category of algebraic \(\IQ\)-linear representations of \(\widehat{G}^{\Gamma_F}\) is semisimple, with simple objects the highest weight representations, which are indexed by \(X_*(T)^+_{\Gamma_F}\).

\begin{lem}\thlabel{unramified in BGmu}
	Let \(\mu\in X_*(T)^+\) with image \(\mu_I\in X_*(T)^+_I\), and consider the irreducible \(\IQ\)-representation \(V_{\mu_I}\) of \(\widehat{G}^I\) of highest weight \(\mu_I\).
	Let \(\tau\in X_*(T)\), with corresponding element \([\varpi^\tau]\in B(G)_{\vsp}\).
	Then \([\varpi^\tau]\in B(G,\mu)\) if and only if the subspace \(\bigoplus_{\lambda_I\in \tau_I+(\sigma-1)X_*(T)_I} V_{\mu_I}(\lambda_I) \subseteq V_{\mu_I}\) does not vanish.
\end{lem}
\begin{proof}
	By replacing \(\tau\) by an element in its \(W_0\)-orbit, we may assume \(\tau_{\Gamma_F}\in X_*(T)^+_{\Gamma_F}\).
	Then \[\bigoplus_{\lambda_I\in \tau_I+(\sigma-1)X_*(T)_I} V_{\mu_I}(\lambda_I)\] is exactly the \(\tau_{\Gamma_F}\)-weight space for \(V_{\mu_I\mid \widehat{G}^{\Gamma_F}}\), which is nontrivial exactly when \(\tau_{\Gamma_F}\leq \mu_{\Gamma_F}\).
	By \cite[Corollary 2.6]{Hamann:Geometric}, this is equivalent to \([\varpi^\tau]\in B(G,\mu)\).
\end{proof}

Let \(m\) be the degree of some unramified extension of \(F\) over which \(G\) becomes residually split, and let
\[\Lambda:=\left\{\lambda\in X_*(T)_I\mid \sum_{i=0}^{m-1} \sigma^i(\lambda)\in X_*(Z_G)_I\right\}.\]
Then, for a \(\widehat{G}^I\)-representation \(V\), we define \begin{equation}\label{Tatep}V^{\Tatep}:= \bigoplus_{\lambda_I\in \Lambda} V(\lambda_I).\end{equation}
Similarly, following \thref{Defi MV sets}, for \(\mu_I\in X_*(T)_I^+\) (resp.~\(\mu\in X_*(T)^+\)), we define
\begin{equation}\label{Tate MV cycles}\IM\IV_{\mu_I}^{\can,\Tatep}:= \bigsqcup_{\lambda_I\in \Lambda} \IM\IV_{\mu_I}^{\can}(\lambda_I), \quad \text{resp. } \quad \IM\IV_{\mu}^{\spl,\Tatep}:=\bigsqcup_{\lambda\in X_*(T) \colon \lambda_I\in \Lambda} \IM\IV_\mu^{\spl}(\lambda) ,\end{equation}
which yield bases of \(V_{\mu_I}^{\Tatep}\) (resp.~\(V_\mu^{\Tatep}\)).

\begin{lem}\thlabel{description of Lambda}
	Let \(\mu_I\in X_*(T)_I^+\).
	There exists \(\tau_I\in \Lambda\) with \(V_{\mu_I}^{\Tatep} = \bigoplus_{\lambda_I\in \tau_I+(\sigma-1)X_*(T)_I} V_{\mu_I}(\lambda_I)\).
	Similarly, if \(\mu\in X_*(T)^+\) and \(V_\mu\) is considered as a \(\widehat{G}^I\)-representation, there exists \(\tau_I \in \Lambda\) such that \(V_\mu^{\Tatep} = \bigoplus_{\lambda_I\in \tau_I+(\sigma-1)X_*(T)_I} V_\mu(\lambda_I)\).
\end{lem}
\begin{proof}
	We start with \(\mu_I\in X_*(T)_I^+\).
	Let \(\lambda_I,\lambda_I'\) be two weights of \(V_{\mu_I}\), and assume they lie in \(\Lambda\).
	Then \([\varpi^{\lambda_I}],[\varpi^{\lambda_I'}]\in B(G)\) are both basic, and since they lie in \(B(G,\mu)\) they must agree.
	As they are also very special, \thref{classification of spherical elements} implies that there exist \(w\in W_0\) and \(\nu_I\in X_*(T)_I\) such that \(\lambda_I=w(\lambda_I')+(\sigma-1)\nu_I\).
	Now, \[\sum_{i=0}^{m-1} \sigma^i(w(\lambda_I')-\lambda_I') = w(\sum_{i=0}^{m-1} \sigma^i(\lambda_I')) - \sum_{i=0}^{m-1} \sigma^i(\lambda_I'),\]
	which vanishes as \(\lambda_I'\in \Lambda\) and \(W_0\) acts trivially on \(X_*(Z_G)_I\).
	It follows that \(w(\lambda_I')-\lambda_I' = (\sigma-1)\nu_I'\) for some \(\nu_I'\in X_*(T)_I\), so that \(\lambda_I-\lambda_I' = (\sigma-1)(\nu_I+\nu_I')\).
	Taking \(\tau_I = \lambda'_I\) then yields the desired element.
	
	To prove the second assertion, i.e., we start with a dominant lift \(\mu\) of \(\mu_I\), the argument above shows that the proposition still holds with the same \(\tau_I\).
\end{proof}

We can now give alternative characterizations of \(B(G)_{\vsp}\), starting with the basic elements.
Recall that the \emph{defect} \(\defect_G(b)\) of \(b\) is defined as the difference between the \(F\)-ranks of \(G\) and \(J_b\).

\begin{prop}\thlabel{When is basic element unramified}
	Let \(\mu\in X_*(T)\), and \(b\in B(G,\mu)\) the basic element.
	The following statements are equivalent:
	\begin{enumerate}
		\item\label{item unramified} \(b\in B(G)_{\vsp}\);
		\item\label{item Tate} \(V_{\mu_I}^{\Tatep}\) is nontrivial;
		\item\label{item defect} \(\defect_G(b) = 0\);
		\item\label{item character} Consider \(\mu_{\adj}\) as a character of \(\widehat{T}_{\sico}\).
		Then the restriction of \(\mu_{\adj}\) to \(Z(\widehat{G}_{\sico})^{\Gamma_F}\) is trivial.
	\end{enumerate}
	Moreover, if the center \(Z_G\subseteq G\) is connected and the above conditions hold, then \(b\) can be represented as \(\varpi^\tau\) for some \(\tau\in X_*(Z_G)\).
\end{prop}
\begin{proof}
	\eqref{item unramified} \(\Leftrightarrow\) \eqref{item Tate}: This follows from Lemmas \ref{unramified in BGmu} and \ref{description of Lambda}.
	
	\eqref{item unramified} \(\Leftrightarrow\) \eqref{item defect}: If \(b\in B(G)_{\vsp}\), then \(J_b\cong G\) and hence \(\defect_G(b)=0\).
	Indeed, \(J_b\) is an inner form of \(G\) for basic \(b\) \cite[§5.2]{Kottwitz:Isocrystals1}, and a Levi subgroup for very special \(b\). 
	Conversely, recall that \(\defect_G(b)=0\) if and only if \(b\) is \(\sigma\)-conjugate to 1 in \(G_{\adj}(\breve{F})\).
	Hence, as \(G_{\adj}(\breve{F})\) is generated by \(\im(G(\breve{F}) \to G_{\adj}(\breve{F}))\) and \(T_{\adj}(\breve{F})\), we see that \(b\in B(G)_{\vsp}\).
		
	\eqref{item defect} \(\Leftrightarrow\) \eqref{item character}: We will again use the fact that \(\defect_G(b)=0\) exactly when \(b = [1]\in B(G_{\adj})\).
	Since \(1\) is always basic, it lies in \(B(G_{\adj},\mu_{\adj})\) if and only if \([\mu_{\adj}] = 0 \in \pi_1(G_{\adj})_\Gamma = (Z(\widehat{G}_{\sico})^\Gamma)^D\), and the equivalence follows.
	
	Finally, if \(Z_G\) is connected, then \(G(\breve{F})\to G_{\adj}(\breve{F})\) is surjective.
	If the above conditions additionally hold, then \(b\) is \(\sigma\)-conjugate to \(1\in G_{\adj}(\breve{F})\), and hence there exists some \(\tau\in X_*(Z_G)\) such that \(b\) is \(\sigma\)-conjugate to \(\varpi^\tau\).
\end{proof}

\begin{cor}\thlabel{unramified iff quasisplit}
	Let \(b\in B(G)\) be any element.
	Then the following conditions are equivalent:
	\begin{enumerate}
		\item\label{item vsp} \(b\) is very special,
		\item\label{item nonbasic defect} \(\defect_G(b)=0\), and
		\item\label{item quasisplit} the twisted centralizer \(J_b\) is quasi-split over \(F\).
	\end{enumerate}
\end{cor}
\begin{proof}
	To see that \eqref{item vsp} and \eqref{item nonbasic defect} are equivalent, we may pass to a Levi subgroup of \(G\) (which has the same \(F\)-rank) where \(b\) is basic.
	Then we can apply \thref{When is basic element unramified}.
	
	For the equivalence between \eqref{item nonbasic defect} and \eqref{item quasisplit}, recall that \(J_b\) is an inner form of a (quasi-split) Levi subgroup of \(G\).
	As above, this Levi has the same \(F\)-rank as \(G\), so we can conclude by the fact that a reductive \(F\)-group has the same rank as its (unique) quasi-split inner form if and only if it is already quasi-split.
\end{proof}

\subsection{Splitting versions of affine Deligne--Lusztig varieties}\label{sec:splitting adlv}

We now define various incarnations of affine Deligne--Lusztig varieties, which we will always consider over \(\Spec \overline{k}\).
We start by recalling their canonical versions, for which we let \(G/F\) be an arbitrary quasi-split reductive group with very special parahoric model \(\Gg/\Oo_F\).

\begin{dfn}
	Let \(\mu_I\in X_*(T)^+_I\) and \(b\in G(\breve{F})\).
	The \emph{(canonical) affine Deligne--Lusztig variety} associated to \(\mu_I\) and \(b\) is the subscheme of \(\Gr_{\Gg}^{\can}\) given by
	\[X_{\leq \mu_I}^{\can}(b):=\left\{g\in LG/L^+\Gg\mid g^{-1}b\sigma(g)\in L^+\Gg\backslash \Gr_{\Gg,\leq \mu_I}^{\can }\right\}.\]
	In other words, it is defined by the following cartesian diagram:
	\[\begin{tikzcd}
		X_{\leq \mu_I}^{\can}(b) \arrow[d, hook] \arrow[r] & \Gr_{\Gg}^{\can}\widetilde{\times} \Gr_{\Gg,\leq \mu_I}^{\can}\arrow[d, "\pr\times m", hook]\\
		\Gr_{\Gg}^{\can} \arrow[r, "\id \times b\sigma"] & \Gr_{\Gg}^{\can}\times \Gr_{\Gg}^{\can}.
	\end{tikzcd}\]
	Analogously, one defines \(X_{\mu_I}^{\can}(b)\), by replacing the closed Schubert variety \(\Gr_{\Gg,\leq \mu_I}^{\can}\) by the Schubert cell \(\Gr_{\Gg,\mu_I}^{\can}\); then
	\[X_{\leq \mu_I}^{\can}(b) = \bigsqcup_{\mu_I' \leq \mu_I\in X_*(T)_I^+}X_{\mu_I'}^{\can}(b).\]
	Finally, for a tuple \(\mu_{\bullet,I}\) of elements in \(X_*(T)_I^+\), the \emph{convolution affine Deligne--Lusztig variety} is defined via the cartesian diagram
	\[\begin{tikzcd}
		X_{\leq \mu_{\bullet,I}}^{\can}(b) \arrow[d] \arrow[r] & \Gr_{\Gg}^{\can}\widetilde{\times} \Gr_{\Gg,\leq \mu_{\bullet,I}}^{\can}\arrow[d, "\pr\times m"]\\
		\Gr_{\Gg}^{\can} \arrow[r, "\id \times b\sigma"] & \Gr_{\Gg}^{\can}\times \Gr_{\Gg}^{\can},
	\end{tikzcd}\]
	i.e., \(X_{\leq \mu_{\bullet,I}}^{\can}(b) := X_{\leq |\mu_{\bullet,I}|}^{\can}(b) \times_{\Gr_{\Gg,\leq |\mu_{\bullet,I}|}^{\can}} \Gr_{\Gg,\leq \mu_{\bullet,I}}^{\can}\).
\end{dfn}

It is well known that \(X_{\leq \mu_I}^{\can}(b)\) (and the other variants) depends, up to isomorphism, only on \(\mu_I\) and the \(\sigma\)-twisted conjugacy class of \(b\) in \(G(\breve{F})\); we will henceforth usually consider \(b\) as an element in \(B(G)\).
By the following lemma, affine Deligne--Lusztig varieties for essentially unramified groups can be understood by studying affine Deligne--Lusztig varieties for unramified groups.

\begin{lem}\thlabel{reductions for adlv}
	\begin{enumerate}
		\item Let \(G_1,G_2\) be reductive \(F\)-groups, and let \(b_i\in G_i(\breve{F})\) and \(\mu_{i,I}\in X_*(T_i)_I^+\) for \(i=1,2\).
		Then there is a canonical isomorphism 
		\[X_{\mu_{1,I}}^{\can}(b_1) \times X_{\mu_{2,I}}^{\can}(b_2) \cong X_{(\mu_{1,I},\mu_{2,I})}^{\can}(b_1,b_2),\] where the latter is an affine Deligne--Lusztig variety for the group \(G_1\times G_2\).
		\item Let \(F/F'\) be a totally ramified extension, and \(G':=\Res_{F/F'} G\) with maximal torus \(T':=\Res_{F/F'} T\subseteq G'\).
		Let \(\mu_{I_F},\mu_{I_{F'}}'\) correspond to each other under the isomorphism \(\in X_*(T)_{I_F}^+ \cong X_*(T')_{I_{F'}}^+\), and similar for \(b,b'\) in \(G(\breve{F}) \cong G'(\breve{F'})\).
		Then there is a natural isomorphism 
		\[X_{\mu_{I_F}}^{\can}(b) \cong X_{\mu_{I_{F'}}'}^{\can}(b').\]
		\item Let \(F'/F\) be an unramified extension of degree \(d\), and \(\sigma'=\sigma^d\) the arithmetic Frobenius for \(F'\).
		Let \(\Sigma\) be the set of embeddings \(F'\to \breve{F}\) over \(F\), fix some \(\tau_0\in \Sigma\), and let \(\tau_i = \sigma^i(\tau_0)\) for \(i=0,\ldots,d-1\).
		Assume \(G/F\) is of the form \(\Res_{F'/F} G'\) for some quasi-split \(G'/F'\), with corresponding parahoric model \(\Gg'/\Oo_{F'}\).
		Let \(b\in G(\breve{F}) = \prod_{\tau\in \Sigma} G'(\breve{F})\), and 
		\[\Nm(b):= b_{\tau_0} \sigma(b_{\tau_1}) \ldots \sigma^{d-1}(b_{\tau_{d-1}})\in G'(\breve{F}).\]
		For \(\mu_I=(\mu_{\tau,I})\in X_*(T)_I^+\cong \prod_{\tau\in \Sigma} X_*(T')_I^+\), let \[\mu_{\bullet,I} = (\mu_{\tau_0},\ldots \sigma(\mu_1),\ldots,\sigma^{d-1}(\mu_{\tau_{d-1}}))\] be a tuple of elements in \(X_*(T')_I^+\) indexed by \(\Sigma\).
		Then there is a canonical isomorphism
		\[X_{\leq \mu_I}^{\can}(b) \cong X_{\leq \mu_{\bullet,I}}^{\can}(\Nm(b)).\]
		Moreover, let \(2\rho_{G}\) and \(2\rho_{G'}\) be the sum of the positive absolute roots for \(G\) and \(G'\) respectively.
		Then
		\[\langle 2\rho_G,\mu_I\rangle = \langle 2\rho_{G'},\sum_i \sigma^i(\mu_{\tau_i,I})\rangle. \quad \text{ and } \quad \langle 2\rho_G,\nu_b\rangle = \langle 2\rho,\nu_{\Nm(b)}\rangle.\]
		Finally, let \(J_b^{G}\) be the twisted centralizer of \(b\in G(\breve{F})\), and \(J_{\Nm(b)}^{G'}\) the twisted centralizer of \(\Nm(b)\in G'(\breve{F})\).
		Then there is a canonical isomorphism
		\[J_b^G \cong \Res_{F'/F} J_{\Nm(b)}^{G'}.\]
		\item\label{it-ADLV isog} Let \(G\to G'\) be a surjective morphism with central kernel, and \(T'\) the schematic image of \(T\) in \(G'\).
		Let \(\mu_I'\in X_*(T')_I^+\) and \(b'\in G'(\breve{F})\) be the images of \(\mu_I\in X_*(T)_I^+\) and \(b\in G(\breve{F})\).
		Then the natural morphism
		\[X_{\leq \mu_I}^{\can}(b) \to X_{\leq \mu_I'}^{\can}(b')\]
		induces isomorphisms between corresponding connected components.
		More precisely, let 
		\[c_{b,\mu_I}\pi_1(G)_I^{\sigma} := \{\gamma\in \pi_1(G)_I\mid (\sigma-1)(\gamma) = [\mu_I] - w(b)\},\]
		where \(w\) denotes the natural map \(LG\to \pi_0(LG_{\overline{k}}) \cong \pi_1(G)_I\).
		Then the following diagram is cartesian:
		\[\begin{tikzcd}
			X_{\leq \mu_I}^{\can}(b) \arrow[d] \arrow[r] & X_{\leq \mu_I'}^{\can}(b')\arrow[d]\\
			c_{b,\mu_I}\pi_1(G)_I^{\sigma} \arrow[r] & c_{b',\mu_I'}\pi_1(G')_I^{\sigma}.
		\end{tikzcd}\]
	\end{enumerate}
\end{lem}
\begin{proof}
	\begin{enumerate}
		\item is immediate from the definitions.
		\item follows since affine Grassmannians (and Schubert varieties) are invariant under totally ramified restrictions of scalars \cite[Lemma 3.2]{HainesRicharz:Smoothness}, compatibly with the Frobenius morphisms.
		\item can be proven verbatim as in \cite[Lemmas 3.5 and 3.7]{Zhu:Affine}.
		\item is shown in \cite[Lemma 5.4.2]{PappasRapoport:Integral}.
	\end{enumerate}
\end{proof}

Using these reductions, we can understand irreducible components of \(X_{\leq \mu_I}^{\can}(b)\) and \(X_{\leq \mu}^{\spl}(b)\) (as defined below) in the essentially unramified case by studying affine Deligne--Lusztig varieties for unramified groups, as in \cite{XiaoZhu:Cycles,Nie:Irreducible,ZhouZhu:Twisted} (cf.~\cite[Appendix A]{ZhouZhu:Twisted} for numerical results in the case of ramified groups).

Now we come to the splitting versions of affine Deligne--Lusztig varieties.
For the rest of this section, we assume that \(G/F\) is quasi-split and essentially unramified.
By definition, there is a map \(X_{\leq \mu_I}^{\can}(b)\to L^+\Gg\backslash \Gr_{\Gg,\leq \mu_I}^{\can}\colon g\mapsto g^{-1}b\sigma(g)\).
This yields the following natural definition.

\begin{dfn}\thlabel{definition splitting adlv}
	Let \(\mu\in X_*(T)^+\) and \(b\in B(G)\).
	The \emph{splitting affine Deligne--Lusztig variety} associated to \(\mu\) and \(b\) is defined via the following cartesian diagram:
	\begin{equation}\label{diagram splitting adlv}\begin{tikzcd}[column sep=huge]
		X_{\leq \mu}^{\spl}(b) \arrow[d] \arrow[r] & L^+\Gg \backslash \Gr_{\Gg,\leq \mu}^{\spl} \arrow[d]\\
		X_{\leq \mu_I}^{\can}(b) \arrow[r, "g\mapsto g^{-1}b\sigma(g)"] & L^+\Gg \backslash \Gr_{\Gg,\leq \mu_I}^{\can}.
	\end{tikzcd}\end{equation}
\end{dfn}

Since the maps in this diagram are \(J_b(F)\)-equivariant (for the trivial actions on \(L^+\Gg\backslash \Gr_{\Gg,\leq \mu_I}^{\can}\) and \(L^+\Gg\backslash \Gr_{\Gg,\leq \mu}^{\spl}\)), we see that \(X_{\leq \mu}^{\spl}(b)\) is equipped with a natural \(J_b(F)\)-action as well.
Moreover, the map \(L^+\Gg\backslash \Gr_{\Gg,\leq \mu}^{\spl}\to L^+\Gg\backslash \Gr_{\Gg,\leq \mu_I}^{\can}\) is schematic and pfp, so that \(X_{\leq \mu}^{\spl}(b)\) is representable by a scheme, locally pfp (since this holds for \(X_{\leq \mu_I}^{\can}(b)\)).
It turns out that the geometry of \(X_{\leq \mu}^{\spl}(b)\) can be understood by studying \(X_{\leq \mu_I}^{\can}(b)\), as well as the fibers of \(\Gr_{\Gg,\leq \mu}^{\spl}\to \Gr_{\Gg,\leq \mu_I}^{\can}\).

\begin{prop}\thlabel{geometry splitting adlv}
	Let \(\mu\in X_*(T)^+\) and \(b\in B(G)\).
	Both \(X_{\leq\mu}^{\spl}(b)\) and \(X_{\leq\mu_I}^{\can}(b)\) are nonempty if and only if \(b\in B(G,\mu)\). 
	If this is the case, then they are both of dimension \(\langle \rho,\mu-\nu_b\rangle-\frac{1}{2}\defect_G(b)\).
\end{prop}
\begin{proof}
	We will appeal to results from \cite{He:Geometric}, where the group \(G\) is assumed to be tamely ramified over a local function field, whose characteristic does not divide the order of \(\pi_1(G_{\der})\).
	However, these assumptions can be eliminated for our purposes.
	First, the restriction on the characteristic ensures that loop groups and partial affine flag varieties are reduced.
	Since we will only be interested in topological properties, this restriction is not important.
	The assumption that \(G/F\) is tame only appears in the reduction to adjoint groups \cite[§1.4]{He:Geometric}; we may instead appeal to \thref{reductions for adlv} \eqref{it-ADLV isog}.
	The restriction to local functions fields arose since the Witt vector affine flag varieties from \cite{Zhu:Affine,BhattScholze:Projectivity} were not yet available, but the arguments go through verbatim, cf.~\cite[§6.3]{He:Geometric}.
	Thus, the non-emptiness statement for the canonical affine Deligne--Lusztig variety follows from \cite[Theorem 7.1]{He:Geometric}.
	The statement for the splitting model follows, since \(X_{\leq \mu}^{\spl}(b)\to X_{\leq \mu_I}^{\can}(b)\) is surjective, as the pullback of a surjective map.
	
	The dimension statement for the canonical version follows from \cite[Theorem 2.29]{He:Hecke}.
	In particular, since \(X_{\leq \mu}^{\spl}(b)\to X_{\leq \mu_I}^{\can}(b)\) is pfp and surjective, we have \(\dim(X_{\leq \mu}^{\spl}(b))\geq \dim(X_{\leq \mu}^{\can}(b)) = \langle \rho,\mu-\nu_b\rangle -\frac{1}{2} \defect_G(b)\), and we need to show this is also an upper bound.
	For this, we will use the diagram \eqref{diagram splitting adlv}: to compute the fibers of \(X_{\leq \mu}^{\spl}(b)\to X_{\leq \mu_I}^{\can}(b)\), it suffices to compute the fibers of \(\Gr_{\Gg,\leq \mu}^{\spl}\to \Gr_{\Gg,\leq \mu_I}^{\can}\).
	Note that by definition, the lower map \(g\mapsto g^{-1}b\sigma(g)\) sends \(X_{\mu_I'}^{\can}(b)\) to \(L^+\Gg\backslash \Gr_{\Gg,\mu_I'}^{\can}\) for any \(\mu_I'\leq \mu_I\in X_*(T)_I^+\).
	But since convolution morphisms are stratified semismall (\cite[Lemma 5.6]{vdH:RamifiedSatake}, based on \cite{MirkovicVilonen:Geometric}), the fiber of \(\Gr_{\Gg,\leq \mu}^{\spl}\to \Gr_{\Gg,\leq \mu_I}^{\can}\) over any point in \(\Gr_{\Gg,\mu_I'}\) has dimension at most \(\langle\rho,\mu-\mu_I'\rangle\).
	Thus, since \(\dim(X_{\mu_I'}^{\can}(b))\leq \langle\rho,\mu_I'-\nu_b\rangle -\frac{1}{2}\defect_G(b)\), the preimage of \(X_{\mu_I'}^{\can}(b)\) in \(X_{\leq \mu}^{\spl}(b)\) has dimension at most 
	\[\langle\rho,\mu_I'-\nu_b \rangle-\frac{1}{2}\defect_G(b) + \langle\rho,\mu-\mu_I'\rangle= \langle\rho,\mu-\nu_b\rangle - \frac{1}{2}\defect_G(b),\]
	as desired.
\end{proof}

Recall that \(J_b(F)\) acts on \(X_{\leq \mu_I}^{\can}(b)\) and \(X_{\leq \mu}^{\spl}(b)\), and hence induces an action on the set of top-dimensional irreducible components of both.
Recall moreover that the MV-cycles yield canonical bases of the irreducible representations of \(\widehat{G}^I\) and \(\widehat{G}\), \thref{remark MV cycles}.
Finally, for an element \(b\in B(G)\), recall its \emph{best integral approximation} \(\lambda_b\in X_*(T)_{\Gamma_F}\), as defined in \cite[Lemma A.2.1]{ZhouZhu:Twisted} (cf.~also \cite[Lemma 2.1]{HamacherViehmann:Irreducible}).

\begin{thm}\thlabel{Irreducible components of splitting adlv}
	Let \(\mu\in X_*(T)^+\) and \(b\in B(G,\mu)\).
	There exists a natural commutative diagram of sets
	\begin{equation}\label{irreducible components of adlv}\begin{tikzcd}
		J_b(F)\backslash \Irr^{\Top}X_{\leq \mu}^{\spl}(b) \arrow[d] \arrow[r, "\cong"] & \bigsqcup_{\lambda\in X^*(\widehat{T}),\lambda_{\Gamma_F}=\lambda_b}\IM\IV_\mu^{\spl}(\lambda) \arrow[d]\\
		\bigsqcup_{\mu_I'\leq \mu_I\in X_*(T)_I^+} J_b(F)\backslash \Irr^{\Top}X_{\mu_I'}^{\can}(b) \arrow[r, "\cong"] & \bigsqcup_{\mu_I'\leq \mu_I\in X_*(T)_I^+} \bigsqcup_{\lambda_I\in X^*(\widehat{T})_I,(\lambda_I)_{\Gamma_F}=\lambda_b} \IM\IV_{\mu_I'}^{\can}(\lambda_I),
	\end{tikzcd}\end{equation}
	where the horizontal maps are bijections.
	Moreover, the right horizontal arrow is given by decomposing the irreducible \(\widehat{G}\)-representation of highest weight \(\mu\) into irreducible \(\widehat{G}^I\)-representations and then comparing weight spaces under the map \(X_*(T)\to X_*(T)_I\).
\end{thm}
\begin{proof}
	We first consider the canonical case.
	By the same argument as \cite[Lemma 3.2]{Nie:Irreducible}, we can assume \(G\) is adjoint.
	Using \thref{reductions for adlv}, it is clear that the validity of the theorem is invariant under passage to products and totally ramified restrictions of scalars.
	The validity is invariant under unramified restrictions of scalars; this is more subtle, but can be handled similarly as the passage from the canonical to the splitting affine Deligne--Lusztig varieties explained below.
	Thus, since \(G\) is assumed essentially unramified, the lower bijection of \eqref{irreducible components of adlv} follows from the case of unramified groups, which was shown in \cite[Theorem 0.5]{Nie:Irreducible} or \cite[Theorem A]{ZhouZhu:Twisted}.
	It remains to construct the upper and left arrows, and show the diagram commutes.
	
	Recall from the proof of \thref{geometry splitting adlv} that \(X_{\leq \mu_I}^{\can}(b) = \bigsqcup_{\mu_I'\leq \mu_I\in X_*(T)_I^+}X_{\mu_I'}^{\can}(b)\), with \(\dim(X_{\mu_I'}^{\can}(b)) = \langle\rho,\mu_I'-\nu_b \rangle-\frac{1}{2} \defect_G(b)\), and that the fiber of any point in \(X_{\mu_I'}^{\can}(b)\) under \(X_{\leq \mu_I}^{\spl}(b)\to X_{\leq \mu_I}^{\can}(b)\) is of dimension at most \(\langle \rho,\mu-\mu_I'\rangle\).
	Moreover, by \cite[Lemma 5.1]{FoxImai:Supersingular}, the map \(X_{\leq \mu}^{\spl}(b) \to X_{\leq \mu_I}^{\can}(b)\) is a locally trivial fibration over each \(X_{\mu_I'}^{\can}(b)\).
	Thus, any top-dimensional irreducible component of \(X_{\leq \mu}^{\spl}(b)\) arises uniquely from an irreducible component in some \(X_{\mu'}^{\can}(b)\) and an \(\langle\rho,\mu-\mu_I'\rangle\)-dimensional irreducible component of the fiber of \(\varpi^{\mu_I'}\in \Gr_{\Gg,\leq \mu_I}^{\can}\) under \(\Gr_{\Gg,\leq \mu}^{\spl}\to \Gr_{\Gg,\leq \mu_I}^{\can}\).
	This gives a canonical map \(\Irr^{\Top}X_{\leq \mu}^{\spl}(b) \to \bigsqcup_{\mu_I'\leq \mu_I\in X_*(T)_I^+} \Irr^{\Top}X_{\mu_I'}^{\can}(b)\), which is \(J_b(F)\)-equivariant and hence induces a map on \(J_b(F)\)-orbits.
	Given the lower bijection in the diagram \eqref{irreducible components of adlv}, the upper bijection follows by identifying the fibers of the vertical maps: it is well-known that the irreducible components of maximal dimension in the fibers of convolution morphisms yield a natural basis for the weight spaces of the tensor product of simple representations; see e.g.~\cite[Proposition 3.1]{Haines:Structure}, or \cite[Corollary 5.1.5]{Zhu:Introduction} for a proof which also works in mixed characteristic.
	(Note that since we are working with \(\mu\)-bounded objects, the subtlety about connected components in the definition of \(\Gr_{\Gg}^{\spl}\) does not play a role here.)
	Commutativity of the diagram then follows by construction.
\end{proof}

The following special case will be the most important towards applications for Shimura varieties, and we record it since it simplifies the situation considerably.

\begin{cor}\thlabel{irreducible comp of ADLV in very special case}
	Assume \(b\in B(G)\) is very special, that \(\mu\in X_*(T)^+\) is minuscule, and that the center \(Z_G\) of \(G\) is connected.
	Fix a very special parahoric subgroup \(J_b(\Oo_F)\subseteq J_b(F)\).
	Then \(X_{\leq \mu}^{\spl}(b)\) and \(X_{\leq \mu_I}^{\can}(b)\) are equidimensional of dimension \(\langle \rho,\mu-\nu_b\rangle\), and there is a natural commutative diagram of sets
	\[\begin{tikzcd}
		\Irr X_{\leq \mu}^{\spl}(b)	\arrow[d] \arrow[r, "\cong"] & \bigsqcup_{\lambda\in X^*(\widehat{T}),\lambda_\Gamma=\lambda_b}\IM\IV_\mu^{\spl}(\lambda) \times J_b(F)/J_b(\Oo_F) \arrow[d]\\
		\bigsqcup_{\mu_I'\leq \mu_I\in X_*(T)_I^+} \Irr X_{\mu_I'}^{\can}(b) \arrow[r, "\cong"] & \bigsqcup_{\mu_I'\leq \mu_I\in X_*(T)_I^+} \bigsqcup_{\lambda_I\in X^*(\widehat{T})_I,(\lambda_I)_\Gamma=\lambda_b} \IM\IV_{\mu_I'}^{\can}(\lambda_I) \times J_b(F)/J_b(\Oo_F).
	\end{tikzcd}\]
\end{cor}
\begin{proof}
	By \thref{reductions for adlv}, the equidimensionality of \(X_{\leq \mu_I}^{\can}(b)\) for very special \(b\) follows from the same assertion for unramified groups, which is explained in \cite[§1.2]{XiaoZhu:Cycles}.
	(Again, the reduction step concerning restrictions of scalars along unramified extensions is subtle, but can be handled similarly as the passage from canonical to splitting objects below).
	The equidimensionality of \(X_{\leq \mu}^{\spl}(b)\) then follows by comparing the dimensions of smaller affine Deligne--Lusztig varieties in \(X_{\leq \mu_I}^{\can}(b)\) with the dimensions of the fibers of the map \(X_{\leq \mu}^{\spl}(b)\to X_{\leq \mu_I}^{\can}(b)\).
	Indeed, the fiber of this map over a point in \(X_{\leq \mu_I'}^{\can}(b)\) is equidimensional of dimension \(\langle \rho,\mu_I-\mu_I'\rangle\) by \cite[Theorem 3.1]{Haines:Equidimensionality} (cf.~\cite[Theorem 5.3]{FoxImai:Supersingular} for a proof which works in mixed characteristic).
	The dimension in question then follows from \thref{geometry splitting adlv}, since the defect vanishes for very special elements.
	
	The required commutative diagram will follow from \thref{Irreducible components of splitting adlv} once we show the stabilizer in \(J_b(F)\) of any element in \(\Irr X_{\mu_I'}^{\can}(b)\) or \(\Irr X_{\leq \mu}^{\spl}(b)\) is very special, as all very special parahorics are conjugate by since \(G\) is essentially unramified and \(Z_G\) is connected.
	(Note that this does not hold for e.g.~odd ramified unitary groups.)
	Since \(X_{\leq \mu}^{\spl}(b)\) is defined via the cartesian diagram \eqref{diagram splitting adlv}, where \(J_b(F)\) acts trivially on the right hand side, the stabilizer of an element in \(\Irr X_{\leq \mu}^{\spl}(b)\) agrees with the stabilizer of its image in \(\Irr X_{\mu_I'}^{\can}(b)\) for some \(\mu_I'\leq \mu_I\in X_*(T)_I^+\).
	Thus it suffices to control the stabilizers for the canonical affine Deligne--Lusztig varieties.
	But these have been shown to be very special in \cite[Theorem 4.1.2, Remark 2.2.3]{HeZhouZhu:Stabilizers} (cf.~also \cite[Theore, 4.4.14]{XiaoZhu:Cycles} in the unramified case, under the mild assumption \cite[Hypothesis 4.4.1]{XiaoZhu:Cycles}).
\end{proof}

\subsection{Affine Deligne--Lusztig varieties and semi-infinite orbits}\label{sec--adlv and semi-infinite orbits}

For applications towards the Tate conjecture later on, we need more refined information about the irreducible components appearing in \thref{irreducible comp of ADLV in very special case}.
Throughout this subsection, we fix \(\mu\in X_*(T)^+\) and \(\tau\in X_*(T)\), for which we assume the very special element \([\varpi^\tau]\) lies in \(B(G,\mu)\).
By \thref{classification of spherical elements}, we may assume \(\tau_{\Gamma_F}\in X_*(T)^+_{\Gamma_F}\) is dominant.
We also choose some \(\lambda_I\in \tau_I + (\sigma-\identity) X_*(T)_I\), a lift \(\lambda \in X_*(T)\), and some \(\bfb\in \IM\IV^{\spl}_\mu(\lambda)\).

\begin{lem}\thlabel{minimal elements in satake to mv}
	Assume \(Z_G\) is connected.
	Then there exists some \(\nu_I\in X_*(T)_I^+\) such that \(\lambda_I + \nu_I - \sigma(\nu_I)\) is dominant, and \(\bfb\) lies in the image of the map
	\[\iota_{\IM\IV}^{\spl} \colon \IS_{\nu_I,\mu\mid \lambda_I+\nu_I}^{\spl} \to \coprod_{\lambda\in X_*(T)\colon \lambda\mapsto \lambda_I\in X_*(T)_I} \IM\IV_\mu^{\spl}(\lambda)\]
	from \eqref{comparison Satake vs MV}.
	Moreover among these \(\nu_I\), there is a minimal \(\nu_I^{\bfb}\), in the sense that for any other \(\nu_I\) satisfying the above, \(\nu_I-\nu_I^{\bfb}\) is dominant.
	Such \(\nu_I^{\bfb}\) is unique up to adding elements in \(X_*(Z_G)_I\).
\end{lem}
\begin{proof}
	Let \(\mu_I'\leq \mu_I\) be the unique element in \(X_*(T)_I^+\) such that \(\beta^{\IM\IV}(\bfb)\) lies in \(\IM\IV_{\mu_I'}^{\can}(\lambda_I)\).
	It suffices to find some minimal \(\nu_I^{\beta^{\IM\IV}(\bfb)}\in X_*(T)_I^+\) such that \(\lambda_I + \nu_I^{\beta^{\IM\IV}(\bfb)} - \sigma(\nu_I^{\beta^{\IM\IV}(\bfb)})\) is dominant and \(\beta^{\IM\IV}(\bfb)\) lies in the image of 
	\[\iota_{\IM\IV}^{\can}\colon \IS_{\nu_I,\mu_I'\mid \lambda_I+\nu_I}^{\can} \to \IM\IV_{\mu_I'}^{\can}(\lambda_I).\]
	Indeed, by the cartesian diagram \eqref{comparison Satake vs MV}, the same element \(\nu_I^{\bfb} = \nu_I^{\beta^{\IM\IV}(\bfb)}\) will then also satisfies the desired properties from the lemma.
	
	Assume first that \(G\) is adjoint, so that it splits up as a product of restrictions of scalars of unramified groups.
	It is clear that the sets of Satake and Mirkovic--Vilonen cycles are compatible with taking products, and so is the map \(\iota_{\IM\IV}^{\can}\).
	Recall that (canonical) affine Grassmannians are invariant under totally ramified restrictions of scalars \cite[Lemma 3.2]{HainesRicharz:Smoothness}, and they are transformed into products of themselves under unramified restrictions of scalars (after base change to \(\Spec \overline{k}\)), compatibly with the map \(\iota_{\IM\IV}^{\can}\).
	Hence, we are reduced to the unramified case, so that \cite[Lemma 4.4.3]{XiaoZhu:Cycles} applies.
	This gives us a unique minimal \(\nu_I\) when \(G\) is adjoint.
	
	We now omit the assumption that \(G\) is adjoint. 
	Consider the quotient \(G\to G_{\adj}\), and let \(T_{\adj}\subseteq G_{\adj}\) be the image of \(T\subseteq G\).
	Then it is easy to see (using e.g.~\cite[Lemma 4.47]{vdH:RamifiedSatake}) that \(\nu_I\in X_*(T)_I^+\) satisfies the conditions of the lemma (for the group \(G\)), if and only if the image of \(\nu_I\) in \(X_*(T_{\adj})_I^+\) satisfies the conditions of the lemma (for the group \(G_{\adj}\)).
	Under the assumption that \(Z_G\) is connected, the map \(X_*(T)_I^+ \to X_*(T_{\adj})_I^+\) is surjective, so that we may lift the element obtained in the adjoint case.
	It is then clearly minimal, and unique up to an element in \(X_*(Z_G)_I\).
\end{proof}

\begin{rmk}\thlabel{centrality of tau}
	Let \(\bfb\) and \(\nu_{I}^{\bfb}\) be as above, and \(\tau_I^{\bfb}=\lambda_I + \nu_I^{\bfb}+\sigma(\nu_I^{\bfb})\).
	Then \(\tau_{\Gamma_F}=\lambda_{\Gamma_F}=\tau_{\Gamma_F}^{\bfb}\), so that \(\varpi^{\tau_I^{\bfb}}\) and \(\varpi^{\tau_I}\) determine the same element in \(B(G)\).
	In particular, if \([\varpi^{\tau_I}]\) is basic, then \(\tau_I^{\bfb}\) is central.
\end{rmk}

In order to use the results of \cite[§4]{XiaoZhu:Cycles}, we will make the following assumption, similar to \cite[Hypothesis 4.4.1]{XiaoZhu:Cycles}.

\begin{ass}\thlabel{assumption on centralizer}
	For the rest of this subsection, we assume \(k\neq \IF_2\) if \(J_{\varpi^{\tau}}\) contains a simple factor of type \(B_n,C_n,F_4\), and \(k\neq \IF_2,\IF_3\) if it contains a factor of type \(G_2\).
\end{ass}

Note that, while the canonical affine Deligne--Lusztig varieties \(X_{\leq \mu}^{\can}(b)\) are naturally subschemes of \(\Gr_{\Gg}^{\can}\), the splitting versions \(X_{\leq \mu}^{\spl}(b)\) are not subschemes of \(\Gr_{\Gg}^{\spl}\).
However, there are natural maps
\[X_{\leq \mu}^{\spl}(b)\to X_{\leq \mu_I}^{\can}(b) \subseteq \Gr_{\Gg}^{\can}.\]
For \(\nu_I\in X_*(T)_I^+\), we define
\[Y_{\nu_I}^{\can}(\tau) := X_{\leq \mu_I}^{\can}(\varpi^\tau)\cap \Ss_{\nu_I}^{\can}\subseteq \Gr_{\Gg}^{\can}\]
and 
\[Y_{\nu_I}^{\spl}(\tau) := X_{\leq \mu}^{\spl}(\varpi^\tau) \times_{\Gr_{\Gg}^{\can}} \Ss_{\nu_I}^{\can}.\]
Alternatively, they are the schemes making the squares in the following diagram cartesian:
\begin{equation}\label{diagram for subadlv}\begin{tikzcd}
	Y_{\nu_I}^{\spl}(\tau) \arrow[r] \arrow[d] &\left(\Ss_{\nu_I}^{\can} \widetilde{\times} (\bigsqcup_{\lambda\in X_*(T)\colon \lambda\mapsto \lambda_I\in X_*(T)_I} \Ss_\lambda^{\spl})\right) \cap (\Gr_{\Gg}^{\can}\widetilde{\times} \Gr_{\Gg,\leq \mu}^{\spl}) \arrow[d]\\
	Y_{\nu_I}^{\can}(\tau) \arrow[d] \arrow[r] & (\Ss_{\nu_I}^{\can} \widetilde{\times} \Ss_{\lambda_I}^{\can}) \cap (\Gr_{\Gg}^{\can} \widetilde{\times} \Gr_{\Gg,\leq \mu_I}^{\can}) \arrow[d, "\pr \times m"] \\
	 \Ss_{\nu_I}^{\can}\arrow[r, "\identity \times \varpi^\tau\sigma"'] & \Ss_{\nu_I}^{\can} \times \Ss_{\tau_I+\sigma(\nu_I)}^{\can},
\end{tikzcd}\end{equation}
where \(\lambda_I = \tau_I+\sigma(\nu_I)-\nu_I\in X_*(T)_I\).
Thus, for any \(\lambda\in X_*(T)\) mapping to \(\lambda_I\) and \(\bfb\in \IM\IV_\mu^{\spl}(\lambda)\), it makes sense to define \(Y_{\nu_I}^{\spl,\bfb}(\tau)\) via the cartesian diagram
\[\begin{tikzcd}
	Y_{\nu_I}^{\spl,\bfb}(\tau) \arrow[r] \arrow[d] &\Ss_{\nu_I}^{\can} \widetilde{\times} (\Ss_\lambda^{\spl} \cap \Gr_{\Gg,\leq \mu}^{\spl})^{\bfb}\arrow[d]\\
	\Ss_{\nu_I}^{\can} \arrow[r] & \Ss_{\nu_I}^{\can} \times \Ss_{\tau_I+\sigma(\nu_I)}^{\can}.
\end{tikzcd}\]
For any \(\bfb_I\in \IM\IV_{\mu_I}^{\can}(\lambda_I)\), we define \(Y_{\nu_I}^{\can,\bfb_I}(\tau)\) similarly.

Now, for \(\bfa'\in \IS_{\nu_I,\mu_I\mid \tau_I+\sigma(\nu_I)}^{\can}\), define the perfect scheme \(X_{\mu_I,\nu_I}^{\can,\bfa'}(\tau)\) via the Cartesian diagram
\[\begin{tikzcd}
	X_{\mu_I,\nu_I}^{\can,\bfa'}(\tau) \arrow[r] \arrow[d] & \Gr_{\nu_I,\mu_I\mid \tau_I+\sigma(\nu_I)}^{\can,\bfa'} \arrow[d]\\
	\Gr_{\Gg,\leq \nu_I}^{\can} \arrow[r, "\identity \times \varpi^\tau \sigma"'] & \Gr_{\Gg,\leq \nu_I}^{\can} \times \Gr_{\Gg,\leq \tau_I+\sigma(\nu_I)}^{\can}.
\end{tikzcd}\]
Similarly, for \(\bfa\in \IS_{\nu_I,\mu\mid \tau+\sigma(\nu_I)}^{\spl}\), we can define \(X_{\mu,\nu_I}^{\spl,\bfa}(\tau)\) via
\[\begin{tikzcd}
	X_{\mu,\nu_I}^{\spl,\bfa}(\tau) \arrow[r] \arrow[d] & \Gr_{\nu_I,\mu\mid \tau+\sigma(\nu_I)}^{\spl,\bfa} \arrow[d]\\
	\Gr_{\Gg,\leq \nu_I}^{\can} \arrow[r, "\identity \times \varpi^\tau \sigma"'] & \Gr_{\Gg,\leq \nu_I}^{\can} \times \Gr_{\Gg,\leq \tau_I+\sigma(\nu_I)}^{\can}.
\end{tikzcd}\]

\begin{prop}\thlabel{Cycle in ADLV}
	Assume that \(Z_G\) is connected, and that \thref{assumption on centralizer} holds.
	Let \(\bfa \in \IS_{\nu_I,\mu\mid \tau_I+\sigma(\nu_I)}^{\spl}\) and \(\bfb = \iota_{\IM\IV}^{\spl}(\bfa)\), and let \(\nu_{\bfb,I}\) be a minimal element obtained from \thref{minimal elements in satake to mv}.
	Then \(X_{\mu,\nu_I}^{\spl,\bfa}(\tau)\) has a unique irreducible component \(X_\mu^{\spl,\bfb}(\tau)\) of dimension \(\langle \rho,\mu-\tau\rangle\), and all other irreducible components have strictly smaller dimension.
	Moreover, the scheme \(X_\mu^{\spl,\bfb}(\tau)\times_{\Gr_{\Gg,\leq \nu_I}^{\can}} \Gr_{\Gg,\nu_I}^{\can}\) is nonempty and irreducible (necessarily of dimension \(\langle \rho,\mu-\tau\rangle\)).
	
	Similar statements hold for the canonical versions.
\end{prop}
\begin{proof}
	For the canonical models, the proposition reduces to the case of unramified groups by \thref{reductions for adlv}, in which case the result was proven in \cite[Theorem 4.4.5]{XiaoZhu:Cycles}.
	(As usual, the reduction step involving unramified restrictions of scalars is more subtle, and can be handled similarly as the reduction from splitting to canonical models below.)
	To deduce the splitting case, we will need the following lemma.
	
	\begin{lem}\thlabel{lemma irred open}
		For any \(\mu_I'\in X_*(T)_I^+\) and \(\bfa'\in \IS_{\nu_I,\mu_I'\mid \tau_I+\sigma(\nu_I)}^{\can}\), the preimage of \((\Gr_{\Gg,\leq \nu_I}^\can \widetilde{\times} \Gr_{\Gg,\mu_I'}^\can) \times_{\Gr_{\Gg}^{\can}} \Gr_{\Gg,\leq \tau_I+\sigma(\nu_I)}^\can\) under the natural map \(X_{\mu_I',\nu_I}^{\can,\bfa'} \to \Gr_{\nu_I,\mu_I'\mid \tau_I+\sigma(\nu_I)}^{\can}\) is an open subscheme in \(X_{\mu_I',\nu_I}^{\can,\bfa'}\) of dimension \(\langle \rho,\mu_I'-\tau_I\rangle\).
	\end{lem}
	\begin{proof}
		The preimage in question is clearly an open subscheme. 
		To show the dimension statement, it suffices to show the generic point of the unique \(\langle \rho,\mu_I'-\tau_I\rangle\)-dimensional component of \(X_{\mu_I',\nu_I}^{\can,\bfa'}\) maps to \((\Gr_{\Gg,\leq \nu_I}^\can \widetilde{\times} \Gr_{\Gg,\mu_I'}^\can) \times_{\Gr_{\Gg}^{\can}} \Gr_{\Gg,\leq \tau_I+\sigma(\nu_I)}^\can \subseteq \Gr_{\nu_I,\mu_I'\mid \tau_I+\sigma(\nu_I)}^{\can}\).
		But this follows from the proof of \cite[Theorem 4.4.5]{XiaoZhu:Cycles}.
		Indeed, the unique irreducible component is obtained as the closure of \(Y_{\nu_I}^{\can,\bfb'} \cap \Gr_{\Gg,\nu_I}^{\can}\), where \(\bfb'=\iota_{\IM\IV}^{\can}(\bfa')\), which is irreducible by \cite[Proposition 4.4.7]{XiaoZhu:Cycles}.
		Moreover, the unique generic point of \(Y_{\nu_I}^{\can,\bfb'} \cap \Gr_{\Gg,\nu_I}^{\can}\) gets sent to \((\Gr_{\Gg,\leq \nu_I}^\can \widetilde{\times} \Gr_{\Gg,\mu_I'}^\can) \times_{\Gr_{\Gg}^{\can}} \Gr_{\Gg,\leq \tau_I+\sigma(\nu_I)}^\can \subseteq \Gr_{\nu_I,\mu_I'\mid \tau_I+\sigma(\nu_I)}^{\can}\) by \cite[(4.4.6)]{XiaoZhu:Cycles}, since the horizontal arrows in that diagram are perfectly smooth and surjective.
	\end{proof}
	
	For the splitting case, consider the cartesian diagram
	\[\begin{tikzcd}
		\coprod_{\mu_I'\leq \mu_I}Z_{\mu_I'} \arrow[r] \arrow[d] & \coprod_{\mu_I'\leq \mu_I} (\Gr_{\Gg,\leq \nu_I}^{\can} \widetilde{\times} \Gr_{\Gg,\mu_I'}^{\can}) \times_{\Gr_{\Gg}^{\can}} \Gr_{\Gg,\leq \tau_I+\sigma(\nu_I)}^{\can}\arrow[d]\\
		\Gr_{\Gg,\leq \nu_I}^{\can} \arrow[r] & \Gr_{\Gg,\leq \nu_I}^\can \times \Gr_{\Gg,\leq \tau_I+\sigma(\nu_I)}^{\can}.
	\end{tikzcd}\]
	Then by \thref{lemma irred open} and the canonical case of the proposition, the irreducible components of each \(Z_{\mu_I'}\) are canonically in bijection with \(\IS_{\nu_I,\mu_I'\mid \tau_I+\sigma(\nu_I)}^{\can}\).
	
	On the other hand, by \cite[Lemma 5.1]{FoxImai:Supersingular}, the restriction of the convolution map \(\Gr_{\Gg,\leq \mu}^{\spl}\to \Gr_{\Gg,\leq \mu_I}^{\can}\) to \(\Gr_{\Gg,\mu_I'}^{\can}\) is a locally trivial fibration, which has dimension \(\leq \langle \rho,\mu-\mu_I'\rangle\).
	Moreover, the \(\langle \rho,\mu-\mu_I'\rangle\)-dimensional irreducible components of its fibers are canonically in bijection with the fibers of the map \(\beta^{\IS}\) from \thref{comparing Satake cycles}.
	This shows that the \(\langle \rho,\mu-\tau\rangle\)-dimensional irreducible components of
	\[\bigsqcup_{\bfa\in \IS_{\nu_I,\mu_I\mid \tau_I+\sigma(\nu_I)}} X_{\mu,\nu_I}^{\spl,\bfa} = \Gr_{\Gg,\leq \nu_I}^{\can} \times_{\Gr_{\Gg,\leq \nu_I}^{\can} \times \Gr_{\Gg,\leq \tau_I+\sigma(\nu_I)}^{\can}} \Gr_{\nu_I,\mu\mid \tau+\sigma(\nu_I)}^{\spl}\] are naturally in bijection with \(\IS_{\nu_I,\mu_I\mid \tau_I+\sigma(\nu_I)}\), and it is clear that each \(X_{\mu,\nu_I}^{\spl,\bfa}\) contains a unique such component.
	Repeating the above argument, but replacing \(\Gr_{\Gg,\leq \nu_I}^{\can}\) by \(\Gr_{\Gg,\nu_I}^{\can}\) everywhere, the nonemptyness and irreducibility of \(X_\mu^{\spl,\bfb}(\tau)\times_{\Gr_{\Gg,\leq \nu_I}^{\can}} \Gr_{\Gg,\nu_I}^{\can}\) also follows from the canonical case.
\end{proof}

\section{Motivic correspondences on moduli of local shtukas}\label{Sec:Motivic corr}

After defining the splitting affine Grassmannians and splitting affine Deligne--Lusztig varieties in the previous sections, we move towards moduli of local shtukas, and define splitting versions of these as well.
We then formulate and prove the main local result of this paper, \thref{Main local theorem}. 
Throughout this section, we will assume \(G/F\) is essentially unramified and use \thref{notation essentially unramified}, but we note that the results involving only canonical models also hold for general quasi-split \(G\) (with very special parahoric model \(\Gg/\Oo_F\)).
Throughout this section, we will consider all our geometric objects over \(\Spec \overline{k}\).

\subsection{Moduli of unbounded splitting local shtukas}

Recall that we defined a notion of (\(\mu\)-bounded) splitting local shtukas in Section \ref{Sec: bounded local shtukas} via a moduli interpretation.
Similar to \cite[§5]{XiaoZhu:Cycles}, this bounded version can actually be extended to an unbounded version, using \thref{definition splitting grassmannians} (although any direct connection between Shimura varieties and moduli of local shtukas, including the splitting versions, will only involve the bounded part).
For this, the group-theoretic methods from Section \ref{Sec:Splitting Gr} will be crucial.

Recall \cite[§5.2]{XiaoZhu:Cycles} that the moduli of local shtukas for \(\Gg\) is defined as the étale quotient \((LG/\Ad_{\sigma^{-1}} L^+\Gg)^{\et}\) for the Frobenius-twisted conjugation action; we will call it the moduli of canonical local shtukas, and denote it \(\Sht_{\Gg}^{\can}\).
It admits a natural map to the Hecke stack \(\Hck_\Gg^{\can}\).

\begin{dfn}
	The \emph{moduli of splitting local shtukas} is the fiber product
	\[\Sht_{\Gg}^{\spl}:=\Sht_{\Gg}^{\can} \times_{\Hck_{\Gg}^{\can}} \Hck_{\Gg}^{\spl},\]
	where \(\Hck_{\Gg}^{\spl}\) is defined in \thref{definition splitting grassmannians}.
	In particular, there is a natural map \(\Sht_{\Gg}^{\spl}\to \Hck_{\Gg}^{\spl}\) (pulled back from \(\Sht_{\Gg}^{\can}\to \Hck_{\Gg}^{\can}\)).
	For any \(\mu\in X_*(T)^+\), there are natural subschemes
	\[\Sht_{\Gg,\leq \mu}^{\spl} := \Sht_{\Gg}^{\can} \times_{\Hck_{\Gg}^{\can}} \Hck_{\Gg,\leq \mu}^{\spl}.\]
\end{dfn}

More generally, we also need iterated versions of the Hecke stacks and moduli of local shtukas, defined via twisted products.
Here, it will be important to distinguish them from the twisted products arising from the splitting models.

\begin{dfn}\thlabel{defi iterated shtuka}
	Let \(\mu_{\bullet,I} = (\mu_{1,I},\ldots,\mu_{t,I})\) be a tuple of elements in \(X_*(T)_I^+\).
	Then the \emph{stack of iterated (canonical) local shtukas} is
	\[\Sht_{\Gg,\leq \mu_{\bullet,I}}^{\can} := \Sht_{\Gg}^{\can} \times_{\Hck_{\Gg}^{\can}} \Hck_{\Gg,\leq \mu_{\bullet,I}}^{\can}.\]
	Similarly, for any convolution Hecke stack as in \thref{defi splitting Schubert}, there is a natural stack of local shtukas, given by taking the fiber product along \(\Sht_{\Gg}^{\can} \to \Hck_{\Gg}^{\can}\), and will be denoted using similar notation as above.
\end{dfn}

It is also possible to define unbounded versions of these objects, but keeping track of the bounds will make the notation clearer later on.
However, it is important that we allow the bounds to be arbitrary, in contrast to Section \ref{Sec: bounded local shtukas}.

Recall also that these moduli of iterated canonical local shtukas admit a partial Frobenius morphism
\begin{equation} F_{\mu_{\bullet,I}}^{\can}\colon \Sht_{\Gg,\leq(\mu_{1,I},\ldots,\mu_{t,I})}^{\can}\cong \Sht_{\Gg,\leq(\sigma(\mu_{t,I}),\mu_{1,I},\ldots,\mu_{t-1,I})}^{\can},\end{equation}
induced by taking the quotient by the Frobenius-twisted conjugation action of \(L^+\Gg\) on
\[LG \widetilde{\times} \ldots \widetilde{\times} LG \cong LG \widetilde{\times} \ldots \widetilde{\times} LG  \colon (g_1,\ldots,g_{t-1},g_t) \mapsto (\sigma(g_{t}),g_1,\ldots,g_{t-1}).\]

\subsection{Correspondences between moduli of splitting local shtukas}

Next, we will introduce certain (geometric) correspondences between moduli of local shtukas.
Throughout, \(\lambda_\bullet=(\lambda_1,\ldots,\lambda_s)\) and \(\mu_\bullet = (\mu_1,\ldots,\mu_t)\) will denote tuples of elements in \(X_*(T)^+\), and similarly \(\lambda_{\bullet,I}\) and \(\mu_{\bullet,I}\) will denote tuples of elements in \(X_*(T)_I^+\).
The canonical versions can be defined similarly as in \cite[Definition 5.2.8]{XiaoZhu:Cycles}.
However, the splitting versions are more subtle, and cannot directly be obtained by base change from the canonical versions.
Instead, we will define them using the stack \(\BB(G)'\), defined as the étale sheafification of the quotient prestack \(LG/\Ad_{\sigma^{-1}} LG\) (cf.~\thref{remark modified Kottwitz stack}).
In particular, there are natural maps 
\[\Sht_{\Gg,\leq \mu_\bullet}^{\spl} \to \Sht_{\Gg,\leq \mu_{\bullet,I}}^{\can} \to \Sht_{\Gg}^{\can} = (LG/\Ad_{\sigma^{-1}} L^+\Gg)^{\et} \to \BB(G)'.\]

\begin{dfn}
	The \emph{Hecke correspondences} 
	\[\begin{tikzcd}
		&\Sht_{\lambda_\bullet\mid \mu_\bullet}^{\spl} \arrow[ld] \arrow[rd] &\\
		\Sht_{\Gg,\leq \lambda_\bullet}^{\spl} && \Sht_{\Gg,\leq \mu_\bullet}^{\spl}
	\end{tikzcd} \text{ and } \begin{tikzcd}
	&\Sht_{\lambda_{\bullet,I}\mid \mu_{\bullet,I}}^{\can} \arrow[ld] \arrow[rd] &\\
	\Sht_{\Gg,\leq \lambda_{\bullet,I}}^{\can} && \Sht_{\Gg,\leq \mu_{\bullet,I}}^{\can}
	\end{tikzcd} \]
	are defined as \(\Sht_{\lambda_\bullet\mid \mu_\bullet}^{\spl} := \Sht_{\Gg,\leq \lambda_\bullet}^{\spl} \times_{\BB(G)'} \Sht_{\Gg,\leq \mu_\bullet}^{\spl}\) and \(\Sht_{\lambda_{\bullet,I}\mid \mu_{\bullet,I}}^{\can} := \Sht_{\Gg,\leq \lambda_{\bullet,I}}^{\can} \times_{\BB(G)'} \Sht_{\Gg,\leq \mu_{\bullet,I}}^{\can}\).
	
	In particular there are natural maps \(\Sht_{\lambda_\bullet\mid \mu_\bullet}^{\spl} \to \Sht_{\lambda_{\bullet,I}\mid \mu_{\bullet,I}}^{\can}\), which are pfp proper and surjective (as this holds for each \(\Sht_{\Gg,\leq \mu_\bullet}^{\spl} \to \Sht_{\Gg,\leq \mu_{\bullet,I}}^{\can}\)).
	
	Finally, the composition maps
	\[\Comp\colon \Sht_{\lambda_\bullet\mid \mu_\bullet}^{\spl} \times_{\Sht_{\Gg,\leq \mu_\bullet}^{\spl}} \Sht_{\mu_\bullet\mid \kappa_\bullet}^{\spl} \to \Sht_{\lambda_\bullet\mid \kappa_\bullet}^{\spl}\]
	are defined as the composition
	\[\Sht_{\lambda_\bullet\mid \mu_\bullet}^{\spl} \times_{\Sht_{\Gg,\leq \mu_\bullet}^{\spl}} \Sht_{\mu_\bullet\mid \kappa_\bullet}^{\spl} = \left(\Sht_{\Gg,\leq \lambda_\bullet}^{\spl} \times_{\BB(G)'} \Sht_{\Gg,\leq \mu_\bullet}^{\spl}\right) \times_{\Sht_{\Gg,\leq \mu_\bullet}^{\spl}} \left(\Sht_{\Gg,\leq \mu_\bullet}^{\spl} \times_{\BB(G)'} \Sht_{\Gg,\leq \kappa_\bullet}^{\spl}\right)\] 
	\[\cong\Sht_{\Gg,\leq \lambda_\bullet}^{\spl} \times_{\BB(G)'} \Sht_{\Gg,\leq \mu_\bullet}^{\spl} \times_{\BB(G)'} \Sht_{\Gg,\leq \kappa_\bullet}^{\spl} \to \Sht_{\Gg,\leq \lambda_\bullet}^{\spl} \times_{\BB(G)'} \Sht_{\Gg,\leq \kappa_\bullet}^{\spl}\]
	These composition maps can also be defined for the canonical versions, and are compatible with the natural maps \(\Sht_{\lambda_\bullet\mid \mu_\bullet}^{\spl} \to \Sht_{\lambda_{\bullet,I}\mid \mu_{\bullet,I}}^{\can}\).
	
	This definition can naturally be generalized to define e.g.~\(\Sht_{\nu_I,\mu\mid \lambda_I}^{\spl}\) and \(\Sht_{\nu_I,\mu_I\mid \lambda_I}^{\can}\) (compare \thref{specific Satake corr}).
\end{dfn}

These Hecke correspondences also admit moduli interpretations.
Recall \cite[Definition 5.1.1]{XiaoZhu:Cycles} that the canonical iterated Hecke stack \(\Hck_{\Gg,\leq \mu_{\bullet,I}}^{\can}\) represents the functor sending a perfect \(\overline{k}\)-algebra \(R\) to the groupoid
\[\left\{\Ee_t\xdashrightarrow{\beta_t} \Ee_{t-1} \xdashrightarrow{\beta_{t-1}} \ldots \xdashrightarrow{\beta_2} \Ee_1 \xdashrightarrow{\beta_1} \Ee_0\right\},\]
where \(\Ee_i\) are \(\Gg\)-bundles over \(D_R=\Spec W_{\Oo_F}(R)\), and the \(\beta_i\) are modifications, i.e., isomorphisms over \(D_R^*=\Spec W_{\Oo_F}(R) \otimes_{\Oo_F} F\), which are moreover bounded by \(\mu_{i,I}\) in the sense of \cite[Definition 3.1.3]{XiaoZhu:Cycles}.
Moreover, for a \(\Gg\)-torsor \(\Ee\) on \(D_R\), let \({}^{\sigma} \Ee := (\sigma \otimes \id)^*\Ee\) be its pullback along the automorphism \(\sigma\otimes \id\) of \(W_{\Oo_F}(R)=W(R) \otimes_{W(k)} \Oo_F\).
Then the moduli of iterated canonical local shtukas represents the functor sending \(R\) to the similar groupoid
\[\left\{\Ee_t\xdashrightarrow{\beta_t} \Ee_{t-1} \xdashrightarrow{\beta_{t-1}} \ldots \xdashrightarrow{\beta_2} \Ee_1 \xdashrightarrow{\beta_1} {}^\sigma \Ee_t\right\}\]
with the modification satisfying the same bounds.
In other words, this classifies a point of \(\Hck_{\Gg,\leq \mu_{\bullet,I}}(R)\) together with an isomorphism \(\Ee_0\cong {}^\sigma \Ee_t\).
Recall from \cite[Lemma 3.23]{Zhu:Tame} that \(\BB(G)'\) classifies modifications pairs \((\Ee,\beta)\), where \(\Ee\) is a \(G\)-torsor on \(D_R^*\) which can be trivialized after an étale cover \(R\to R'\), and \(\beta\colon \Ee\cong {}^\sigma \Ee\) is an isomorphism.
Moreover, by the proof of \cite[Lemma 1.3]{Zhu:Affine} any \(G\)-torsor on \(D_R^*\) which is the restriction of a \(\Gg\)-torsor on \(D_R\) can be trivialized after such an étale cover of \(R\).
Consequently, the correspondence \(\Sht_{\lambda_{\bullet,I}\mid \mu_{\bullet,I}}^{\can}\) parametrizes commutative diagrams
\begin{equation}\label{diagram Hecke correspondences}\begin{tikzcd}
	\Ee_s' \arrow[d, dashed, "\beta"] \arrow[r, dashed, "\beta'_s"] & \Ee_{s-1}' \arrow[r, dashed, "\beta'_{s-1}"] & \ldots \arrow[r, dashed, "\beta'_2"] & \Ee'_1 \arrow[r, dashed, "\beta'_1"] & {}^\sigma \Ee'_s \arrow[d, dashed, "\sigma(\beta)"]\\
	\Ee_t\arrow[r, dashed, "\beta_t"] &\Ee_{t-1} \arrow[r, dashed, "\beta_{t-1}"] & \ldots \arrow[r, dashed, "\beta_2"] & \Ee_1 \arrow[r, dashed, "\beta_1"]& {}^\sigma \Ee_t,
\end{tikzcd}\end{equation}
where the upper row determines a point of \(\Sht_{\Gg,\leq \lambda_{\bullet,I}}^{\can}\), and the lower row a point of \(\Sht_{\Gg,\leq\mu_{\bullet,I}}^{\can}\).
If we moreover require the modification \(\beta\) to be bounded by \(\nu_I\in X_*(T)_I^+\), we get a closed substack \(\Sht_{\lambda_{\bullet,I}\mid \mu_{\bullet,I}}^{\nu_I,\can}\subseteq \Sht_{\lambda_{\bullet,I}\mid \mu_{\bullet,I}}^{\can}\).
Finally, we define a similar closed substack in the splitting case as the fiber product
\[\begin{tikzcd}
	\Sht_{\lambda_\bullet\mid \mu_\bullet}^{\nu_I,\spl} \arrow[r] \arrow[d] & \Sht_{\lambda_\bullet\mid \mu_\bullet}^{\spl}\arrow[d] \\
	\Sht_{\lambda_{\bullet,I}\mid \mu_{\bullet,I}}^{\nu_I,\can} \arrow[r] & \Sht_{\lambda_{\bullet,I}\mid \mu_{\bullet,I}}^{\can}.
\end{tikzcd}\]
Using the moduli interpretation \eqref{diagram Hecke correspondences}, the map \(\Comp\) is obtained by composing the vertical arrows.
Thus, \(\Comp\) restricts to morphisms
\[\Comp\colon \Sht_{\lambda_{\bullet,I}\mid \mu_{\bullet,I}}^{\nu_I,\can} \times_{\Sht_{\Gg,\leq \mu_{\bullet,I}}^{\can}} \Sht_{\mu_{\bullet,I}\mid \kappa_{\bullet,I}}^{\nu_I',\can} \to \Sht_{\lambda_{\bullet,I}\mid \kappa_{\bullet,I}}^{\nu_I+\nu_I',\can},\]
and hence similarly for the splitting versions.

\begin{rmk}\thlabel{remarks Hecke correspondences}
	We can also define Hecke correspondences between moduli of unbounded local shtukas.
	At least when \(s=t=1\), we will write them as \(\Sht_{\infty \mid \infty}^{\spl}\) and \(\Sht_{\infty \mid \infty}^{\can}\).
	They also admit closed substacks \(\Sht_{\infty \mid \infty}^{\nu_I,\spl}\) and \(\Sht_{\infty \mid \infty}^{\nu_I,\can}\) for \(\nu_I\in X_*(T)_I^+\).
\end{rmk}

\begin{ex}\thlabel{example hecke corr central}
	The moduli stack of local shtukas can be used to define affine Deligne--Lusztig varieties.
	Indeed, let \(\mu_I\in X_*(T)_I^+\) and \(b\in B(G)\).
	Recall from \cite[Proposition 3.27]{Zhu:Tame} that the underlying set of \(\Bb(G) = (LG/\Ad_\sigma LG)^{\et}\) agrees with \(B(G)\), so that \(b\) yields a point of \(\BB(G)\).
	Changing to \(\BB(G)'\) according to our conventions, we get the following cartesian diagram
	\begin{equation}\label{adlv as fibers can}\begin{tikzcd}
		X_{\leq \mu_I^*}^{\can}(b) \arrow[r] \arrow[d] & \{b^{-1}\} \arrow[d]\\
		\Sht_{\Gg,\leq \mu_I}^{\can} \arrow[r] & \BB(G)'. 
	\end{tikzcd}\end{equation}
	In particular, if \(\tau_I\in X_*(Z_G)_I\) is a central cocharacter, then Lang's theorem implies that 
	\begin{equation}\label{shtuka for central}\Sht_{\Gg,\leq \tau_I}^{\can} \cong \Gg(\Oo_F)\backslash\Spec \overline{k},\end{equation} 
	and we see that 
	\[\Sht_{\tau_I\mid \mu_I}^{\can} \cong \Gg(\Oo_F)\backslash X_{\leq \mu_I^*}^{\can}(\varpi^{\tau^*}).\]
	This moreover restricts to
	\[\Sht_{\tau_I\mid \mu_I}^{\nu_I,\can} \cong \Gg(\Oo_F)\backslash X_{\mu_I^*,\nu_I^*}^{\can}(\varpi^{\tau^*}).\]
	
	Similar statements hold for splitting models: let \(\mu\in X_*(T)^+\) and \(\tau\in X_*(Z_G)\).
	Then pulling back the above along \(\Sht_{\Gg,\leq \mu}^{\spl}\to \Sht_{\Gg,\leq \mu_I}^{\can}\), we obtain a cartesian square
	\begin{equation}\label{adlv as fibers spl}\begin{tikzcd}
		X_{\leq \mu^*}^{\spl}(b) \arrow[r] \arrow[d] & \{b^{-1}\} \arrow[d]\\
		\Sht_{\Gg,\leq \mu}^{\spl} \arrow[r] & \BB(G)',
	\end{tikzcd}\end{equation}
	and isomorphisms
	\[\Sht_{\tau\mid \mu}^{\nu_I,\spl} \cong \Gg(\Oo_F) \backslash X_{\mu^*,\nu_I^*}^{\spl}(\varpi^{\tau^*}).\]
\end{ex}

The same argument as \cite[Lemma 5.2.12]{XiaoZhu:Cycles} shows the following.
\begin{lem}\thlabel{representability of bounded correspondence spaces}
	The maps \(\Sht_{\lambda_\bullet\mid \mu_\bullet}^{\nu,\spl}\to \Sht_{\Gg,\leq\lambda_\bullet}^{\spl}\) and \(\Sht_{\infty \mid \infty}^{\nu,\spl}\to \Sht_{\Gg}^{\spl}\) are representable by pfp perfectly proper schemes.
	In particular, \(\Sht_{\lambda_\bullet\mid \mu_\bullet}^{\spl}\to \Sht_{\Gg,\leq \lambda_\bullet}^{\spl}\) and \(\Sht_{\infty \mid \infty}^{\spl}\to \Sht_{\Gg}^{\spl}\) are representable by ind-(perfectly proper) schemes.
	Similar statements hold for the canonical versions.
\end{lem}
More precisely, at least if \(\lambda_\bullet=\lambda\) and \(\mu_\bullet=\mu\) are 1-tuples, the fibers of \(\Sht_{\lambda\mid \mu}^{\spl}\to \Sht_{\Gg,\leq \lambda}^{\spl}\) can be identified with the splitting affine Deligne--Lusztig varieties from \thref{definition splitting adlv}.

Many of the properties established in \cite[§5]{XiaoZhu:Cycles} also hold for moduli of canonical local shtukas, and we will use these throughout (in fact, since we do not work with restricted local shtukas, we can avoid many of the technicalities appearing in \cite[§5.3]{XiaoZhu:Cycles}).
However, for splitting models, there are no obvious analogues of \cite[Lemmas 5.2.14--5.2.16]{XiaoZhu:Cycles}, and hence also not of \cite[Lemma 5.2.18]{XiaoZhu:Cycles}.
By \thref{spl vs can fully faithful} below, this will not yield any issues.

\begin{rmk}\thlabel{remark on restricted shtukas}
	Let \(\mu_I\in X_*(T)_I^+\), and let \(m\gge n> 0\) be two integers, such that the \(\sigma^{-1}\)-twisted conjugation action of \(L^+\Gg\) on \(\Gr_{\Gg,\leq \mu_I}^{(n),\can}\) factors through \(L^m\Gg\).
	Then, as in \cite[§5.3]{XiaoZhu:Cycles}, we can define the stack of \((m,n)\)-restricted \(\mu_I\)-bounded canonical local shtukas as
	\[\Sht_{\Gg,\leq \mu_I}^{(m,n),\can} := (\Gr_{\Gg,\leq \mu_I}^{(n),\can}/\Ad_{\sigma^{-1}}L^m\Gg)^{\et}.\]
	If \((m',n')\) and \((m,n)\) are as above, with \(m'\geq m\) and \(n'\geq n\), we get natural restriction maps
	\[\Sht_{\Gg{,}\leq \mu_I}^{(m',n'),\can} \to \Sht_{\Gg,\leq \mu_I}^{(m,n),\can}.\]
	These are the composition of affine bundles (since this holds for \(\Gr_{\Gg,\leq \mu_I}^{(n'),\can} \to \Gr_{\Gg,\leq \mu_I}^{(n),\can}\)), and base changes along \(*/L^{m'}\Gg \to */L^m\Gg\).
	Since \(L^{m'}\Gg\to L^m\Gg\) has split unipotent kernel, and we clearly have \(\Sht_{\Gg,\mu_I}^{\can} \cong \varprojlim_{m\gge n>0} \Sht_{\Gg,\leq \mu_I}^{(m,n),\can}\), we see that \(\Sht_{\Gg,\leq \mu_I}^{\can}\) satisfies \thref{assumption corr} \eqref{assumption 1}.
	
	Now, let \(\lambda_I,\nu_I\in X_*(T)_I^+\) be further dominant cocharacters.
	Then we would like the map \(\Sht_{\lambda_I\mid \mu_I}^{\nu_I,\can}\to \Sht_{\Gg,\leq \lambda_I}^{\can}\) to satisfy \thref{assumption corr} \eqref{assumption 2}.
	For \((m,n)\) with \(m\gge n\), the same (complicated) construction as in \cite[Definition 5.3.16]{XiaoZhu:Cycles} yields a pfp algebraic stack \(\Sht_{\lambda_I\mid \mu_I}^{\nu_I,(m,n),\can} \to \Sht_{\Gg,\leq \lambda_I}^{(m,n),\can}\).
	For a second pair \((m',n')\), with \(m'\gge n'\), \(m'\geq m\) and \(n'\geq n\), the analogue of \cite[Lemma 5.3.17]{XiaoZhu:Cycles} yields a cartesian diagram
	\[\begin{tikzcd}
		\Sht_{\lambda_I\mid \mu_I}^{\nu_I,(m',n'),\can} \arrow[d] \arrow[r] &  \Sht_{\Gg,\leq \lambda_I}^{(m',n'),\can} \arrow[d] \\
		\Sht_{\lambda_I\mid \mu_I}^{\nu_I,(m,n),\can} \arrow[r] & \Sht_{\Gg,\leq \lambda_I}^{(m,n),\can}.
	\end{tikzcd}\]
	This even holds for \(m'=n'=\infty\), so that \(\Sht_{\lambda_I\mid \mu_I}^{\nu_I,\can} \cong \varprojlim_{m\gge n>0} \Sht_{\lambda_I\mid ´\mu_I}^{\nu_I,(m,n),\can}\), and we see that \(\Sht_{\lambda_I\mid \mu_I}^{\nu_I,\can}\to \Sht_{\Gg,\leq \lambda_I}^{\can}\) indeed satisfies \thref{assumption corr} \eqref{assumption 2}.
	
	By symmetry, \(\Sht_{\lambda_I\mid \mu_I}^{\nu_I,\can}\to \Sht_{\Gg,\leq \mu_I}^{\can}\) also satisfies \thref{assumption corr} \eqref{assumption 2}.
	We note that the stack descending \(\Sht_{\lambda_I\mid \mu_I}^{\nu_I,\can}\to \Sht_{\Gg,\leq \mu_I}^{\can}\) to \(\Sht_{\Gg,\leq \mu_I}^{(m,n),\can}\) does not agree with the stack \(\Sht_{\lambda_I\mid \mu_I}^{\nu_I,(m,n),\can}\) from the previous paragraph.
	However, thanks the the additional level of flexibility of our sheaf theory as explained in Appendix \ref{Appendix:motives}, this is not an issue.
	In particular, this allows us to avoid many of the technical difficulties appearing in \cite[§5.3]{XiaoZhu:Cycles}.
	The discussion above also works verbatim for the moduli of iterated canonical local shtukas, as well as for the moduli of splitting local shtukas.
\end{rmk}

\subsection{Motives on moduli of local shtukas}\label{subsec:motives on shtukas}

With the geometry from the previous subsections at hand, we now explain how we can define suitable categories of motives on moduli of local shtukas.
As reviewed in the Appendix, one can define motives on arbitrary (perfect) prestacks \eqref{motives on prestacks}. 
This will simplify certain aspects of \cite{XiaoZhu:Cycles}; in particular the notion of restricted local shtukas will only be used in a lesser degree.

In Section \ref{Sec:Splitting Gr}, we explained how to define categories of stratified Tate motives on splitting affine Grassmannians and local Hecke stacks, and equip them with t-structures.
The case of moduli of local shtukas is similar, but additional subtleties arise.
Namely, fix some tuple \(\mu_\bullet\) of elements in \(X_*(T)^+\), and consider the cartesian diagram
\begin{equation}\label{diagram to define MATM of shtukas}\begin{tikzcd}
	\Gr_{\Gg,\leq \mu_\bullet}^{(\infty),\spl} \arrow[r] \arrow[d] & \Gr_{\Gg,\leq \mu_\bullet}^{\spl}\arrow[d]\\
	\Sht_{\Gg,\leq \mu_\bullet}^{\spl} \arrow[r] & \Hck_{\Gg,\leq \mu_\bullet}^{\spl},
\end{tikzcd}\end{equation}
where all arrows are \(L^+\Gg\)-fibrations.
Here, \(\Gr_{\Gg,\leq \mu_\bullet}^{(\infty),\spl}\) is the natural \(L^+\Gg\)-torsor over \(\Gr_{\Gg,\leq \mu_\bullet}^{\spl}\) defined similarly as \eqref{WT torsor over Fl}.
Moreover, the \(L^+\Gg\)-action on \(\Gr_{\Gg,\leq \mu_\bullet}^{\spl}\) making it an étale torsor over \(\Sht_{\Gg,\leq \mu_\bullet}^{\spl}\) corresponds to the \(\sigma^{-1}\)-twisted conjugation action.
In particular, with respect to this action, we have \[\DM(\Sht_{\Gg,\leq \mu_\bullet}^{\spl}) \cong \DM_{L^+\Gg}(\Gr_{\Gg,\leq \mu_\bullet}^{(\infty),\spl}).\]
Since the stratification \(\Gr_{\Gg,\leq \mu_\bullet}^{(\infty),\spl} = \bigsqcup_{\mu_\bullet'\leq \mu_\bullet \in (X_*(T)^+)^t} \Gr_{\Gg,\mu_\bullet'}^{(\infty),\spl}\) is \(L^+\Gg\)-stable, it makes sense to define
\[\DM(\Sht_{\Gg,\leq \mu_\bullet}^{\spl})\supseteq \DTM(\Sht_{\Gg,\leq \mu_\bullet}^{\spl}):=\DTM_{L^+\Gg}(\Gr_{\Gg,\leq \mu_\bullet}^{(\infty),\spl}),\]
where the motives are stratified Tate with respect to the above stratification.
We can moreover equip this category with a t-structure, normalized as in \thref{remarks about DATM on LG}, with heart \[\MTM(\Sht_{\Gg,\leq \mu_\bullet}^{\spl}) \cong \MTM_{L^+\Gg}(\Gr_{\Gg,\leq \mu_\bullet}^{(\infty),\spl}).\]
By construction, the !-pullbacks along the morphisms in the diagram \eqref{diagram to define MATM of shtukas} preserve \(\DTM\) and are t-exact.
Finally, there are obvious transition maps between the diagrams \(\eqref{diagram to define MATM of shtukas}\) for varying \(\mu_\bullet'\leq \mu_\bullet \in (X_*(T)^+)^t\), pushforward along which preserves \(\DTM\) and \(\MTM\), so that we can also define categories of (mixed) Tate motives on unbounded versions of moduli of local shtukas.
In particular, in the case \(t=1\), !-pullback gives a natural functor
\begin{equation}\label{functor from hecke to shtuka}\MTM(\Hck_{\Gg}^{\spl})\to \MTM(\Sht_{\Gg}^{\spl}).\end{equation}
Again, all of the above also works in the case of moduli of canonical local shtukas.

Finally, we want to define a category of motives on moduli of local shtukas where morphisms are given by (motivic) correspondences, generalizing \cite[§5.4]{XiaoZhu:Cycles}.
We will only do this for the unbounded, non-iterated moduli of local sthukas.
For simplicity, we will also only define this category for mixed Tate motives which are moreover constructible, rather than all of \(\DM\).
Recall that the category of constructible motives on an arbitrary perfect prestack has been defined in \eqref{constructible motives on prestacks}; we will denote the intersection \(\MTM(\Sht_{\Gg}^{\spl}) \cap \DM_{\cons}(\Sht_{\Gg}^{\spl})\subseteq \DM(\Sht_{\Gg}^{\spl})\) by \(\MTM_{\cons}(\Sht_{\Gg}^{\spl})\), and similar for the canonical versions.

By \thref{remark on restricted shtukas}, \(\Sht_{\infty\mid \infty}^{\nu_I,\spl} \to \Sht_{\Gg}^{\spl} \times \Sht_{\Gg}^{\spl}\) and \(\Sht_{\infty\mid \infty}^{\nu_I,\can} \to \Sht_{\Gg}^{\can} \times \Sht_{\Gg}^{\can}\) satisfy \thref{assumption corr}, so that the techniques from §\ref{App:Corr} apply in this setting.
(Strictly speaking, they only apply to the stacks of bounded local shtukas.
However, all the motives we will consider are constructible, so they will be supported on some bounded part, so that §\ref{App:Corr} does in fact apply.)

\begin{dfn}
	The category \(\MTM_{\cons}^{\Corr}(\Sht_{\Gg}^{\spl})\) is the (ordinary) category with the same objects as \(\MTM_{\cons}(\Sht_{\Gg}^{\spl})\).
	For \(\Ff_1,\Ff_2\in \MTM_{\cons}(\Sht_{\Gg}^{\spl})\), the set of morphisms \(\Ff_1\to \Ff_2\) is given by the filtered colimit
	\[\varinjlim_{\nu_I\in X_*(T)_I^+} \Corr_{\Sht_{\infty \mid \infty}^{\nu_I,\spl}} \left(\left(\Sht_{\Gg}^{\spl},\Ff_1\right),\left(\Sht_{\Gg}^{\spl},\Ff_2\right)\right),\]
	where the transition maps are given by the proper pushforward of correspondences along
	\[\begin{tikzcd}
		\Sht_{\Gg}^{\spl} \arrow[d, equal] & \Sht_{\infty\mid \infty}^{\nu_I,\spl} \arrow[l] \arrow[r] \arrow[d, hook] & \Sht_{\Gg}^{\spl} \arrow[d, equal]\\
		\Sht_{\Gg}^{\spl} & \Sht_{\infty\mid \infty}^{\nu_I',\spl} \arrow[l] \arrow[r] & \Sht_{\Gg}^{\spl}.
	\end{tikzcd}\]
	It remains to define the composition of correspondences, so let \(\Ff_1,\Ff_2,\Ff_3\in \MTM_{\cons}(\Sht_{\Gg}^{\spl})\), and \[f\in \Corr_{\Sht_{\infty \mid \infty}^{\nu_I,\spl}} \left(\left(\Sht_{\Gg}^{\spl},\Ff_1\right),\left(\Sht_{\Gg}^{\spl},\Ff_2\right)\right) \quad \text{ and } \quad f'\in \Corr_{\Sht_{\infty \mid \infty}^{\nu_I',\spl}} \left(\left(\Sht_{\Gg}^{\spl},\Ff_2\right),\left(\Sht_{\Gg}^{\spl},\Ff_3\right)\right).\]
	Then one can compose these motivic correspondences as in \thref{composition of correspondences} to get a correspondence supported on 
	\[\Sht_{\Gg}^{\spl} \leftarrow \Sht_{\infty\mid \infty}^{\nu_I,\spl} \times_{\Sht_{\Gg}^{\spl}} \Sht_{\infty\mid \infty}^{\nu_I',\spl} \to \Sht_{\Gg}^{\spl}.\]
	Proper pushforward of this correspondence (in the sense of \thref{construction pushforward correspondence}) along the diagram\
	\[\begin{tikzcd}
		\Sht_{\Gg}^{\spl} \arrow[d, equal] & \Sht_{\infty\mid \infty}^{\nu_I,\spl} \times_{\Sht_{\Gg}^{\spl}} \Sht_{\infty\mid \infty}^{\nu_I',\spl} \arrow[l] \arrow[r] \arrow[d, "\Comp"] & \Sht_{\Gg}^{\spl} \arrow[d, equal]\\
		\Sht_{\Gg}^{\spl} & \Sht_{\infty\mid \infty}^{\nu_I+\nu_I',\spl} \arrow[l] \arrow[r] & \Sht_{\Gg}^{\spl}
	\end{tikzcd}\]
	then yields the desired composition \(f'\circ f\in \Corr_{\Sht_{\infty \mid \infty}^{\nu_I+\nu_I'\spl}} \left(\left(\Sht_{\Gg}^{\spl},\Ff_1\right),\left(\Sht_{\Gg}^{\spl},\Ff_3\right)\right)\).
	
	To see that this is well-defined (i.e., compatible with changing \(\nu_I\in X_*(T)_I^+\)), it suffices to apply \thref{compatibility pushforward correspondence}, since the maps \(\Comp\) are compatible with changing \(\nu_I\).
	Similarly, associativity of this composition also follows from \thref{compatibility pushforward correspondence}, since the following diagram is readily seen to commute:
	\[\begin{tikzcd}
		\Sht_{\infty\mid \infty}^{\spl} \times_{\Sht_{\Gg}^{\spl}} \Sht_{\infty\mid \infty}^{\spl} \times_{\Sht_{\Gg}^{\spl}} \Sht_{\infty\mid \infty}^{\spl} \arrow[d, "\id \times \Comp"'] \arrow[r, "\Comp \times \id"] & \Sht_{\infty\mid \infty}^{\spl} \times_{\Sht_{\Gg}^{\spl}} \Sht_{\infty\mid \infty}^{\spl} \arrow[d, "\Comp"]\\
		\Sht_{\infty\mid \infty}^{\spl} \times_{\Sht_{\Gg}^{\spl}} \Sht_{\infty\mid \infty}^{\spl} \arrow[r, "\Comp"'] & \Sht_{\infty\mid \infty}^{\spl}.
	\end{tikzcd}\]
	In particular, by not using restricted local shtukas, we have avoided many technicalities used when checking the similar statements in \cite[§5.4.2]{XiaoZhu:Cycles}.
\end{dfn}

Again, a similar construction works for the case of canonical local shtukas, leading to the category \(\MTM_{\cons}^{\Corr}(\Sht_{\Gg}^{\can})\).
Moreover, proper pushforward of correspondences yields a natural functor \(\MTM_{\cons}^{\Corr}(\Sht_{\Gg}^{\spl})\to \MTM_{\cons}^{\Corr}(\Sht_{\Gg}^{\can})\).

\begin{prop}\thlabel{spl vs can fully faithful}
	The functor
	\[\MTM_{\cons}^{\Corr}(\Sht_{\Gg}^{\spl})\to \MTM_{\cons}^{\Corr}(\Sht_{\Gg}^{\can})\]
	is fully faithful.
\end{prop}
\begin{proof}
	Let \(\nu_I\in X_*(T)_I^+\), and consider the diagram
	\[\begin{tikzcd}
		\Sht_{\infty \mid \infty}^{\nu_I{,}\spl} \arrow[d, "p_1^{\spl} \times p_2^{\spl}"'] \arrow[r, "\widetilde{m}"] & \Sht_{\infty\mid \infty}^{\nu_I{,}\can} \arrow[d, "p_1^{\can} \times p_2^{\can}"]\\
		\Sht_\Gg^{\spl} \times \Sht_\Gg^{\spl} \arrow[r, "m\times m"'] \arrow[d, "q_i^{\spl}"] & \Sht_\Gg^{\can} \times \Sht_\Gg^{\can} \arrow[d, "q_i^{\can}"]\\
		\Sht_\Gg^{\spl} \arrow[r, "m"'] & \Sht_\Gg^{\can},
	\end{tikzcd}\]
	where the upper square is cartesian.
	Let \(\Ff_1,\Ff_2\in \MTM_{\cons}(\Sht_{\Gg}^{\spl})\); since these are constructible, they are supported on some \(\Sht_{\Gg{,}\leq \mu}^{\spl}\).
	Passing to restricted local shtukas, we may pretend the stacks in the above diagram are pfp algebraic stacks.
	Then, we compute
	\begin{align*}
	&(p_1^{\can}\times p_2^{\can})_*\IHom_{\Sht_{\infty\mid \infty}^{\nu_I,\can}}(p_1^{\can,*}m_*\Ff_1,p_2^{\can,!}m_*\Ff_2)\\ &\cong (p_1^{\can}\times p_2^{\can})_*(p_1^{\can}\times p_2^{\can})^! \IHom_{\Sht_\Gg^{\can} \times \Sht_\Gg^{\can}}(q_1^{\can,*}m_*\Ff_1,q_2^{\can,!}m_*\Ff_2) & \text{\eqref{exchagne3}}\\
	& \cong (p_1^{\can}\times p_2^{\can})_*(p_1^{\can}\times p_2^{\can})^!(\ID(m_*\Ff_1) \boxtimes m_*\Ff_2) & \text{\eqref{Kunneth!up} and \eqref{KunnethHom}} \\
	& \cong (p_1^{\can}\times p_2^{\can})_*(p_1^{\can}\times p_2^{\can})^! (m_*\ID(\Ff_1) \boxtimes m_*\Ff_2) & \text{\eqref{exchange2}} \\
	& \cong (p_1^{\can}\times p_2^{\can})_*(p_1^{\can}\times p_2^{\can})^! (m\times m)_* \IHom_{\Sht_\Gg^{\spl} \times \Sht_\Gg^{\spl}}(q_1^{\spl,*}\Ff_1,q_2^{\spl,*}\Ff_2)& \text{\eqref{Kunneth*down}} \\
	& \cong (p_1^{\can}\times p_2^{\can})_* \widetilde{m}_* (p_1^{\spl} \times p_1^{\spl})^! \IHom_{\Sht_\Gg^{\spl} \times \Sht_\Gg^{\spl}}(q_1^{\spl,*}\Ff_1,q_2^{\spl,*}\Ff_2)&\\
	& \cong (m\times m)_* (p_1^{\spl} \times p_2^{\spl})_*(p_1^{\spl} \times p_1^{\spl})^! \IHom_{\Sht_\Gg^{\spl} \times \Sht_\Gg^{\spl}}(q_1^{\spl,*}\Ff_1,q_2^{\spl,*}\Ff_2)&\\
	& \cong (m\times m)_*(p_1^{\spl} \times p_2^{\spl})_* \IHom_{\Sht_{\infty \mid \infty}^{\nu_I,\spl}}(p_1^{\spl,*}\Ff_!,p_2^{\spl,!}\Ff_2).
	\end{align*}
	Pushing forward the outer terms of this series of equivalences to \(\Spec \overline{k}\) and taking \(\Hom(\unit,-)\) yields the desired groups of correspondences, which implies the proposition.
\end{proof}
Since \(\Sht^{\can}_{\infty\mid \infty} = \Sht_{\Gg}^{\can} \times_{\BB(G)'} \Sht_{\Gg}^{\can}\) and similarly for \(\Sht_{\Gg}^{\spl}\), both \(\MTM_{\cons}^{\Corr}(\Sht_{\Gg}^{\can})\) and \(\MTM_{\cons}^{\Corr}(\Sht_{\Gg}^{\spl})\) can be thought of as subcategories of \(\DM(\BB(G)')\) via h-descent.
This is visibly independent of whether one uses the splitting or canonical model, and hence explains why the functor should indeed be fully faithful.
To avoid a long digression on \(\DM(\BB(G)')\), we have given a direct argument instead; we will study motivic aspects of \(\BB(G)\) in detail in future work.

We finish this section by describing morphisms between certain distinguished objects in the categories constructed above.
By \thref{spl vs can fully faithful}, it suffices to consider the canonical version.
For \(M\in \MTM_{\cons}(\Spec \overline{k})\), denote by \(\delta_{\mathbf{1}}^{\can}(M)\in \MTM_{\cons}^{\Corr}(\Sht_{\Gg}^{\can})\) the object pulled back from \(\IC_0(M)\in \MTM_{\cons}(\Hck_{\Gg}^{\can})\) (as in \thref{notation representations}) under the map \(\Sht_{\Gg}^{\can}\to \Hck_{\Gg}^{\can}\).
Note that \(\delta_{\mathbf{1}}^{\can}(M)\) is the pushforward of the similarly defined \(\delta_{\mathbf{1}}^{\spl}(M)\in \MTM_{\cons}^{\Corr}(\Sht_{\Gg}^{\spl})\).
Let
\[\Hh_{\Gg}:=C_c(\Gg(\Oo)\backslash G(F)/\Gg(\Oo),\IZ[\frac{1}{p}])\]
denote the very special Hecke algebra of \(G\) with \(\IZ[\frac{1}{p}]\)-coefficients.
Then the following proposition can be proven similarly to \cite[Proposition 5.4.4]{XiaoZhu:Cycles}

\begin{prop}
	Let \(M,N\in \MTM_{\cons}(\Spec \overline{k})\).
	Then there are natural group isomorphisms
	\[\Hom_{\MTM_{\cons}^{\Corr}(\Sht_{\Gg}^{\can})}(\delta_{\mathbf{1}}^{\can}(M),\delta_{\mathbf{1}}^{\can}(N)) \cong \Hh_{\Gg} \otimes \Hom_{\MTM_{\cons}(\Spec \overline{k})}(M,N),\]
	which are compatible with composition.
	Consequently, there is a canonical ring isomorphism
	\[\End_{\MTM_{\cons}^{\Corr}(\Sht_{\Gg}^{\can})}(\delta_{\mathbf{1}}^{\can}(\unit)) \cong \Hh_{\Gg}.\]
\end{prop}

\subsection{Motivic correspondences between moduli of local shtukas}\label{Subsec:Motivic correspondences on shtukas}

After all the preparations above, we can finally state and prove the main local theorem of this paper.
As usual, our geometric objects are base changed to \(\Spec \overline{k}\).
Moreover, from now on we will use motives with rational coefficients, and write \(\DM(-):= \DM(-,\IQ)\) and so on.

Consider the Langlands dual group \(\widehat{G}\) of \(G\), defined over \(\IQ\).
The motivic Satake equivalence from \cite{vdH:RamifiedSatake} describes mixed Tate motives on \(\Hck_{\Gg}^{\can}\) as representations of \(\widehat{G}^I\), internally in \(\MTM(\Spec \overline{k})\).
Recall that \(\MTM(\Spec \overline{k}) := \MTM(\Spec \overline{k},\IQ)\) is equivalent to \(\IZ\)-graded \(\IQ\)-vector spaces \cite[Corollary 2.8]{vdH:RamifiedSatake}.
Thus, to simplify the notation, let us write
\[\Gmot^I:=\widehat{G}^I \rtimes \IG_m,\]
where the action is as in \cite[§9]{vdH:RamifiedSatake}, so that there is a monoidal equivalence
\[\MTM(\Hck_{\Gg}^{\can}) \cong \Rep_{\Gmot^I}(\QVect) =: \Rep(\Gmot^I).\]
by \cite[Corollary 9.6]{vdH:RamifiedSatake}.
Similarly, we can write
\[\Gmot :=  \widehat{G} \rtimes \IG_m,\]
which agrees with Deligne's modification of the Langlands dual group \cite{Deligne:Letter2007,FrenkelGross:Irregular}.
Then \thref{Satake for splitting models} yields an equivalence
\[\MTM(\Hck_{\Gg}^{\spl}) \cong \Rep_{\Gmot}(\QVect)=: \Rep(\Gmot),\]
but we emphasize that the left hand side does not a priori have a monoidal structure.
Under the above equivalences, constructible objects in \(\MTM\) correspond to finite dimensional representations.

Recall that \(\widehat{G}^I\) admits a natural \(\Gamma_k\)-action, which commutes with the \(\IG_m\)-action above (cf.~\cite[§9]{vdH:RamifiedSatake}).
This induces a \(\IG_m\times \Gamma_k\)-action on \(\Gmot^I=\widehat{G}^I\rtimes \IG_m\), such that the induced action on the subgroup \(\IG_m\) is trivial.

\begin{dfn}
	Let \(x\in \IG_m\times \Gamma_k\) and \(V\in \Rep(\Gmot^I)\).
	Then the \(x\)-twisted representation \(xV\) is defined as the composition \(\Gmot^I \xrightarrow{x^{-1}} \Gmot^I \to \GL(V)\), using the above action of \(\IG_m\times \Gamma_k\) on \(\Gmot^I\).
\end{dfn}

\begin{rmk}
	\begin{enumerate}
		\item Under the equivalence \(\Rep(\Gmot^I)\cong \Rep_{\widehat{G}^I}(\MTM(\Spec \overline{k}))\), the twisting of representations does not change the underlying object in \(\MTM(\Spec \overline{k})\).
		Indeed, this follows from the fact that the \(\IG_m\times \Gamma_k\)-action restricts to the trivial action on \(\IG_m\subseteq \widehat{G}^I\).
		\item The \(\Gamma_k\)-action on \(X_*(T)_I = X^*(\widehat{T}^I)\) is given by precomposition with \(\gamma^{-1}\).
		In particular, we see that for \(\gamma\in \Gamma_k\) and \(\mu\in X_*(T)_I^+\), we have \(\gamma V_\mu^{\can}(n) = V_{\gamma(\mu)}^{\can}(n)\).
		Moreover, the pullback of \(\Sat(V_{\gamma(\mu)}^{\can}(n))\) along the automorphism \(\Gr_{\Gg}^{\can} \xrightarrow{\gamma} \Gr_{\Gg}^{\can}\) will be supported on \(\Gr_{\Gg,\leq \mu}^{\can}\subseteq \Gr_{\Gg}^{\can}\), and we see that
		\begin{equation}\thlabel{error in XZ}\gamma^* \Sat(\gamma V_\mu^{\can}(n)) = \Sat(V_\mu^{\can}(n)).\end{equation}
		This fixes the erroneous \cite[(3.4.5)]{XiaoZhu:Cycles} (but we note that in \cite[Lemma 6.1.11]{XiaoZhu:Cycles}, the correct version was used). 
		A similar discussion holds for splitting models as well.
	\end{enumerate}
\end{rmk}

Now, consider the quotient stack \(\widehat{G}^Iq^{-1}\sigma/\Gmot^I\).
If we denote by \(\Rep^{\fd}\) the full subcategory of representations on finite-dimensional vector spaces, we have functors
\[\Rep^{\fd}(\Gmot) \to \Rep^{\fd}(\Gmot^I) \to \Coh^{\Gmot^I}(\widehat{G}^Iq^{-1}\sigma).\]
We denote the essential image of the second functor by
\[\Coh_{\widehat{G}^I\text{-}\mathrm{fr}}^{\Gmot^I}(\widehat{G}^Iq^{-1}\sigma),\]
and the essential image of the composition by
\[\Coh_{\widehat{G}\text{-}\mathrm{fr}}^{\Gmot^I}(\widehat{G}^Iq^{-1}\sigma);\]
these are the \emph{free} coherent sheaves on \(\widehat{G}^Iq^{-1}\sigma/\Gmot^I\), generated by representations of \(\Gmot^I\) and \(\Gmot\) respectively.
Now the main theorem of this section is:

\begin{thm}\thlabel{Main local theorem}
	There exists a commutative diagram
	\begin{equation}\label{big local diagram}
		\begin{tikzcd}[column sep=small]
			\Coh_{\widehat{G}\text{-}\mathrm{fr}}^{\Gmot^I}(\widehat{G}^Iq^{-1}\sigma) \arrow[ddd] \arrow[rrr, "\IL^{\spl}"] &&&\MTM_{\cons}^{\Corr}(\Sht_{\Gg}^{\spl}) \arrow[ddd]\\
			& \Rep_{\widehat{G}}^{\fd}(\MTM(\Spec \overline{k})) \arrow[d] \arrow[r, "\Sat"] \arrow[ul] & \MTM_{\cons}(\Hck_{\Gg}^{\spl}) \arrow[d] \arrow[ur] & \\
			& \Rep_{\widehat{G}^I}^{\fd}(\MTM(\Spec \overline{k})) \arrow[r, "\Sat"'] \arrow[ld] & \MTM_{\cons}(\Hck_{\Gg}^{\can}) \arrow[rd] & \\
			\Coh_{\widehat{G}^I\text{-}\mathrm{fr}}^{\Gmot^I}(\widehat{G}^Iq^{-1}\sigma) \arrow[rrr, "\IL^{\can}"'] &&& \MTM_{\cons}^{\Corr}(\Sht_{\Gg}^{\can}),
		\end{tikzcd}
	\end{equation}
	where the two vertical arrows on the left are obtained by restricting representations, and the vertical arrows on the right by (proper) pushforward.
\end{thm}

\begin{rmk}\thlabel{remark main local theorem}
	The lower trapezoid of \eqref{big local diagram} exists in general (at very special level), without assuming \(G\) to be essentially unramified.
\end{rmk}

This theorem will be proven throughout this subsection, following \cite[§6]{XiaoZhu:Cycles}.
Recall that the functors \(\MTM_{\cons}(\Hck_{\Gg}^{\spl}) \to \MTM_{\cons}^{\Corr}(\Sht_{\Gg}^{\spl})\) and \(\MTM_{\cons}(\Hck_{\Gg}^{\can}) \to \MTM_{\cons}^{\Corr}(\Sht_{\Gg}^{\can})\) have been defined in \eqref{functor from hecke to shtuka}.
In particular, this determines \(\IL^{\spl}\) and \(\IL^{\can}\) on the level of objects, so it remains to describe what they do to morphisms.

Let us start by making the morphism groups in \(\Coh_{\widehat{G}^I\text{-}\mathrm{fr}}^{\Gmot^I}(\widehat{G}^Iq^{-1}\sigma)\) explicit.
For \(V\in \Rep(\widehat{G}^I)\), we denote by \(\bfJ(V):=(\Oo_{\widehat{G}^Iq^{-1}\sigma} \otimes V)^{\widehat{G}^I}\) the space of global sections of the pullback of \(V\) along \(\widehat{G}^Iq^{-1}\sigma/\widehat{G}^I \to \Spec \IQ /\widehat{G}^I\).
In particular, we can write \(\bfJ(V)\) for \(V\in \Rep^{\fd}(\Gmot ^I)\), where we implicitly consider the restriction of \(V\) along \(\widehat{G}^I \to \Gmot^I\); in that case \(\bfJ(V)\) has an additional \(\IG_m\)-action.
If \(\IQ\in \Rep^{\fd}(\Gmot^I)\) is the trivial representation, we also write \(\bfJ:=\bfJ(\IQ)\).

\begin{lem}
	The \(\IG_m\)-action on \(\bfJ\) is trivial.
\end{lem}
\begin{proof}
	Recall that \(\widehat{G}^I\rtimes \IG_m\subseteq V_{\widehat{G},\rho_{\adj}}^I\), where \(V_{\widehat{G},\rho_{\adj}}^I\) are the inertia-invariants of the Vinberg monoid, as defined in e.g.~\cite[§10.1]{vdH:RamifiedSatake}.
	Recall also the commutative monoid \(V_{\widehat{T}}^I\to (\widehat{T}_{\adj}^+)^I\) defined in \cite[§10.1]{vdH:RamifiedSatake}.
	In particular, the natural \(\IG_m\)-action (given by conjugation) is trivial on \(V_{\widehat{T}}^I\), and hence also on its global sections.
	The lemma then follows from the Chevalley restriction theorem \cite[Proposition 10.8]{vdH:RamifiedSatake}.
\end{proof}

Next, for any \(V\in \Rep^{\fd}(\Gmot^I)\), let us denote by \(\widetilde{V}\) its pullback to \(\Coh_{\widehat{G}^I\text{-}\mathrm{fr}}^{\Gmot^I}(\widehat{G}^Iq^{-1}\sigma)\).
Then \(\IL^{\can}(\widetilde{V})\) agrees with the !-pullback of \(\Sat(V)\) along \(\Sht_{\Gg}^{\can} \to \Hck_{\Gg}^{\can}\).
Moreover, for \(V_1,V_2\in \Rep^{\fd}(\Gmot^I)\) we have
\[\Hom_{\Coh^{\Gmot^I}(\widehat{G}^Iq^{-1}\sigma)}(\widetilde{V_1},\widetilde{V_2}) \cong \Hom_{\Oo_{\widehat{G}^Iq^{-1}\sigma}}(\Oo_{\widehat{G}^Iq^{-1}\sigma} \otimes V_1, \Oo_{\widehat{G}^Iq^{-1}\sigma}\otimes V_2)^{\Gmot^I}\]
\[\cong (\Oo_{\widehat{G}^Iq^{-1}\sigma}\otimes V_1^* \otimes V_2)^{\Gmot^I} \cong \bfJ(V_1^*\otimes V_2)^{\IG_m}.\]

Now, for any auxiliary \(W\in \Rep^{\fd}(\Gmot^I)\), with basis \(\{e_i\}\) and dual basis \(\{e_i^*\}\) of \(W^*\), there is a natural map
\begin{equation}\label{reps to coh}\Xi_W\colon \Hom_{\Gmot^I}(q^{-1}\sigma W\otimes V_1\otimes W^*,V_2)\to (\Oo_{\widehat{G}q^{-1}\sigma}\otimes V_1^*\otimes V_2)^{\Gmot^I}\end{equation}
\[\bfa\mapsto \left(\Xi_W(\bfa)\colon g\mapsto \sum_i \bfa(g\cdot e_i\otimes e_i^*)\in \Hom_{\IQ}(V_1,V_2)\cong V_1^*\otimes V_2\right).\]
Then \(\Xi_W(\bfa)\) is independent of the choice of basis \(\{e_i\}\), and takes values in \(\Gmot^I\)-invariant functions from \(\widehat{G}^Iq^{-1}\sigma\) to \(V_1^*\otimes V_2\), as the morphisms are assumed to be \(\Gmot^I\)-equivariant.
Moreover, by the Peter--Weyl theorem describing the global sections of \(\widehat{G}^I\cong \widehat{G}^Iq^{-1}\sigma\) (which also holds for disconnected reductive groups), the family of maps \(\{\Xi_W\mid W\in \Rep^{\fd}(\Gmot^I)\}\) is jointly surjective.
In particular, to define \(\IL^{\can}\), which it suffices to do on morphisms, we may define the compositions
\[\mathscr{C}_W^{\can} \colon \Hom_{\Gmot^I}(q^{-1}\sigma W \otimes V_1 \otimes W^*,V_2) \xrightarrow{\Xi_W} (\Oo_{\widehat{G}^Iq^{-1}\sigma} \otimes V_1^* \otimes V_2)^{\Gmot^I} \to \Hom_{\MTM_{\cons}^{\Corr}(\Sht_{\Gg}^{\can})}(\IL^{\can}(\widetilde{V_1}), \IL^{\can}(\widetilde{V_2})),\]
and show this is constant along the fibers of \(\{\Xi_W\}_W\) (as well as compatible with composition).
Similar considerations hold when defining \(\IL^{\spl}\), in which case \(V_1,V_2\in \Rep^{\fd}(\Gmot)\), and we want to define the compositions
\[\mathscr{C}_W^{\spl} \colon \Hom_{\Gmot^I}(q^{-1}\sigma W \otimes V_1 \otimes W^*,V_2) \xrightarrow{\Xi_W} (\Oo_{\widehat{G}^Iq^{-1}\sigma} \otimes V_1^* \otimes V_2)^{\Gmot^I} \to \Hom_{\MTM_{\cons}^{\Corr}(\Sht_{\Gg}^{\spl})}(\IL^{\spl}(\widetilde{V_1}), \IL^{\spl}(\widetilde{V_2})).\]
However, we emphasize that the auxiliary \(W\)'s are still representations of \(\Gmot^I\).

The additional \(q^{-1}\) appearing in contrast to \cite[§6.2]{XiaoZhu:Cycles} ensures that the target of the maps \(\Xi_W\) consist of morphisms between equivariant vector bundles on \(\widehat{G}^Iq^{-1}\sigma\) rather than \(\widehat{G}^I\sigma\).
This is important since we do not fix a square root of \(q\) in our coefficients, and do not trivialize the Tate twist.
(For the same reason, it is \(\widehat{G}^Iq^{-1}\sigma\) that shows up in the stack of spherical Langlands parameters \cite[Definition 2.1]{vdH:SphericalParameters}.)
However, to construct the functors \(\IL^{\spl}\) and \(\IL^{\can}\), it is easier to work with \(\sigma W\otimes W^*\) instead of \(q^{-1}\sigma W\otimes W^*\).
It turns out that the two representations are isomorphic, and that we can in fact construct an explicit isomorphism, independent of choices or a square root of \(q\).

\begin{lem}\thlabel{twisting has no effect}
	Let \(W\) be a representation of \(\Gmot^I\).
	Then \(W\) and \(q^{-1}W\) admit a natural equivariant isomorphism, which is moreover compatible with the tensor product of representations.
\end{lem}
\begin{proof}
	The representation \(q^{-1}W\) is given by
	\[\Gmot^I \xrightarrow{\Ad_{\rho_{\adj}(q)}} \Gmot^I \to \Aut(W).\]
	Under the equivalences \(\Rep(\Gmot^I)\cong \Rep_{\widehat{G}^I}(\MTM(\Spec \overline{k}))\) and \(\MTM(\Spec \overline{k})\cong \grQVect\), we can decompose \(W = \bigoplus_{i\in \IZ} W_i(i)\), where \(W_i\) lies in degree 0.
	We claim that the automorphism \(\phi\in \Aut(W)\) (as a vector space) which preserves this decomposition and induces the multiplication by \(q^{-i}\) on \(W_i\) satisfies the desired conditions.
	It is clear that this is compatible with the tensor product of representations.
	
	We need to check the following diagram commutes:
	\[\begin{tikzcd}
		\Gmot^I \arrow[d, "\Ad_{\rho_{\adj}(q)}"'] \arrow[r] & \Aut(W) \arrow[d, "\phi\circ - \circ \phi^{-1}"]\\
		\Gmot^I \arrow[r] & \Aut(W).
	\end{tikzcd}\]
	Recall from \cite[§9]{vdH:RamifiedSatake} that \(\widehat{G}^I\) and its \(\IG_m\)-action are determined by the maximal torus \(\widetilde{T}^I\) and its minimal Levi's, along with the \(\IG_m\)-action on those.
	Thus, it suffices to show the diagram above commutes when \(G\) is replaced by either a torus or a group of semisimple rank 1.
	
	In case of a torus, the \(\IG_m\)-action is trivial, so the twist by \(q^{-1}\) does nothing.
	And as any action of a torus would preserve the decomposition \(W=\bigoplus_{i\in \IZ} W_i(i)\), conjugation by \(\phi\) also has no effect.
	
	For groups of semisimple rank 1, we may reduce to the case where \(G\) is adjoint.
	Thus we have either \(G=\PGL_2\) and \(\widehat{G}^I=\widehat{G}=\SL_2\), or \(G=\PU_3\) defined by a ramified quadratic extension, in which case \(\widehat{G}^I=\PGL_2\subseteq \widehat{G} = \SL_3\).
	We may thus assume \(\widehat{G}^I = \SL_2\), so that \(\widehat{G}^I\rtimes \IG_m = \GL_2\).
	Since every representation of \(\GL_2\) can be obtained from the standard representation by twisting the \(\IG_m\)-action and taking direct summands, we may assume \(W\in \Rep(\GL_2)\) is the standard representation, this decomposes as \(\IQ \otimes \IQ(-1)\) as an object in \(\MTM(\Spec \overline{k})\).
	In that case we can directly check that \(\phi\) yields the desired isomorphism, using that \(\Ad_{\rho_{\adj}(q)}\) acts on \(\GL_2\) by
	\[\begin{pmatrix}
		a & b \\ c & d
	\end{pmatrix} \mapsto \begin{pmatrix}
	a & qb \\ q^{-1}c & d
	\end{pmatrix},\]
	cf.~\cite[(6.12)]{CassvdHScholbach:Geometric}.
\end{proof}

Before we define the maps \(\mathscr{C}_W^{\can}\), let us explain how to construct certain motivic correspondences supported on the moduli of local shtukas.
For a representation \(V\in \Rep^{\fd}(\Gmot^I)\) we will denote by \(\Gr_{\Gg,V}^{\can} \subseteq \Gr_{\Gg}^{\can}\) the support of \(\Sat(V)\); this is a finite closed union of Schubert cells in \(\Gr_{\Gg}^{\can}\).
More generally, if \(V_\bullet = V_1\boxtimes \ldots \boxtimes V_t\) for \(V_i\in \Rep^{\fd}(\Gmot^I)\), we denote 
\[\Gr_{\Gg,V_\bullet}^{\can} := \Gr_{\Gg,V_1}^{\can} \widetilde{\times} \ldots \widetilde{\times} \Gr_{\Gg,V_t}^{\can}\subseteq \Gr_{\Gg}^{\can,\widetilde{\times} t}.\]
We can then define \(\Hck_{\Gg,V_\bullet}^{\can}\) by taking the quotient by \(L^+\Gg\), and define \(\Sht_{\Gg,V_\bullet}^{\can}\) as in \thref{defi iterated shtuka}.
We will use similar notation for splitting models.

For \(V_1,\ldots,V_t,V_1',\ldots,V_s'\in \Rep^{\fd}(\Gmot^I)\), we define correspondences
\[\Gr_{V_\bullet\mid V'_\bullet}^{\can} := \Gr_{\Gg,V_\bullet}^{\can} \times_{\Gr_{\Gg}^{\can}} \Gr_{\Gg,V'_\bullet}^{\can},\]
\[\Hck_{V_\bullet\mid V'_\bullet}^{\can} := L^+\Gg \backslash \Gr_{V_\bullet\mid V'_\bullet}^{\can},\]
and\[\Sht_{V_\bullet\mid V'_\bullet}^{\can} \colon \Sht_{\Gg,V_\bullet}^{\can} \times_{\BB(G)'} \Sht_{\Gg,V_\bullet'}^{\can}.\]
Then \thref{basic properties of Satake correspondences} (3) yields isomorphisms
\[\mathscr{C}_{\Hck}^{\can} \colon \Hom_{\Gmot^I}(V_\bullet,V_\bullet') \cong \Corr_{\Hck_{V_\bullet\mid V'_\bullet}^{\can}}\left(\Sat(V_\bullet),\Sat(V_\bullet')\right),\]
which are compatible with composition.
This allows us to construct the so-called Satake correspondences on moduli of local shtukas.

\begin{lem}
	Let \(U_\bullet,V_\bullet,W_\bullet\) be tuples of objects in \(\Rep^{\fd}(\Gmot^I)\).
	Then pullback of correspondences induces maps
	\[\mathscr{C}_{\Sht}^{\can} \colon \Hom_{\Gmot^I}(V_\bullet,W_\bullet) \to \Corr_{\Sht_{V_\bullet\mid W_\bullet}^{0,\can}}(\IL^{\can}(\widetilde{V_\bullet}),\IL^{\can}(\widetilde{W_\bullet})).\]
	These are moreover compatible with composition, in the sense that the diagram
	\[\begin{tikzcd}
		\Hom_{\Gmot^I}(U_\bullet,V_\bullet) \otimes \Hom_{\Gmot^I}(V_\bullet,W_\bullet) \arrow[r] \arrow[d] & \Corr_{\Sht_{V_\bullet\mid W_\bullet}^{0,\can}}(\IL^{\can}(\widetilde{U_\bullet}),\IL^{\can}(\widetilde{V_\bullet})) \otimes \Corr_{\Sht_{V_\bullet\mid W_\bullet}^{0,\can}}(\IL^{\can}(\widetilde{V_\bullet}),\IL^{\can}(\widetilde{W_\bullet})) \arrow[d] \\
		\Hom_{\Gmot^I}(U_\bullet,W_\bullet) \arrow[r] & \Corr_{\Sht_{V_\bullet\mid W_\bullet}^{0,\can}}(\IL^{\can}(\widetilde{U_\bullet}),\IL^{\can}(\widetilde{W_\bullet}))
	\end{tikzcd}\]
	commutes, where the right vertical arrow is given by proper pushforward of correspondences along
	\[\Sht_{U_\bullet\mid V_\bullet}^{\can} \times_{\Sht_{\Gg,V_\bullet}^{\can}} \Sht_{V_\bullet\mid W_\bullet}^{\can} \to \Sht_{U_\bullet\mid W_\bullet}^{\can}.\]
\end{lem}
\begin{proof}
	The map \(\mathscr{C}_{\Sht}^{\can}\) is given by the pro-smooth pullback of \(\mathscr{C}_{\Hck}^{\can}\) (\thref{construction pullback correspondence}) along the diagram
	\begin{equation}\label{diagram spl}\begin{tikzcd}
			\Sht_{\Gg,V_\bullet}^{\can} \arrow[d] & \Sht_{U_\bullet\mid V_\bullet}^{0,\can} \arrow[l] \arrow[d] \arrow[r] & \Sht_{\Gg,W_\bullet}^{\can} \arrow[d]\\
			\Hck_{\Gg,V_\bullet}^{\can} & \Hck_{V_\bullet\mid W_\bullet}^{\can} \arrow[l] \arrow[r] & \Hck_{\Gg,W_\bullet}^{\can},
	\end{tikzcd}\end{equation}
	where both squares are cartesian.
	The compatibility with composition then follows from the compatibility of pro-smooth pullback of correspondences with proper pushforward of correspondences, \thref{Compatibility pullback and pushforward correspondences}.
\end{proof}

Aside from the Satake correspondences, we will also need correspondences induced by the partial Frobenius.
Namely, let \(V_\bullet\in \Rep^{\fd}((\Gmot^I)^t)\) and \(W\in \Rep^{\fd}(\Gmot^I)\).
Then there are natural morphisms
\[\Sht_{\Gg,V_\bullet \otimes W}^{\can} \to \Hck_{\Gg,V_\bullet \boxtimes W}^{\can} \to \Hck_{\Gg,V_\bullet}^{\can} \times \Hck_{\Gg,W}^{\can},\]
where both maps are \(L^+\Gg\)-fibrations, and in particular pro-(perfectly smooth). 
Consider the commutative diagram
\begin{equation}\label{diagram to define frobenius correspondence}\begin{tikzcd}
	\Sht_{\Gg,\sigma W\boxtimes V_\bullet}^{\can} \arrow[d] \arrow[r, equal] & \Sht_{\Gg,\sigma W\boxtimes V_\bullet}^{\can} \arrow[d] \arrow[r, "F_{V_\bullet\boxtimes W}^{-1}"] & \Sht_{\Gg,V_\bullet \boxtimes W}^{\can} \arrow[d]\\
	\Hck_{\Gg,\sigma W}^{\can} \times \Hck_{\Gg,V_\bullet}^{\can} & \Hck_{\Gg,W}^{\can} \times \Hck_{\Gg,V_\bullet}^{\can} \arrow[l, "\sigma \times \id"] \arrow[r, equal] & \Hck_{\Gg,V_\bullet}^{\can} \times \Hck_{\Gg,W}^{\can},
\end{tikzcd}\end{equation}
where all horizontal arrows are isomorphisms, and vertical maps are pro-(perfectly smooth) (in particular, the squares are cartesian).
Using \eqref{error in XZ}, \thref{examples of correspondences} \eqref{pullback correspondence} gives a correspondence
\[\Gamma_{\sigma \times \identity}^*\colon \left(\Hck_{\Gg,\sigma W}^{\can} \times \Hck_{\Gg,V_\bullet}^{\can},\Sat(\sigma W) \boxtimes \Sat(V_\bullet)\right) \to \left(\Hck_{\Gg,V_\bullet}^{\can} \times \Hck_{\Gg,W}^{\can}, \Sat(V_\bullet) \boxtimes \Sat(W)\right).\]
Pulling back this correspondence along \eqref{diagram to define frobenius correspondence} then yields the \emph{partial Frobenius correspondence}
\[\ID \Gamma^{\can}_{F_{V_\bullet \boxtimes W}^{-1}} \colon \left(\Sht_{\Gg,\sigma W\boxtimes V_\bullet}^{\can},\IL^{\can}(\widetilde{\sigma W} \boxtimes \widetilde{V_\bullet})\right) \to \left(\Sht_{\Gg,V_\bullet \boxtimes W}^{\can}, \IL^{\can}(\widetilde{V_\bullet} \boxtimes \widetilde{W}) \right).\]

Now, let \(V_1,V_2,W\in \Rep^{\fd}(\Gmot^I)\).
Using \thref{twisting has no effect}, consider the map
\[\mathscr{C}_W^{\can}\colon \colon\Hom_{\Gmot^I}(q^{-1}\sigma W\otimes V_1\otimes W^*,V_2) \cong \Hom_{\Gmot^I}(\sigma W\otimes V_1\otimes W^*,V_2)\to \Hom_{\MTM_{\cons}^{\Corr}(\Sht_{\Gg}^{\can})}(\IL^{\can}(\widetilde{V_1}), \IL^{\can}(\widetilde{V_2})),\]
given by sending \(\bfa \in \Hom_{\Gmot^I}(\sigma W\otimes V_1 \otimes W^*,V_2)\) to the composition
\[\IL^{\can}(\widetilde{V_1}) \xrightarrow{\mathscr{C}_{\Sht}^{\can}(\delta_{\sigma W} \otimes \identity_{V_1})} \IL^{\can}(\widetilde{\sigma W^*} \boxtimes (\widetilde{\sigma W} \otimes \widetilde{V_1})) \xrightarrow{\ID \Gamma^{\can}_{F^{-1}_{(\sigma W \otimes V_1) \boxtimes W^*}}} \IL^{\can} ((\widetilde{\sigma W} \otimes \widetilde{V_1}) \boxtimes \widetilde{W^*}) \xrightarrow{\mathscr{C}_{\Sht}^{\can}(\bfa)} \IL^{\can}(\widetilde{V_2}),\]
where \(\delta_{\sigma W} \colon \IQ \to \sigma W^* \otimes \sigma W\) is the unit map.

\begin{prop}\thlabel{functor well defined}
	Let \(V_1,V_2\in \Rep^{\fd}(\Gmot^I)\).
	The collection of maps \(\mathscr{C}_W^{\can}\), for \(W\in \Rep^{\fd}(\Gmot^I)\), factors as in the following diagram
	\[\begin{tikzcd}[column sep=tiny]
		\coprod_{W} \Hom_{\Gmot^I}(q^{-1}\sigma W \otimes V_1 \otimes W^*, V_2) \arrow[rd, "\coprod_{W} \Xi_W"', two heads] \arrow[rr, "\coprod_{W} \mathscr{C}_W^{\can}"] && \Hom_{\MTM_{\cons}^{\Corr}(\Sht_{\Gg}^{\can})}(\IL^{\can}(\widetilde{V_1}), \IL^{\can}(\widetilde{V_2})).\\
		& (\Oo_{\widehat{G}^Iq^{-1}\sigma} \otimes V_1^* \otimes V_2)^{\Gmot^I} \arrow[ur, "\IL^{\can}"']&
	\end{tikzcd}\]
\end{prop}
\begin{proof}
	Let \(\bfc\in (\Oo_{\widehat{G}^Iq^{-1}\sigma} \otimes V_1^* \otimes V_2)^{\Gmot^I}\), and fix some \(W\in \Rep^{\fd}(\Gmot^I)\) and \(\bfa\in \Hom_{\Gmot^I}(q^{-1}\sigma W \otimes V_1 \otimes W^*,V_2)\) such that \(\bfc = \Xi_W(\bfa)\).
	
	Consider \(\Oo_{\widehat{G}^I}\in \Rep_{\widehat{G}^I}(\MTM(\Spec \overline{k}))\), corresponding to right translation of \(\widehat{G}^I\).
	Under the equivalence \(\Rep_{\widehat{G}^I}(\MTM(\Spec \overline{k})) \cong \Rep(\Gmot^I)\), this yields an action of \(\Gmot^I\) on \(\Oo_{\widehat{G}^I}\).
	Now, for any \(U\in \Rep_{\widehat{G}^I}(\MTM(\Spec \overline{k}))\), we have the coaction map \(U \to \Oo_{\widehat{G}^I} \otimes U\).
	This yields a \(\Gmot^I\)-equivariant map
	\[\operatorname{act}_U\colon U \to \Oo_{\widehat{G}^I} \otimes \underline{U}.\]
	Here, \(\underline{U}\) has the trivial action under \(\widehat{G}^I \subseteq \Gmot^I\), but carries the induced \(\IG_m\)-action.
	Equivalently, it is the object in \(\Rep(\Gmot^I) \cong \Rep_{\Gmot^I}(\MTM(\Spec \overline{k}))\) with the trivial \(\widehat{G}^I\)-action.
	Similarly, we have an equivariant map 
	\[m_U\colon \underline{U^*} \otimes U \to \Oo_{\widehat{G}^I}.\]
	This yields a morphism
	\[\sigma W^* \otimes V_2 \otimes W \xrightarrow{\operatorname{act}_{\sigma W^*}} \underline{W^*} \otimes \sigma \Oo_{\widehat{G}^I} \otimes V_2 \otimes W \xrightarrow{m_W} \sigma \Oo_{\widehat{G}^I} \otimes V_2 \otimes \Oo_{\widehat{G}^I}.\]
	Note that \(\Gamma_k\) acts trivially on \(\IG_m\), so that \(\underline{\sigma W^*} = \underline{W^*}\).
	More generally, we can twist representations in \(\Rep_{\widehat{G}^I}(\MTM(\Spec \overline{k}))\) by \(\sigma\), which does not change the underlying object in \(\MTM(\Spec \overline{k})\).
	
	Now, let \(\bfa'\) correspond to \(\bfa\) under
	\[\Hom_{\Gmot^I}(q^{-1}\sigma W \otimes V_1 \otimes W^*,V_2) \cong \Hom_{\Gmot^I}(V_1,\sigma W^* \otimes V_2 \otimes W).\]
	Consider the map \(d_{q^{-1}\sigma}\colon \Oo_{\widehat{G}^Iq^{-1}\sigma} \to q^{-1}\sigma \Oo_{\widehat{G}^I} \otimes \Oo_{\widehat{G}^I}\),
	intertwining the (twisted) conjugation action of \(\Gmot^I\) on \(\Oo_{\widehat{G}^Iq^{-1}\sigma}\) with the diagonal action on \(q^{-1}\sigma \Oo_{\widehat{G}^I} \otimes \Oo_{\widehat{G}^I}\), given as the composition
	\[d_{q^{-1}\sigma}\colon \Oo_{\widehat{G}^Iq^{-1}\sigma} \xrightarrow{q^{-1}\sigma} \Oo_{\widehat{G}^Iq^{-1}\sigma} \xrightarrow{\text{comultiplication}} q^{-1}\sigma \Oo_{\widehat{G}^I} \otimes \Oo_{\widehat{G}^I} \xrightarrow{\text{antipode} \otimes \identity} q^{-1}\sigma \Oo_{\widehat{G}^I} \otimes \Oo_{\widehat{G}^I};\]
	compare \cite[Proof of Lemma 6.2.5]{XiaoZhu:Cycles}.
	This yields a map
	\[(\Oo_{\widehat{G}^Iq^{-1}\sigma} \otimes V_1^* \otimes V_2)^{\Gmot^I} \xrightarrow{d_{q^{-1}\sigma}} (q^{-1}\sigma \Oo_{\widehat{G}^I} \otimes \Oo_{\widehat{G}^I} \otimes V_1^* \otimes V_2)^{\Gmot^I} \cong \Hom_{\Gmot^I}(V_1,\sigma \Oo_{\widehat{G}^I} \otimes V_2 \otimes \Oo_{\widehat{G}^I}),\]
	under which the image of \(\bfc\) is denoted by \(d_{q^{-1}\sigma}(\bfc')\).
	
	Next, we have two commutative diagrams
	\[\begin{tikzcd}[column sep=large]
		V_1 \arrow[r, "\bfa'"] \arrow[d, equal] & \sigma W^* \otimes V_2 \otimes W \arrow[d, "m_W \circ \operatorname{act}_{\sigma W^*}"] \\
		V_1 \arrow[r, "d_{q^{-1}\sigma}(\bfc')"'] & \sigma \Oo_{\widehat{G}^I} \otimes V_2 \otimes \Oo_{\widehat{G}^I},
	\end{tikzcd} \quad \text{and} \quad
	\begin{tikzcd}[column sep=large]
		V_2 \otimes W \otimes W^* \arrow[r, "\identity_{V_2} \otimes e_W"] \arrow[d, "m_W \circ \operatorname{act}_{W^*}"'] & V_2 \arrow[d, equal]\\
		V_2 \otimes \Oo_{\widehat{G}^I} \otimes \Oo_{\widehat{G}^I} \arrow[r, "\operatorname{ev}_{(1,1)}"']	& V_2,
	\end{tikzcd}\]
	where \(e_W \colon W\otimes W^*\to \IQ\) is the evaluation map, and \(\operatorname{ev}_{(1,1)}\) is the evaluation map at \((1,1)\in \widehat{G}^I \times \widehat{G}^I\).
	Then, by the analogue of \cite[Lemma 6.2.3]{XiaoZhu:Cycles} in our setting, \(\mathscr{C}_{W}^{\can}(\bfa)\) can be computed as
	\[\mathscr{C}_{\Sht}^{\can}(\identity_{V_2} \otimes e_W) \circ \ID_{F_{(V_2\otimes W)\boxtimes W^*}^{-1}}^{\can} \circ \mathscr{C}_{\Sht}^{\can}(\bfa'),\]
	which in turn, by \cite[Remark 6.2.4]{XiaoZhu:Cycles}, agrees with
	\[\mathscr{C}_{\Sht}^{\can}(\operatorname{ev}_{(1,1)}) \circ \ID \Gamma_{F_{(V_2 \otimes \Oo_{\widehat{G}^I}) \boxtimes \Oo_{\widehat{G}^I}}^{-1}}^{\can} \circ \mathscr{C}_{\Sht}^{\can}(d_{q^{-1}\sigma}(\bfc')).\]
	Since this only depends on \(\bfc\), rather than on \(W\) and \(\bfa\), we are done.
\end{proof}

\begin{proof}[Proof of \thref{Main local theorem}]
	The diagram \eqref{big local diagram} determines \(\IL^{\can}\) and \(\IL^{\spl}\) on the level of objects: for \(V\in \Rep^{\fd}(\Gmot^I) \cong \Rep_{\widehat{G}^I}^{\fd}(\MTM(\Spec \overline{k}))\), the object \(\IL^{\can}(\widetilde{V})\) must agree with the pullback of \(\Sat(V)\) along \(\Sht_{\Gg}^{\can} \to \Hck_{\Gg}^{\can}\), and similarly for \(\IL^{\spl}\) and \(V\in \Rep^{\fd}(\Gmot) \cong \Rep_{\widehat{G}}^{\fd}(\MTM(\Spec \overline{k}))\).
	Next, we define \(\IL^{\can}\) and \(\IL^{\spl}\) on the level of morphisms, and check the compatibility with composition; this will imply that \(\IL^{\can}\) and \(\IL^{\spl}\) are moreover well-defined on the level of objects.
	Since the outer vertical arrows of \eqref{big local diagram} are fully faithful, it suffices to consider \(\IL^{\can}\).
	Then \(\IL^{\can}\) was constructed in \thref{functor well defined}.
	The fact that \(\IL^{\can}\) is actually a functor, i.e., compatible with composition, can then be checked verbatim as in \cite[Lemma 6.2.7]{XiaoZhu:Cycles}.
	(Although this is a tedious check, the situation is slightly simplified by not using moduli of restricted local shtukas.
	In particular, the references to \cite[Lemma 5.3.23]{XiaoZhu:Cycles} can be replaced by \cite[Lemma 5.2.18]{XiaoZhu:Cycles}.)
	
	It remains to show \eqref{big local diagram} commutes.
	Commutativity of the inner square follows from \thref{Satake for splitting models}, whereas for the leftmost trapezoid this follows from functoriality of pullback.
	The rightmost trapezoid commutes since the pro-smooth pullback of correspondences commutes with proper pushforward by \thref{Compatibility pullback and pushforward correspondences}.
	By construction, the lower trapezoid commutes on the level of objects.
	To show it also commutes for morphisms, consider a \(\Gmot^I\)-equivariant map \(V_1\to V_2\).
	Then the pullback of this map to \(\Coh_{\widehat{G}^I\text{-}\mathrm{fr}}^{\Gmot^I}(\widehat{G}^I q^{-1}\sigma)\) lies in the image of \(\Xi_{\IQ}\), so that \(\mathscr{C}_{\Sht}^{\can}(\delta_{\sigma \IQ} \otimes \identity_{V_1})\) and \(\ID \Gamma_{F_{(\sigma \IQ \otimes V_1) \otimes \IQ}}^{\can}\) are trivial.
	Thus, \(\IL^{\can}\) sends this morphism to a correspondence pulled back from a Satake correspondence on Hecke stacks, which by definition agrees with the functor \(\MTM_{\cons}(\Hck_{\Gg}^{\can}) \to \MTM_{\cons}^{\Corr}(\Sht_{\Gg}^{\can})\).
	Finally, the outer square commutes by construction, and the commutativity of the upper trapezoid follows from the commutativity of the lower trapezoid and the fully faithfulness of the outer vertical arrows.
\end{proof}

\subsection{Examples of the functor}\label{subsec:examples of local Langlands}

Let us make the functor \(\IL^{\can}\) (and hence also \(\IL^{\spl}\)) explicit in certain cases, following \cite[§6.3]{XiaoZhu:Cycles}.
Let \(\tau\in X_*(Z_G)\) be central, \(\mu\in X_*(T)^+\) minuscule, and \(\nu_I\in X_*(T)_I^+\) arbitrary.
In that case, we have
\[\Sht_{\tau\mid \mu}^{\nu_I,\spl} \cong \Gg(\Oo_F)\backslash X_{\mu^*,\nu_I^*}^{\spl}(\varpi^{\tau^*})\]
by \thref{example hecke corr central}.
We will use the following diagram:
\begin{equation}\label{diagram of correspondences}\begin{tiny}\begin{tikzcd}
	\Spec \overline{k} & \Gr_{\Gg,\leq \sigma(\nu_I^*)}^{\can} \arrow[l] \arrow[r, "\id \times \varpi^{\tau_I^*}"] & \Gr_{\Gg,\leq \sigma(\nu_I^*)}^{\can} \times \Gr_{\Gg,\leq \tau_I^* + \sigma(\nu_I^*)}^{\can} \arrow[r, "\sigma^{-1} \times \id"] & \Gr_{\Gg,\leq \nu_I^*}^{\can} \times \Gr_{\Gg,\leq \tau_I^* + \sigma(\nu_I^*)}^{\can} & \Gr_{\nu_I^*,\mu^*\mid \tau_I^* + \sigma(\nu_I^*)}^{\spl} \arrow[r] \arrow[l] & \Spec \overline{k}\\
	\Gr_{\Gg,\leq \tau}^{(\infty),\spl} \arrow[u] \arrow[d] & Z_{\tau_I} \arrow[u] \arrow[l] \arrow[d] \arrow[r] & \Gr_{\Gg,\leq (\sigma(\nu_I^*),\tau_I+\sigma(\nu_I))}^{(\infty),\can} \arrow[u] \arrow[d] \arrow[r, dashed] & \Gr_{\Gg,\leq (\tau_I+\sigma(\nu_I),\nu_I^*)}^{(\infty),\can} \arrow[u, dashed] \arrow[d] & Z_\mu \arrow[u, dashed] \arrow[r] \arrow[d] \arrow[l] \arrow[ul, phantom, "(\clubsuit)"] & \Gr_{\Gg,\leq \mu}^{(\infty),\spl} \arrow[u] \arrow[d]\\
	\Sht_{\Gg,\leq \tau}^{\spl} & \Sht_{\tau\mid \sigma(\nu_I^*),\tau_I+\sigma(\nu_I)}^{0,\spl \mid \can} \arrow[l] \arrow[r] & \Sht_{\Gg,\leq (\sigma(\nu_I^*),\tau_I+\sigma(\nu_I))}^{\can} \arrow[r, "F^{-1}"'] & \Sht_{\Gg,\leq( \tau_I + \sigma(\nu_I),\nu_I^*)}^{\can} & \Sht_{\tau_I+\sigma(\nu_I),\nu_I^*\mid \mu}^{0,\can\mid \spl} \arrow[l] \arrow[r] & \Sht_{\Gg,\leq \mu}^{\spl}.
\end{tikzcd}\end{tiny}\end{equation}
Here,
\begin{itemize}
	\item we have used that \(\Gr_{\Gg,\leq \tau}^{\spl} \cong \Gr_{\Gg,\leq \tau_I}^{\can}\) since \(\tau\) is central, and hence also \(\Sht_{\Gg,\leq \tau}^{\spl} \cong \Sht_{\Gg,\leq \tau_I}^{\can}\),
	\item \(Z_{\tau_I}\) is defined as
	\[Z_{\tau_I}:=\{(g_1,g_2)\in \Gr_{\Gg,\leq (\sigma(\nu_I^*),\tau_I+\sigma(\nu_I))}^{(\infty),\can}\mid g_1g_2\in \varpi^{\tau_I}L^+\Gg\} \cong \Gr_{\Gg,\leq \sigma(\nu_I*)}^{\can} \times L^+\Gg,\]
	making the four leftmost squares cartesian (note that all vertical morphisms involved are \(L^+\Gg\)-fibrations),
	\item \(Z_\mu\) is defined as
	\[Z_\mu := \{(g_1,g_2) \in \Gr_{\Gg,\leq (\tau_I+\sigma(\nu_I),\nu_I^*)}^{(\infty),\can} \mid g_1,g_2\in L^+\Gg \varpi^\mu L^+\Gg\},\]
	making the two rightmost bottom squares cartesian,
	\item the map \(\Gr_{\Gg,\leq (\sigma(\nu_I^*),\tau_I+\sigma(\nu_I))}^{(\infty),\can} \to \Gr_{\Gg,\leq \sigma(\nu_I^*)}^{\can} \times \Gr_{\Gg,\leq \tau_I^* + \sigma(\nu_I^*)}^{\can}\) is induced by \((g_1,g_2) \mapsto (g_1,g_2^{-1})\).
\end{itemize}
We note that there is no actual map \(\Gr_{\Gg,\leq (\sigma(\nu_I^*),\tau_I+\sigma(\nu_I))}^{(\infty),\can} \to \Gr_{\Gg,\leq (\tau_I+\sigma(\nu_I),\nu_I^*)}^{(\infty),\can}\) making \eqref{diagram of correspondences} commute.
However, consider the diagram
\[\begin{tiny}\begin{tikzcd}
	\Gr_{\Gg,\leq (\sigma(\nu_I^*),\tau_I+\sigma(\nu_I))}^{(\infty),\can} \arrow[r] & \Gr_{\Gg,\leq (\sigma(\nu_I^*),\tau_I+\sigma(\nu_I))}^{\can} \arrow[r] & \Hck_{\Gg,\leq (\sigma(\nu_I^*),\tau_I+\sigma(\nu_I))}^{\can} \arrow[r] & \Hck_{\Gg,\leq \sigma(\nu_I^*)}^{\can} \times \Hck_{\Gg, \leq \tau_I + \sigma(\nu_I)}^{\can} \arrow[d]\\
	\Gr_{\Gg,\leq (\tau_I+\sigma(\nu_I),\nu_I^*)}^{(\infty),\can} \arrow[r] & \Gr_{\Gg,\leq (\tau_I+\sigma(\nu_I),\nu_I^*)}^{\can} \arrow[r] & \Hck_{\Gg,\leq (\tau_I+\sigma(\nu_I),\nu_I^*)}^{\can} \arrow[r] & \Hck_{\Gg,\leq \tau_I+\sigma(\nu_I)}^{\can} \times \Hck_{\Gg,\leq \sigma(\nu_I^*)}^{\can}.
\end{tikzcd}\end{tiny}\]
Then, when restricting to \(\MTM\), the !-pullback along the two leftmost arrows is fully faithful by \cite[Lemma 3.2.12]{RicharzScholbach:Intersection}, whereas !-pullback along the other arrows is an equivalence by \thref{Satake for convolution Gr}.
Since our main use of \eqref{diagram of correspondences} will be to pull back correspondences, which will only involve objects in \(\MTM(\Gr_{\Gg,\leq (\sigma(\nu_I^*),\tau_I+\sigma(\nu_I))}^{(\infty),\can})\) pulled back from \(\MTM(\Hck_{\Gg,\leq (\sigma(\nu_I^*),\tau_I+\sigma(\nu_I))}^{\can})\), we may safely pretend the dashed arrow \(\Gr_{\Gg,\leq (\sigma(\nu_I^*),\tau_I+\sigma(\nu_I))}^{(\infty),\can} \dasharrow \Gr_{\Gg,\leq (\tau_I+\sigma(\nu_I),\nu_I^*)}^{(\infty),\can}\) exists.
In fact, this dashed arrow behaves like it is an isomorphism, and !-pullback along this arrow (which is well-defined for the objects we will be interested in) commutes with the !-pullback along the other arrows.

Similarly, the dashed arrow \(\Gr_{\Gg,\leq (\tau_I+\sigma(\nu_I),\nu_I^*)}^{(\infty),\can} \dasharrow \Gr_{\Gg,\leq \nu_I^*}^{\can} \times \Gr_{\Gg,\leq \tau_I^* + \sigma(\nu_I^*)}^{\can}\) does not actually exist, but we can pretend there is a !-pullback along it using the diagram
\[\begin{tiny}
	\begin{tikzcd}
		\Gr_{\Gg,\leq (\tau_I+\sigma(\nu_I),\nu_I^*)}^{(\infty),\can} \arrow[r] & \Gr_{\Gg,\leq (\tau_I+\sigma(\nu_I),\nu_I^*)}^{\can} \arrow[r] & \Hck_{\Gg,\leq (\tau_I + \sigma(\nu_I))}^{\can} \times \Hck_{\Gg,\leq \nu_I^*}^{\can} \arrow[d]\\
		\Gr_{\Gg,\leq \nu_I^*}^{\can} \times \Gr_{\Gg,\leq \tau_I^* + \sigma(\nu_I^*)}^{\can} \arrow[rr] && \Hck_{\Gg,\leq (\tau_I^* + \sigma(\nu_I^*))}^{\can} \times \Hck_{\Gg,\leq \nu_I^*}^{\can},
	\end{tikzcd}
\end{tiny}\]
where the vertical arrow is induced by \(g\mapsto g^{-1} \colon \Hck_{\Gg,\leq (\tau_I + \sigma(\nu_I))}^{\can} \cong \Hck_{\Gg,\leq (\tau_I^* + \sigma(\nu_I^*))}^{\can}\). 
Again, !-pullback along all arrows induce equivalences on \(\MTM\) (except possibly the rightmost upper horizontal arrow, where !-pullback is still fully faithful).
Moreover, the dashed arrow \(\Gr_{\Gg,\leq (\tau_I+\sigma(\nu_I),\nu_I^*)}^{(\infty),\can} \dasharrow \Gr_{\Gg,\leq \nu_I^*}^{\can} \times \Gr_{\Gg,\leq \tau_I^* + \sigma(\nu_I^*)}^{\can}\) behaves like an \(L^+\Gg\)-torsor.

Finally, we can similarly pretend the dashed arrow \(Z_\mu \to \Gr_{\nu_I^*,\mu^*\mid \tau_I^* + \sigma(\nu_I^*)}^{\spl}\) exists, and makes the square \((\clubsuit)\) behave like it is cartesian.
Using this, all squares are cartesian (or at least behave like they are, in case the arrows do not actually exist), except for the upper rightmost square.
This finishes the explanation of the diagram \eqref{diagram of correspondences}.

Next, we have the following sequence of isomorphisms, using that \(\mu\) is minuscule, so that \(\Gr_{\Gg,\leq \mu}^{\spl}\) is smooth:
\[\Hom_{\Gmot^I}(q^{-1}\sigma V_{\nu_I}^{\can} \otimes V_{\tau}^{\spl} \otimes V_{\nu_I}^{\can,*},V_\mu^{\spl}) \cong \Hom_{\Gmot^I}(V_{\nu_I^*}^{\can} \otimes V_{\mu^*}^{\spl}(\langle 2\rho,\mu\rangle),V_{\tau_I^*}^{\can} \otimes V_{\sigma(\nu_I^*)}^	{\can})\]
\[\cong \Corr_{\Gr_{\nu_I^*,\mu^*\mid \tau_I^*+\sigma(\nu_I^*)}^{\spl}}\left( (\Gr_{\Gg,\leq \nu_I^*}^{\can} \widetilde{\times} \Gr_{\Gg,\leq \mu^*}^{\spl}, \IC_{\nu_I^*}(\IQ) \widetilde{\boxtimes} \IC_{\mu^*}(\IQ)(\langle 2\rho,\mu\rangle)),(\Gr_{\Gg,\leq \tau_I^* + \sigma(\nu_I^*)}^{\can},\IC_{\tau_I^*+\sigma(\nu_I^*)}(\IQ))\right)\]
\[\cong \Corr_{\Gr_{\nu_I^*,\mu^*\mid \tau_I^*+\sigma(\nu_I^*)}^{\spl}}\left((\Gr_{\Gg,\leq \nu_I^*}^{\can}, \IC_{\nu_I^*}(\IQ)[\langle 2\rho,\mu\rangle](\langle 2\rho,\mu\rangle)), (\Gr_{\Gg,\leq \tau_I^* + \sigma(\nu_I^*)}^{\can}, \IC_{\tau_I^*+\sigma(\nu_I^*)}(\IQ)) \right)\]
\[\cong \Corr_{\Gr_{\nu_I^*,\mu^*\mid \tau_I^*+\sigma(\nu_I^*)}^{\spl}}\left((\Gr_{\Gg,\leq \nu_I^*} \times \Gr_{\Gg,\leq \tau_I^*+\sigma(\nu_I^*)}^{\can}, \IC_{\nu_I^*}(\IQ) \boxtimes \ID(\IC_{\tau_I^*+\sigma(\nu_I^*)}(\IQ))),(\Spec \overline{k}, \IQ[-\langle 2\rho,\mu\rangle](-\langle 2\rho,\mu\rangle))\right),\]
where the last isomorphism follows from \thref{products and correspondences}.
We denote this composition by \(\bfa \mapsto \Sat(\bfa)^\sharp\).

Now, we can apply the pro-smooth pullback of correspondences from the lower to the middle row of \eqref{diagram of correspondences}, as well as from the upper to the middle row of \eqref{diagram of correspondences}.
Then:
\begin{itemize}
	\item Let 
	\[\ID((\Gamma_{\varpi^{\tau^*}})^\sharp)\colon (\Spec \overline{k},\IQ) \to (\Gr_{\Gg,\leq \sigma(\nu_I^*)} \times \Gr_{\Gg,\leq \tau^*+\sigma(\nu_I^*)},\IC_{\nu_I^*} \boxtimes \ID(\IC_{\tau_I^*+\sigma(\nu_I^*)}))\]
	be obtained by applying the constructions from \thref{examples of correspondences} \eqref{dual correspondence}, \thref{products and correspondences} and \thref{examples of correspondences} \eqref{pullback correspondence} respectively, and the fact that pullback along \(\varpi^{\tau^*}\colon \Gr_{\Gg,\leq \sigma(\nu_I^*)} \cong \Gr_{\Gg,\leq \tau^*+\sigma(\nu_I^*)}\) identifies \(\IC_{\tau_I^*+\sigma(\nu_I^*)}\) and \(\IC_{\sigma(\nu_I^*)}\).
	Then the pullback of \(\ID((\Gamma_{\varpi^{\tau^*}})^\sharp)\) agrees with the pullback of \(\mathscr{C}_{\Sht}^{\spl}(\delta_{V_{\sigma(\nu_I*)}^{\can}} \otimes \identity_{V_\tau^{\spl}})\).
	\item The pullback of \(\ID(F^{-1})\) agrees with the pullback of \(\Gamma_{\sigma^{-1}\times \id}^*\).
	\item For \(\bfa\in  \Hom_{\Gmot^I}(q^{-1}\sigma V_{\nu_I}^{\can} \otimes V_{\tau}^{\spl} \otimes V_{\nu_I}^{\can,*},V_\mu^{\spl})\), the pullback of \(\mathscr{C}_{\Sht}^{\spl}(\bfa)\) agrees with the pullback of \(\Sat(\bfa)^\sharp\).
\end{itemize}

Thus, by \thref{compatibility pullback correspondence}, the pullback of \(\mathscr{C}_{V_{\nu_I}^{\can}}^{\spl}(\bfa)\) to the middle row agrees with the pullback of 
\begin{equation}\label{new correspondence}
	\mathscr{C}_{\nu_I}^{\Gr}(\bfa):=\Sat(\bfa)^\sharp \circ \Gamma_{\sigma \times \identity}^* \circ \ID((\Gamma_{\varpi^{\tau^*}})^\sharp).
\end{equation}

We now specialize the above discussion to two cases.
The first will be to relate the functor \(\IL^{\can}\) to Hecke correspondences, via the Satake isomorphism.

\begin{rmk}\thlabel{sign switch Satake iso}
	Recall that for general quasi-split groups, the Satake isomorphism with \(\IZ\)-, and hence with \(\IZ[\frac{1}{p}]\),-coefficients was deduced from the motivic Satake equivalence in \cite[§10]{vdH:RamifiedSatake}.
	However, this Satake isomorphism has to be modified in order to be compatible with \(\IL^{\can}\), as we now explain.
	Recall that \(\sigma\) denotes an arithmetic Frobenius, whereas \(\phi\) denotes a geometric Frobenius.
	
	Let \(R(\Gmot^I)=K_0(\Rep^{\fd}(\Gmot^I))[\frac{1}{p}]\) be the Grothendieck ring of the category of finite dimensional \(\Gmot^I\)-representations, tensored with \(\IZ[\frac{1}{p}]\).
	By \cite[Lemma 10.9]{vdH:RamifiedSatake}, taking traces of representations yields a surjection \(R(\Gmot^I) \to \IZ[\frac{1}{p}][\widehat{G}^Iq\phi]^{\widehat{G}^I}\), where \(\widehat{G}^I\) acts on \(\widehat{G}^Iq\phi\) via conjugation.
	Using the motivic Satake equivalence and the trace of geometric Frobenius, the Grothendieck-Lefschetz trace formula yields an isomorphism 
	\[\Sat^{\cl}\colon \IZ[\frac{1}{p}][\widehat{G}^Iq\phi]^{\widehat{G}^I} \cong \Hh_{\Gg},\] 
	\cite[Theorem 10.11]{vdH:RamifiedSatake}.
	Moreover, the Satake transform \(\CT^{\cl}\) from \cite[§10.2]{vdH:RamifiedSatake} identifies \(\Hh_\Gg\) with a subring of \(\IZ[\frac{1}{p}][X^*(\widehat{T}^I)^{\phi}]\); note that this version of the Satake transform does not involve the modulus character, and hence does not need a square root of \(q\).
	
	On the other hand, taking traces also yields a surjection \(R(\Gmot^I) \to \IZ[\frac{1}{p}][\widehat{G}^Iq^{-1}\sigma]^{\widehat{G}^I}\).
	Moreover, a modified version of the Chevalley restriction isomorphism (replacing the geometric Frobenius by the arithmetic Frobenius in \cite[Proposition 10.8]{vdH:RamifiedSatake}) identifies \(\IZ[\frac{1}{p}][\widehat{G}^Iq^{-1}\sigma]^{\widehat{G}^I}\) with a subring of \(\IZ[\frac{1}{p}][X^*(\widehat{T}^I)^{\phi}]\); it follows from \cite[Theorem 10.7]{vdH:RamifiedSatake} that the resulting subring agrees with the one obtained by embedding \(\IZ[\frac{1}{p}][\widehat{G}^Iq\phi]^{\widehat{G}^I}\) into it.
	Since this subring moreover agrees with the image of the Satake transform \(\CT^{\cl}\), we obtain an isomorphism \[\Sat^{\cl'}\colon \IZ[\frac{1}{p}][\widehat{G}^Iq^{-1}\sigma]^{\widehat{G}^I} \cong \Hh_{\Gg}.\]

	Although \(\Sat^{\cl}\) and \(\Sat^{\cl'}\) do not agree, they are related via the commutative diagram
	\[\begin{tikzcd}[row sep=tiny]
		& \IZ[\frac{1}{p}][\widehat{G}^Iq\phi]^{\widehat{G}^I} \arrow[dd, "\operatorname{inv}"'] \arrow[r, "\Sat^{\cl}", "\cong"'] & \Hh_{\Gg} \arrow[r, "\CT^{\cl}"] & \IZ[\frac{1}{p}][X^*(\widehat{T}^I)^{\phi}] \arrow[dd, "\lambda\mapsto -\lambda"]\\
		R(\Gmot^I) \arrow[ur] \arrow[dr]&§&\\
		&\IZ[\frac{1}{p}][\widehat{G}^Iq^{-1}\sigma]^{\widehat{G}^I} \arrow[r, "\Sat^{\cl'}"', "\cong"] & \Hh_{\Gg} \arrow[r, "\CT^{\cl}"'] & \IZ[\frac{1}{p}][X^*(\widehat{T}^I)^{\phi}],
	\end{tikzcd}\]
	where the map \(\operatorname{inv} \colon \IZ[\frac{1}{p}][\widehat{G}^Iq\phi]^{\widehat{G}^I} \cong \IZ[\frac{1}{p}][\widehat{G}^Iq^{-1}\sigma]^{\widehat{G}^I}\) is induced by the \(\widehat{G}^I\)-equivariant isomorphism \(\widehat{G}^I q\phi\cong \widehat{G}^Iq^{-1}\sigma\colon g\cdot q\phi \mapsto (q^{-1}\sigma)(g) \cdot q^{-1} \sigma\).
\end{rmk}

Let us go back to the comparison between \(\IL^{\can}\) with Hecke correspondences, and assume \(\nu_I = \sigma(\nu_I)\), and \(\tau=\mu=0\).
Then the upper row of \eqref{diagram of correspondences} can be identified with
\[\Spec \overline{k} \leftarrow \Gr_{\Gg,\nu_I^*}^{\can} \xrightarrow{\Delta} \Gr_{\Gg,\leq \nu_I^*}^{\can} \times \Gr_{\Gg,\leq \nu_I^*}^{\can} \xrightarrow{\sigma^{-1} \times \identity}\Gr_{\Gg,\leq \nu_I^*}^{\can} \times \Gr_{\Gg,\leq \nu_I^*}^{\can} \xleftarrow{\Delta} \Gr_{\Gg,\leq \nu_I^*} \to \Spec \overline{k}.\]
Let \(\bfa\) be induced by the counit \(\bfe_{\nu_I}\colon V_{\nu_I}^{\can} \otimes V_{\nu_I}^{\can,*} \to \IQ\), so that \(\Gr_{\nu_I^*,\mu^*\mid \tau_I^*+\sigma(\nu_I^*)}^{\spl,\bfa} \cong \Gr_{\leq \nu_I^*}^{\can}\).
Then \(\mathscr{C}_{\nu_I}^{\Gr}(\bfa)\) can be identified with \(\delta_{\IC_{\nu_I^*}} \circ \Gamma_{\sigma\times \id}^* \circ e_{\IC_{\nu_I^*}}\), where \(\delta_{\IC_{\nu_I^*}}\) and \(e_{\IC_{\nu_I^*}}\) are the unit and counit correspondence from \thref{examples of correspondences} \eqref{co/unit corr}.
Then by \thref{motivic Grothendieck-Lefschetz}, this correspondence is supported on \(\Gr_{\Gg,\leq \nu_I^*}^{\can}(k)\), and the value at a point x is exactly the trace of geometric Frobenius on the stalk of \(\IC_{\nu_I^*}\) at a geometric point over \(x\).
By \thref{sign switch Satake iso}, we conclude:

\begin{thm}\thlabel{compatibility with Satake iso}
	The functor \(\IL^{\can}\) sends the structure sheaf to \(\delta_{\mathbf{1}}^{\can}(\IQ)\), and the induced map on endomorphism rings agrees with the Satake isomorphism
	\[\Sat^{\cl'}\colon \IQ[\widehat{G}^Iq^{-1}\sigma]^{\widehat{G}^I} \cong \Hh_\Gg\otimes \IQ.\]
\end{thm}

Next, we compare \(\IL^{\spl}\) with the cycle class map.
For this, we again assume \(\nu_I\in X_*(T)_I^+\) and \(\tau\in X_*(Z_G)\) to be arbitrary, and \(\mu\in X_*(T)^+\) to be minuscule.
Moreover, we will simplify the situation by fixing a prime \(\ell\neq p\) and working with \(\ell\)-adic étale cohomology instead of motives; this will be sufficient for our applications.
Moreover, using a fixed square root \(\sqrt{q}\in \overline{\IQ}_\ell\), we can consider half-twists \((\frac{1}{2})\).
Using this, we define for \(\lambda_I\in X_*(T)^+_I\) the corresponding intersection complex as
\[\IC_{\lambda_I}^{\et}:= \im\left(\pH^0\left(\iota_{\lambda_I,!}\overline{\IQ_\ell}[\langle 2\rho,\lambda_I\rangle](\langle \rho,\lambda_I\rangle)\right) \to \pH^0\left(\iota_{\lambda_I,*}\overline{\IQ_\ell}[\langle 2\rho,\lambda_I\rangle](\langle \rho,\lambda_I\rangle)\right)\right),\]
where \(\iota_{\lambda_I}\colon \Gr_{\Gg,\lambda_I}^{\can} \into \Gr_{\Gg}^{\can}\) is the inclusion, and similarly for the splitting versions.
We will also denote \(V_{\lambda_I}^{\et}\) for the simple \(\overline{\IQ_\ell}\)-linear \(\widehat{G}^I\)-representation of highest weight \(\lambda_I\), which corresponds to \(\IC_{\lambda_I}^{\et}\) under the \(\ell\)-adic geometric Satake equivalence \cite[Theorem 9.9]{vdH:RamifiedSatake}.
Again, we have similar notation for the splitting versions.
Since \(\Sht_{\tau\mid \mu}^{\nu_I,\spl} \cong \Gg(\Oo_F) \backslash X_{\mu,\nu_I}(\varpi^{\tau})\) has dimension \(\leq \langle \rho,\nu\rangle\) by \thref{geometry splitting adlv}, we get
\[\Corr_Y^{\et}(\IL^{\spl}(\widetilde{V_{\tau}^{\spl}}),\IL^{\spl}(\widetilde{V_\mu^{\spl}})) \cong \mathrm{H}^{\BM}_{\langle 2\rho,\mu\rangle}(X_{\mu,\nu_I}^{\spl}(\varpi^{\tau^*}))^{\Gg(\Oo_F)}\]
by \cite[§A.2.18]{XiaoZhu:Cycles}.
By applying the methods of §\ref{Subsec:Motivic correspondences on shtukas} after the \(\ell\)-adic étale realization (cf.~also \cite[§6.2]{XiaoZhu:Cycles}), we can define for any \(\bfa\in \Hom_{\widehat{G}^I}(\sigma V_{\nu_I}^{\et} \otimes V_{\tau}^{\et} \otimes V_{\nu_I}^{\et,*}, V_{\mu}^{\et})\) a correspondence \(\mathscr{C}_{V_{\nu_I}^{\et}}(\bfa) \in \Corr_{Y}^{\et}(\widetilde{V_{\tau}^{\et}},\widetilde{V_{\mu}^{\et}})\), which we may thus view as an element in \(\mathrm{H}^{\BM}_{\langle 2\rho,\mu\rangle}(X_{\mu,\nu_I}^{\spl}(\varpi^{\tau^*}))\).

On the other hand, the support of the correspondence \eqref{new correspondence} (or more precisely, its \(\ell\)-adic realization) can be identified with \(X_{\mu,\nu_I}^{\spl}(\varpi^{\tau})\).
Then \eqref{new correspondence} can also be identified with an element in \(\mathrm{H}^{\BM}_{\langle 2\rho,\mu\rangle}(X_{\mu,\nu_I}^{\spl}(\varpi^{\tau^*}))\), and the discussion preceding \eqref{new correspondence} shows that this agrees with the element defined by \(\mathscr{C}_{V_{\nu_I}^{\et}}(\bfa)\).

\begin{prop}\thlabel{Langlands vs cycle class}
	Let \(\bfa\in \Hom_{\widehat{G}^I}(\sigma V_{\nu_I}^{\et} \otimes V_{\tau}^{\et} \otimes V_{\nu_I^*}^{\et}, V_{\mu}^{\et})\) be induced by an element in \(\IS_{\nu_I^*,\mu^*\mid \sigma(\nu_I^*)+\tau^*}^{\spl}\) via \thref{basic properties of Satake correspondences} (3).
	Let \(\bfb= \iota_{\IM\IV}^{\spl}(\bfa)\) be the element induced via \thref{Satake to MV cycles}.
	Then \(\mathscr{C}_{V_{\nu_I}^{\et}}(\bfa) \in \mathrm{H}^{\BM}_{\langle 2\rho,\mu\rangle}(X_{\mu,\nu_I}^{\spl}(\varpi^{\tau^*}))\) agrees with the fundamental class of \(X_{\mu^*}^{\spl,\bfb}(\tau^*) \subseteq X_{\mu^*,\nu_I^*}^{\spl}(\varpi^{\tau^*})\).
\end{prop}
\begin{proof}
	Recall from \thref{Cycle in ADLV} that \(X_{\mu^*}^{\spl,\bfb}(\tau^*)\) is the unique irreducible component of \(X_{\mu^*,\nu_I^*}^{\spl,\bfa}(\tau^*)\) of dimension \(\langle \rho,\mu-\tau\rangle\), and that all other irreducible components have strictly smaller dimension.
	Note that the correspondence \(\eqref{new correspondence}\) (viewed as an element in \(\mathrm{H}_{\langle2\rho,\mu\rangle}^{\BM}(X_{\mu^*,\nu_I^*}^{\spl}(\tau^*))\)) clearly lies in the image of
	\[\mathrm{H}_{\langle2\rho,\mu\rangle}^{\BM}(X_{\mu^*,\nu_I^*}^{\spl,\bfa}(\tau^*)) \to \mathrm{H}_{\langle2\rho,\mu\rangle}^{\BM}(X_{\mu^*,\nu_I^*}^{\spl}(\tau^*)).\]
	Since the source is 1-dimensional, it follows that \(\mathscr{C}_{V_{\nu_I}^{\et}}(\bfa)\) is a scalar multiple of the fundamental class of \(X_{\mu^*}^{\spl,\bfb}(\tau^*)\).
	Since \(X_{\mu^*}^{\spl,\bfb}(\tau^*)\times_{\Gr_{\Gg,\leq \nu_I^*}^{\can}} \Gr_{\Gg,\nu_I^*}^{\can} \neq \varnothing\) by \thref{Cycle in ADLV}, we may compute the cohomology class after restricting to \(\Gr_{\Gg,\nu_I^*}^{\can} \subseteq \Gr_{\Gg,\leq \nu_I^*}^{\can}\).
	
	Let \(Z:= \Gr_{\nu_I^*,\mu^*\mid \tau_I^*+\sigma(\nu_I^*)} \times_{\Gr_{\Gg,\leq \nu_I^*}^{\can} \times \Gr_{\Gg,\leq \tau_I^*+\sigma(\nu_I^*)}} \Gr_{\Gg,\nu_I^*}^{\can} \times \Gr_{\Gg,\tau_I^*+\sigma(\nu_I^*)}\).
	Then the restriction of the top row of \eqref{diagram of correspondences} to \(\Gr_{\Gg,\nu_I^*}^{\can} \subseteq \Gr_{\Gg,\leq \nu_I^*}^{\can}\) agrees with the intersection of the diagonal 
	\begin{equation}\label{corr1}\Delta\colon \Gr_{\Gg,\sigma(\nu_I^*)}^{\can} \to \Gr_{\Gg,\sigma(\nu_I^*)}^{\can} \times \Gr_{\Gg,\sigma(\nu_I^*)}^{\can}\end{equation}
	 with the correspondence
	\begin{equation}\label{corr2}Z\xrightarrow{c_1\times c_2} \Gr_{\Gg,\sigma(\nu_I^*)}^{\can} \times \Gr_{\Gg,\sigma(\nu_I^*)}^{\can}\end{equation}
	\[c_1\colon Z \to \Gr_{\Gg,\nu_I^*}^{\can} \xrightarrow{\sigma} \Gr_{\Gg,\sigma(\nu_I^*)}^{\can} \qquad c_2 \colon Z \to \Gr_{\tau_I^* + \sigma(\nu_I^*)} \xrightarrow{\varpi^{\tau_I}} \Gr_{\Gg,\sigma(\nu_I^*)}.\]
	Thus, the restriction of \eqref{new correspondence} is given by \(\delta_{\overline{\IQ}_\ell} \circ u^{\sharp}\), where \(\delta_{\overline{\IQ}_\ell}\) is defined in \cite[Example A.2.3 (4)]{XiaoZhu:Cycles}, and \(u\colon c_1^*\overline{\IQ}_\ell[\langle 2\rho,\mu\rangle](\langle \rho,\mu\rangle) \to c_2^!\overline{\IQ}_\ell\) is given by the fundamental class of \(Z\).
	By passing to smooth deperfections, and using that the Frobenius has trivial differential, the correspondences \eqref{corr1} and \eqref{corr2} intersect each other properly smoothly.
	Thus, we conclude by \cite[Lemma A.2.21]{XiaoZhu:Cycles}.
\end{proof}

\section{Motivic Jacquet--Langlands transfers}\label{Sec:JacquetLanglands}

In this section, we combine the construction of exotic Hecke correspondences between special fibers of Shimura varieties from \thref{Exotic Hecke correspondences}, with the motivic correspondences on the moduli of local shtukas from \thref{Main local theorem} to construct motivic realizations of the Jacquet--Langlands correspondence.
Throughout this section, we assume \(p\neq 2\).

We fix a setup as in Section \ref{Sec:Exotic}, i.e., we fix two global integral (at \(p\)) PEL data \((B,*,\Oo_B,V_i,(,)_i,\Ll_i,h_i)\) with \(i=1,2\), as well as an isomorphism \(V_1\otimes \IA_f \cong V_2\otimes \IA_f\) which is compatible with the pairings \((,)_i\) and the \(\Oo_B\)-structure.
Let \((\IG_i,\IX_i)\) be the corresponding Shimura data, with Hodge cocharacters \(\mu_i\) and reflex fields \(\IE_i\).
Then we have an isomorphism \(\IG_{1,\IA_f} \cong \IG_{2,\IA_f}\), which induces an Galois-equivariant isomorphism on dual groups, and assume that \(\mu_{1\mid Z(\widehat{G})^{\Gamma_{\IQ_p}}} = \mu_{2\mid Z(\widehat{G})^{\Gamma_{\IQ_p}}}\).
Let \(\IE:=\IE_1\IE_2\subseteq \mathbb{C}\), fix a place \(\pp\) of \(\IE\) lying over \(p\), and let \(E:=\IE_\pp\) be the corresponding completion with residue field \(\IF_q\).
We fix a small enough compact open subgroup \(K^p\subseteq \IG_1(\IA_f^p) \cong \IG_2(\IA_f^p)\), and assume that the integral models of \(G:=\IG_{1,\IQ_p} \cong \IG_{2,\IQ_p}\) corresponding to \(\Ll_1\) and \(\Ll_2\) agree and are very special; we denote it by \(\Gg\).
Finally, we assume that the splitting models associated to the above PEL data are flat and absolutely weakly normal, i.e., that \thref{assumption splitting flat} holds for both data.

In the situation above, we have the Shimura varieties \(\Sh_K(\IG_i,\IX_i)\) defined over \(\IE_i\) (hence over \(\IE\)),  and the integral models \(\mathscr{S}_{K^pK_p}^{\spl}(\IG_i,\IX_i)\) and \(\mathscr{S}_{K^pK_p}^{\can}(\IG_i,\IX_i)\) over \(\Oo_{E^{\Gal}}\) and \(\Oo_{E}\) respectively. 
We denote the perfections of their special fibers by \(\Sh_{\mu_i,K}^{\spl}\) and \(\Sh_{\mu_i,K}^{\can}\); these live over \(\IF_q\).
Throughout this section, all geometric objects (moduli of local shtukas, special fibers of Shimura varieties) will be considered over \(\overline{\IF}_q\); we denote the structure maps of the Shimura varieties by \(\gamma_i^{\spl} \colon \Sh_{\mu_i{,}K}^{\spl}\to \Spec \overline{\IF}_q\) and \(\gamma_i^{\can}\colon \Sh_{\mu_i,K}^{\can} \to \Spec \overline{\IF}_q\).

\subsection{Motivic realizations of the Jacquet--Langlands correspondence}

We first construct Jacquet--Langlands correspondences in the case of trivial coefficients.
Note that the (underived) prime-to-\(p\) Hecke algebras for \((\IG_1,\IX_1)\) and \((\IG_2,\IX_2)\) (at level \(K^p\) and with \(\IQ\)-coefficients) can be identified; we denote it by \(\Hh_{K^p}\).
The following theorem verifies further cases of \cite[Conjecture 4.60]{Zhu:Coherent}; although we only work at the underived level, we allow ramified cases, and provide motivic enhancements of the conjectural geometric Jacquet--Langlands correspondence.
Recall the functor \(\IL^{\can}\) from \thref{Main local theorem}, and the crystalline period map \(\loc_p^{\can}\colon \Sh_{\mu_i,K}^{\can} \to \Sht_{\Gg,\leq \mu_{i,I}}^{\can}\) from \eqref{shimura varieties to shtukas}.

\begin{thm}\thlabel{JL for canonical models}
	Let \(W_1,W_2\in \Rep^{\fd}_{\widehat{G}^I}(\MTM(\Spec \overline{\IF}_p))\), with pullbacks \(\widetilde{W_1},\widetilde{W_2}\in \Coh^{\Gmot^I}(\widehat{G}^I p^{-1} \sigma)\).
	Then there is a natural map
	\[\Hom_{\Coh^{\widehat{G}^I}(\widehat{G}^I p^{-1} \sigma)}\left(\widetilde{W_1},\widetilde{W_2}\right)
	\to \Hom_{\DM(\Spec \overline{\IF}_q)}\left(\gamma_{1,!}^{\can}\loc_p^{\can,!}\IL^{\can}(\widetilde{W_1}),\gamma_{2,!}^{\can}\loc_p^{\can,!}\IL^{\can}(\widetilde{W_2})\right),\]
	taking values in \(\Hh^p\)-equivariant morphisms.
	This map is moreover compatible with compositions.
\end{thm}
\begin{proof}
	The map will be defined as the composition
	\[\Hom_{\Coh^{\Gmot^I}(\widehat{G}^I p^{-1} \sigma)}\left(\widetilde{W_1},\widetilde{W_2}\right) \to \Hom_{\MTM_{\cons}^{\Corr}(\Sht_{\Gg}^{\can})}\left(\IL^{\can}(\widetilde{W_1}),\IL^{\can}(\widetilde{W_2})\right)\]
	\[\to \varinjlim_{\nu_I\in X_*(T)_I^+} \Corr_{\Sh_{\mu_1\mid \mu_2}^{\can}}\left(\loc_p^{\can,!}\IL^{\can}(\widetilde{W_1}),\loc_p^{\can,!}\IL^{\can}(\widetilde{W_2})\right)\]
	\[\to \Hom_{\DM(\Spec \overline{\IF}_q)}\left(\gamma_{1,!}^{\can}\loc_p^{\can,!}\IL^{\can}(\widetilde{W_1}),\gamma_{2,!}^{\can}\loc_p^{\can,!}\IL^{\can}(\widetilde{W_2})\right),\]
	where \(\Sh_{\mu_1\mid \mu_2}^{\can} \to \Sh_{\mu_1,K}^{\can} \times \Sh_{\mu_2,K}^{\can}\) is the exotic Hecke correspondence from \thref{Exotic Hecke correspondences}.
	We explain how each of these arrows is defined.
	
	The first arrow arises from \thref{Main local theorem} (since we only need the canonical part of the diagram, this exists without the assumption that the local groups are essentially unramified, by \thref{remark main local theorem}).
	The second arrow is given by the pro-smooth pullback of correspondences along the middle part of \eqref{big diagram of correspondences}, as in \thref{construction pullback correspondence} (using \thref{smoothness of shimura to shtuka}).
	The fact that pullback of correspondences is compatible with the transition maps in the colimit, which are given by proper pushforward, follows from \thref{Compatibility pullback and pushforward correspondences}.
	The last arrow will follow from the proper pushforward of correspondences, using the diagram
	\[\begin{tikzcd}
		\Sh_{\mu_1,K}^{\can} \arrow[d, "\gamma_1^{\can}"'] & \Sh_{\mu_1\mid \mu_2}^{\can} \arrow[d] \arrow[l] \arrow[r] & \Sh_{\mu_2,K}^{\can} \arrow[d, "\gamma_2^{\can}"]\\
		\Spec \overline{\IF}_q & \Spec \overline{\IF}_q \arrow[l, equal] \arrow[r, equal] & \Spec \overline{\IF}_q,
	\end{tikzcd}\]
	as long as this is compatible with the transition morphisms for varying \(\nu_I\in X_*(T)_I^+\).
	But this follows immediately from the compatibility of proper pushforward with composition, \thref{compatibility pushforward correspondence}.
	
	This map is compatible with compositions (when fixing a third global integral PEL data satisfying the same conditions as \((B,*,\Oo_B,V_i,(,)_i,\Ll_i,h_i)\)), since this holds for each of the arrows above.
	Indeed, for the first arrow this follows since \thref{Main local theorem} gives a functor, and for the two other arrows this follows from \thref{compatibility pullback correspondence} and \thref{compatibility pushforward correspondence} respectively.
\end{proof}

\begin{rmk}\thlabel{remarks on JL}
	\begin{enumerate}
		\item An interesting special case arises when \(W_i\) are the representations \(V_{\mu_i}^{\spl}\) from \thref{notation representations}, so that their restriction to \(\widehat{G}^I\subseteq \Gmot^I\) agrees with the restriction of the irreducible \(\widehat{G}\)-representation of highest weight \(\mu_i\).
		Namely, in that case, the objects \(\loc_p^{\can,!}\IL^{\can}(\widetilde{W_i})\) should agree with the nearby cycles motive of \(\Sh_{\mu_i,K}^{\can}\).
		Indeed, at least on the level of \(\ell\)-adic cohomology, this follows from the fact that \(\loc_p^{\can}\) admits a pro-smooth formal and deperfected lift by \thref{smoothness of shimura to display}, the relation of ramified Satake with nearby cycles \cite{Zhu:Ramified,Richarz:Affine,vdH:RamifiedSatake}, as well as \cite{Berkovich:Vanishing} (which tells that the nearby cycles of \(\mathscr{S}_{K^pK_p}^{\can}\) agree with the nearby cycles of its formal completion).
		The same will hold motivically as well, once a motivic analogue of \cite{Berkovich:Vanishing} is known.
		\item If \(G=\IG_{\IQ_p}\) is essentially unramified, a similar morphism can be constructed using splitting models instead (see the next subsection for a more general construction).
		These two maps are moreover compatible with proper pushforward along \(\Sh_{\mu_i,K}^{\spl} \to \Sh_{\mu_i,K}^{\can}\).
		Since splitting models for essentially unramified groups are smooth (at very special level), the objects \(\loc_p^{\spl,!}(\IL^{\spl}(\widetilde{V_{\mu_i}^{\spl}}))\) agree with the shifted monoidal unit \(\IQ[\langle 2\rho,\mu_i\rangle]\in \DM(\Sh_{\mu_i,K}^{\spl})\).
		Thus, the compatibility of nearby cycles with proper pushforward and smooth pullback gives a different argument in this case for why \(\loc_p^{\can,!}\IL^{\can}(\widetilde{V_{\mu_i}^{\spl}})\) agrees with the nearby cycles of \(\Sh_{\mu_i,K}^{\can}\).
	\end{enumerate}
\end{rmk}

\begin{rmk}\thlabel{remark on S=T}
	Assume that \(G\) is essentially unramified, so that \(\loc_p^{\can,!}\IL^{\can}(\widetilde{V_{\mu_1}^{\spl}})\) is indeed the nearby cycles of \(\Sh_{\mu_1,K}^{\can}\).
	In that case, \thref{JL for canonical models} induces an action of the spherical Hecke algebra
	\[\Hh_{\Gg}\otimes \IQ \cong \bfJ \subseteq \End_{\Coh^{\Gmot^I}(\widehat{G}^I p^{-1} \sigma)}(V_{\mu_1}^{\spl})\]
	on \(\gamma_{1,!}^{\can}\loc_p^{\can,!}\IL^{\can}(\widetilde{V_{\mu_1}^{\spl}})\).
	On the other hand, \(\bfJ\) also acts on the same object via Hecke correspondences, and we conjecture that the two actions agree.
	
	In the unramified case and for \(\ell\)-adic cohomology, this was proven in \cite{Wu:S=T} by using diamonds and the Fargues--Scholze local Langlands correspondence \cite{FarguesScholze:Geometrization}.
	It is plausible that similar methods, together with the recent motivic refinement \cite{Scholze:Geometrization}, could prove the above conjecture as well.
	In order to avoid a long digression on \(p\)-adic geometry and Berkovich motives \cite{Scholze:Berkovich}, we leave this for future work.
	
	However, for 0-dimensional Shimura varieties, the compatibility of the two \(\bfJ\)-actions can be proven as in \cite[Proposition 7.3.14]{XiaoZhu:Cycles}, and it follows essentially from \thref{compatibility with Satake iso}.
\end{rmk}

Finally, we explain how the motivic nature of \thref{JL for canonical models} allows us to construct Jacquet--Langlands transfers between higher (rationalized) Chow groups.
For some smooth variety \(X\) over \(\overline{\IF}_q\) and \(m,n\in \IZ\), we denote by
\[\CH^{m,n}(X)_{\IQ} = \Hom_{\DM(X)}(\IQ,\IQ[m](n))\]
the (rationalized) higher Chow groups of \(X\) \cite[Example 11.2.3]{CisinskiDeglise:Triangulated}.

\begin{cor}\thlabel{JL for Chow and crystalline}
	Assume that \(G\) is essentially unramified, and that the \(\Sh_{\mu_i{,}K}^{\can}\) are perfectly proper.
	Then for any \(n_1,n_2\in \IZ\), there is a natural map
	\[\Hom_{\Coh^{\Gmot^I}(\widehat{G}^I p^{-1} \sigma)}(\widetilde{V_{\mu_1}^{\spl}}(n_1),\widetilde{V_{\mu_2}^{\spl}}(n_2)) \to \Hom_{\Hh^p}(\CH^{\langle 2\rho,\mu_1\rangle,n_1}(\Sh_{\mu_1,K}^{\spl})_{\IQ},\CH^{\langle 2\rho,\mu_2\rangle,n_2}(\Sh_{\mu_2,K}^{\spl})_{\IQ}),\]
	compatible with composition.
\end{cor}
By further applying the rigid realization from \cite[Example 17.2.23]{CisinskiDeglise:Triangulated}, which computes crystalline cohomology for smooth projective varieties, we further deduce a Jacquet--Langlands transfer between the crystalline cohomology of \(\Sh_{\mu_1{,}K}^{\spl}\) and \(\Sh_{\mu_2{,}K}^{\spl}\).
\begin{proof}
	Our assumptions imply that \(\Sh_{\mu_i,K}^{\spl}\) are perfectly proper and perfectly smooth.
	Recall from \thref{remarks on JL} that 
	\[\gamma_{i,!}^{\can} \loc_p^{\can,!} \IL^{\can}(\widetilde{V_{\mu_i}^{\spl}}) \cong \gamma_{i,!}^{\spl} \loc_p^{\spl,!} \IL^{\spl}(\widetilde{V_{\mu_i}^{\spl}}) \cong \gamma_{i,*}^{\spl}(\IQ[\langle 2\rho,\mu_i\rangle]).\]
	The corollary then follows by applying \(\Hom_{\DM(\Spec \overline{\IF}_q)}(\IQ,-)\) to the map from \thref{JL for canonical models}.
\end{proof}

\subsection{The case of nontrivial local systems}

Next, we consider the case of non-trivial motivic local systems.
For this, we also assume the PEL data \((B,*,\Oo_B,V_i,(,)_i,\Ll_i,h_i)\) are such that the \(G_i\) are essentially unramified and \(\Gg_i\) very special, so that the corresponding splitting models are smooth.

Recall that in \cite{Yu:Geometric}, Yu constructs geometric Jacquet--Langlands transfers between the cohomology of the special fibers of different Shimura varieties, with coefficients in local systems associated with representations of the group \(\IG_{1,\IQ_\ell}\cong \IG_{2,\IQ_\ell}\) for some \(\ell\neq p\).
To generalize this motivically, one needs to lift these local systems to objects in \(\DM\), but one also needs to find \(\IQ\)-vector spaces on which both \(\IG_1\) and \(\IG_2\) act, such that the representations become isomorphic over \(\IQ_\ell\) for all \(\ell\neq p\).
Since we use PEL type Shimura varieties, one such representation is given by \(V_1\), which has the same dimension as \(V_2\), as \(V_{1,\IA_f}\cong V_{2,\IA_f}\).
More generally, let \(U_1\) be any \(\IQ\)-linear \(\IG_1\)-representation obtained from \(V_1\) by taking direct sums, tensor products, duals, and exterior products, and let \(U_2\) denote the representation obtained from \(V_2\) via the same operations.
By \cite[Théorème 8.6]{Ancona:Decomposition}, these yield objects \(\Ll_{U_i}\in \DM(\Sh_{\mu_i,K}^{\spl})\); it is not hard to check that their \(\ell\)-adic realizations agree with the local systems constructed in \cite[§6.2]{Yu:Geometric} in the unramified case.
(Note that \cite{Ancona:Decomposition} considers canonical models of Shimura varieties over number fields, but the proofs work for general smooth quasi-projective moduli spaces of abelian varieties over a field; in particular for smooth splitting models.
Moreover, loc.~cit.~works with relative Chow motives, but these are canonically fully faithfully embedded in \(\DM\) \cite[§11.3.8]{CisinskiDeglise:Triangulated}.)

\begin{thm}\thlabel{JL transfer for local systems}
	Let \(\nu_I\in X_*(T)_I^+\), and \(U_1,U_2\) be as above.
	Then there is a natural map
	\[\Corr_{\Sh_{\mu_1\mid \mu_2}^{\nu_I,\spl}}((\Sh_{\mu_1,K}^{\spl},\IQ),(\Sh_{\mu_2,K}^{\spl},\IQ)) \to \Corr_{\Sh_{\mu_1\mid \mu_2}^{\nu_I,\spl}}((\Sh_{\mu_1,K}^{\spl},\Ll_{U_1}),(\Sh_{\mu_2,K}^{\spl},\Ll_{U_2})),\]
	which is compatible with changing \(\nu_I\) (using proper pushforward of correspondences), as well as with composition.
\end{thm}
\begin{proof}
	First, we assume \(U_i=V_i^\vee\) is the dual of \(V_i\).
	Consider the diagram
	\[\begin{tikzcd}
		A_1 \arrow[d, "f_1"'] & & \widetilde{A_1} \arrow[rd, "\widetilde{f_1}"'] \arrow[rr, "\alpha", dashed] \arrow[ll, "\widetilde{p_1}"'] && \widetilde{A_2}\arrow[ld, "\widetilde{f_2}"] \arrow[rr, "\widetilde{p_2}"] && A_2 \arrow[d, "f_2"]\\
		\Sh_{\mu_1,K}^{\spl} &&& \Sh_{\mu_1\mid \mu_2}^{\nu_I,\spl} \arrow[lll, "p_1"] \arrow[rrr, "p_2"'] &&& \Sh_{\mu_2,K}^{\spl},
	\end{tikzcd}\]
	where \(A_i\to \Sh_{\mu_i,K}^{\spl}\) is the universal abelian variety (the additional structure will not be relevant for the proof) with pullback \(\widetilde{A_i}\to \Sh_{\mu_1\mid \mu_2}^{\nu_I,\spl}\), and \(\alpha\) is the universal \(p\)-quasi-isogeny (arising from the proof of \thref{Exotic Hecke correspondences}).
	Note that these abelian schemes are of relative dimension \(g=\frac{\dim_\IQ V_i}{2}\).
	By \cite[Corollary 3.2]{DeningerMurre:Motivic}, \(\Ll_{U_i}\in \DM(\Sh_{\mu_i,K}^{\spl})\) is characterized as the direct summand of \(f_{i,*}\IQ\) through which the multiplication by \(n\) map on \(A_i\) acts via multiplication by \(n\).
	By proper base change, the same characterization holds for the pullback \(p_i^*\Ll_{U_i}\in \DM(\Sh_{\mu_1\mid \mu_2}^{\nu_I,\spl})\), as a direct summand of \(\widetilde{f_i}_{,*} \IQ\).
	
	We claim that the \(p\)-quasi-isogeny \(\alpha\) (which is induced by \(V_{1,\IA_f}\cong V_{2,\IA_f}\)) induces an isomorphism \(\widetilde{f_1}_{,*}\IQ\cong \widetilde{f_2}_{,*}\IQ\), identifying the direct summands \(p_1^*\Ll_{U_1}\) and \(P_2^*\Ll_{U_2}\).
	Indeed, we can choose homomorphisms
	\[\widetilde{A_1}\to \widetilde{A_2}\to \widetilde{A_1} \to \widetilde{A_2}\]
	which are \(p\)-power multiples of \(\alpha\) and its inverse, and
	such that the composition of two maps is the multiplication by a power of \(p\) on \(\widetilde{A_i}\).
	By the above paragraph, and since \(p\) is invertible in our coefficients, multiplication by a power of \(p\) induces an isomorphism on \(\widetilde{f_i}_{,*}\IQ\), preserving the direct summands \(p_i^*\Ll_{U_i}\).
	This readily implies the claim.
	
	We then obtain canonical morphisms
	\[\Hom_{\Sh_{\mu_1\mid \mu_2}^{\nu_I, \spl}}(p_1^*\IQ,p_2^!\IQ)\to \Hom_{\Sh_{\mu_1\mid \mu_2}^{\nu_I, \spl}}(\widetilde{f_2}_{,*}\widetilde{f_2}^*p_1^*\IQ,\widetilde{f_2}_{,*}\widetilde{f_2}^*p_2^!\IQ)\]
	\[\cong \Hom_{\Sh_{\mu_1\mid \mu_2}^{\nu_I,\spl}}(\widetilde{f_1}_{,*}\widetilde{f_1}^*p_1^*\IQ,\widetilde{f_2}_{,*}\widetilde{f_2}^!p_2^!\IQ(-g)[-2g]) \cong \Hom_{\Sh_{\mu_1\mid \mu_2}^{\nu_I,\spl}}(\widetilde{f_1}_{,*}\widetilde{p_1}^*f_1^*\IQ,\widetilde{f_2}_{,*}\widetilde{p_2}^!f_2^!\IQ(-g)[-2g])\]
	\[\cong\Hom_{\Sh_{\mu_1\mid \mu_2}^{\nu_I,\spl}}(p_1^*f_{1,*}f_1^*\IQ, p_2^! f_{2,*} f_2^!\IQ(-g)[-2g]) \cong \Hom_{\Sh_{\mu_1\mid \mu_2}^{\nu_I,\spl}}(p_1^*f_{1,*}\IQ, p_2^! f_{2,*}\IQ)\]
	\[\to \Hom_{\Sh_{\mu_1\mid \mu_2}^{\nu_I,\spl}}(p_1^*\Ll_{U_1},p_2^!\Ll_{U_2}),\]
	where the last map is the projection obtained by taking the suitable direct summands.
	
	In general, \(U_i\) can be obtained from \(V_i^\vee\) by taking direct sums, tensor products, duals, and exterior products (by assumption).
	These cases can be handled similarly, by either replacing the \(A_i\) by disjoint unions, products, or dual abelian schemes respectively, or by replacing \(\widetilde{f_i}_{,*}\IQ\) by a different direct summand.
	
	Finally, the compatibility with changing \(\nu_I\) holds since each of the arrows above is compatible with proper pushforward.
\end{proof}

\begin{rmk}
	Although the previous theorem had some restrictions on the representations of \(\IG_1\) and \(\IG_2\), any \(\IQ\)-linear representation \(U\) of \(\IG_1\) is a direct summand of some representation \(U_1\) as above \cite[Theorem 4.14]{Milne:Algebraic}.
	The cohomology of \(\Sh_{\mu_1,K}^{\spl}\) with coefficients in \(\Ll_{U}\) is then a direct summand of the cohomology with coefficients in \(\Ll_{U_1}\), so we do not lose much by this restriction.
\end{rmk}

Combining \thref{JL transfer for local systems} with (the proof of) \thref{JL for canonical models}, we get the following Jacquet--Langlands transfer for the cohomology of Shimura varieties with coefficients in non-trivial local systems.
Recall the representations \(V_{\mu_i}^{\spl}\) from \thref{notation representations}.

\begin{cor}\thlabel{coro JL}
	Let \(U_1,U_2\) be as above, and \(n\in \IZ\).
	Then there is a natural map
	\[\Hom_{\Coh^{\Gmot^I}(\widehat{G}^I p^{-1} \sigma)}\left(\widetilde{V_{\mu_1}^{\spl}}, \widetilde{V_{\mu_2}^{\spl}}(n)\right) \to \Hom_{\DM(\Spec \overline{\IF}_p)}\left(\gamma_{1,!}^{\spl}(\Ll_{U_1}[\langle2\rho,\mu_1\rangle]),\gamma_{2,!}^{\spl}(\Ll_{U_2}[\langle2\rho,\mu_2\rangle](n))\right),\]
	taking values in \(\Hh^p\)-equivariant morphisms.
	This map is moreover compatible with composition.
\end{cor}
\begin{proof}
	Note that the pullback of \(\IL^{\spl}(\widetilde{V_{\mu_i}^{\spl}}) \in \DM(\Sht_{\Gg}^{\spl})\) to \(\DM(\Sh_{\mu_i,K}^{\spl})\) agrees with \(\IQ[\langle2\rho,\mu_i\rangle]\).
	The same proof as \thref{JL for canonical models} yields a morphism
	\[\Hom_{\Coh^{\Gmot^I}\widehat{G}^I p^{-1} \sigma)}\left(\widetilde{V_{\mu_1}^{\spl}}, \widetilde{V_{\mu_2}^{\spl}}(n)\right) \to \Hom_{\DM(\Spec \overline{\IF}_p)}\left(\gamma_{1,!}^{\spl}(\IQ[\langle 2\rho,\mu_1\rangle]),\gamma_{2,!}^{\spl}(\IQ[\langle 2\rho,\mu_2\rangle](n))\right).\]
	On the other hand, \thref{JL transfer for local systems} induces by proper pushforward of correspondences a morphism
	\[\Hom_{\DM(\Spec \overline{\IF}_p)}\left(\gamma_{1,!}^{\spl}(\IQ),\gamma_{2,!}^{\spl}(\IQ(n))\right) \to \Hom_{\DM(\Spec \overline{\IF}_p)}\left(\gamma_{1,!}^{\spl}(\Ll_{U_1}),\gamma_{2,!}^{\spl}(\Ll_{U_2}(n))\right).\]
	Composing these two yields the desired map, and it satisfies the required properties since these hold for the maps in Theorems \ref{JL for canonical models} and \ref{JL transfer for local systems}.
\end{proof}

\section{On the Tate conjecture for splitting models of Shimura varieties}\label{Sec:Tate}
We are now ready to tackle (generic instances of) the Tate conjecture for splitting models of Shimura varieties.
Namely, we will consider cycles in the basic Newton strata, and show that under some genericity assumption (and when sufficient understanding of the cohomology of the Shimura variety is available), they generate all Tate classes.
Roughly speaking, these cycles arise from exotic correspondences in the case where one of the Shimura varieties is a Shimura set.
Throughout this section, we will work with \(\ell\)-adic cohomology rather than motives, and fix an auxiliary prime \(\ell\), as well as a square root of \(q\) in \(\overline{\IQ}_\ell\), which we use to fix a half-Tate twist.
We will also assume \(p\neq 2\).

\subsection{Splitting models beyond the PEL type}

In the previous section, we considered Shimura varieties of PEL type, since their moduli interpretations have allowed us to construct exotic Hecke correspondences.
Since the methods that will be used in this section are mostly group-theoretic, we can work with Shimura varieties of Hodge type instead (in fact, the techniques should also be applicable to Shimura varieties of abelian type, but certain necessary inputs from the literature have only been recorded in the Hodge type case).
We will use this subsection to define and study splitting models in greater generality then \thref{Defi:splitting model}.

Let \((\IG,\IX)\) be a Shimura datum of Hodge type with reflex field \(\IE\), Hodge cocharacter \(\mu\), and Shimura variety \(\Sh_K(\IG,\IX)/\IE\) for sufficiently small compact open subgroup \(K\subseteq \IG(\IA_f)\).
We further assume that \(G:=\IG_{\IQ_p}\) is essentially unramified, and assume that \(K=K^pK_p\) with \(K_p\subseteq G(\IQ_p)\) a very special subgroup.
Finally, we assume that \((p,\IG,\IX,K)\) is of \emph{global Hodge type} in the sense of \cite[Definition 4.5.1]{PappasRapoport:padic}, i.e., that the corresponding parahoric \(\Gg/\IZ_p\) is a connected Bruhat--Tits stabilizer group scheme.
Let \(E\) be the completion of \(\IE\) at a place \(\pp\) lying over \(p\), with residue field \(k=k_\pp\).
Then \cite[Theorem 4.5.2]{PappasRapoport:padic} gives a canonical (in the sense of \cite[Conjecture 4.2.2]{PappasRapoport:padic}) integral model \(\mathscr{S}_{K^pK_p}^{\can}(\IG,\IX)/\Oo_E\) of \(\Sh_K(\IG,\IX)\).
We also put the following hypothesis on the Shimura datum:
\begin{ass}\thlabel{assumptions Shimura variety}
	\begin{enumerate}
		\item \(\mathscr{S}_{K^pK_p}^{\can}(\IG,\IX)\) satisfies \cite[Conjecture 4.9.2]{PappasRapoport:padic}, i.e., there exists a suitable smooth morphism \(\mathscr{S}_{K^pK_p}^{\can}(\IG,\IX) \to \Gg_{\Oo_E} \backslash \mathscr{M}_{\Gg,\preccurlyeq\mu}^{\can}\) of relative dimension \(\dim \Gg\), where \(\mathscr{M}_{\Gg,\preccurlyeq \mu}^{\can}\) is the local model from \cite{ScholzeWeinstein:Berkeley,AGLR:Local,GleasonLourenco:Tubular} (in particular, the perfection of its special fiber is identified with \(\Gr_{\Gg,\leq \mu_I}^{\can}\)).
		\item Consider the crystalline period map \(\loc_p\colon \Sh_{\mu,K}^{\can} \to \Sht_{\Gg,\leq \mu_I}^{\can}\), obtained by the reduction (in the sense of \cite[Definition 3.12]{Gleason:Specialization}) of the similar map on the level of v-sheaves arising from the definition of the canonical models as in \cite[Conjecture 4.2.2 (b)]{PappasRapoport:padic}.
		Then for \(m\geq n+1\), the composition of this map with the truncation \(\Sht_{\Gg{,}\leq \mu_I}^{\can} \to \Sht_{\Gg{,}\leq \mu_I}^{\can,(m,n)}\) is perfectly smooth.
		\item If \((\IG,\IX)\) is not of PEL type, we assume that the triple \((\IG,\IX,K^pK_p)\) is strongly admissible in the sense of \cite[Definition 2.1.6]{KisinZhou:Strongly}, i.e., we additionally have that the centralizer of a maximal \(\breve{\IQ}_p\)-split torus in \(G\) is an R-smooth torus in the sense of \cite[Definition 2.4.3]{KisinZhou:Independence}.
	\end{enumerate}
\end{ass}

\begin{rmk}
	\begin{enumerate}
		\item \cite[Conjecture 4.9.2]{PappasRapoport:padic} is known in many cases, cf.~\cite[Theorem 4.3.6]{DvHKZ:Conjecture} and \cite[Theorem 7.2.20]{KisinPappasZhou:Integral}.
		It also follows from \eqref{diagram of models} in the PEL cases previously considered in this paper.
		\item The smoothness of the (truncated) crystalline period map is known under certain assumptions by \cite[Corollary 2.57]{Hoff:Parahoric}.
		The smoothness is general seems to be known to experts, cf.~\cite[Theorem 6.3]{Zhu:Tame}.
		Since this has not yet been documented in the literature, we have given a proof in the PEL type cases relevant for this paper in \thref{smoothness of shimura to shtuka}.
		\item The fact that the crystalline period map, arising as the reduction from the construction of \cite{PappasRapoport:padic}, agrees with the constructions from \cite{XiaoZhu:Cycles,ShenYuZhang:EKOR,Hoff:Parahoric} and Section \ref{Sec: bounded local shtukas}, follows from the discussion in \cite[§5.3]{DvHKZ:Igusa}.
		\item The assumption that \((\IG,\IX,K^pK_p)\) is strongly admissible will arise in \thref{RZuniformization}.
		The condition that a torus is R-smooth is not a strong restriction, and we refer to \cite[Proposition 2.4.6]{KisinZhou:Independence} for examples.
	\end{enumerate}
\end{rmk}

\begin{rmk}\thlabel{recollection on PR model}
	Let us recall the construction of \(\mathscr{S}_{K^pK_p}^{\can}(\IG,\IX)\) as in \cite[§4.5]{PappasRapoport:padic} (which works without the assumption that \(G\) is essentially unramified).
	By definition, there exists an embedding of Shimura data \((\IG,\IX) \into (\GSp(V,\psi),S^{\pm})\), where \(V\) is a \(\IQ\)-vector space with a perfect alternating pairing \(\psi\), and \(S^{\pm}\) is the Siegel double space.
	Then there exists a parahoric group scheme \(\Hh\) of \(\GSp(V_{\IQ_p},\psi_{\IQ_p})\) such that the closed embedding \(G\to \GSp(V_{\IQ_p},\psi_{\IQ_p})\) extends to a dilated immersion \(\Gg\to \Hh\).
	Let \(K^{\flat}_p=\Hh(\IZ_p)\), and let \(K^{\flat p}\subseteq \GSp(\IA_f^p)\) be such that there is a closed immersion \(\Sh_{K^pK_p}(\IG,\IX)_{E} \to \Sh_{K^{\flat p}K^{\flat}_p}(\GSp(V,\Psi),S^{\pm}) \otimes_{\IQ} E\); this exists by \cite[Lemma 2.1.2]{Kisin:Integral}.
	
	Now, \(\Hh\) arises as the stabilizer of some periodic self-dual lattice chain \(\Lambda_\bullet\) in \(V_{\IQ_p}\).
	This yields a moduli interpretation of \(\Sh_{K^{\flat p}K^{\flat}_p}(\GSp(V,\Psi),S^{\pm})\), parametrizing chains of p-isogenies between polarized abelian varieties.
	This moduli interpretation moreover gives rise to a (canonical) integral model \(\mathscr{S}_{K^{\flat p}K^{\flat}_p}^{\can}(\GSp(V,\Psi),S^{\pm})\), as in §\ref{Sec:Models of Shimura varieties}.
	Then \(\mathscr{S}_{K^pK_p}^{\can}(\IG,\IX)\) is defined as the normalization of the (reduced) schematic closure of \(\Sh_{K^pK_p}(\IG,\IX)\) in \(\mathscr{S}_{K^{\flat p}K^{\flat}_p}^{\can}(\GSp(V,\Psi),S^{\pm}) \otimes_{\IZ_p} \Oo_E\).
	As usual, we denote its special fiber by \(\Sh_{\mu,K}^{\can}\).
	
	In particular, the choice of such a Hodge embedding as above attaches to each point \(x\in \Sh_{\mu,K}\) a polarized abelian variety \(A_x\).
	In the PEL case, they come equipped with extra structure, in the form of endomorphism and level structures.
	In the Hodge type case, this is generally not available, but instead one can attach to each \(x\in \Sh_{\mu,K}\) various ``tensors" in the (crystalline and \(\ell\)-adic) cohomology of \(A_x\).
	This will only play a minor role in what follows, and we refer to \cite[§4.6]{PappasRapoport:padic} for details (cf.~also \cite{Kisin:Integral,KisinPappas:Integral}).
\end{rmk}

Going back to the case where \(G\) is essentially unramified, we can define a splitting local model \(\mathscr{M}_{\Gg,\preccurlyeq\mu}^{\spl}\).
This is a smooth \(\Oo_{E^{\Gal}}\)-scheme \(\mathscr{M}_{\Gg,\preccurlyeq\mu}^{\spl}\) equipped with an \(\Gg_{\Oo_{E^{\Gal}}}\)-action, an equivariant map \(\mathscr{M}_{\Gg,\preccurlyeq\mu}^{\spl} \to \mathscr{M}_{\Gg,\preccurlyeq\mu}^{\can} \times_{\Spec \Oo_E} \Spec \Oo_{E^{\Gal}}\), which induces an isomorphism over the generic fiber of \(\Spec \Oo_{E^{\Gal}}\), and whose special fiber identifies with the map \(\Gr_{\Gg,\leq \mu}^{\spl} \to \Gr_{\Gg,\leq \mu_I}^{\can}\) from \thref{remarks on splitting grassmannians}.
Indeed, since local models are invariant under central isogenies and compatible with products, we can reduce to the case of restrictions of scalars.
In that case, the desired local model has been constructed in \cite[Definition 5.4.2]{Levin:Local}.
(More precisely, one should take the closed subscheme of the splitting Grassmannian defined in loc.~cit., given by the schematic closure of the minuscule Schubert cell in its generic fiber.
The fact that its special fiber has the correct description then follows from \cite[Proposition 5.4.4]{Levin:Local} and \cite[(5.4.4.2)]{Levin:Local}.)

\begin{dfn}
	The splitting integral model of \(\Sh_K(\IG,\IX)\) is defined via the following cartesian diagram:
	\[\begin{tikzcd}
		\mathscr{S}_{K^pK_p}^{\spl}(\IG,\IX) \arrow[r] \arrow[d] & \Gg_{\Oo_{E^{\Gal}}} \backslash \mathscr{M}_{\Gg,\preccurlyeq\mu}^{\spl} \arrow[d]\\
		\mathscr{S}_{K^pK_p}^{\can}(\IG,\IX) \arrow[r] & \Gg_{\Oo_{E}} \backslash \mathscr{M}_{\Gg,\preccurlyeq\mu}^{\can}.
	\end{tikzcd} \]
\end{dfn}

This is clearly representable by an \(\Oo_{E^{\Gal}}\)-scheme. 
Moreover, since \(\mathscr{M}_{\Gg,\preccurlyeq\mu}^{\spl}\) is smooth, the map \(\mathscr{M}_{\Gg,\preccurlyeq\mu}^{\spl} \to \mathscr{M}_{\Gg,\preccurlyeq\mu}^{\can}\) is proper, and the horizontal arrows are smooth, it follows that \(\mathscr{S}_{K^pK_p}^{\spl}(\IG,\IX) \to \mathscr{S}_{K^pK_p}^{\can}(\IG,\IX)\) is a resolution of singularities.
Finally, we denote by \(\Sh_{\mu,K}^{\spl}\) and \(\Sh_{\mu,K}^{\can}\) the perfections of the geometric special fibers of \(\mathscr{S}_{K^pK_p}^{\spl}(\IG,\IX)\) and \(\mathscr{S}_{K^pK_p}^{\can}(\IG,\IX)\).

Recall that \(\Sh_{\mu,K}^{\can}\) admits a \emph{Newton stratification}, where the strata are defined as the fibers of a natural map \(\Sh_{\mu,K}^{\can}(\overline{k}) \to B(G)\) \cite[Remark 4.2.1]{PappasRapoport:padic}.
This is well known to be indexed by \(B(G,\mu^*)\).

\begin{dfn}
	The Newton stratification of \(\Sh_{\mu,K}^{\spl}\) is defined by pulling back the Newton stratification of \(\Sh_{\mu,K}^{\can}\).
	In particular, it is indexed by \(B(G,\mu^*)\), and the non-empty Newton strata in \(\Sh_{\mu,K}^{\spl}\) correspond exactly to the non-empty Newton strata in \(\Sh_{\mu,K}^{\can}\) (by surjectivity of \(\Sh_{\mu,K}^{\spl} \to \Sh_{\mu,K}^{\can}\)).
	
	The \emph{basic} Newton strata, corresponding to the unique basic element in \(B(G,\mu^*)\), are denoted by \(\Sh_{\mu,K,\bas}^{\spl} \to \Sh_{\mu,K,\bas}^{\can}\).
\end{dfn}

By \cite[Theorem 1.3.14]{KisinMadapusiShin:HondaTate}, the basic Newton stratum \(\Sh_{\mu,K,\bas}^{\can}\) is non-empty, hence the same holds for \(\Sh_{\mu,K,\bas}^{\spl}\).
Let us fix a point \(x\in \Sh_{\mu,K,\bas}^{\spl}\), and denote its image in \(\Sh_{\mu,K,\bas}^{\can}\) the same way.
To understand the geometry of \(\Sh_{\mu,K,\bas}^{\spl}\) (both local and global), it suffices to understand the geometry of \(\Sh_{\mu,K,\bas}^{\can}\), as well as the map \(\Gr_{\Gg,\leq \mu}^{\spl} \to \Gr_{\Gg,\leq \mu_I}^{\can}\) (compare the proof of \thref{Irreducible components of splitting adlv}).
To make this precise, we prove a uniformization result à la Rapoport--Zink \cite[Theorem 6.23]{RapoportZink:Period}.

\begin{thm}\thlabel{RZuniformization}
	Let \(b\in B(G,\mu^*)\) be the unique basic element.
	There exists a commutative diagram
	\[\begin{tikzcd}
		\II(\IQ) \backslash X_{\leq \mu^*}^{\spl}(b) \times \IG(\IA_f^p) /K^p \arrow[d] \arrow[r, "\cong"] & \Sh_{\mu,K,\bas}^{\spl} \arrow[d]\\
		\II(\IQ) \backslash X_{\leq \mu_I^*}^{\can}(b) \times \IG(\IA_f^p)/K^p \arrow[r, "\cong"'] & \Sh_{\mu,K,\bas}^{\can}
	\end{tikzcd}\]
	of perfect schemes, which is equivariant for prime-to-\(p\) Hecke operators.
	Here, \(\II/\IQ\) is an inner form of \(\IG\) for which \(\II \otimes_{\IQ} \IA_f^p \cong \IG \otimes_{\IQ} \IA_f^p\) and \(\II \otimes_{\IQ} \IQ_p \cong J_b\), where \(J_b\) is the \(\sigma\)-twisted centralizer of \(b\in G(\breve{\IQ}_p)\), and \(\II(\IR)\) is compact modulo center.
\end{thm}
\begin{rmk}
	The assumption that \((\IG,\IX,K^pK_p)\) is strongly admissible arises when applying \cite[Theorem 2.2.7]{KisinZhou:Strongly}, which is needed to show Rapoport--Zink uniformization for the canonical model.
	It can be omitted in the PEL type case, by instead appealing to the Rapoport--Zink uniformization from \cite[Theorem 6.23]{RapoportZink:Period}, which applies to the naive model.
	A similar remark also holds when applying \thref{RZuniformization} below.
\end{rmk}
\begin{proof}
	We first handle the case of canonical models; this has already been studied extensively in the literature \cite{XiaoZhu:Cycles,HeZhouZhu:Stabilizers,PappasRapoport:padic}.
	In this case, Pappas--Rapoport \cite{PappasRapoport:padic} have shown uniformization of the formal completion of \(\mathscr{S}_{K^pK_p}^{\can}(\IG,\IX)\) along an isogeny class by integral local Shimura varieties.
	Recall that the reduced special fibers of such integral local Shimura varieties agree with affine Deligne--Lusztig varieties \cite[Proposition 2.61]{Gleason:Geometric}, and that an isogeny class in \(\Sh_{\mu,K}^{\can}\) consists of those points whose associated abelian variety are quasi-isogenous, respecting the polarization (up to scalar) and the étale and crystalline tensors (as in \thref{recollection on PR model}).
	Thus, the theorem will follow from \cite[Theorem 4.10.6]{PappasRapoport:padic} and \cite[Corollary 6.3]{GleasonLimXu:Connected}, as soon as we know that \(\Sh_{\mu,K,\bas}^{\can} \subseteq \Sh_{\mu,K}^{\can}\) consists of a single isogeny class.
	For the integral models constructed in \cite{KisinPappas:Integral}, this was shown in \cite[Proposition 5.2.12]{HeZhouZhu:Stabilizers}, the key ingredient being that any isogeny class in the basic Newton stratum contains a point that lifts to a special point in the generic fiber.
	When \((\IG,\IX,K^pK_p)\) is strongly admissible, this last fact has recently been generalized in \cite[Theorem 2.2.7]{KisinZhou:Strongly}.
	Using this fact, the proof of \cite[Proposition 5.2.12]{HeZhouZhu:Stabilizers} goes through verbatim to show that \(\Sh_{\mu,K,\bas}^{\can}\subseteq \Sh_{\mu,K}^{\can}\) is a single isogeny class.
	This yields the canonical case.
	
	Now, the isomorphism \(\II(\IQ) \backslash X_{\leq \mu_I^*}^{\can}(b) \times \IG(\IA_f^p)/K^p \cong \Sh_{\mu,K,\bas}^{\can}\) is induced from a map \(X_{\leq \mu^*_I}^{\can}(b) \to \Sh_{\mu,K,\bas}^{\can}\), by letting \(\IG(\IA_f^p)\) act on \(\Sh_{\mu,K,\bas}^{\can}\) \cite[(4.10.6) and (4.10.7)]{PappasRapoport:padic}.
	Moreover, the composition
	\[X_{\leq \mu_I^*}^{\can}(b) \to \Sh_{\mu,K,\bas}^{\can} \to L^+\Gg \backslash \Gr_{\Gg,\leq \mu_I}^{\can}\] 
	agrees with the map \(g\mapsto (g^{-1}b\sigma(g))^{-1}\).
	In particular, we get the following diagram with cartesian squares:
	\[\begin{tikzcd}
		X_{\leq \mu^*}^{\spl}(b) \arrow[d] \arrow[r]& \Sh_{\mu,K,\bas}^{\spl}\arrow[d] \arrow[r] & L^+\Gg \backslash \Gr_{\Gg,\leq\mu}^{\spl} \arrow[d]\\
		X_{\leq \mu_I^*}^{\can}(b) \arrow[r] & \Sh_{\mu,K,\bas}^{\can} \arrow[r] & L^+\Gg \backslash \Gr_{\Gg,\leq \mu_I}^{\can}.
	\end{tikzcd}\]
	Letting \(\IG(\IA_f^p)\) act on the Shimura varieties then yields the following diagram, which still has cartesian squares:
	\[\begin{tikzcd}
		\II(\IQ) \backslash X_{\leq \mu^*}^{\spl}(b) \times \IG(\IA_f^p)/K^p \arrow[r] \arrow[d]& \Sh_{\mu,K,\bas}^{\spl} \arrow[r] \arrow[d] & L^+\Gg \backslash \Gr_{\Gg,\leq \mu}^{\spl} \arrow[d]\\
		\II(\IQ)\backslash X_{\leq \mu_I^*}^{\can} \times \IG(\IA_f^p)/K^p \arrow[r] & \Sh_{\mu,K,\bas}^{\can} \arrow[r] & L^+\Gg \backslash \Gr_{\Gg,\leq \mu_I}^{\can}.
	\end{tikzcd}\]
	But then the upper left horizontal arrow is an isomorphism, as the pullback of an isomorphism, yielding the desired diagram.
	The equivariance for prime-to-\(p\) Hecke operators follows since the isomorphisms are induced by the \(\IG(\IA_f^p)\)-action on the Shimura varieties.
\end{proof}

Following \cite[Proposition 7.3.5]{XiaoZhu:Cycles}, we can interpret this Rapoport--Zink uniformization as an exotic Hecke correspondence between the Shimura variety for \((\IG,\IX)\) and a 0-dimensional (weak) Shimura variety.
In the following, we will assume that the center \(Z_{\IG}\) of \(\IG\) is connected, and that the basic element \(b\in B(G,\mu^*)\) is very special.
In that case, \(b=p^{\tau^*}\) for some \(\tau\in X_*(Z_G)\) by \thref{When is basic element unramified}; this satisfies \(\mu_{\mid Z(\widehat{G})^{\Gamma_{\IQ_p}}} = \tau_{\mid Z(\widehat{G})^{\Gamma_{\IQ_p}}}\).
Moreover, \thref{When is basic element unramified} \eqref{item character} and \cite[Corollary 2.1.6]{XiaoZhu:Cycles} imply that there is an inner form \(\IG'\) of \(\IG\), which is trivial at the finite places, and for which \(\IG'_{\IR}\) is compact modulo center.
Consider the pair \((\IG',\IX')=(\IG',\{h\})\), with \(h\colon \Res_{\mathbb{C}/\IR} \IG_{m,\mathbb{C}} \to \IG_{\IR}\) the descent to \(\IR\) of \(\IG_{m,\mathbb{C}} \times \IG_{m,\mathbb{C}} \to \IG_{\mathbb{C}}\colon (x,y)\mapsto \tau(x)\cdot \overline{\tau(y)}\), where \(\overline{(-)}\) denotes complex conjugation. 
This is a weak Shimura datum in the sense of \cite[§2.2]{TianXiao:GorenOort}, whose Hodge cocharacter is \(\tau\).
The corresponding complex weak Shimura variety \(\Sh_K(\IG',\IX')/\mathbb{C}\) (at level \(K\subseteq \IG(\IA_f^p) \cong \IG'(\IA_f^p)\)) is just a finite discrete union of points, indexed by \(\IG'(\IQ) \backslash \IG'(\IA_f)/K\).
Let \(\IE'\) denote its reflex field and \(\widetilde{\IE}:=\IE\IE'\) the composite with the reflex field of \((\IG,\IX)\), let \(\widetilde{E}\) be a completion at a place lying over \(\pp\) (which was the place of \(\IE\) whose completion is \(E\)), and let \(\widetilde{k}\) be the residue field of \(\widetilde{E}\).
Then by \cite[§2.8]{TianXiao:GorenOort}, \((\IG',\IX')\) gives rise to a canonical weak Shimura variety \(\Sh_K(\IG',\IX')/\IE'\) (which we will usually consider over \(\widetilde{\IE}\)), as well as a canonical integral model \(\mathscr{S}_K^{\can}(\IG',\IX')/\Oo_{\widetilde{E}}\).
We then denote by \(\Sh_{\tau,K}^{\can}\) the perfection of the geometric special fiber of \(\mathscr{S}_K^{\can}(\IG',\IX')\).
Note that by \eqref{shtuka for central}, \(\Sht_{\Gg,\tau_I}^{\can}\) is just the classifying stack of a profinite group, so that there exists a natural pro-(perfectly smooth) map \(\Sh_{\tau,K}^{\can} \to \Sht_{\Gg,\tau_I}^{\can}\).
Moreover, since \(\Sht_{\Gg,\tau_I}^{\can} \cong \Sht_{\Gg,\tau}^{\spl}\) in this case, it makes sense to define \(\Sh_{\tau,K}^{\spl}:=\Sh_{\tau,K}^{\can}\), so that there is also a pro-(perfectly smooth) map \(\Sh_{\tau,K}^{\spl} \to \Sht_{\Gg,\tau}^{\spl}\).

\begin{cor}\thlabel{exotic corr via RZ unif}
	Assume that \(Z_{\IG}\) is connected, and that the basic element \(b\in B(G,\mu^*)\) is very special.
	Then there is a commutative diagram
	\[\begin{tikzcd}
		\Sh_{\tau{,}K}^{\spl} \arrow[dd] \arrow[rd,equal]&& \Sh_{\tau\mid \mu}^{\spl} \arrow[ll] \arrow[rr] \arrow[dd] \arrow[rd] \arrow[ld, phantom, "X"] && \Sh_{\mu{,}K}^{\spl} \arrow[dd, "\loc_p^{\spl}"' {yshift=8pt}] \arrow[rd] &\\
		&\Sh_{\tau{,}K}^{\can} \arrow[dd]&& \Sh_{\tau\mid \mu}^{\can} \arrow[ll] \arrow[rr] \arrow[dd] && \Sh_{\mu{,}K}^{\can} \arrow[dd, "\loc_p^{\can}"]\\
		\Sht_{\Gg{,}\tau}^{\spl} \arrow[rd, equal] && \Sht_{\tau\mid \mu}^{\spl} \arrow[ll] \arrow[rr] \arrow[rd] \arrow[ld, phantom, "X"] && \Sht_{\Gg{,}\leq \mu}^{\spl} \arrow[rd]\\
		& \Sht_{\Gg{,}\tau_I}^{\can} && \Sht_{\tau_I\mid \mu_I}^{\can} \arrow[ll] \arrow[rr] && \Sht_{\Gg{,}\leq \mu_I}^{\can},
	\end{tikzcd}\]
	where all squares are cartesian, except those marked ``\(X\)''.
	Moreover, \(\Sh_{\tau\mid \mu}^{\spl}\) and \(\Sh_{\tau\mid \mu}^{\can}\) are ind-(perfect schemes), which admit a \(\widetilde{k}\)-structure such that the upper half of the diagram is defined over \(\widetilde{k}\).
\end{cor}
\begin{proof}
	Once the diagram is constructed with the desired cartesianness properties, \(\Sh_{\tau\mid \mu}^{\spl}\) and \(\Sh_{\tau\mid \mu}^{\can}\) will be ind-(perfect schemes) by \thref{representability of bounded correspondence spaces}.
	Moreover, since \(b\) is assumed very special, both \(\IG'\) and the group \(\II\) from \thref{RZuniformization} are inner forms of \(\IG\) which are trivial at the finite places.
	Applying \cite[Corollary 2.1.6]{XiaoZhu:Cycles}, we can fix an isomorphism \(\IG'\cong \II\), compatible with the inner twists to \(\IG\).
	
	We first discuss the existence of the diagram over \(\Spec \overline{k}\).
	Let
	\[\Sh_{\tau\mid\mu}^{\spl} = \IG'(\IQ) \backslash G(\IQ_p) \overset{\Gg(\IZ_p)}{\times} X_{\leq \mu^*}^{\spl}(b) \times \IG(\IA_f^p)/K^p\]
	and 
	\[\Sh_{\tau\mid\mu}^{\can} = \IG'(\IQ) \backslash G(\IQ_p) \overset{\Gg(\IZ_p)}{\times} X_{\leq \mu_I^*}^{\can}(b) \times \IG(\IA_f^p)/K^p.\]
	Moreover, the fixed point \(x\in \Sh_{\mu,K,\bas}^{\spl}\) (and its image in \(\Sh_{\mu,K,\bas}^{\can}\)) satisfy
	\[\Sht_{\tau\mid\mu}^{\spl} \times_{\Sht_{\Gg,\leq \mu}^{\spl}} x \cong X_{\leq \mu^*}^{\spl}(b) \cong G(\IQ_p)/\Gg(\IZ_p)\]
	and 
	\[\Sht_{\tau_I\mid\mu_I}^{\can} \times_{\Sht_{\Gg,\leq \mu_I}^{\can}} x \cong X_{\leq \mu_I^*}^{\can}(b) \cong G(\IQ_p)/\Gg(\IZ_p)\]
	by \eqref{adlv as fibers can} and \eqref{adlv as fibers spl}, since \(b=p^{\tau^*}\) with \(\tau\) central.
	Thus, the Rapoport--Zink uniformization from \thref{RZuniformization} implies that
	\[\Sht_{\tau\mid \mu}^{\spl} \times_{\Sht_{\Gg,\leq \mu}^{\spl}} \Sh_{\mu,K}^{\spl} \cong \Sht_{\tau\mid \mu}^{\spl} \times_{\Sht_{\Gg,\leq \mu}^{\spl}} \Sh_{\mu,K,\bas}^{\spl}\cong \Sht_{\tau\mid \mu}^{\spl} \times_{\Sht_{\Gg,\leq \mu}^{\spl}} (\IG'(\IQ) \backslash X_{\leq \mu^*}^{\spl}(b) \times \IG(\IA_f^p)/K^p)\]
	\[\cong \IG'(\IQ) \backslash G(\IQ_p) \overset{\Gg(\IZ_p)}{\times} X_{\leq \mu^*}^{\spl}(b) \times \IG(\IA_f^p)/K^p = \Sh_{\tau\mid \mu}^{\spl},\]
	and similarly for the canonical models.
	This shows that the desired diagram exists, and the squares in the right half are cartesian.
	The two remaining marked squares on the left are cartesian by \eqref{adlv as fibers can} and \eqref{adlv as fibers spl}.
	
	We are left to show the upper half is defined over \(\Spec \widetilde{k}\).
	It suffices to do this for canonical models; the case of splitting models will follow by pullback.
	In that case, we can follow \cite[Proposition 7.3.5]{XiaoZhu:Cycles}; we sketch the argument for convenience of the reader.
	
	The fixed point \(x\) yields a tuple \((A_x,\eta_x,s_{\et,x},s_{0,x})\), where \(A_x\) is an abelian variety, \(\eta_x\) a level structure away from \(p\), and \(s_{\et,x}\) and \(s_{0,x}\) are étale and crystalline tensors for \(A_x\) (cf.~\thref{recollection on PR model}).
	It also corresponds to a (canonical) local shtuka \(x^*\Ee\), determined by the map \(\loc_p^{\can}\).
	On the other hand, a point \(y=(g_p,y_0,g^p)\in \Sh_{\tau\mid \mu}^{\can}\) yields a modification of this local shtuka via
	\[\Ee_y \xrightarrow{\beta_{y_0}} x^*\Ee \xrightarrow{g^p} x^*\Ee,\]
	which in turns induces a modification of \(A_x\), i.e., a \(p\)-power quasi-isogeny \(A_y\xrightarrow{\beta_y} A_x\).
	Via this modification, the level structure \(\eta_xg^p\) on \(A_x\) also induces a level structure \(\eta_y\) on \(A_y\), and the image of \(y\in \Sh_{\tau\mid \mu}^{\can}\) in \(\Sh_{\mu,K}^{\can}\) corresponds to the tuple \((A_y,\eta_y,\beta_y^{-1}(s_{\et,x}),\beta_y^{-1}(s_{0,x}))\).
	
	Consider the Frobenius action on \(\Sh_{\tau\mid \mu}^{\can}\) given by
	\[(g_p,y_0,g^p) \mapsto (p^{\IN_v(\tau^*)g_p,\sigma_v(y_0)},g^p) \quad \text{ for } \quad \IN_v(\tau^*) := \sum_i^{[k_v\colon \IF_p]} \sigma^i(\tau^*),\]
	where \(\sigma_v\) acts on \(X_{\leq \mu_I^*}^{\can}(b)\) as in \cite[(5.2)]{Rapoport:Guide}
	This makes the map \(\Sh_{\tau\mid \mu}^{\can} \to \Sh_{\tau,K}^{\can}\) \(\sigma_v\)-equivariant, where the action on the target is given by the reciprocity map as in \cite[§2.8]{TianXiao:GorenOort}.
	The same argument as in \cite[Proposition 7.3.5]{XiaoZhu:Cycles} then shows that the map \(\Sh_{\tau\mid \mu}^{\can} \to \Sh_{\mu,K}^{\can}\) is also \(\sigma_v\)-equivariant, concluding the proof.
\end{proof}

As another consequence of the Rapoport--Zink uniformization, we can describe irreducible components of the basic Newton stratum.
Recall that we have defined \(\IM\IV_{\mu^*}^{\spl,\Tatep}\) and \(\IM\IV_{\mu_I^*}^{\can,\Tatep}\) in \eqref{Tate MV cycles}.

\begin{cor}\thlabel{irreducible components of basic stratum}
	Assume that \(Z_{\IG}\) is connected, and that the basic element in \(B(G,\mu^*)\) is very special.
	Then \(\Sh_{\mu,K,\bas}^{\spl}\) and \(\Sh_{\mu,K,\bas}^{\can}\) are equidimensional of dimension \(\langle\rho,\mu\rangle\).
	Moreover, there is a commutative diagram
	\[\begin{tikzcd}
		\Irr(\Sh_{\mu,K,\bas}^{\spl}) \arrow[d] \arrow[r, "\cong"] & \IG'(\IQ) \backslash \IG(\IA_f)/K \times \IM\IV_{\mu^*}^{\spl,\Tatep} \arrow[d]\\
		\{*\} \sqcup \Irr(\Sh_{\mu,K,\bas}^{\can}) \arrow[r, "\cong"'] & \{*\} \sqcup \left(\IG'(\IQ)\backslash \IG(\IA_f)/K \times \IM\IV_{\mu_I^*}^{\can,\Tatep}\right).
	\end{tikzcd}\]
\end{cor}
\begin{proof}
	Rapoport--Zink uniformization, \thref{RZuniformization}, and \thref{irreducible comp of ADLV in very special case} imply that \(\Sh_{\mu,K,\bas}^{\spl}\) and \(\Sh_{\mu,K,\bas}^{\can}\) are equidimensional of dimension \(\langle \rho,\mu\rangle\), and describe the irreducible components (since we have already seen in the proof of \thref{exotic corr via RZ unif} that \(\II \cong \IG'\) under our assumptions). 
	To see that this description indeed agrees with the desired one, we can additionally use \thref{description of Lambda}.
	
	The image of an irreducible component of \(\Sh_{\mu,K,\bas}^{\spl}\) in \(\Sh_{\mu,K,\bas}^{\can}\) is automatically irreducible and closed, but can be of dimension less than \(\langle\rho,\mu\rangle\).
	The left vertical map is then given by sending such an irreducible component to its image in \(\Sh_{\mu,K,\bas}^{\can}\) if this is of dimension \(\langle\rho,\mu\rangle\), and to \(\{*\}\) otherwise.
	Similarly, the right vertical map sends irreducible components of \(\Gr_{\Gg,\leq \mu^*}^{\spl}\cap \Ss_\lambda^{\spl}\) to their image in \(\Gr_{\Gg,\leq \mu_I}^{\can} \cap \Ss_{\lambda_I}^{\can}\) when they have the correct dimension, and to \(\{*\}\) otherwise.
	The commutativity of the diagram then follows from \thref{irreducible comp of ADLV in very special case}.
\end{proof}

\subsection{Cycles on splitting models}\label{subsec:Tate}

Recall that we have fixed a prime \(\ell\neq p\).
We will also denote by \(I\) the inertia group of \(\IQ_p\).

Throughout this section, we fix a Hodge type Shimura datum \((\IG,\IX)\) as in the previous section; in particular \(G:=\IG_{\IQ_p}\) is essentially unramified, and we have the canonical and splitting special fibers \(\Sh_{\mu,K}^{\can}\) and \(\Sh_{\mu,K}^{\spl}\) over \(\Spec \overline{k}\), for some sufficiently small \(K=K^pK_p\subseteq \IG(\IA_f)\) with \(K_p\subseteq G(\IQ_p)\) very special.
Moreover, \(Z_G\) is assumed to be connected.
Let \(\Hh_K := C_c^{\infty}(K\backslash \IG(\IA_f)/K, \overline{\IQ}_\ell)\) be the global Hecke algebra at level \(K\), with coefficients in \(\overline{\IQ}_\ell\).
Similarly, we denote the Hecke algebras away from \(p\) (resp.~at \(p\)) by \(\Hh_{K^p}:=C_c^{\infty}(K^p\backslash \IG(\IA_f^p)/K^p, \overline{\IQ}_\ell)\) (resp.~by \(\Hh_{\Gg}=\Hh_{K_p}:=C_c(K_p\backslash G(\IQ_p)/K_p,\overline{\IQ}_\ell)\)).
Let \(\phi_v\) denote the geometric Frobenius of the residue field \(k=k_v\) of \(E\).
For an irreducible \(\Hh_K\)-module \(\pi_f\) and \(i\in \IZ_{\geq 0}\), we write
\[W^{2i}(\pi_f,\overline{\IQ}_\ell(i)) := \Hom_{\Hh_K}(\pi_f,\mathrm{H}_c^{2i}(\Sh_{\mu,K}^{\spl},\overline{\IQ}_\ell(i))),\]
and let
\[T^i(\pi_f,\overline{\IQ}_\ell) := \bigcup_{j\geq 1} W^{2i}(\pi_f,\overline{\IQ}_\ell(i))^{\phi_v^j},\]
be the space of Tate classes in the \(\pi_f\)-isotypic part of the cohomology of \(\Sh_{\mu,K}^{\spl}\).
Recall that, at least when \(\Sh_{\mu,K}^{\spl}\) is proper, the Tate conjecture \cite{Tate:Algebraic} predicts that these arise from algebraic cycles of codimension \(i\).
We will show this under certain genericity conditions for \(\pi_f\), which we define below.

Since we are now working with \(\overline{\IQ}_\ell\)-coefficients, we will use the fixed square root of \(p\) to consider the stack of spherical Langlands parameters over \(\overline{\IQ}_\ell\) as \(\widehat{G}^Ip^{-1}\sigma_p/\widehat{G}^I \cong \widehat{G}^I\sigma_p/\widehat{G}^I\).
Similarly, we now redefine \(\bfJ:=\Gamma(\widehat{G}^I\sigma_p/\widehat{G}^I,\Oo)\), which can be identified with the spherical Hecke algebra \(\Hh_{\Gg}\) by \cite[Corollary 10.12]{vdH:RamifiedSatake}.
More generally, for a \(\overline{\IQ}_\ell\)-linear \(\widehat{G}^I\)-representation \(V\), we redefine
\[\bfJ(V):=\Gamma(\widehat{G}^I\sigma_p/\widehat{G}^I, \widetilde{V}) \cong (\Oo_{\widehat{G}^I} \otimes V)^{c_{\sigma_p}(\widehat{G}^I)}.\]
Then we still have \(\Hom_{\Coh^{\widehat{G}^I}(\widehat{G}^I\sigma_p)}(\widetilde{V},\widetilde{W}) \cong \bfJ(V^*\otimes W)\).
Note also that by using the fixed \(p^{\frac{1}{2}}\), the twisted Chevalley restriction theorem \cite[Proposition 10.8]{vdH:RamifiedSatake} induces an isomorphism \(\overline{\IQ}_\ell[\widehat{G}^I\sigma_p]^{\widehat{G}^I} \cong \overline{\IQ}_\ell[\widehat{T}^I_{\sigma_p}]^{W_0}\), where \(W_0=W^{\Gamma_{\IQ_p}}\) is the relative Weyl group of \(G\),
and \(\widehat{T}^I_{\sigma_p} := (\widehat{T}^I)_{\sigma_p} := \widehat{T}^I/(\sigma_p-1)\widehat{T}^I\).
Note that the neutral component of \(\widehat{T}_{\sigma_p}^I\) can be canonically identified with the dual group of the maximal split torus \(A\subseteq G\).
Let \(\Phi(G,A)^\vee\) be the relative (with respect to \(A\subseteq G\)) coroot system, which we will view as a subset \(\Phi(G,A)^\vee \subseteq X^*(\widehat{T})_I^{\sigma_p}\).

\begin{rmk}
	In general, there are multiple notions of root data for fixed point of reductive groups, cf.~\cite{Haines:Dualities,ALRR:Fixed}.
	However, by our assumption that \(G\) is essentially unramified, the positive roots of \(\widehat{G}\) all lie in an equivalence class of type (i') in the notation of \cite[§4.1]{ALRR:Fixed}.
	This shows that the different notions of root data from \cite[Proposition 6.1]{ALRR:Fixed} agree.
	More precisely, the roots \(\RR^{(I)}\) of \(\widehat{G}^I\) are exactly the image of the roots \(\RR\) of \(\widehat{G}\) under the quotient \(X^*(\widehat{T}) = X_*(T) \to X_*(T)_I = X^*(\widehat{T}^I)\).
	Note that \(X^*(\widehat{T}^I)\) may contain torsion, but this will not be problematic for us.
\end{rmk}

\begin{prop}\thlabel{determinant of pairing}
	Let \(V\) be a finite-dimensional \(\overline{\IQ}_\ell\)-linear \(\widehat{G}^I\)-representation.
	Then
	\begin{enumerate}
		\item\label{it-projective} \(\bfJ(V)\) is a finite projective \(\bfJ\)-module.
		\item\label{it-determinant} The determinant of the pairing \(\bfJ(V) \otimes_{\bfJ} \bfJ(V^*)\) is the divisor on \(\widehat{G}^I\sigma_p\GIT\widehat{G}^I \cong \widehat{T}_{\sigma_p}^I \GIT W_0\) defined by the function
		\[\prod_{\alpha \in \Phi(G,A)^\vee\mid \frac{\alpha}{2}\notin \Phi(G,A)^\vee} (e^\alpha-1)^{\zeta_\alpha} \prod_{\alpha \in \Phi(G,A)^\vee\mid \frac{\alpha}{2}\in \Phi(G,A)^\vee} (e^\alpha +1)^{\zeta_\alpha},\]
		where 
		\[\zeta_\alpha = \sum_{n\geq 0} \dim V_{\mid \widehat{G}^{\Gamma_F}}(n\alpha_{\sigma}),\]
		and \(\alpha_{\sigma}\in \Phi(G,S)^\vee\) is a coroot for the relative system \((G,S)\) such that the sum of its \(\sigma\)-translates is \(\alpha\).
	\end{enumerate}
\end{prop}
\begin{proof}
	\begin{enumerate}
		\item Since \(Z_G\) is connected by assumption, the derived subgroup \(\widehat{G}_{\der}\) of \(\widehat{G}\) is simply connected.
		Moreover, using \thref{notation essentially unramified}, the adjoint quotient \(G \to G_{\adj} \cong \prod_i \Res_{F_i/F} H_i\) induces a dual morphism \(\prod_i \widehat{H_i}^{e_if_i} \to \widehat{G}\), which factors through the derived subgroup \(\widehat{G}_{\der}\).
		Since the group \(\prod_i \widehat{H_i}^{e_if_i}\) is also simply connected (as the dual of an adjoint group), we get an isomorphism \(\prod_i \widehat{H_i}^{e_if_i} \to \widehat{G}_{\der}\), and hence an isomorphism 
		\[\prod_i \widehat{H_i}^{f_i} \cong (\widehat{G}_{\der})^I\]
		by passing to inertia-invariants.
		Now, the two diagonal morphisms in
		\[\begin{tikzcd}
			(\widehat{G}_{\der})^I \arrow[rd] \arrow[rr] & & (\widehat{G}^I)_{\der} \arrow[ld]\\
			& \widehat{G}^I&
		\end{tikzcd}\]
		are closed immersions.
		Hence, the natural map \((\widehat{G}_{\der})^I \to (\widehat{G}^I)_{\der}\) is a closed immersion as well.
		We claim that it is also surjective: indeed, since \(\pi_0(\widehat{G}^I)=\pi_0(\widehat{T}^I)\) \cite[Proposition 5.4 (7)]{ALRR:Fixed}, which is abelian, \((\widehat{G}^I)_{\der}\) agrees with the derived subgroup of the neutral connected component of \(\widehat{G}^I\); this implies the claim.
		Using this, we can and will denote \(\widehat{G}_{\der}^I:=(\widehat{G}_{\der})^I \cong (\widehat{G}^I)_{\der}\).
		Thus, we have shown that the derived subgroup of \(\widehat{G}^I\) is a split connected simply connected reductive group.
		(Note that this does not hold in general, by \cite[Example 5.9 (1)]{ALRR:Fixed}).
		
		Now, let \(\widehat{G}':= \widehat{G}_{\der}^I \times Z(\widehat{G}^I)\); the \(\sigma_p\)-action on \(\widehat{G}^I\) lifts uniquely to a \(\sigma_p\)-action to \(\widehat{G}'\) \cite[9.16]{Steinberg:Endomorphisms}.
		We will write \(\bfJ_{\widehat{G}^I}=\bfJ\), and \(\bfJ_{\widehat{G}'}\) for the ring of global sections of \(\widehat{G}'\sigma_p/\widehat{G}'\).
		Similarly, we define \(\bfJ_{\widehat{G}^I}(V)\) and \(\bfJ_{\widehat{G}'}(V)\).
		Consider the exact sequence
		\begin{equation}\label{enlarging dual group}1 \to \Delta \to \widehat{G}' \to \widehat{G}^I \to 1,\end{equation}
		and let \(\Delta_{\sigma_p}\) be the coinvariants of \(\Delta\) with respect to the induced \(\sigma_p\)-action.
		Let \(\widehat{T}'\) be the preimage of \(\widehat{T}^I\) in \(\widehat{G}'\), yielding an exact sequence
		\[1\to X^*(\widehat{T}^I) \to X^*(\widehat{T}') \to X^*(\Delta) \to 1.\]
		Passing to \(\sigma_p\)-invariants, we obtain a sequence
		\[1\to X^*(\widehat{T}^I)^{\sigma_p} \to X^*(\widehat{T}')^{\sigma_p} \to X^*(\Delta)^{\sigma_p},\]
		which is however not necessarily exact on the right.
		Let \(\Omega\subseteq X^*(\Delta)^{\sigma_p}\) be the image of \(X^*(\widehat{T}')^{\sigma_p}\) in \(X^*(\Delta)^{\sigma_p}\), corresponding to a diagonalizable quotient \(D(\Omega)\) of \(\Delta_{\sigma_p}\), and an exact sequence
		\[1\to D(\Omega) \to \widehat{T}'_{\sigma_p} \to \widehat{T}^I_{\sigma_p} \to 1.\]
		Then the action of the finite group \(\Delta\) on \(\bfJ_{\widehat{G}'}\) factors through the quotients \(\Delta \onto \Delta_{\sigma_p} \onto D(\Omega)\).
		We claim that the natural map \(\Spec \bfJ_{\widehat{G}'}\to \Spec \bfJ_{\widehat{G}^I}\) is a \(D(\Omega)\)-torsor.
		
		Indeed, since we are working over a field of characteristic 0, the Peter--Weyl theorem implies that
		\[\overline{\IQ}_\ell[\widehat{G}^I]^{c_{\sigma_p}(\widehat{G}^I)} \cong \bigoplus_{\mu\in X^*(\widehat{T}^I)^+} (V_\mu \otimes V_\mu^*)^{c_{\sigma_p}(\widehat{G}^I)} \cong \bigoplus_{X^*(\widehat{T}^I)^{+,\sigma_p}} \overline{\IQ}_\ell,\]
		and similarly
		\[\overline{\IQ}_\ell[\widehat{G}']^{c_{\sigma_p}(\widehat{G}')} \cong \bigoplus_{X^*(\widehat{T}')^{+,\sigma_p}} \overline{\IQ}_\ell.\]
		Since \(D(\Omega)\) acts freely on \(\bfJ_{\widehat{G}'}\), and by comparing the basis above, we see that \(\Spec \bfJ_{\widehat{G}'} \to \Spec \bfJ_{\widehat{G}^I}\) is indeed a \(D(\Omega)\)-torsor.
		
		Now, since the \(\Delta_{\sigma_p}\)-action on \(\bfJ_{\widehat{G}'}\) factors through \(D(\Omega)\), \cite[(4.1.3)]{XiaoZhu:Vector} implies that \(\bfJ_{\widehat{G}'}(V)^{D(\Omega)} = \bfJ_{\widehat{G}^I}(V)\).
		Then the same argument as in \cite[Lemma 4.1.3]{XiaoZhu:Vector} implies that the natural map \(\bfJ_{\widehat{G}^I}(V) \otimes_{\bfJ_{\widehat{G}^I}} \bfJ_{\widehat{G}'} \to \bfJ_{\widehat{G}'}(V)\) is an isomorphism.
		Thus, by faithfully flat descent, we are reduced to showing the proposition for \(\widehat{G}'=Z(\widehat{G}^I) \times \widehat{G}_{\der}^I\), and hence for \(G_{\der}^I\).
		But we have seen above that this is a connected, semisimple, simply connected reductive group, so that \cite[Theorem 4.3.2]{XiaoZhu:Vector} applies.
		
		\item By (1), it suffices to compute the determinant of the pairing \(\bfJ_{\widehat{G}'}(V) \otimes_{\bfJ_{\widehat{G}'}} \bfJ_{\widehat{G}'}(V^*)\).
		By \cite[(4.1.1)]{XiaoZhu:Vector}, everything is compatible with products.
		Since we clearly have a perfect pairing for diagonalizable groups, it suffices to consider the semisimple simply connected group \(\widehat{G}_{\der}^I\), which agrees with the inertia-invariants of the dual group of the adjoint quotient \(G_{\adj}\).
		Again using the compatibility with products, we may assume that \(G=\Res_{F'/\IQ_p} H\), where \(H/F'\) is adjoint and unramified over \(F'\).
		Since the proposition for \(H\) was shown in \cite[Theorem 6.1.2]{XiaoZhu:Vector}, it remains to handle restrictions of scalars.
		
		Let \(F\) be the maximal unramified subextension of \(F'/\IQ_p\), and set \(K := \Res_{F'/F} H\), so that \(G=\Res_{F/\IQ_p} K\).
		Then there is a Frobenius-equivariant isomorphism \(\widehat{H}^{I_{F'}} \cong \widehat{K}^{I_F}\), and the relative root systems of \(K\) and \(H\) agree (using the choices of compatible tori from \thref{notation essentially unramified}).
		Finally the passage from \(K\) to \(G\) along the unramified extension \(F/\IQ_p\) can be handled by \cite[Lemma 4.1.2]{XiaoZhu:Vector}.
	\end{enumerate}
\end{proof}

We can now define the desired genericity conditions.
We denote by \(\sigma_p\) and \(\phi_p\) the arithmetic and geometric \(p\)-Frobenii respectively, and let \(m>0\) be some integer such that the \(\Gamma_{\IF_p}\)-action on \(\widehat{G}^I_{\overline{\IQ}_\ell}\) factors through \(\Gal(\IF_{p^m}/\IF_p)\).

\begin{dfn}\thlabel{defi general}
	Let \(\gamma\sigma_p\in \widehat{G}^I\sigma_p \GIT\widehat{G}^I\) be an element, and \(V\in \Rep^{\fd}(\widehat{G}^I_{\overline{\IQ}_\ell})\).
	\begin{enumerate}
		\item \(\gamma\sigma_p\) is \emph{\(V\)-general} if it does not lie in the divisor from \thref{determinant of pairing} \eqref{it-determinant}.
		\item \(\gamma\sigma_p\) is \emph{strongly \(V\)-general} if for every dominant cocharacter \(\lambda_{\Gamma_{\IQ_p}} \in X_*(\widehat{T}^{\Gamma_{\IQ_p}})\) such that \(V_{\widehat{G}^{\Gamma_{\IQ_p}}}(\lambda_{\Gamma_{\IQ_p}})\neq 0\) and which does not factor through \(\widehat{Z_G}^{\Gamma_{\IQ_p}}\), we have \(\lambda_{\Gamma_{\IQ_p}}((\gamma \phi_p)^{mn})\neq 1\) for all \(n>1\).
	\end{enumerate}
\end{dfn}
Recall that \(Z_G\) is assumed to be connected, so that it is a torus and \(\widehat{Z_G}\) is defined.
Note also that any regular semisimple element is \(V\)-general for any \(V\in \Rep^{\fd}(\widehat{G}_{\overline{\IQ_\ell}}^I)\).

Let \(V_{\mu^*}\) be the \(\overline{\IQ}_\ell\)-linear simple \(\widehat{G}\)-representation of highest weight \(\mu^*\), and denote its restriction to \(\widehat{G}^I\) similarly.
Then we can define \(V_{\mu^*}^{\Tatep}\) as in \eqref{Tatep}.
Similarly, \(V_{\mu_I^*}\) denotes the simple \(\widehat{G}^I\)-representation of highest weight \(\mu_I^*\), and again we have defined \(V_{\mu_I^*}^{\Tatep}\) in \eqref{Tatep}.
The following lemma explains the relevance of the strong generality condition.

\begin{lem}\thlabel{lemma strongly general}
	Let \(\gamma\phi_p \in \widehat{T}^I\rtimes \Gal(\IF_{p^m}/\IF_p)\subseteq \widehat{G}^I\rtimes \Gal(\IF_{p^m}/\IF_p)\) be such that its image in \(\widehat{Z_G}^I\rtimes \Gal(\IF_{p^m}/\IF_p)\) has finite order.
	Then we have
	\begin{equation}\label{Tatep as fixed points}V_{\mu^*}^{\Tatep} \subseteq \bigcup_{j\geq 1} V_{\mu^*}^{r_{\mu^*}((\gamma\phi_p)^{jm})}.\end{equation}
	Moreover, if \(\gamma\) is strongly general, then this inclusion is an isomorphism.
\end{lem}
\begin{proof}
	By \thref{description of Lambda} (and since the corresponding \(\tau_I\) is central by \thref{centrality of tau}), \(V_{\mu^*}^{\Tatep}\) decomposes (as a vector space) as
	\[V_{\mu^*}^{\Tatep} = \bigoplus_{\lambda_{\Gamma_{\IQ_p}} \in X^*(\widehat{Z_G}^{\Gamma_{\IQ_p}}) = X^*(\widehat{Z_G})_{\Gamma_{\IQ_p}}} V_{\mu^*\mid \widehat{G}^{\Gamma_{\IQ_p}}}(\lambda_{\Gamma_{\IQ_p}}).\]
	By the assumption that \(\gamma\phi_p\) has finite image in \(\widehat{Z_G}^I\rtimes \Gal(\IF_{p^m}/\IF_p)\), we see that \(r_{\mu^*}((\gamma \phi_p)^{jm})\) acts trivially on each direct summand, which gives the inclusion \eqref{Tatep as fixed points}.
	On the other hand, if \(\lambda_{\Gamma_{\IQ_p}}\in X^*(\widehat{T}^{\Gamma_{\IQ_p}})\) does not factor through \(\widehat{Z_G}^{\Gamma_{\IQ_p}}\) and \(\gamma\phi_p\) is strongly general, then \((V_{\mu^*\mid \widehat{G}^{\Gamma_{\IQ_p}}}(\lambda_{\Gamma_{\IQ_p}}))^{r_\mu^*((\gamma\phi_p)^{jm})} = 0\) for \(j>1\).
	This shows that \eqref{Tatep as fixed points} is an isomorphism when \(\gamma \phi_p\) is strongly general.
\end{proof}

Let us recall certain Shimura varieties that will be of special interest, and have also been the focus of e.g.~\cite{Kottwitz:lambda-adic,ScholzeShin:Cohomology}.

\begin{dfn}
	Let \(F\) be a CM field, with totally real subfield \(F_1\), such that \(F/F_1\) is unramified at all places above \(p\).
	Let \(D\) be a central division algebra over \(F\), of dimension \(n^2\), with an involution \(*\) that restricts to complex conjugation on \(F\).
	Let \(b\in D^\times\) with \(b^*=b\) such that \((v,w)\mapsto vbw^*\) is a positive pairing on \(B\otimes_{\IQ} \IR\); this exists by \cite[Lemma 2.8]{Kottwitz:Points}.
	Define an involution \(\sharp\colon D^{\op}\cong D^{\op}\colon v\mapsto bv^*b^{-1}\), let \(\beta\in F\) such that \(\beta^\sharp=-\beta\), and consider the inner product \(\langle v,w\rangle = \tr_{D^{\op}}(vb^{-1}\beta w^{\sharp})\).
	
	Then the tuple \((B,*,V,(,),h) = (D^{\op},\sharp,D^{\op},\langle \rangle,h_0)\), where \(D^{\op}\) acts on itself via left-multiplication, and \(h_0\colon\mathbb{C} \to D\otimes_{\IQ} \IR \cong \End_{D^{\op}\otimes_{\IQ} R}(D^{\op}\otimes_{\IQ} \IR)\) is any *-homomorphism such that \((v,w)\mapsto \langle v,h_0(i)w\rangle\) is positive-definite, is a global rational PEL datum (in the sense of \thref{PEL data}).
	Following \cite{ScholzeShin:Cohomology}, the resulting Shimura variety (which is automatically proper) will be called a \emph{compact unitary group Shimura variety}.
\end{dfn}

In the terminology of \cite{ScholzeShin:Cohomology}, the local group appearing in the above Shimura data at a \emph{split place} is essentially unramified.
Thus, our methods apply to these Shimura varieties.

Finally, we prove the main theorem of the paper.

\begin{thm}\thlabel{main global theorem}
	Let \((\IG,\IX)\) be a Shimura datum of Hodge type, satisfying \thref{assumptions Shimura variety}, such that \(G:=\IG_{\IQ_p}\) is essentially unramified.
	We fix a very special parahoric model \(\Gg/\IZ_p\) of \(G\).
	Assume that the center \(Z_G\) of \(G\) is connected, and that \(V_{\mu^*}^{\Tatep}\neq 0\).
	Let \(\IG'\) denote the unique inner form of \(\IG\), which is trivial at the finite places, and compact modulo center at the archimedean place (it exists by \thref{When is basic element unramified} \eqref{item character} and \cite[Corollary 2.1.6]{XiaoZhu:Cycles}).
	Then the following hold:
	\begin{enumerate}
		\item The basic Newton strata \(\Sh_{\mu,K,\bas}^{\spl}\) and \(\Sh_{\mu,K,\bas}^{\can}\) are equidimensional of dimension \(\langle \rho,\mu\rangle\).
		Moreover, there is a canonical \(\Hh_{K,\overline{\IQ}_\ell}\)-equivariant diagram
		\[\begin{tikzcd}
			\mathrm{H}^{\BM}_{\langle2 \rho,\mu\rangle}(\Sh_{\mu,K,\bas}^{\spl}) \arrow[r, "\cong"] \arrow[d] & C(\IG'(\IQ)\backslash \IG'(\IA_f)/K,\overline{\IQ}_\ell) \otimes_{\overline{\IQ}_\ell} V_{\mu^*}^{\Tatep}\arrow[d]\\
			\mathrm{H}^{\BM}_{\langle2 \rho,\mu\rangle}(\Sh_{\mu,K,\bas}^{\can}) \arrow[r, "\cong"'] & C(\IG'(\IQ)\backslash \IG'(\IA_f)/K,\overline{\IQ}_\ell) \otimes_{\overline{\IQ}_\ell} V_{\mu_I^*}^{\Tatep},
		\end{tikzcd}\]
		where \(C(\IG'(\IQ)\backslash \IG'(\IA_f)/K,\overline{\IQ}_\ell)\) denotes the space of functions \(\IG'(\IQ)\backslash \IG'(\IA_f)/K \to \overline{\IQ}_\ell\).
		\item Let \(\pi_f\) be an irreducible \(\Hh_{K,\overline{\IQ}_\ell}\)-module, and
		\[\mathrm{H}^{\BM}_{\langle2 \rho,\mu\rangle}(\Sh_{\mu,K,\bas}^{\spl})[\pi_f] := \Hom_{\Hh_{K,\overline{\IQ}_\ell}}(\pi_f,\mathrm{H}^{\BM}_{\langle 2\rho,\mu\rangle}(\Sh_{\mu,K,\bas}^{\spl}))\otimes \pi_f\]
		its \(\pi_f\)-isotypical component.
		If the Satake parameter of the component \(\pi_{f,p}\) of \(\pi_f\) at \(p\) is \(V_{\mu^*}\)-general in the sense of \thref{defi general}, then the restriction of the cycle class map
		\[\mathrm{H}^{\BM}_{\langle2 \rho,\mu\rangle}(\Sh_{\mu,K,\bas}^{\spl})[\pi_f] \subseteq \mathrm{H}^{\BM}_{\langle2 \rho,\mu\rangle}(\Sh_{\mu,K,\bas}^{\spl}) \xrightarrow{\cl} \mathrm{H}^{\langle 2\rho,\mu\rangle}_{\comp}(\Sh_{\mu,K}^{\spl},\overline{\IQ}_\ell(\langle \rho,\mu\rangle))\]
		is injective.
		\item
		Assume that \(\Sh_{K^pK_p}(\IG,\IX)\) is a compact unitary group Shimura variety, and let \(\pi_f^p\) be an irreducible \(\Hh_{K^p,\overline{\IQ}_\ell}\)-module.
		Then the \(\pi_f^p\)-isotypical component of the cycle class map \(\cl\) surjects onto
		\[\sum_{\pi_p} T^{\langle \rho,\mu\rangle}(\pi_p \pi_f^p,\overline{\IQ}_\ell) \otimes \pi_p\pi_f^p\]
		if the Satake parameters of the \(\pi_p\) are all strongly \(V_{\mu^*}\)-general, where the \(\pi_p\) range over those irreducible representations of the local Hecke algebra \(\Hh_{\Gg}\) such that \(\pi_p\pi_{f}^p\) appears in \(C(\IG'(\IQ)\backslash \IG'(\IA_f)/K,\overline{\IQ}_\ell)\).
	\end{enumerate}
\end{thm}
\begin{rmk}
	Once the conjecture from \thref{remark on S=T} is known, it will be possible to replace the sum in part (3) above by a single irreducible \(\Hh_{K,\overline{\IQ}_\ell}\)-module \(\pi_f\).
\end{rmk}
\begin{proof}
	(1) By \thref{When is basic element unramified}, the condition that \(V_{\mu^*}^{\Tatep}\neq 0\) is equivalent to the basic element \(b\in B(G,\mu^*)\) being very special.
	The dimension statement and computation of Borel--Moore homology then follow from \thref{irreducible components of basic stratum}.
		
	(2) By \thref{irreducible comp of ADLV in very special case} and \thref{irreducible components of basic stratum}, we get canonical decompositions
	\[X_{\leq \mu^*}^{\spl}(b) = \bigcup_{\bfb\in \IM\IV_{\mu^*}^{\spl,\Tatep}} X_{\leq \mu^*}^{\bfb,\spl}(b)\]
	and
	\[\Sh_{\mu,K,\bas}^{\spl} = \bigcup_{\bfb\in \IM\IV_{\mu^*}^{\spl,\Tatep}} \Sh_{\mu,K,\bas}^{\bfb,\spl}.\]
	As in \cite[§A.2.18]{XiaoZhu:Cycles}, there are the cycle class maps
	\[\cl(\bfb)\colon \mathrm{H}^{\BM}_{\langle 2\rho,\mu\rangle}(\Sh_{\mu,K,\bas}^{\bfb,\spl}) \cong C(\IG'(\IQ)\backslash \IG'(\IA_f)/K,\overline{\IQ}_\ell) \to \mathrm{H}^{\langle 2\rho,\mu\rangle}_\comp(\Sh_{\mu,K},\overline{\IQ}_\ell (\langle \rho,\mu\rangle))\]
	and 
	\[\cl:= \bigoplus_{\bfb\in \IM\IV^{\spl,\Tatep}}\cl(\bfb):\bigoplus_{\bfb\in \IM\IV^{\spl,\Tatep}}\mathrm{H}^{\BM}_{\langle2 \rho,\mu\rangle}(\Sh_{\mu,K,\bas}^{\bfb,\spl}) \cong \mathrm{H}^{\BM}_{\langle 2\rho,\mu\rangle}(\Sh_{\mu,K,\bas}^{\spl})\to \mathrm{H}^{\langle 2\rho,\mu\rangle}_\comp(\Sh_{\mu,K},\overline{\IQ}_\ell (\langle \rho,\mu\rangle)).\]
	Let \(\tau^*\in X_*(Z_G)\subseteq X_*(T)\) such that \(p^{\tau^*}\) represents the (very special) basic element \(b\in B(G,\mu^*)\), \thref{When is basic element unramified}.
	Then we can apply the proof of \thref{JL for canonical models} to the exotic Hecke correspondence from \thref{exotic corr via RZ unif} to get a map
	\[\JL_{\tau_I,\mu}\colon \mathrm{H}^0_c(\Sh_{\tau,K}^{\can}, \overline{\IQ}_\ell) \otimes_{\bfJ} \Hom_{\bfJ}(\widetilde{V_{\tau_I}^{\can}},\widetilde{V_\mu^{\spl}}) \to \mathrm{H}_c^{\langle 2\rho,\mu\rangle}(\Sh_{\mu,K}^{\spl},\overline{\IQ}_\ell(\langle\rho,\mu\rangle)).\]
	
	Now, for \(\bfb\in \IM\IV^{\spl,\Tatep}_{\mu^*}\), let \(\lambda\in X_*(T)\) be the unique element such that \(\bfb\in \IM\IV^{\spl}_{\mu^*}(\lambda^*)\).
	Then by \thref{minimal elements in satake to mv}, there exists a minimal \(\nu_{\bfb,I}\in X_*(T)_I^+\) such that \(\lambda_I + \nu_{\bfb,I} - \sigma(\nu_{\bfb,I})\in X_*(T)_I^+\) and \(\bfb\) lies in the image of the map \(\iota_{\IM\IV}^{\spl}\colon \IS_{\nu_{\bfb,I}^*,\mu^*\mid \lambda_I^*+\nu_I^*}^{\spl} \to \coprod_{\lambda\in X_*(T)\colon \lambda\mapsto \lambda_I\in X_*(T)_I} \IM\IV_{\mu^*}^{\spl}(\lambda^*)\) from \thref{Satake to MV cycles}.
	Let \(\tau_{\bfb,I}:= \lambda_I + \nu_{\bfb,I} - \sigma(\nu_{\bfb,I})\), which is central by \thref{centrality of tau}.
	Then \(\bfa := (\iota_{\IM\IV}^{\spl})^{-1}(\bfb)\) yields an irreducible component of \(\Gr_{\nu_{\bfb,I}^*,\mu^*\mid \tau_{\bfb,I}^* + \sigma(\nu_{\bfb,I}^*)}^{\spl}\), and hence an element
	\[\bfa_{\mathrm{in}} \in \Hom_{\widehat{G}^I_{\overline{\IQ}_\ell}}(V_{\nu_{\bfb,I}^*}^{\can} \otimes V_{\mu^*}^{\spl}, V_{\tau_{\bfb,I}^*}^{\can} \otimes V_{\sigma(\nu_{\bfb,I}^*)}^{\can}) \cong \Hom_{\widehat{G}^I_{\overline{\IQ}_\ell}}(V_{\sigma(\nu_{\bfb,I})}^{\can} \otimes V_{\tau_{\bfb,I}}^{\can} \otimes V_{\nu_{\bfb,I}^*}^{\can}, V_{\mu}^{\spl})\]
	by \thref{basic properties of Satake correspondences} \eqref{homs and irr}.
	By \eqref{reps to coh}, this yields an element in \(\Hom_{\bfJ}(\widetilde{V_{\tau_{\bfb,I}}^{\can}},\widetilde{V_{\mu}^{\spl}})\).
	This element can be plugged into the map \(\JL_{\tau_{\bfb,I},\mu}\), and the diagram
	\[\begin{tikzcd}
		\mathrm{H}_c^0(\Sh_{\tau_{\bfb},K}^{\can},\overline{\IQ}_\ell) \arrow[rr, "\cong"] \arrow[rd, "\JL_{\tau_{\bfb,I},\mu}(\bfa_{\mathrm{in}})"'] && \mathrm{H}_{\langle2\rho,\mu\rangle}^{\BM}(\Sh_{\mu,K,\bas}^{\bfb,\spl}) \arrow[dl, "\cl(\bfb)"]\\
		&\mathrm{H}_c^{\langle2\rho,\mu\rangle}(\Sh_{\mu,K,}^{\spl}, \overline{\IQ}_\ell(\langle\rho,\mu\rangle))&
	\end{tikzcd}\]
	commutes by \thref{Langlands vs cycle class}; here \(\Sh_{\tau_{\bfb},K}^{\can}\) is the special fiber of the weak Shimura variety associated with \(\IG'\) and some lift \(\tau_{\bfb}\) of \(\tau_{\bfb,I}\).
	
	On the other hand, we have a map \(\cl'(\bfb)\colon \mathrm{H}_c^{\langle2\rho,\mu\rangle}(\Sh_{\mu,K}^{\spl},\overline{\IQ}_\ell(\langle \rho,\mu\rangle)) \to \mathrm{H}_c^0(\Sh_{\tau_{\bfb},K}^{\can},\overline{\IQ}_\ell)\) which is dual to the cycle class map \cite[§A.2.18]{XiaoZhu:Cycles}; here we crucially use the smoothness of \(\Sh_{\mu,K}^{\spl}\).
	Moreover, similar to the above paragraph, we have maps
	\[\JL_{\mu,\tau_{\bfb,I}} \colon \mathrm{H}_c^{\langle2\rho,\mu\rangle}(\Sh_{\mu,K}^{\spl},\overline{\IQ}_\ell(\langle \rho,\mu\rangle)) \otimes_{\bfJ} \Hom_{\bfJ}(\widetilde{V_\mu^{\spl}},\widetilde{V_{\tau_{\bfb,I}}^{\can}}) \to \mathrm{H}_c^0(\Sh_{\tau_{\bfb},K}^{\can},\overline{\IQ}_\ell)\]
	and
	\[\bfa_{\mathrm{out}} \in \Hom_{\widehat{G}^I_{\overline{\IQ}_\ell}}(V_{\tau_{\bfb,I}^*}^{\can} \otimes V_{\sigma(\nu_{\bfb,I}^*)}^{\can},V_{\nu_{\bfb,I}^*}^{\can} \otimes V_{\mu^*}^{\spl}) \cong \Hom_{\widehat{G}^I_{\overline{\IQ}_\ell}}(V_{\sigma(\nu_{\bfb,I}^*)}^{\can} \otimes V_{\mu}^{\spl} \otimes V_{\nu_{\bfb,I}}^{\can}, V_{\tau_{\bfb,I}}^{\can}),\]
	and the maps \(\cl'(\bfb)\) and \(\JL_{\mu,\tau_{\bfb,I}}(\bfa_{\mathrm{out}})\) agree.
	Thus, if \(\{\bfb_1,\ldots,\bfb_r\}\) are the elements of \(\IM\IV_{\mu^*}^{\spl,\Tatep}\), the intersection matrix of the cycle classes coming from the basic Newton stratum has \((i,j)\)-entry
	\begin{equation}\label{intersection matrix}\JL_{\mu,\tau_{\bfb_j,I}}(\bfa_{j,\mathrm{out}}) \circ \JL_{\tau_{\bfb_i,I},\mu}(\bfa_{i,\mathrm{in}}).\end{equation}
	
	For \(i\) and \(j\), consider the morphism \(\widetilde{V_{\tau_{\bfb_i,I}}^{\can}} \to \widetilde{V_{\tau_{\bfb_j,I}}^{\can}}\) induced by \[\Xi_{V_{\nu_{\bfb_j,I}^*}^{\can}}(\bfa_{j,\mathrm{out}})\circ \Xi_{V_{\nu_{\bfb_i,I}}^{\can}}(\bfa_{i,\mathrm{in}}) \colon \widetilde{V_{\tau_{\bfb_i,I}}^{\can}} \to \widetilde{V_{\mu}^{\spl}} \to \widetilde{V_{\tau_{\bfb_j,I}}^{\can}},\]
	where the \(\Xi_-\) are defined similarly to \eqref{reps to coh}.
	This induces an endomorphism \(\gamma\) of the vector bundle \(\bigoplus_{\IM\IV_{\mu^*}^{\spl,\Tatep}} \widetilde{V_{\tau_{\bfb_i,I}}^{\can}}\) on \(\widehat{G}^I\sigma/\widehat{G}\).
	In order to determine when the determinant of the intersection matrix \eqref{intersection matrix} is invertible, we may consider the matrix coming from the endomorphism \(\gamma\), since the Jacquet--Langlands maps \(\JL\) are compatible with composition.
	
	The determinant of \(\gamma\) can be viewed as a global section of \(\widehat{G}^I\sigma_p /\widehat{G}^I\), i.e., as an element of \(\bfJ=\bfJ_{\widehat{G}^I}\).
	To compute it, consider the short exact sequence
	\[1 \to \Delta \to \widehat{G}':= \widehat{G}^I_{\der} \times Z(\widehat{G}^I) \to \widehat{G}^I \to 1\]
	as in \eqref{enlarging dual group}, on which \(\sigma_p\) acts.
	Then \cite[(4.1.3)]{XiaoZhu:Vector} yields an isomorphism \(\Gamma(\widehat{G}^I\sigma_p /\widehat{G}^I,\Oo_{\widehat{G}^I \sigma_p/\widehat{G}^I}) \cong \Gamma(\widehat{G}'\sigma_p /\widehat{G}', \Oo_{\widehat{G}'\sigma_p/\widehat{G}'})^{\Delta_{\sigma_p}}\), where \(\Delta_{\sigma_p}\) denotes the \(\sigma_p\)-coinvariants of \(\Delta\).
	Thus, it suffices to compute the determinant as a global section of \(\widehat{G}'\sigma_p/\widehat{G}'\), i.e., as an element in \(\bfJ':=\bfJ_{\widehat{G}'}\).
	
	Now, the restriction of \(\tau_I\in X_*(T)_I\cong X^*(\widehat{T}^I)\) to \(X^*(Z(\widehat{G})^I)\) is the central character of \(V_{\mu}^{\spl}\) (or more precisely, the central character of each direct summand of \(V_\mu^{\spl}\) as a \(\widehat{G}^I\)-representation).
	The restriction of \(\lambda_I = \tau_{\bfb_i,I} + \sigma(\nu_{\bfb_i,I}) - \nu_{\bfb_i,I} \in X^*(\widehat{T}^I)\) to \(Z(\widehat{G})^I\subseteq \widehat{T}^I\) then agrees with \(\tau_I\), which we may also view as a character of \(\widehat{G}'\) via restriction along \(\widehat{G}' = \widehat{G}_{\der}^I\times Z(\widehat{G})^I \to Z(\widehat{G})^I\).
	This yields an isomorphism \(\widetilde{V_{\tau_I}^{\can}} \cong \widetilde{V_{\tau_{\bfb_i,I}}^{\can}} \colon f\mapsto fe^{\nu_{\bfb_i,I\mid Z(\widehat{G})^I}}\) of vector bundles on \(\widehat{G}'\sigma_p/\widehat{G}'\), and we may replace all \(\tau_{\bfb,I}\) above by \(\tau_I\) when working on \(\widehat{G}'\sigma_p/\widehat{G}'\).
	Moreover, the set \(\left\{\widetilde{V_{\tau_I}^{\can}} \xrightarrow{\cong} \widetilde{V_{\tau_{\bfb_i,I}}^{\can}} \xrightarrow{\bfa_{i,\mathrm{in}}} \widetilde{V_{\mu}^{\spl}}\mid i\in \IM\IV_{\mu^*}^{\spl,\Tatep}\right\}\) forms a basis of \(\Hom_{\bfJ'}(\widetilde{V_{\tau_I}^{\can}},\widetilde{V_{\mu}^{\spl}})\) by \cite[Theorem 4.3.2]{XiaoZhu:Vector}, and dually we get a basis for \(\Hom_{\bfJ'}(\widetilde{V_{\mu}^{\spl}},\widetilde{V_{\tau_I}^{\can}})\).
	The determinant \(\Dd\) (as a global section on \(\widehat{G}'\sigma_p/\widehat{G}'\)) can thus be identified with the determinant of the pairing
	\[\Hom_{\bfJ'}(\widetilde{V_{\tau_I}^{\can}},\widetilde{V_{\mu}^{\spl}}) \otimes_{\bfJ'} \Hom_{\bfJ'}(\widetilde{V_{\mu}^{\spl}},\widetilde{V_{\tau_I}^{\can}}) \to \bfJ',\]
	which in turn can be identified with
	\[\bfJ'(V_\mu^{\spl} \otimes V_{\tau_I}^{\can,*}) \otimes_{\bfJ'} \bfJ'(V_{\tau_I}^{\can} \otimes V_{\mu}^{\spl,*}) \to \bfJ'.\]
	Since the cycle class map \(\mathrm{H}_{\langle 2\rho,\mu\rangle}^{\BM}(\Sh_{\mu,K,\bas}^{\spl})[\pi_f] \to \mathrm{H}_c^{\langle 2\rho,\mu\rangle}(\Sh_{\mu,K}^{\spl},\overline{\IQ}_\ell(\langle \rho,\mu\rangle))\) is injective if the Satake parameter of \(\pi_{f,p}\) does not belong to the divisor determined by \(\Dd\), part (2) of the theorem follows from \thref{determinant of pairing}.
	
	(3) Let \(\pi_f^p\) be an irreducible \(\Hh_{K^p}\)-module, and \(\pi_p\) an irreducible \(\Hh_{K_p}\)-module.
	Let  \(m_{\IG'}(\pi_p\pi_f^p)\) denote the multiplicity of \(\pi_p\pi_f^p\) in \(C(\IG'(\IQ) \backslash \IG'(\IA_f) / K, \overline{\IQ}_\ell)\).
	
	Now, \(W(\pi_p\pi_f^p):=\bigoplus_{i=0}^{\langle 2\rho,\mu\rangle} W^{2i}(\pi_p\pi_f^p,\overline{\IQ}_\ell(i))\) admits an action of \(\Gamma_{\IE}\), and its restriction to \(\Gamma_{E^{\Gal}}\) factors through \(\Gamma_k\) (since \(\mathscr{S}_{K}^{\spl}(\IG,\IX)\) is proper smooth over \(\Oo_{E^{\Gal}}\)).
	On the other hand we have a representation of \(\Gamma_k\) on \(V_{\mu^*}\) obtained via the composition
	\[\Gamma_k \to \widehat{G}^I \rtimes \Gamma_k \to \GL(V_{\mu^*}).\]
	Here, the map \(\Gamma_k \to \widehat{G}^I \rtimes \Gamma_k\) is the spherical Langlands parameter (in the sense of \cite{Zhu:Ramified,vdH:SphericalParameters}) attached to \(\pi_p\) under the Satake isomorphism \cite[Theorem 10.11]{vdH:RamifiedSatake} (after base change to \(\overline{\IQ}_\ell\); in particular this depends on the choice of \(q^{\frac{1}{2}}\)).
	Moreover, the representation \(\widehat{G}^I\rtimes \Gamma_k \to \GL(V_\mu^*)\) is the one from \cite[Definition 5.2]{vdH:SphericalParameters}, i.e., its restriction to \(\widehat{G}^I\subseteq \widehat{G}^I\rtimes \Gamma_k\) agrees with the restriction along \(\widehat{G}^I\subseteq \widehat{G}\) of the irreducible representation of highest weight \(\mu^*\), and the \(\Gamma_k\)-action fixes this highest weight space.
	
	Then by combining \cite[Theorem 9.1]{ScholzeShin:Cohomology} and \cite[Theorem 2.2.1]{XiaoZhu:Cycles}, we see that in the Grothendieck group of \(\phi_v\)-representations, we have
	\begin{equation}\label{trace formula}m_{\IG'}(\pi_p\pi_f^p)[V_{\mu^*}] = [W(\pi_p\pi_f^p)].\end{equation}
	Indeed, note that \cite[§2.2]{XiaoZhu:Cycles} do not need any unramified assumptions.
	Moreover, \cite{ScholzeShin:Cohomology} use the Langlands parameter associated to \(\pi_p\) constructed by Harris--Taylor \cite{HarrisTaylor:Geometry}.
	But this is known to agree with the Fargues--Scholze local Langlands correspondence \cite[Theorem IX.7.4]{FarguesScholze:Geometrization}, which in turn is compatible with our construction by \cite[Corollary 4.5]{vdH:SphericalParameters}.
	
	Finally, part (3) of the theorem follows from \eqref{trace formula} by applying \(\bigcup_{j\geq 1}(-)^{\phi_v^j}\) to both sides, using \thref{lemma strongly general}.
\end{proof}

\appendix
\appendicestocpagenum

\section{Complements on étale motives}\label{Appendix:motives}

In this appendix, we recall and gather the necessary background on étale motives, which is the sheaf theory used throughout most of the main text.
We will refer to \cite[§2]{vdH:RamifiedSatake} for more details, and emphasize the parts that have not been discussed in loc.~cit.
For the sake of accessibility, and since everything is probably well-known, we will be brief and focus on the applications necessary in the main text, rather than developing the six functor formalism in the highest possible generality.
Throughout this appendix, \(k\) will denote a finite field of characteristic \(p\), or an algebraic closure thereof.
Moreover, \(\Lambda\) will denote any (discrete) \(\IZ[\frac{1}{p}]\)-algebra.

\subsection{Review of the motivic six functor formalism}

The motivic theory of interest in this paper is the theory of étale motives, constructed in \cite{CisinskiDeglise:Etale} via the h-topology (and upgraded to the level of \(\infty\)-categories in \cite{Robalo:Theorie,Khan:Motivic,Preis:Motivic}); in particular it satisfies h-descent.
As explained in \cite[Theorem 2.10]{RicharzScholbach:Witt} (cf.~also \cite[Proposition 2.2]{vdH:RamifiedSatake}), this theory is invariant under the perfection functor, and hence descends to a functor
\[\DM(-) := \DM(-,\Lambda) \colon (\Sch_k^{\pfp})^{\op} \to \Pr_{\Lambda}^{\St}\colon(f\colon X\to Y)\mapsto (f^!\colon \DM(Y)\to \DM(X)),\]
which comes equipped with a full six-functor formalism.
Here, \(\Sch^{\pfp}_k\) is the category of perfect schemes which are \emph{perfectly of finite presentation} over \(k\) (i.e., locally the perfection of a scheme of finite presentation over \(k\)), and \(\Pr_{\Lambda}^{\St}\) is the \(\infty\)-category of stably presentably \(\Lambda\)-linear \(\infty\)-categories, with colimit-preserving \(\Lambda\)-linear functors.

By applying Kan extensions as in \cite[Definition 2.2.1]{RicharzScholbach:Intersection} (after fixing a large enough regular cardinal \(\kappa\) and only considering affine \(k\)-schemes which \(\kappa\)-filtered colimits of affine schemes in \(\Sch_k^{\pfp}\)), this can be extended to a functor
\begin{equation}\label{motives on prestacks}\DM(-):=\DM(-,\Lambda)\colon (\PreStk_k^{\perf})^{\op}\to \Pr_{\Lambda}^{\St}\colon (f\colon X\to Y) \mapsto (f^!\colon \DM(Y) \to \DM(X)),\end{equation}
where \(\PreStk_k^{\perf}=\Fun((\AffSch_k^{\perf})^{\op},\Ani)\) is the category of perfect prestacks over \(k\), i.e., presheaves on the category of perfect affine schemes over \(k\), with values in anima (i.e., \(\infty\)-groupoids).
More precisely, this means that for a perfect prestack \(X\), we have
\begin{equation}\label{kanlim}\DM(X) \cong \varprojlim_{T\to X} \DM(T),\end{equation}
where \(T\) ranges over the category of perfect affine schemes equipped with a map to \(X\), and
\begin{equation}\label{kancolim}\DM(T) \cong \varinjlim_{T\to T'} \DM(T'),\end{equation}
where \(T'\) ranges over the perfect affine schemes which are pfp over \(\Spec k\); cf.~\cite[Remark 2.2.2 (i)]{RicharzScholbach:Intersection}.
However, although \(f^!\) exists for all morphisms \(f\) of prestacks, this Kan extended formalism does not admit the six functors for all morphisms \(f\).

It is explained in \cite[§2.2]{vdH:RamifiedSatake} why \(f^*,f_*,f_!\) exist for morphisms of ind-schemes (which are required to be schematic in the case of \(f^*\)), and quotients of morphisms of schemes by pro-smooth group actions.
More generally, \cite[Proposition 2.3.3]{RicharzScholbach:Intersection} established the existence of a left adjoint \(f_!\) of \(f^!\) for a general morphism of ind-algebraic stacks (locally of finite type).
In this paper, we will also need the *-functors for representable morphisms of algebraic stacks.
To construct these, recall that for a category \(\Cc\), there is the \(\infty\)-category of correspondences
\(\Corr(\Cc)\), with the same objects as \(\Cc\), and where the morphisms \(X\to Y\) are diagrams 
\begin{equation}\label{geometric corr}\begin{tikzcd}
	&Z\arrow[ld, "f"'] \arrow[rd, "g"]&\\
	X&&Y
\end{tikzcd}\end{equation}
in \(\Cc\); we refer to \cite[Chapter 7]{GaitsgoryRozenblyum:Study1}, \cite[§2]{HeyerMann:6functor}, or \cite[§8.1]{Zhu:Tame} for the precise definition, as well as the symmetric monoidal structure on this category.
We also note that this category is different from the notion of motivic (or cohomological) correspondences that we will recall in §\ref{App:Corr}.
By requiring the map \(f\colon Z\to Y\) to lie in a class of morphisms \(E\) in \(\Cc\), we obtain a full subcategory \(\Corr(\Cc)_E\subseteq \Corr(\Cc)\).
Then the motivic six-functor formalism \(\DM\colon (\Sch_k^{\pfp})^{\op} \to \Pr_{\Lambda}^{\St}\) can be packaged into a lax symmetric monoidal functor
\[\DM\colon \Corr(\Sch_k^{\pfp}) \to \Pr_{\Lambda}^{\St},\]
sending a correspondence as in \eqref{geometric corr} to the functor \(f_*g^!\).
(Note that this is different from the usual convention, which would be \(f_!g^*\), but is in line with the formalism from \cite[§10.4]{Zhu:Tame}.
Moreover, the lax monoidality requires the compatibility of \(\boxtimes\) with !-pullback and *-pushforward, which holds under our assumptions by \cite[Propositions 2.3.5 and 2.1.20]{JinYang:Kunneth}.)
Let \(\AlgStk^{\pfp}\) be the category of perfect algebraic stacks, perfectly of finite presentation, and \(\Repr\) the class of morphisms of such stacks which are representable in schemes.
(For simplicity, in order to avoid having to worry about *-functors between general morphisms of algebraic spaces, we will assume the diagonals of all algebraic stacks are representable by schemes.
This holds for all quotient stacks, and all algebraic stacks that appear in this paper.)
Then \cite[Proposition 8.45]{Zhu:Tame} yields an extension
\[\DM\colon \Corr(\AlgStk^{\pfp})_{\Repr} \to \Pr_{\Lambda}^{\St},\]
which agrees with the previously constructed functor \(\DM\).
The resulting \((-)_*\) and \((-)^!\) clearly commute with colimits, and the same argument as in \cite[Proposition 2.3.3]{RicharzScholbach:Intersection} shows that they commute with limits.
Thus, we get left adjoints \((-)^*\) for representable morphisms, and \((-)_!\) in general.

Moreover, for an algebraic stack \(X\), the category \(\DM(X)\) admits a symmetric monoidal structure, making the *-pullback functors symmetric monoidal; we will denote the monoidal unit by \(\unit_X\), or even by \(\unit\).
Indeed, since we assumed the diagonal \(\Delta_X\) of \(X\) to be representable by schemes, the functor is given by \(-\otimes - := \Delta^*(-\boxtimes-)\), where the exterior product \(\boxtimes\) is given by the lax symmetric monoidal structure of \(\DM\).
This tensor product commutes with colimits in each variable, and hence \(\DM(X)\) is closed symmetric monoidal.
The internal Hom is denoted by
\[\IHom_X(-,-)\colon \DM(X)^{\op} \times \DM(X) \to \DM(X).\]

Let us also record various Künneth formulas for motives on algebraic stacks.

\begin{prop}\thlabel{Kunneth}
	Let \(f_i\colon X_i\to Y_i\), for \(i=1,2\), be representable morphisms of pfp algebraic stacks, and consider the commutative diagram
	\[\begin{tikzcd}
		X_1 \arrow[d, "f_1"] & X_1 \times_k X_2 \arrow[l] \arrow[r] \arrow[d, "f"] & X_2 \arrow[d, "f_2"]\\
		Y_1 & Y_1\times_k Y_2 \arrow[l] \arrow[r] & Y_2.
	\end{tikzcd}\]
	Then there are natural isomorphisms of functors
	\begin{equation}\label{Kunneth*up}f_1^*(-) \boxtimes f_2^*(-) \cong f^*(-\boxtimes -),\end{equation}
	\begin{equation}\label{Kunneth!down}f_{1,!}(-) \boxtimes f_{2,!}(-) \cong f_!(-\boxtimes-),\end{equation}
	\begin{equation}\label{Kunneth*down}f_{1,*}(-) \boxtimes f_{2,*}(-) \cong f_*(-\boxtimes -),\end{equation}
	\begin{equation}\label{Kunneth!up}f_1^!(-) \boxtimes f_2^!(-) \cong f^!(-\boxtimes -),\end{equation}
	and
	\begin{equation}\label{KunnethHom}\IHom_{X_1}(-,-)\boxtimes \IHom_{X_2}(-,-) \cong \IHom_{X_1\times_k X_2}(-\boxtimes-,-\boxtimes-).\end{equation}
\end{prop}
\begin{proof}
	Let us fix smooth atlases \(v_i\colon V_i\to Y_i\), and set \(u_i\colon U_i:=V_i \times_{Y_i} X_i \to X_i\).
	We also have the smooth atlases \(V:=V_1\times_k V_2\to Y_1\times_k Y_2\) and \(U:= U_1 \times_k U_2 \to X_1 \times_k X_2\).
	Then the functors \(f^*,f_*,f_!,f^!\) are compatible with pullback to the smooth atlas.
	When the \(X_i,Y_i\) are schemes, the proposition was shown in \cite[Theorem 1.2.2]{JinYang:Kunneth}.
	In general, the maps can be constructed in the same way, by using formal properties of the 6-functor formalism.
	To check that these maps are isomorphisms, we may pull back to the smooth atlases, which is conservative and compatible with the construction of the maps, which reduces to the case of schemes.
\end{proof}

A similar reduction to the case of schemes, combined with \cite[Theorem 2.4.50 (5)]{CisinskiDeglise:Triangulated}, yields the following.

\begin{prop}\thlabel{Exchange}
	Let \(f\colon X\to Y\) be a representable morphism of pfp algebraic stacks, and let \(M,N\in \DM(Y)\) and \(L\in \DM(X)\).
	Then there are natural isomorphisms
	\begin{equation}\label{exchange1}(f_!L) \otimes M \cong f_!(L\otimes f^*M),\end{equation}
	\begin{equation}\label{exchange2}\IHom_Y(f_!L,M) \cong f_* \IHom_X(L,f^!M),\end{equation}
	and
	\begin{equation}\label{exchagne3}f^!\IHom_Y(M,N) \cong \IHom_X(f^*M,f^!N).\end{equation}
\end{prop}

Recall that (under our assumptions on the base scheme \(k\)), for any pfp scheme \(X\), the category of compact objects in \(\DM(X)\) agrees with the category \(\DM_{\cons}(X)\) of \emph{constructible} objects (i.e., the thick stable subcategory generated by Tate twists of the motives of perfectly smooth \(X\)-schemes) \cite[Theorem 5.2.4]{CisinskiDeglise:Etale}.
Although compactness is a purely categorical notion, and compactness in \(\DM(X)\) for a pfp algebraic stack \(X\) cannot be checked after pullback to a smooth atlas, we can define a subcategory of \(\DM(X)\), which does satisfy descent, as well as certain finiteness properties.
Namely, since the six 6 functors for étale motives on pfp schemes preserve constructible objects \cite[Corollary 6.2.14]{CisinskiDeglise:Etale}, we obtain a six-functor formalism
\[\DM_{\cons}(-):=\DM_{\cons}(-,\Lambda)\colon (\Sch_k^{\pfp})^{\op} \to \Cat_{\infty,\Lambda}^{\perf}\colon X\mapsto \DM_{\cons}(X),\]
taking values in small idempotent-complete stably \(\Lambda\)-linear \(\infty\)-categories, with exact functors.
By applying the same Kan extension process as for \(\DM\), we can extend \(\DM_{\cons}\) to a functor
\begin{equation}\label{constructible motives on prestacks}\DM_{\cons}(-) := \DM_{\cons}(-,\Lambda) \colon (\PreStk_k^{\perf})^{\op}\to \Cat_{\infty,\Lambda}^{\perf} \colon X \mapsto \DM_{\cons}(X).\end{equation}
This yields, for any perfect prestack \(X\), a full subcategory \(\DM_{\cons}(X)\subseteq \DM(X)\), called the category of \emph{constructible motives} (but it does not agree with the category of compact objects in general).
Moreover, \eqref{kanlim} and \eqref{kancolim} still hold for \(\DM_{\cons}\), but we have to take colimits in \(\Cat_{\infty,\Lambda}^{\perf}\) instead of \(\Pr_{\Lambda}^{\St}\) (in the case of limits, this does not make a difference).
In particular, for a pfp algebraic stack \(X\), a motive \(M\in \DM(X)\) is constructible if and only if its !-pullback to some (equivalently, any) smooth atlas is constructible.
If \(X\) is a colimit of such pfp algebraic stacks \(X_i\) along closed immersions, then we have
\[\DM(X) \cong \varinjlim_i \DM(X_i),\]
with the (fully faithful) transition maps given by pushforward, and \(M\in \DM(X)\) is constructible exactly when it lies in the essential image of \(\DM_{\cons}(X_i) \subseteq \DM(X_i) \to \DM(X)\) for some \(X_i\).

\subsection{Recollections on motivic correspondences}\label{App:Corr}

Next, we review the formalism of motivic correspondences.
Most of the proofs are easy and will follow \cite[§A.2]{XiaoZhu:Cycles}, but we will work with a slightly different setup.
Namely, instead of working with algebraic stacks, with will allow a class of geometric objects which will also incorporate the moduli stacks of local shtukas from Section \ref{Sec:Motivic corr}.

\begin{ass}\thlabel{assumption corr}
	Whenever we are working with motivic correspondences, we will assume the following setup: \(c_X\colon C\to X\) and \(c_Y\colon C\to Y\) are morphisms of prestacks such that
	\begin{enumerate}
		\item\label{assumption 1} \(X =\varprojlim_{i\geq 0} X_{i}\), where each \(X_{i}\) is an algebraic stack over \(\Spec k\), such that the transition morphisms \(X_{i+1}\to X_i\) are either affine bundles, or the base change along a morphism \(*/H \to */H'\), where \(H\to H'\) is a surjection of algebraic \(k\)-groups with split unipotent kernel. 
		The same condition is required for \(Y\).
		\item\label{assumption 2} There exists a map \(c_{X,0}\colon C_0\to X_0\) which is representable by schemes, such that \(C_0 \times_{X_0} X \cong C\), and the induced map \(C \cong C_0\times_{X_0} X \to X\) agrees with \(c_X\).
		The same condition is required for \(c_Y\).
	\end{enumerate}
	In particular, we get a presentation \(C=\varprojlim_i C_0 \times_{X_0} X\), which satisfies condition (1).
	We get a similar presentation induced by \(c_{Y,0}\), but we emphasize that we do not require these two presentations to agree.
\end{ass}

Under the above assumptions, we can compute motives on \(C,X,Y\) as follows.

\begin{lem}\thlabel{DM as colimit}
	Let \(c_X\colon C\to X\) be as in \thref{assumption corr}.
	Then there is an equivalence
	\[\DM(C) \cong \varinjlim_i \DM(C_i),\]
	where the transition morphisms, given by !-pullback, are fully faithful.
	In particular, \(\DM(C)\) is a closed symmetric monoidal category.
	A similar equivalence holds for \(X\), compatibly with the (non monoidal) !-pullback along \(c_X\).
\end{lem}
\begin{proof}
	Since the maps \(C_j\to C_i\) are given either by an affine bundle, or by the base change along \(*/H \to */H'\), where \(H\onto H'\) has split unipotent kernel, the transition morphisms, given by !-pullback, are indeed fully faithful by homotopy invariance.
	To compute \(\DM(C)\), we may assume that all the transition maps \(C_{i+1}\to C_i\) are affine bundles. 
	Indeed, in the other case, \(C_{i+1}\) will be of the form \(C_i/U\), where \(U\) is some split unipotent group scheme acting trivially, so that !-pullback \(\DM(C_{i})\to \DM(C_{i+1})\) is an equivalence, and we may safely ignore it in what follows.
	Note that in the colimit \(\varinjlim_i \DM(C_i)\), the transition maps, given by !-pullback, can be replaced by *-pullback by purity.
	Since these *-pullbacks are symmetric monoidal, the sequential colimit admits a symmetric monoidal structure.
	It is even closed, since the tensor product commutes with colimits in both variables.
	On the other hand, passing to right adjoints, the colimit may even be computed as \(\varprojlim_i \DM(C_i)\), where the transition maps are given by *-pushforward.
	
	Now, choose an affine smooth atlas \(V_0\to C_0\); pullback to \(C_i\) gives an affine smooth atlas \(V_i:= V_0 \times_{C_0} C_i\to C_i\); we also set \(V:= V_0\times_{C_0} C\), which is still affine.
	In that case we do have \(\DM(V) \cong \varprojlim \DM(V_i)\), with the transition morphisms given by *-pushforward.
	Indeed, this follows from \eqref{kancolim}, since any map \(V \to T'\), with \(T'\) a perfect affine scheme, perfectly of finite type, will factor through some \(V_i\).
	Applying the same argument to the \v{C}ech nerves of \(V_i\to C_i\), and using that the limit along *-pushforward commute with the limit along the !-pullback in the descent equivalence \cite[Example 2.2.19]{RicharzScholbach:Intersection} by base change, we see that \(\DM(C) \cong \varinjlim_i \DM(C_i)\).
	
	The final assertion is clear.
\end{proof}

Consequently, we can define
\[c_X^*\colon \DM(X)\to \DM(C)\]
termwise using this colimit, since *-pullback \(\DM(X_i)\to \DM(C_i)\) commutes with the transition maps (given by !-pullback) by purity.
Similarly, we can define \(c_{X,*},c_{X,!}\colon \DM(C)\to \DM(X)\) using this colimit, and we get adjunctions \((c_X^*,c_{X,*})\) and \((c_{X,!},c_X^!)\).

Let us now define motivic correspondences.
This notion is classical, going back to \cite{SGA5} for \(\ell\)-adic cohomology, and has been studied for motives in e.g.~\cite[Definition 3.4.2.1]{Cisinski:Cohomological}.

\begin{dfn}\thlabel{defi motivic corr}
	Let \((c_X, c_Y) \colon C\to X\times Y\) satisfy \thref{assumption corr}, and let \(M\in \DM(X)\) and \(N\in \DM(Y)\).
	A \emph{motivic correspondence from \((X,M)\) to \((Y,N)\) supported on \(C\)} is a morphism
	\[\alpha\colon c_X^*M \to c_Y^!N\]
	in \(\DM(C)\).
	We will sometimes also write \(\alpha\colon (X,M)\to (Y,N)\), and denote by \(\Corr_C(M,N)\) the set of isomorphism classes of correspondences between \((X,M)\) and \((Y,N)\) which are supported on \(C\) (for the natural notion of isomorphisms between motivic correspondences).
\end{dfn}

In particular, although \(\Corr_C(M,N)\) admits the structure of an anima, it will suffice to consider the underlying set for the purposes of this paper.

By \thref{DM as colimit}, \(\DM(C)\) admits an internal Hom.
Since any prestack admits a dualizing motive \(\omega\), we can define the Verdier duality functor
\[\ID:= \IHom_C(-,\omega_C)\colon \DM(C)^{\op} \to \DM(C).\]
When restricted to constructible objects \(M\), Verdier duality interchanges *- and !-pullback, as well as *- and !-pushforward.
Moreover, \(\ID\) induces an involution on constructible objects.
Both these claims can be checked after pullback to a smooth atlas of some \(C_i\), where they follow from \cite[Corollary 6.3.15]{CisinskiDeglise:Etale} and \cite[Theorem 2.4.50 (5)]{CisinskiDeglise:Triangulated}.

\begin{ex}\thlabel{examples of correspondences}
	Here are some examples of motivic correspondences, keeping the notation from \thref{assumption corr}.
	\begin{enumerate}
		\item For any \(M\in \DM(C)\), there is the \emph{pushforward correspondence}
		\[(\Gamma_{c_Y})_!\colon M\to c_Y^!c_{Y,!}M\in \Corr_C(M,c_{Y,!}M),\]
		arising from the adjunction \((c_{Y,!},c_Y^!)\).
		\item\label{pullback correspondence} For any \(N\in \DM(C)\), there is the \emph{pullback correspondence}
		\[\Gamma_{c_Y}^*\colon c_Y^*N\to c_Y^*N\in \Corr_C(N,c_Y^*N),\]
		given by the identity morphism.
		\item\label{dual correspondence} Let \(M\in \DM_{\cons}(X)\), and \(N\in \DM_{\cons}(Y)\).
		Then for any correspondence \(\alpha\in \Corr_C(X,Y)\), there is a natural \emph{dual correspondence} \(\ID(\alpha)\in \Corr_C(\ID(N),\ID(M))\).
		\item\label{co/unit corr} Assume now that \(X\) is a pfp scheme, and let \(M\in \DM(X)\).
		Then the unit map \(\unit \to \Delta^!(M \boxtimes \ID(M))\) induces a correspondence 
		\[\delta_M\in \Corr_X\left((\Spec k, \unit),(X \times X,M \boxtimes \ID(M))\right).\]
		Dually, the counit \(M \otimes \ID(M) \to \omega_X\) induces a correspondence 
		\[e_M\in \Corr_X\left((X\times X,M\boxtimes \ID(M)),(\Spec k,\unit) \right).\]
	\end{enumerate}
\end{ex}

\begin{constr}\thlabel{composition of correspondences}
	Motivic correspondences can be composed.
	Indeed, let \((C,c_X,c_Y,\alpha)\in \Corr_C(M,N)\) and \((D,d_Y,d_Z,\beta)\in \Corr_D(N,L)\) be correspondences, for \(M\in \DM(X)\), \(N\in \DM(Y)\), and \(L\in \DM(Z)\).
	Then we define a correspondence
	\[\beta\circ \alpha \in \Corr_{C\times_Y D}(M,L)\] via the composition
	\[p^*c_X^*M \xrightarrow{p^*\alpha} p^*c_Y^!N \to q^!d_Y^*N \xrightarrow{q^!\beta} q^!d_Z^!L,\]
	where \(p\colon C\times_Y D\to C\) and \(q\colon C\times_Y D\to D\) are the projection maps, and the middle morphism is adjoint to the base change isomorphism.
	More precisely, we may restrict to some \(Y_i\), consider the induced stacks \(C_i\) and \(D_i\), as well as \(C_i \times_{Y_i} D_i\).
	Since these are all pfp algebraic stacks, we indeed have the desired base change isomorphism by the six functor formalism.
	
	See also \cite[Definition 3.4.2.6]{Cisinski:Cohomological} for an alternative, but equivalent, formulation.
\end{constr}

The following lemma is easy, but we give a proof to illustrate why we do not require the two presentations of \(C\) in \thref{assumption corr} to agree.
Similar arguments will work for the constructions below (which will not be given in detail), and this will simplify discussions related to moduli of restricted local shtukas in the main body of the paper.

\begin{lem}
	The composition of correspondences, defined in \thref{composition of correspondences}, is associative.
\end{lem}
\begin{proof}
	Consider the diagram
	\[\begin{tikzcd}
		&&& C\times_Y D \times_Z E \arrow[ld, "p_1"'] \arrow[rd, "p_2"] \arrow[dd, phantom, "\clubsuit"] &&&\\
		&& C\times_Y D \arrow[ld] \arrow[rd, "q_1"'] && D \times_Z E \arrow[ld, "q_2"] \arrow[rd] &&\\
		& C \arrow[ld] \arrow[rd] && D \arrow[ld] \arrow[rd] && E \arrow[ld] \arrow[rd]&\\
		X && Y && Z && W
	\end{tikzcd}\]
	with cartesian squares, and let \(\alpha\in \Corr_C(X,Y)\), \(\beta\in \Corr_D(Y,Z)\) and \(\gamma\in \Corr_E(Z,W)\) be motivic correspondences.
	The equality \((\gamma\circ \beta) \circ \alpha=\gamma\circ (\beta\circ \alpha)\in \Corr_{C\times_Y D \times_Z E}(X,W)\) can be constructed as usual via base change morphisms; the only slightly subtle point involves the square marked ``\(\clubsuit\)''.
	There, it suffices to construct a natural transformation \(p_1^*\circ q_1^! \to p_2^! \circ q_2^*\).
	Consider \(D_0\), as well as \((C \times_Y D)_0\) and \((D\times_Z E)_0\) as in \thref{assumption corr}, and let \((C \times_Y D \times_Z E)_0\) be the obvious fiber product.
	Let us denote by \((p_1)_0,(p_2)_0,(q_1)_0,(q_2)_0\) the obvious maps between these algebraic stacks.
	Then we can construct a natural transformation of functors \((p_1)_0^*(q_1)_0^! \to (p_2^!)_0 (q_2)_0^*\), since all objects are usual (perfect) algebraic stacks.
	Since the relevant functors commute with smooth !-pullback by purity, we can conclude by passing to the colimit, as in \thref{DM as colimit}.
\end{proof}

Here is another way to produce motivic correspondences.

\begin{lem}\thlabel{products and correspondences}
	Let \((c_X,c_Y)\colon C\to X\times_k Y\) be as in \thref{assumption corr}, and let \(M\in \DM(X)\) and \(N\in \DM_{\cons}(Y)\).
	Then there is a natural bijection
	\[\Corr_C\left((X,M),(Y,N)\right) \cong \Corr_C\left((X\times_k Y,M \otimes \ID(N)),(\Spec k,\unit)\right)\]
	\[\alpha \leftrightarrow \alpha^\sharp\]
\end{lem}
\begin{proof}
	The bijection is given by sending \(\alpha\) to the composition
	\[(c_X\times c_Y)^*(M \boxtimes \ID(N)) \cong c_X^*M \otimes c_Y^*\ID(N) \xrightarrow{\alpha \otimes \identity} c_Y^!M \otimes \ID (c_Y^!N) \to \omega_C.\]
\end{proof}

\begin{constr}\thlabel{construction pushforward correspondence}
	Consider a commutative diagram
	\begin{equation}\label{diagram pushforward correspondence}\begin{tikzcd}
		X \arrow[d, "f_X"'] & C \arrow[l, "c_X"'] \arrow[r, "c_Y"] \arrow[d, "f"] & Y \arrow[d, "f_Y"]\\
		X' & D \arrow[l, "d_X"] \arrow[r, "d_Y"'] & Y',
	\end{tikzcd}\end{equation}
	where \((c_X,c_Y)\colon C\to X\times_k Y\) and \((d_X,d_Y)\colon D\to X'\times_k Y'\) satisfy \thref{assumption corr}, and the natural map \(C\to X \times_{X'} D\) is representable by perfectly proper schemes.
	Then for \(M\in \DM(X)\), \(N\in \DM(Y)\) and \(\alpha\in \Corr_C(M,N)\) we can define the \emph{pushforward correspondence} \(f_!\alpha\in \Corr_D(f_{X,!}M,f_{Y,!}N)\) as the composition
	\[d_X^*f_{X,!}M \to f_!c_X^*M \xrightarrow{f_!\alpha} f_! c_Y^!N\to d_Y^!f_{Y,!} N,\]
	where the first and last arrows are constructed via the natural adjunctions (cf.~\cite[(A.1.3) and (A.1.4)]{XiaoZhu:Cycles}).
\end{constr}

\begin{lem}\thlabel{compatibility pushforward correspondence}
	Pushforward of motivic correspondences is compatible with both vertical and horizontal composition of diagrams \eqref{diagram pushforward correspondence}.
\end{lem}
\begin{proof}
The proofs of \cite[Lemma A.2.8]{XiaoZhu:Cycles} and \cite[Lemma A.2.10]{XiaoZhu:Cycles} carry over to our setting.
\end{proof}

Next, we introduce the pro-smooth pullback of correspondences, where we will slightly deviate from \cite[A.2.11]{XiaoZhu:Cycles}.

\begin{lem}\thlabel{purity for pro-smooth morphisms}
	Let 
	\[\begin{tikzcd}
		C' \arrow[d, "c_X'"'] \arrow[r, "f'"] & C \arrow[d, "c_X"]\\
		X' \arrow[r, "f"'] & X
	\end{tikzcd}\]
	be a cartesian diagram of perfect prestacks, where \(c_X\colon C\cong \varprojlim_i C_i\to X\cong \varprojlim_i X_i\) and \(c_X'\) satisfy \thref{assumption corr}, and \(f\cong \varprojlim_i f_i\) is a limit of perfectly smooth schematic morphisms \(f_i\colon X_i'\to X_i\).
	Then relative purity induces a natural equivalence of functors
	\begin{equation}\label{pro-purity}(f')^!c_X^*\cong (c_X')^*f^!.\end{equation}
\end{lem}
For example, \(f\) could be a torsor under a pro-algebraic group, such as the positive loop group of a parahoric group scheme.
\begin{proof}
	Recall from \thref{DM as colimit} that \(\DM(X)\cong\varinjlim_i \DM(X_i)\) etc.
	For each \(i\), let \(f_i'\colon C_i'\to C_i\) be the base change of \(f_i\); then relative purity yields an equivalence 
	\[(f_i')^!c_X^*\cong (c_X')^*f_i^!.\]
	This implies the lemma by passing to the colimit.
\end{proof}

\begin{constr}\thlabel{construction pullback correspondence}
	Consider a commutative diagram
	\begin{equation}\label{diagram pullback correspondence}\begin{tikzcd}
		X' \arrow[d, "f_X"'] & D \arrow[l, "d_X"'] \arrow[r, "d_Y"] \arrow[d, "f"] & Y' \arrow[d, "f_Y"]\\
		X & C \arrow[l, "c_X"] \arrow[r, "c_Y"'] & Y,
	\end{tikzcd}\end{equation}
	where \((c_X,c_Y)\colon C\to X\times_k Y\) and \((d_X,d_Y)\colon D\to X'\times_k Y'\) satisfy \thref{assumption corr}, the left square is cartesian, and \(f_X\) is a limit of perfectly smooth schematic maps as in \thref{purity for pro-smooth morphisms}.
	Let \(M\in \DM(X)\), \(N\in \DM(Y)\) and \(\alpha\in \Corr_{C}(M,N)\).
	Then we define the \emph{pro-smooth pullback correspondence} \(f^!\alpha\in \Corr_D(f_X^!M,f_Y^!N)\) as the composition
	\[d_X^*f_X^!M \xrightarrow{\eqref{pro-purity}} f^!c_X^*M\xrightarrow{f^!\alpha} f^! c_Y^!N \cong d_Y^!f_Y^!N \in \Corr_D(f_X^!M,f_Y^!N).\]
\end{constr}

Again, this construction satisfies natural compatibilities.

\begin{lem}\thlabel{compatibility pullback correspondence}
	Pullback of motivic correspondences is compatible with both vertical and horizontal composition of diagrams \eqref{diagram pullback correspondence}.
\end{lem}
\begin{proof}
	This can be proven as \cite[Lemma A.2.14]{XiaoZhu:Cycles} and \cite[Lemma A.2.15]{XiaoZhu:Cycles} respectively.
\end{proof}

We will also need the following compatibility between proper pushforward and pro-smooth pullback of motivic correspondences.

\begin{lem}\thlabel{Compatibility pullback and pushforward correspondences}
	Consider a diagram
	\[\begin{tikzcd}
		 &X_2' \arrow[dd, "f_X'"' {yshift=10pt}] \arrow[dl, "g_X'"'] && D_2 \arrow[ll, "d_{X{,}2}"'] \arrow[rr, "d_{Y{,}2}"] \arrow[dd, "f'"' {yshift=10pt}] \arrow[dl, "g'"'] && Y_2' \arrow[dd, "f_Y'" {yshift=10pt}] \arrow[dl, "g_Y'"] \\
		X_1' \arrow[dd, crossing over, "f_X"' {yshift=10pt}] && D_1 \arrow[ll, crossing over, "d_{X,1}"' {xshift=12pt}] \arrow[rr, crossing over, "d_{Y,1}" {xshift=12pt}] && Y_1'&\\
		&X_2 \arrow[dl, "g_X"'] && C_2 \arrow[ll, "c_{X{,}2}" {xshift=-12pt}] \arrow[rr, "c_{Y{,}2}"' {xshift=-12pt}] \arrow[dl, "g"] && Y_2 \arrow[dl, "g_Y"] \\
		X_1 && C_1 \arrow[from=uu, crossing over, "f"' {yshift=10pt}] \arrow[ll, "c_{X,1}"] \arrow[rr, "c_{Y,1}"'] && Y_1 \arrow[from=uu, crossing over, "f_Y" {yshift=10pt}] &
	\end{tikzcd}\]
	where the horizontal rectangles are in \thref{construction pullback correspondence}, the vertical rectangles are as in \thref{construction pushforward correspondence}, the square with vertices \(D_2,D_1,C_2,C_1\) is cartesian, and the maps \(f_X,f_X',f,f'\) are additionally representable by perfectly proper schemes.
	Let \(M\in \DM(X_1')\) and \(N\in \DM(Y_1')\), and let \(\alpha\in \Corr_{D_1}(M,N)\) be a motivic correspondence.
	Then there is a natural commutative diagram of correspondences
	\[\begin{tikzcd}[column sep=large]
		(X_2,g_{X}^!f_{X,!}M) \arrow[r, "g^!f_!(\alpha)"] \arrow[d] & (Y_2,g_{X}^!f_{X,!}N) \\
		(X_2,f_{X,!}' (g_{X}')^!M) \arrow[r, "f'_!(g')^!(\alpha)"'] & (Y_2,f_{X,!}'(g_X')^!N) \arrow[u],
	\end{tikzcd}\]
	where the base change morphisms \(g_X',f_{X,!}M\to f_{X,!}'(g_X')^!M\) and \(f_{X,!}'(g_X')^!N \to g_X^!f_{X,!}N\) are the ones used in \thref{construction pushforward correspondence}.
\end{lem}
\begin{proof}
	The usual argument reduces us to the case of algebraic stacks, where it follows from usual properties of the six-functor formalism, as in \cite[Lemma A.2.17]{XiaoZhu:Cycles}.
\end{proof}

Finally, recall the following form of the Grothendieck--Lefschetz trace formula.

\begin{prop}\thlabel{motivic Grothendieck-Lefschetz}
	Let \(X/\IF_q\) be a finite type scheme, \(\Ff\in \DM(X,\IQ)\) a locally constructible motive (in the sense of \cite[Definition 3.3.1.12]{Cisinski:Cohomological}), and consider their pullback to \(\overline{\IF}_q\), denoted \((\overline{X},\overline{\Ff})\).
	Then the motivic correspondence
	\[\delta_{\overline{\Ff}} \circ (\Gamma_{\sigma_q}^*)^\sharp \in \Corr_{\overline{X}}\left((\overline{X}\times \overline{X},\overline{\Ff} \boxtimes \ID(\sigma_q^*\overline{\Ff})),(\overline{X} \times \overline{X},\overline{\Ff} \boxtimes \ID(\overline{\Ff}))\right)\]
	is supported on \(X(\IF_q)\subset \overline{X}\), and the value at a point \(x\in X(\IF_q)\) (with geometric point \(\overline{x}\)) is given by the trace of the geometric \(q\)-Frobenius on the stalk \(\Ff_{\overline{x}}\).
\end{prop}
\begin{proof}
	This is a restatement of \cite[Theorem 3.4.2.25]{Cisinski:Cohomological}.
	Indeed, by \cite[Example 3.4.2.12]{Cisinski:Cohomological}, the object denoted \(\operatorname{Tr}\) in loc.~cit.~is viewed as a 0-cycle on the Frobenius-fixed point locus of \(\overline{X}\), which is exactly \(X(\IF_q)\).
\end{proof}

\bibliographystyle{alphaurl}
\bibliography{bib}
	
\end{document}